\documentclass[12pt, reqno]{amsart}
\usepackage{mathrsfs}
\usepackage{amsfonts}
\usepackage[centertags]{amsmath}

\usepackage{amssymb}
\usepackage{amsthm}
\usepackage{graphicx}
\usepackage{caption}
\usepackage{dynkin-diagrams}
\usepackage[colorlinks]{hyperref}
\hypersetup{
bookmarksnumbered,
pdfstartview={FitH},
breaklinks=true,
linkcolor=blue,
urlcolor=blue,
citecolor=blue,
bookmarksdepth=2
}
\usepackage[sorted,compressed-cites,sorted-cites,initials]{amsrefs}
\usepackage{enumitem}
\usepackage[margin=1in]{geometry}
\usepackage{tikz-cd}
\usepackage{rotating}
\usetikzlibrary{matrix,arrows,decorations.pathmorphing}
\usepackage{tikz-cd}
\usetikzlibrary{calc}
\usetikzlibrary{decorations.pathmorphing}

\tikzset{curve/.style={settings={#1},to path={(\tikztostart)
    .. controls ($(\tikztostart)!\pv{pos}!(\tikztotarget)!\pv{height}!270:(\tikztotarget)$)
    and ($(\tikztostart)!1-\pv{pos}!(\tikztotarget)!\pv{height}!270:(\tikztotarget)$)
    .. (\tikztotarget)\tikztonodes}},
    settings/.code={\tikzset{quiver/.cd,#1}
        \def\pv##1{\pgfkeysvalueof{/tikz/quiver/##1}}},
    quiver/.cd,pos/.initial=0.35,height/.initial=0}

\tikzset{tail reversed/.code={\pgfsetarrowsstart{tikzcd to}}}
\tikzset{2tail/.code={\pgfsetarrowsstart{Implies[reversed]}}}
\tikzset{2tail reversed/.code={\pgfsetarrowsstart{Implies}}}
\tikzset{no body/.style={/tikz/dash pattern=on 0 off 1mm}}

\usepackage{listing}
\usepackage{color}
\usepackage{float}

\DeclareMathOperator{\supp}{supp}
\newcommand{\brd}[2]{\underbrace{#1\cdots}_{#2}}

\newcommand{\CC}{{\mathbb C}}

\newcommand{\OEIS}[1]{\href{https:oeis.org/#1}{#1}}
\newcommand{\mP}{\mathsf P}
\newcommand{\la}{\langle}
\newcommand{\ra}{\rangle}
\newcommand{\wh}[1]{\widehat{#1}}

\newcommand{\ascprod}{\mathop{\overrightarrow{\prod}}}
\newcommand{\dscprod}{\mathop{\overleftarrow{\prod}}}
\newcommand{\ascprodst}{\mathop{\overrightarrow{\prod}^\times}}
\newcommand{\dscprodst}{\mathop{\overleftarrow{\prod}^\times}}
\newcommand{\prodst}{\mathop{{\prod}^\times}}

\usepackage{xstring}
\makeatletter
\newcommand{\partref}[1]{\StrCut{#1}{.}\@pref@\@suff@\ref{\@pref@}\StrLen{\@suff@}[\@slen@]\ifnum\@slen@>0\ref{#1}\fi}
\makeatother

\allowdisplaybreaks[4]
\usetikzlibrary{matrix,arrows,decorations.pathmorphing}

\newtheorem{theorem}{Theorem}[section]
\newtheorem{corollary}[theorem]{Corollary}
\newtheorem{lemma}[theorem]{Lemma}
\newtheorem{proposition}[theorem]{Proposition}
\theoremstyle{definition}
\newtheorem{definition}[theorem]{Definition}
\newtheorem{remark}[theorem]{Remark}
\newtheorem{example}[theorem]{Example}

\newtheorem{problem}{Problem}

\newtheorem{conjecture}[theorem]{Conjecture}

\newcommand{\cx}[2]{c_{#1\to #2}}

\newcommand{\cxr}[2]{c_{#1\leftarrow #2}}

\newcommand{\Cx}[2]{C_{#1\to #2}}
\newcommand{\Cxr}[2]{C_{#1\leftarrow #2}}

\newenvironment{enmalph}{\begin{enumerate}[label={\rm(\alph*)},leftmargin=*]}{\end{enumerate}}

\SetLabelAlign{center}{\hfil#1\hfil}
\newenvironment{enmroman}{\begin{enumerate}[label={\rm(\roman*)},align=center,leftmargin=*]}{\end{enumerate}}

\numberwithin{equation}{section}
\newcommand{\Cox}[1]{\mathbf{Cox}(#1)}
\newcommand{\tensor}{\otimes}

\newcommand{\kk}{\Bbbk}

\newcommand{\ZZ}{{\mathbb Z}}

\DeclareMathOperator{\id}{id}

\DeclareMathOperator{\End}{End}

\DeclareMathOperator{\Br}{Br}

\DeclareMathOperator{\Hom}{Hom}

\DeclareMathOperator{\Inv}{Inv}

\DeclareMathOperator{\SQF}{SQF}
\makeatletter
\renewcommand*{\Hdynkin}{\Adynkin%
\dynkinEdgeLabel{\numexpr\dynkin@rank-1}{\dynkin@rank}{5}%
}
\newcommand{\plink}[1]{\hypertarget{#1}{}\label{page:#1}}
\makeatother
\pgfkeys{/Dynkin diagram,
edge length=1.0cm,text style/.style={scale=0.8},Coxeter,root radius=0.07}

\begin{document}

\title{Hecke and Artin monoids and their homomorphisms}
\author{Arkady Berenstein, Jacob Greenstein and Jian-Rong Li}
\address{Arkady Berenstein, Department of Mathematics, University of Oregon, Eugene, OR 97403, USA}
\email{arkadiy@math.uoregon.edu}
\address{Jacob Greenstein, Department of Mathematics, University of California, Riverside, CA 92521, USA}
\email{jacobg@ucr.edu}
\address{Jian-Rong Li, Faculty of Mathematics, University of Vienna, Oskar-Morgenstern-Platz 1, 1090 Vienna, Austria}
\email{lijr07@gmail.com}
\date{}

\thanks{This work was partially supported by the Simons Foundation Collaboration Grant no.~636972 (A.~Berenstein), the Simons
foundation collaboration grant no.~245735 (J.~Greenstein), Austrian Science Fund (FWF): P 34602, Grant DOI: 10.55776/P34602, and PAT 9039323, Grant-DOI 10.55776/PAT9039323 (J.-R. Li).
}

\begin{abstract}
The aim of the present work is to systematically study
homomorphisms of Hecke and Artin monoids and thus to develop their comprehensive theory. 
Our original motivation was the striking observation that parabolic projections of Hecke monoids respect all parabolic elements. We found other classes of
 homomorphisms of Hecke monoids with the same property and discovered that many of them lift to
 homomorphisms of covering Artin monoids with a similar property. It turned out that they belong to a much larger class (in fact, a category) of homomorphisms of Artin monoids, most of which appear to be new.
\end{abstract}

\maketitle

\tableofcontents

\section{Introduction and main results}
The aim of the present work is to systematically study
homomorphisms of Hecke and Artin monoids and thus to develop their comprehensive theory. 

We start with {\em parabolic} homomorphisms of Hecke monoids, also referred to as 0-Hecke monoids, Demazure monoids and even Coxeter monoids in the literature. Our interest in them was motivated by the following surprising
observation which emerged in the framework
of geometric cactus group actions (\cites{BGL,BGLa}).
Let $M$ be a Coxeter matrix over a finite set~$I$ and let $W(M)=\la s_i\,:\, i\in I\ra$ 
(respectively, $(W(M),\star)$) be the corresponding Coxeter group 
(respectively, the Hecke
monoid generated by simple idempotents also denoted by $s_i$, $i\in I$, see~\S\S\ref{subs:Br(M)W(M)}, \ref{subs:Hecke} for details).
Given~$J\subset I$, let $W_J(M)=\la s_j\,:\, j\in J\ra$ be the {\em parabolic} subgroup of~$W(M)$
corresponding to~$J$ (\S\ref{subs:parab}). 
It is well-known that the assignments
$s_j\mapsto s_j$, $j\in J$ and $s_j\mapsto 1$, $j\in I\setminus J$, define a homomorphism of monoids $p_J:(W(M),\star)\to 
(W_J(M),\star)$, which we refer to as a {\em parabolic projection}.

More generally, we say that a homomorphism $\phi:(W(M'),\star)\to 
(W(M),\star)$ is {\em light}
if the image of any generator of~$W(M')$ is either 
a generator of~$W(M)$ or~$1$
(Definition~\ref{defn:light}). Clearly,  all
light homomorphisms form a category
(see~\S\ref{subs:parab proj} for its detailed description), and  parabolic projections form its subcategory.
For any~$K\subset I$
such that~$W_K(M)$ is finite, there is a distinguished
family of elements $w_{J;K}\in W_K(M)$, $J\subset K$ called {\em parabolic} (cf.~\S\ref{subs:parab elts}). Our first main result is
the following
\begin{theorem}[Theorem~\ref{thm:light->parab}]\label{thm:main thm 1}
A light homomorphism of Hecke monoids maps all parabolic elements to parabolic elements.
\end{theorem}
While it is easy to see that longest
parabolic elements
$w_{\emptyset;K'}$ for~$K'\subset I'$ such that~$W_{K'}(M')$ is finite
are mapped by
a light homomorphism~$\phi:(W(M'),\star)\to 
(W(M),\star)$ to $w_{\emptyset;K}$ for some~$K\subset I$, the equation
$w_{\emptyset;K}=w_{\emptyset;J}\star x$
obtained by
applying~$\phi$ to the identity $w_{\emptyset;K'}=
w_{\emptyset;J'}\star w_{J';K'}$
has too many solutions~$x\in W_{K}(M)$, which makes the above result highly non-trivial.

More generally, in~\S\ref{subs:multiparab} we introduce multiparabolic elements. We expect that
Theorem~\ref{thm:main thm 1} holds for them as well (see Conjecture~\ref{conj:pJ multiparab}). They
also have some interesting combinatorial properties.
For instance, the generating function of their number depends only
on the shape of the Coxeter graph of~$M$  (Theorem~\ref{thm:count multipar}) and
can be expressed in terms of Chebyshev polynomials
(Proposition~\ref{prop:Chebyshev}). 

It turns out that for every finite~$W(M)$ the set $\mP(W(M),\star):=
\{ w_{K;I}\,:\, K\subset I\}$
 is an abelian submonoid of~$(W(M),\star)$. For crystallographic
 $W(M)$ this
 was established in~\cite{BK07}*{Proposition~2.30},
 while remaining cases were treated in~\cite{He09}
 (see also Proposition~\ref{prop:submonoid *}).
 Thus,
Theorem~\ref{thm:main thm 1} implies that
a light homomorphism~$\phi:(W(M'),\star)\to
(W(M),\star)$ restricts to a homomorphism
of monoids $\mP(W(M'),\star)\to \mP(W_J(M),\star)$
where~$J=\supp\phi(W(M'))\subset I$ (see~\S\ref{subs:Br(M)W(M)}).
The set of multiparabolic elements is almost a submonoid, in the sense that it becomes one after
adjoining (very few of) their pairwise products 
(Conjecture~\ref{conj:multiparab submon} and Example~\ref{ex:multipar submon}).

In view of this discussion and Theorem~\ref{thm:main thm 1}, we say that
a homomorphism $\phi:(W(M'),\star)\to (W(M),\star)$
of finite Hecke monoids is {\em parabolic}
if it restricts to a homomorphism of monoids $\mP(W'(M),\star)\to
\mP(W(M),\star)$. So, it is natural to pose the following
\begin{problem}\label{prob:1.2}
Classify parabolic homomorphisms of finite Hecke monoids (clearly, such homomorphisms form a category).
\end{problem}
The problem is open so far.
In addition to light ones we found several other families of parabolic homomorphisms.
\begin{example}\label{ex:1.3}
The assignments $s'_i\mapsto s_i$, $1\le i\le n-1$,
$s'_n\mapsto s_n s_{n+1} s_n=s_{n+1}s_n s_{n+1}$
define a parabolic homomorphism of Hecke monoids
$\phi:(W',\star)\to (W,\star)$ where~$W'$ is of type~$B_n$ and~$W$ is of type~$A_{n+1}$
(in fact, $\phi=p_{[1,n+1]}\circ
\phi'$ where $\phi':(W(B_n),\star)\to (W(A_{2n}),\star)$
is the standard unfolding; see~\S\ref{subs:parab hom Artin} for details).
Likewise, each of the assignments
\begin{alignat*}{3}
&s'_1\mapsto s_1,& &s'_2\mapsto s_2 s_4,&
&s'_3\mapsto s_3 s_4 s_3,\\
&s'_1\mapsto s_1 s_4,&\quad &s'_2\mapsto s_2,&\quad
&s'_3\mapsto s_3 s_4 s_3,\\
&s'_1\mapsto s_1 s_4,&
&s'_2\mapsto s_1 s_3,&
&s'_2\mapsto s_1 s_2 s_1
 \end{alignat*}
 defines a parabolic homomorphism~$(W(B_3),\star)
 \to (W(A_4),\star)$. In fact, this list, together
 with the generic family described above, exhausts
 all such homomorphisms, up to the diagram automorphism of~$W(A_4)$.
\end{example}

Quite unexpectedly, as a byproduct of our approach to Problem~\ref{prob:1.2}, the 
classification of light homomorphisms, and Example~\ref{ex:1.3}, it is possible to describe, to some extent, {\em all} homomorphisms of Hecke monoids. To begin with,
since the generators $s_i$, $i\in I$ of~$(W(M),\star)$
are idempotents, to find
all homomorphisms of Hecke monoids it is necessary to find all idempotents in Hecke monoids.
One can show (see~\S\ref{subs:Hecke} and e.g.~\cite{K14}) that an element of~$(W(M),\star)$
is an idempotent if and only if it is the longest element~$w_\circ^J=w_{\emptyset;J}$
in a finite parabolic Hecke submonoid
$(W_J(M),\star)$ (and in the Coxeter subgroup~$W_J(M)$, see~\S\ref{subs:w0J}).
\begin{theorem}[Theorem~\ref{thm:Hom Heck Mon}]\label{thm:1.3}
Let $M=(m_{ij})_{i,j\in I}$, $M'=(m'_{ij})_{i,j\in I'}$ be
Coxeter matrices and let 
$(W(M),\star)$, $(W(M'),\star)=
\la s'_i: i\in I'\ra$ be the respective Hecke monoids.
\begin{enmalph}
    \item Given a family $\boldsymbol K=\{K_i\,:\,i\in I'\}$
    of subsets of~$I$ such that each~$W_{K_i}(M)$ is finite,
    the assignments $s'_i\mapsto w_\circ^{K_i}$, $i\in I'$,
    define a homomorphism $\phi_{\boldsymbol K}:(W(M'),\star)\to (W(M),\star)$
    if and only if  for all~$i\not=j\in I'$ either~$m'_{ij}=
    \infty$ or
    $W_{K_i\cup K_j}(M)$ is finite
    and
    $$m'_{ij}\ge \max(\mu_M(
    K_i,K_j),\mu_M(K_j,K_i)),$$
    where
    $\mu_M(K_i,K_j)$ is the minimal positive integer~$m$
    such that $\brd{w_\circ^{K_i}\star w_\circ^{K_j}\star}{m}=w_\circ^{K_i\cup K_j}$.
    \item All homomorphisms of Hecke monoids
    are of this form.
\end{enmalph}
\end{theorem}
This provides the classification of homomorphisms of
Hecke monoids (for example, $p_J=\phi_{\boldsymbol K(J)}$
where~$K(J)_i=\{i\}$ if~$i\in J$ and~$K(J)_i=\emptyset$
otherwise)
and will enhance representation theory of Hecke monoids developed in~\cite{Cart}. 
However, this classification does not immediately solve Problem~\ref{prob:1.2}, as illustrated by the following example.
\begin{example}\label{ex:1.5}
The assignments
$$
s'_i\mapsto s_i,\, 1\le i\le n-1,\quad
s'_n\mapsto w_\circ^{[n,m]}
$$
$m\ge n+2$
define non-parabolic homomorphisms~$(W(B_n),\star)\to (W(A_m),\star)$.
Likewise, each of the assignments
\begin{alignat*}{3}
&s'_1\mapsto s_1,&&s'_2\mapsto w_\circ^{\{2,4\}},
&&s'_3\mapsto w_\circ^{\{3,4,5\}},\\
&s'_1\mapsto s_1,&\quad &s'_2\mapsto w_\circ^{\{2,4,5\}},
&\quad &s'_3\mapsto w_\circ^{\{3,4\}},
\end{alignat*}
defines a non-parabolic homomorphism
$(W(B_3),\star)\to (W(A_5),\star)$.
\end{example}

It turns out that many solutions of Problem~\ref{prob:1.2} originate from homomorphisms of {\em Artin monoids}. Recall 
(see~\S\ref{subs:Br(M)W(M)})  that
the Artin monoid~$\Br^+(M)$ associated with~$M$ is, in the presentation of Tits (\cite{Tits}), generated by
the~$T_w$, $w\in W(M)$ subject
to appropriate relations (see Theorem~\ref{thm:Tits}).
Analogously to the above, we call
elements $T_{w_{J;I}}$, $J\subset I$
(assuming that~$W(M)$ is finite)
{\em parabolic}. 
We can thus
formulate the following
\begin{problem}\label{prob:II}
Classify {\em strict parabolic homomorphisms} of Artin monoids, that is, those 
homomorphisms which map parabolic 
elements to parabolic elements
(like parabolic homomorphisms of Hecke monoids, they form a category).
\end{problem}
Parabolic elements in Artin monoids were already instrumental in our study of actions of the cactus group on  representation  categories of the corresponding quantum group (\cite{BGL}), and we expect that solving Problem~\ref{prob:II} will shed more light on symmetries of these categories.

When investigating such homomorphisms of Artin monoids,
we exploit their main advantage over their  Hecke counterparts. Namely, an Artin monoid embed into the corresponding Artin group (\cite{Par}*{Theorem~1.1}) where
every parabolic element~$T_{w_{J;I}}$ can be written as $T_{w_\circ^J}^{-1}T_{w_\circ^I}$. Thus, it would suffice to  determine the images of $T_{w_\circ^J}$ for all~$J$ for any given homomorphism, and
the most natural (``default'') solutions of Problem~\ref{prob:II} are those  which map~$T'_{w_\circ^{J'}}$, $J'\subset I'$
to~$T_{w_\circ^J}$ for some~$J\subset I$ depending on~$J'$.
We prove (Theorem~\ref{thm:artin parab}) that LCM homomorphisms studied e.g. in~\cites{Cas,Cri,God} are strict
parabolic and that these are the only ``default'' ones. Yet, surprisingly, there are new infinite families of solutions of Problem~\ref{prob:II} of a rather different nature.
\begin{theorem}[Propositions~\partref{prop:admissible hom from BrI22m to BrBn.a} and \ref{prop:Bm Bn parab}, 
Theorem~\partref{thm:monomial brd.c}]
\label{thm:intro strict parab}
\begin{enmalph}
\item\label{thm:intro strict parab.a} For any~$2\le m\le n$,
the assignments 
$$T'_1\mapsto (T_m T_{m+1}\cdots T_{n-1} T_n T_{n-1}\cdots T_m)(T_{m-2} T_{m-4}\cdots ),\qquad 
T'_2\mapsto T_{m-1}T_{m-3}\cdots
$$
define a strict parabolic homomorphism
$\Br^+(I_2(2m))\to \Br^+(B_n)$.

\item\label{thm:intro strict parab.b} For any~$2\le m\le n$ the 
assignments 
$$
T'_i\mapsto T_i,\quad i\in[1,m-1],
\quad T'_m\mapsto T_{m}T_{m+1}\cdots T_{n-1}T_n T_{n-1}\cdots T_m,
$$
define a strict parabolic homomorphism 
$\Br^+(B_m)\to \Br^+(B_n)$.

\item\label{thm:intro strict parab.c} For all $m,n\in\ZZ_{\ge 2}$, 
the assignments 
\begin{align*}
&T'_i\mapsto 
T_{w_{[(i-1)m+1,im-1]\cup[im+1,(i+1)m-1];[(i-1)m+1,(i+1)m-1]}},\qquad i\in[1,n-1],\\
&T'_n\mapsto T_{w_{[(n-1)m+1,nm-1];[(n-1)m+1,nm]}}
\end{align*}
define a strict parabolic homomorphism $\Br^+(B_n)\to\Br^+(B_{mn})$.
\end{enmalph}
\end{theorem}
The restriction of
the homomorphism from part~\ref{thm:intro strict parab.c} to $\Br^+_n:=\Br^+(A_{n-1})$
yields a homomorphism $\Br^+_n\to\Br^+_{nm}$
(Theorem~\partref{thm:monomial brd.a})
which is closely related to natural braidings between tensor powers of objects in braided monoidal categories and was inspired by our work in progress \cite{BGLmon}. 
These homomorphisms seem to provide most of solutions of Problem~\ref{prob:II} (see below for the discussion of ``missing'' ones).

It turns out that relaxing strict parabolicity is quite beneficial as it
yields additional interesting families of homomorphisms of Artin monoids. Namely, assume, for simplicity, that the Coxeter graph of~$M$ is connected, 
and denote $\mP(\Br^+(M))$ the 
(necessarily infinite) submonoid of~$\Br^+(M)$
generated by all parabolic elements $T_{w_{J;I}}$, $J\subset I$ (see~\S\ref{subs:parab elts} for the details and the general case). We call a homomorphism 
$\Phi:\Br^+(M')\to \Br^+(M)$ of Artin monoids {\em parabolic} if $\Phi(\mP(\Br^+(M')))\subset 
\mP(\Br^+_J(M))$ where~$J=\supp \Phi(\Br^+(M'))$ (see Definition~\ref{defn:parabhom}).
We can thus
pose the following
\begin{problem}\label{prob:III}
Classify parabolic homomorphisms of Artin monoids (which, like parabolic homomorphisms of Hecke monoids, form a category).
\end{problem}
Unlike Problem~\ref{prob:II}, this problem is wide open and highly non-trivial. 
The first class of solutions we found mimics  Theorem~\ref{thm:main thm 1} for Artin monoids. Namely,
we say that a homomorphism of Artin monoids 
is {\em light} if its maps generators of its domain to
powers of generators of its codomain (see Definition~\ref{defn:light Artin}).
\begin{theorem}[Theorem~\ref{thm:light->parabolic}]\label{thm:main thm 2}
All light homomorphisms of Artin monoids are parabolic.
\end{theorem}
For example, one of the very few existing parabolic 
projections of Artin monoids $P_{[1,n-1]}:\Br^+(B_n)\to \Br^+(A_{n-1})$, $T_i\mapsto T_i^{1-\delta_{i,n}}$, $i\in[1,n]$, is a parabolic homomorphism (Proposition~\ref{prop:PJ Bn An-1}).
Other examples of light homomorphisms are those 
which verify the famous {\em Tits conjecture} whose generalization was proved in~\cite{CP}. We extend it to Conjecture~\ref{conj:gen Tits} which essentially characterizes all injective light homomorphisms of 
Artin monoids.

Another class of non-strict parabolic homomorphisms is the following new infinite family.
\begin{theorem}[Propositions~\partref{prop:Cox hom from TJ}]
\label{thm:intro parab Cox}
For any $1\le a<b\le n$, the 
assignments
\begin{align*}
&T'_1\mapsto (T_a T_{a+1}\cdots T_{b-1} T_b T_{b-1} \cdots T_a)
(T_{a-2}T_{a-4}\cdots)(T_{b+2}T_{b+4}\cdots),\\
&T'_2\mapsto (T_{a-1} T_{a-3}\cdots ) 
(T_{b+1}T_{b+3}\cdots )
\end{align*}
define a parabolic homomorphism $\Br^+(I_2(2(n-b+(1-\delta_{a+b,n+1})(a-1))))\to \Br^+(A_n)$.
\end{theorem}
Note that this yields more solutions of Problem~\ref{prob:III} because every Artin
monoid is ``assembled'' out of dihedral ones $\Br^+(I_2(r))$. As an application, these
yield new actions of the corresponding Artin groups on any braided monoidal category.
The fact that these assignments yield homomorphisms of monoids is one of the most technically involved results of the paper, and its proof occupies most of Section~\ref{sec:SQF AB}. We also show, using Burau representation~\cite{Bu} of the Hecke algebra (see~\S\ref{subs:Burau}),
that there are no other homomorphisms of that type (see Theorem~\ref{thm:adm I2m converse } for the precise statement). The parabolicity of these homomorphisms is even more astonishing. 

The homomorphisms from Theorems~\ref{thm:intro strict parab} and~\ref{thm:intro parab Cox} are
``cleaned up'' (or more precisely ``undecorated'' in the sense of~\S\ref{subs:decor hom} and Remark~\ref{rem:undercoration}) versions 
of solutions of the following
\begin{problem}\label{prob:1.4}
Classify all homomorphisms of Artin monoids
$\Phi:\Br^+(M')\to \Br^+(M)$ such that $\Phi(T'_i)
=T_{w_\circ^{J_i}}$, $J_i\subset I$ for all~$i\in I'$
(we refer to them as {\em standard}, see Definition~\ref{defn:Heck Coxeter}).
\end{problem}
We completely solved Problem~\ref{prob:1.4} in the important case when all the~$J_i$, $i\in I'$
are pairwise disjoint and~$M$ is of type~$A$ or~$B$
(Theorems~\ref{thm:main thm adm}
and~\ref{thm:higher rank adm AB}). It is worth
mentioning that the central role in
proving these results is played by homomorphisms from
Theorems~\partref{thm:intro strict parab.a}\ref{thm:intro strict parab.b} and~\ref{thm:intro parab Cox}.
Unfolding these homomorphisms via~\eqref{eq:unfold Bn Dn+1} we obtain homomorphisms with the same property for~$M=D_{n+1}$, $n\ge 3$.
We expect that, apart from very few sporadic examples, that completes the 
classification of disjoint standard 
homomorphisms (see~\S\ref{subs:rank > 2}). 

The non-disjoint case seems to be even richer than the disjoint one. 
We have already discovered a plethora of families of standard non-disjoint homomorphisms in~\S\S\ref{subs:mon braid}, \ref{subs:inf ser non-disj} and~\ref{subs:non-disj Hecke}. For instance, parabolic homomorphisms from Theorem~\partref{thm:intro strict parab.c} are ``cleaned up'' versions of standard homomorphisms from
Theorem~\partref{thm:monomial brd.d}.
We exhibit even more conjectural non-disjoint homomorphisms in~\S\ref{subs:conj families}.

Compositions of standard homomorphisms are also quite interesting and recover
some known homomorphisms which are not standard.
\begin{example}\label{ex:1.6}
A well-known homomorphism $\Br^+(B_n)\to \Br^+(A_n)$
defined by $T'_i\mapsto T_i^{1+\delta_{i,n}}$, $i\in[1,n]$ is the composition of the standard
unfolding $\Br^+(B_n)\to \Br^+(D_{n+1})$ given by~\eqref{eq:unfold Bn Dn+1} and the standard
folding $\Br^+(D_{n+1})\to \Br^+(A_n)$
defined by $T'_i\mapsto T_{i-\delta_{i,n+1}}$, $i\in [1,n+1]$. Similarly, a well-known homomorphism
$\Br^+(G_2)\to \Br^+(A_2)$ defined by $T'_1\mapsto T_1^3$, $T'_2\mapsto T_2$ is the composition of
the standard unfolding $\Br^+(G_2)\to\Br^+(D_4)$,
$T'_1\mapsto T_1T_3 T_4$, $T'_2\mapsto T_2$
with the standard folding $\Br^+(D_4)\to \Br^+(A_2)$, $T'_i\mapsto T_1$, $i\in\{1,3,4\}$, $T'_2\mapsto T_2$
(see Example~\ref{ex:affine} for similar homomorphisms in affine types).
\end{example}
While these examples show that standard homomorphisms do not form a category, they still
share an important property with their compositions. Namely,
we say that a homomorphism~$\Phi:\Br^+(M')\to
\Br^+(M)$ is of {\em Hecke type} 
if it descends, necessarily uniquely, to a homomorphism~$\overline\Phi_\star$ of respective Hecke monoids; the notion
of a {\em Coxeter type} homomorphism is defined similarly (see Definition~\ref{defn:Heck Coxeter}).
 Artin monoids and their Hecke type homomorphisms form a subcategory~$\mathscr{AH}$ of the category~$\mathscr A$ of Artin monoids and all their homomorphisms, and
the assignments
$\Br^+(M)\mapsto (W(M),\star)$, $\Phi\mapsto \overline\Phi_\star$
define a functor~$\mathsf H$
from~$\mathscr{AH}$ to the category~$\mathscr H$ of Hecke monoids and their homomorphisms (Corollary~\ref{cor:cat AH AC}). Likewise,
we obtain a functor~$\mathsf C:\mathscr{AC}\to\mathscr C$ where~$\mathscr{AC}$ is the subcategory of~$\mathscr A$ with morphisms being of Coxeter type, and~$\mathscr C$ is the category of Coxeter groups. 

 Neither of the functors~$\mathsf H$ and~$\mathsf C$ is full or faithful (see Examples~\ref{ex:non-faithful}, \ref{ex:non-liftable} and~\ref{ex:non-liftable-1}
 and Remark~\ref{rem:C not full}). We refer to homomorphisms in their images as 
{\em liftable}, and they would be interesting to describe.

Based on the above discussion, we expect that the most interesting  Hecke and Coxeter type homomorphisms of Artin monoids are the standard ones and their compositions. 
For instance, every light homomorphism is in that class, i.e. belongs to the subcategory~$\mathscr {ACH}_{st}$ of~$\mathscr{AH}\cap\mathscr{AC}$ generated by all standard homomorphisms (Proposition~\ref{prop:Tits standard}). Moreover, we expect that a homomorphism in~$\mathscr {ACH}_{st}$
is injective if and only it is injective on all parabolic
submonoids of rank 2
which, 
in particular, recovers the aforementioned extension 
of Tits conjecture (Conjecture~\ref{conj:gen Tits}).
Besides, one can often obtain
homomorphisms of Artin monoids which are not of Hecke type by ``cleaning'' standard ones (that was how Theorems~\ref{thm:intro strict parab} and~\ref{thm:intro parab Cox} were derived from Theorem~\partref{thm:monomial brd.b}\ref{thm:monomial brd.d}). Sometimes this 
procedure yields homomorphisms which are neither Hecke nor Coxeter (Corollary~\ref{cor:strange homs}) or rather exotic Coxeter-Hecke type homomorphisms (Remark~\ref{rem:non-std B2 An}).

Surprisingly, we classified (yet conjecturally) {\em all} standard homomorphisms from~$\Br^+(A_2)$
and~$\Br^+(B_2)$ (Theorems~\ref{thm:Hom A2}, \ref{thm:Hom B2}, \ref{thm:B2 A2n-1 spec} and Proposition~\ref{prop:Hom A2|B2 EF}) to Artin monoids of finite types. 
They have rather intriguing combinatorial properties
(see Theorem~\ref{thm:combinatoric std}
and Remark~\ref{rem:OEIS} for OEIS appearances of the related numerology). For 
instance, there are precisely~$2^m$ fully supported standard homomorphisms~$\Br^+_3\to \Br^+_{3m}$, $m\ge 1$. 
Homomorphisms from~$\Br^+(B_2)\to\Br^+(M)$ are even more affluent. 
For instance,
their number grows asymptotically as $\frac14(1+\sqrt 3)^{\frac12(n+5)}$
for~$M$ of type~$A_n$, $3\cdot 2^n$
for~$M$ of type~$B_n$ and faster than~$\frac73\cdot 2^{n+1}$ if~$M$ is of type~$D_{n+1}$. There is an 
additional series of homomorphisms~$\Br^+(B_2)\to \Br^+_{2n}$ which grows at least as fast as $2^{\frac12 n}n$.
These results, in conjunction with conjectures in~\S\ref{subs:conj families}, give a hope to {\em completely} classify standard homomorphisms of Artin monoids of finite and affine types.

The above constructions and examples demonstrate that our theory of parabolic, Hecke and Coxeter type (especially standard) homomorphisms is rich and advances our understanding of Hecke and Artin monoids, Coxeter groups and their representations.

This paper is organized as follows. In Section~\ref{sec:Prelim} we introduce the notation and collect definitions and properties of Artin and Hecke monoids which are used throughout the rest of the paper. Section~\ref{sec:Gen homs} is dedicated to general properties of homomorphisms of Hecke and Artin monoids. In Section~\ref{sec:Parab proj} we study light homomorphisms of Hecke and Artin monoids and, in particular, prove Theorems~\ref{thm:main thm 1} and~\ref{thm:main thm 2}.
In Section~\ref{sec:SQF AB} we classify disjoint standard homomorphisms of Artin monoids whose codomain is of finite type~$A$ or~$B$ and describe several infinite families of non-disjoint standard or parabolic Coxeter type homomorphisms.

\subsection*{Acknowledgments}
The main part of this work was carried out while the authors were visiting Erwin Schr\"odinger
International Institute for Theoretical Physics (ESI), Vienna, Austria,
in the framework of the ``Research in teams'' program. It is our pleasure to thank the ESI for its hospitality. This work took its present shape while the first author was visiting Max Planck Institute for Mathematics in the Sciences (MIS), Leipzig, Germany and the second author was visiting Institut des Hautes \'Etudes Scientifiques (IHES), Bures-sur-Yvette, France. The hospitality of both institutions is gratefully acknowledged.

\section{Preliminaries}\label{sec:Prelim}

\subsection{General notation}
We extend the natural order on~$\mathbb Z$ to~$\mathbb Z\cup\{\infty\}$
via $\infty>n$ for all~$n\in\mathbb Z$
and use the convention that $n \infty=n+\infty=\infty$ for all~$n\in\mathbb Z_{>0}\cup\{\infty\}$.
In particular, $\infty$ is assumed 
to be divisible by all elements of~$\mathbb Z_{>0}\cup\{\infty\}$. 
Given~\plink{bar s}$s\in\mathbb Z$, let $\bar s\in\{0,1\}$
be the remainder of~$s$ when divided by~$2$.
For any $a,b\in\mathbb Z$ we denote $[a,b]=\{ i\in\mathbb Z\,:\,
a\le i\le b\}$ and \plink{[a,b]2}$[a,b]_2=\{ k\in [a,b]\,:\, \overline{b-k}=0\}$. Given $a,b\in\ZZ$ and~$J\subset \ZZ$, set $a+b J:=\{a +b j\,:\, j\in J\}$.
The power set of a set~$S$ will be denoted~\plink{PS}$\mathscr P(S)$. Given a category~$\mathscr{D}$, we denote $\Hom_{\mathscr D}(X,Y)$ the set of morphisms from $X\in\mathscr D$ to~$Y\in\mathscr D$.

\subsection{Monoids}\label{subs:monoids}
Throughout this paper, a homomorphism of monoids is assumed to map the identity element of the domain
to the identity element of the codomain.

Let~$\mathsf M$ be a multiplicative monoid. 
Given any finite subset~$I\subset \ZZ$ and
a family $X_i$, $i\in I$ of elements of~$\mathsf M$ we set
$$\plink{ascp}
\ascprod_{i\in I} X_i=X_{i_1}\cdots X_{i_r},\qquad
\dscprod_{i\in I} X_i=X_{i_r}\cdots X_{i_1}.
$$
where $I=\{i_1,\dots,i_r\}$ with~$i_1<\cdots<i_r$. This notation
will also be used for infinite families with all but finitely many $X_i$
equal to~$1$.

Given a family~$S$ of generators of~$\mathsf M$,
the length function $\ell_S:\mathsf M\to \mathbb Z_{\ge0}$ is defined by setting $\ell_S(x)$,
$x\in \mathsf M$
to be
the minimal length of a word in~$S$ which is equal
to~$x$. Clearly, $\ell_S(xy)\le \ell_S(x)+\ell_S(y)$
for all~$x,y\in\mathsf M$.

An equivalence relation~$\mathcal C\subset\mathsf M\times\mathsf M$
is called a {\em congruence relation} on~$\mathsf M$
if $(x,y),(x',y')\in\mathcal C$ implies that~$(xx',yy')\in\mathcal C$. In that case,
the set~$\mathsf M/\mathcal C$ of equivalence classes with respect to~$\mathcal C$ is also a monoid, with the multiplication defined by $[x]_{\mathcal C}[y]_{\mathcal C}=[xy]_{\mathcal C}$, $x,y\in\mathsf M$,
where~$[x]_{\mathcal C}$ is the equivalence class of~$\mathcal C$. Furthermore, the canonical map~$\pi_{\mathcal C}:\mathsf M\to \mathsf M/\mathcal C$, $x\mapsto [x]_{\mathcal C}$,
$x\in\mathsf M$ is a surjective homomorphism monoids.

We say that~$\mathsf M$ is {\em left} (respectively, {\em right}) {\em cancellative} if $x y=x y'$ (respectively, $y x=y' x$),
$x,x',y\in\mathsf M$ implies~$y=y'$. We say that~$\mathsf M$
is {\em cancellative} if it is left and right cancellative. 
For any $x,y\in \mathsf M$ and~$m\in\ZZ_{\ge0}$  denote
$$\plink{brd}
\brd{xy}{m}:=(xy)^{\lfloor \frac12 m\rfloor} x^{\bar m}.
$$
In other words, $\brd{xy}0=1$, $\brd{xy}{m+1}=
\brd{xy}m\,x$ if~$m$ is even, $\brd{xy}{m+1}=
\brd{xy}m\,y$ if~$m$ is odd, while $\brd{xy}{m+1}=
x\brd{yx}{m}$ for all~$m\in\mathbb Z_{\ge 0}$. Given $x,y\in\mathsf M$, define\plink{B(x,y)}
$$
B(x,y)=\{ k\in\ZZ_{>0}\,:\, \brd{xy}k=\brd{yx}k\}.
$$
\begin{lemma}\label{lem:taut homs}
Let~$\mathsf M$ be a multiplicative monoid suppose that~$B(x,y)\not=\emptyset$ for some~$x,y\in \mathsf M$. 
\begin{enmalph}
    \item \label{lem:taut homs.a}
    If~$m\in B(x,y)$ then
    \begin{equation}\brd{xy}{km}=(\brd{xy}m)^k, \qquad 
    k\ge 1\label{eq:taut homs}
    \end{equation}
    and so~$\mathbb Z_{>0}m\subset B(x,y)$. 
\item\label{lem:taut homs.b} If~$\mathsf M$ is left (or right) cancellative  then~$B(x,y)=\mathbb Z_{>0}\min B(x,y)$.
\end{enmalph}
\end{lemma}
\begin{proof}
Let~$t_0=x$ and~$t_1=y$. 
Note first that for~$r\le s\in\mathbb Z_{>0}$ and
$a\in \ZZ$
we have 
\begin{equation}\label{eq:brd tail}
\brd{t_{\overline{a\vphantom1}} t_{\overline{a+1}}}{s}=\brd{t_{\overline{a\vphantom1}}t_{\overline{a+1}}}r \,\brd{t_{\overline{a+r}}t_{\overline{a+r+1}}}{s-r}.
\end{equation}

We use induction on~$k$
to prove~\eqref{eq:taut homs}, the induction base being trivial. For the inductive step we have for~$k\in\mathbb Z_{>0}$ 
\begin{alignat*}{3}
\brd{t_0t_1}{(k+1)m}=&\brd{t_0t_1}{km} \,\brd{t_{\overline{km}}t_{\overline{km+1}}}{m}
&\qquad&\text{by~\eqref{eq:brd tail}}\\
=&(\brd{t_0t_1}{m})^k \brd{t_0t_1}{m}&&\text{by induction hypothesis and since~$m\in B(t_0,t_1)$}\\
=&(\brd{t_0t_1}{m})^{k+1}.
\end{alignat*}
The second assertion in part~\ref{lem:taut homs.a} is now immediate.

We only prove part~\ref{lem:taut homs.b} for
left cancellative monoids. Let~$m=\min B(t_0,t_1)$, $m'\in B(x,y)\setminus\{m\}$
and write~$m'=d m+r$,
$d\in\mathbb Z_{> 0}$,
$0\le r<m$. Suppose that~$r>0$. Then 
\begin{alignat*}{2}
\brd{t_1t_0}{dm}\,\brd{t_{\overline{dm+1}}t_{\overline{dm}}}{r}
=&\brd{t_1t_0}{m'}&&\text{by~\eqref{eq:brd tail}}\\
=&\brd{t_0t_1}{m'}&&\text{since $m'\in B(t_0,t_1)$}\\
=&\brd{t_0t_1}{dm}\,\brd{t_{\overline{dm}}t_{\overline{dm+1}}}{r}&&
\text{by \eqref{eq:brd tail}}\\
=&\brd{t_1t_0}{dm}\,\brd{t_{\overline{dm}}t_{\overline{dm+1}}}{r}&\qquad&
\text{by part~\ref{lem:taut homs.a}}.
\end{alignat*}
Since~$\mathsf M$ is left cancellative,
it follows that $\brd{t_{\overline{dm+1}}t_{\overline{dm}}}{r}=
\brd{t_{\overline{dm}}t_{\overline{dm+1}}}{r}$
whence~$r\in B(t_0,t_1)$ which is a contradiction. Thus, $r=0$.
\end{proof}

\subsection{Artin monoids and Coxeter groups}\label{subs:Br(M)W(M)}
Let~$I$ be a finite set and let~$M=(m_{ij})_{i,j\in I}$ be
a symmetric matrix with $m_{ii}=1$ and $m_{ij}\in\ZZ_{>1}\cup\{\infty\}$, $i\not=j$.
Such a matrix is called a {\em Coxeter matrix} (over~$I$), and we denote the set of 
all Coxeter matrices over~$I$ by~\plink{Cox I}$\Cox I$.
The Coxeter
graph~\plink{Gamma(M)}$\Gamma(M)$ associated with~$M$ is the undirected graph with vertex set~$I$
and with a unique edge connecting $i,j\in I$ if and only if~$m_{ij}>2$.
The edge is labeled with~$m_{ij}$ if~$m_{ij}>3$.

The {\em Artin monoid} $\Br^+(M)$\plink{Br+(M)} associated with~$M$ (see for example~\cites{BrSa,Del,Tits}) is generated by the~$T_i$, $i\in I$ subject to relations
$$
\brd{T_iT_j}{m_{ij}}=\brd{T_jT_i}{m_{ij}},\qquad i\not=j\in I,\, m_{ij}<\infty.
$$
The Artin group~$\Br(M)$ associated with~$M$ has the same generators satisfying the same relations. By~\cite{Par}*{Theorem~1.1}, the natural homomorphism $\Br^+(M)\to \Br(M)$
is injective. In particular, $\Br^+(M)$ is cancellative (see
also~\cite{BrSa}*{Proposition~2.3} which establishes the cancellativty
of~$\Br^+(M)$ without using the embedding into its Artin group). The following example illustrates
Lemma~\partref{lem:taut homs.b}.
\begin{example}
Let~$I=\{1,2\}$ and let~$M=\left(\begin{smallmatrix}1&3\\3&1\end{smallmatrix}\right)$. Then $\Br^+(M)$ is generated by $T_1$, $T_2$ subject to the relation $T_1T_2T_1=T_2T_1T_2$.
Yet $(T_1T_2)^2\not=(T_2T_1)^2$ for otherwise 
we would have $T_2T_1T_2^2=T_1T_2T_1T_2=T_2T_1T_2T_1$
which, since~$\Br^+(M)$ is cancellative, yields~$T_1=T_2$.
\end{example}
Since defining relations of~$\Br^+(M)$ are homogeneous in the number of generators, the
length function with respect to~$\{T_i\}_{i\in I}$
is a homomorphism of monoids~\plink{ell}$\ell:\Br^+(M)\to (\ZZ_{\ge 0},+)$. If~$|I|=1$ this homomorphism
is actually an isomorphism.

Since defining relations of~$\Br^+(M)$ are palindromic,
$\Br^+(M)$
admits a unique \plink{op}anti-involu\-tion~${}^{op}$ defined on
generators by~$(T_i)^{op}=T_i$, $i\in I$. It clearly extends to
an anti-involution of~$\Br(M)$.

The {\em Coxeter group}\plink{W(M)}
$W=W(M)$ associated with~$M$ is generated by the~$s_i$, $i\in I$ subject
to relations
$$
(s_i s_j)^{m_{ij}}=1,\qquad i,j\in I,\, m_{ij}\not=\infty.
$$
Clearly, $W(M)$ is the quotient of~$\Br(M)$ by the minimal normal
subgroup containing the $T_i^2$, $i\in I$. Let~\plink{piM}$\pi_M:\Br(M)\to W(M)$,
$T_i\mapsto s_i$, $i\in I$,
be the canonical projection, which obviously restricts to a
surjective homomorphism of monoids~$\Br^+(M)\to W(M)$. Note also that~$W(M)$ is isomorphic
to the quotient monoid of~$\Br^+(M)$ by the minimal congruence relation containing the~$(T_i^2,1)$, $i\in I$.

We denote~$\ell$ the length function for~$W(M)$
with respect to~$\{s_i\}_{i\in I}$.
An expression~$w=s_{i_1}\cdots s_{i_k}$,
$i_1,\dots,i_k\in I$ is called {\em reduced} if~$k=\ell(w)$. Clearly,
$\ell(\pi_M(T))\le \ell(T)$ for all~$T\in\Br^+(M)$ and we
set\plink{SQF}
$$
\SQF^+(M)=\{ T\in\Br^+(M)\,:\,\ell(\pi_M(T))=\ell(T)\}.
$$
Elements of~$\SQF^+(M)$ are called {\em square free}. The following is well-known.
\begin{theorem}[\cite{Tits}*{Theorem~3}] \label{thm:Tits}
\begin{enmalph}
\item\label{thm:Tits.a} $\pi_M$ restricts to a bijection~$\SQF^+(M)\to W(M)$.
\item\label{thm:Tits.b} Given $w\in W(M)$, denote $T_w$ the unique
element of~$\SQF^+(M)\cap \pi^{-1}_M(\{w\})$.
Then $T_w T_{w'}=T_{ww'}$ if and only if
$\ell(ww')=\ell(w)+\ell(w')$. In particular,
for any~$w\in W(M)$,
an expression $w=s_{i_1}\cdots s_{i_k}$, $i_1,\dots,i_k\in I$ is reduced
if and only if~$T_w=T_{i_1}\cdots T_{i_k}$.
\end{enmalph}
\end{theorem}
\begin{lemma}\label{lem:sq free fact}
Let~$M$ be a Coxeter matrix, $T\in\SQF^+(M)$ and suppose that~$T=T'T''$, $T',T''\in\Br^+(M)$.
Then both $T'$ and~$T''$ are square free.
\end{lemma}
\begin{proof}
Since $\ell(X)\ge \ell(\pi_M(X))$ for all~$X\in\Br^+(M)$,
we have~$\ell(\pi_M(T))=\ell(T)=\ell(T')+\ell(T'')\ge
\ell(\pi_M(T'))+\ell(\pi_M(T''))$.
On the other hand, since $\pi_M(T)=
\pi_M(T')\pi_M(T'')$,
$\ell(\pi_M(T))
\le \ell(\pi_M(T'))+\ell(\pi_M(T''))$, whence $\ell(\pi_M(T))=
\ell(\pi_M(T'))+\ell(\pi_M(T''))$. This forces~$\ell(T')=\ell(\pi_M(T))$
and $\ell(T'')=\ell(\pi_M(T''))$.
\end{proof}

The anti-involution~${}^{op}$ factors through to an anti-involution of~$W(M)$ which coincides with
the anti-involution $w\mapsto w^{-1}$, $w\in W(M)$. The following is
immediate.
\begin{lemma}\label{lem:can image op inv}
If~$T\in\Br^+(M)$ is ${}^{op}$-invariant then~$\pi_M(T)\in W(M)$ is an involution.
In particular, $T\in\SQF^+(M)$ is ${}^{op}$-invariant
if and only if~$\pi_M(T)$ is an involution.
\end{lemma}

\subsection{Parabolic submonoids and subgroups}\label{subs:parab}
Given~$J\subset I$, let $M_J=(m_{ij})_{i,j\in J}\in \Cox J$ be the corresponding
submatrix of~$M$. Then the submonoid \plink{Br+J(M)}$\Br^+_J(M):=\la T_j\,:\, j\in J\ra$ of~$\Br^+(M)$ is isomorphic to~$\Br^+(M_J)$. The subgroups~$\Br_J(M)$ of~$\Br(M)$ and~$W_J(M)$ of~$W(M)$ are defined similarly and are
isomorphic to respective objects corresponding to~$M_J$. Those subobjects are called {\em parabolic} submonoids (subgroups).
We will usually identify $W_J(M)$ with~$W(M_J)$ and so on and denote~\plink{iotaJ}$\iota_{J}$
the natural inclusion of~$W_J(M)$ (respectively, $\Br^+_J(M)$)
into~$W(M)$ (respectively, $\Br^+(M)$).

We say that~$J\subset I$ is of {\em finite type} if~$W(M_J)$ is finite.
The corresponding subgroups and submonoids are often referred to as being of {\em spherical type} in the literature. We denote~\plink{F(M)}$\mathscr F(M)$ the
set of all subsets of~$I$ of finite type. Clearly,
$\mathscr F(M)=\mathscr P(I)$ if and only if~$I\in\mathscr F(M)$, in which case we also say that~$M$ is of finite type. Note that~$\emptyset\in\mathscr F(M)$, the corresponding
parabolic subgroups and submonoids being trivial.

Define \plink{supp}$\supp:\Br(M)\to \mathscr P(I)$
by $$
\supp T=\bigcap_{J\subset I\,:\, T\in\Br_J(M)} J,
\qquad T\in\Br(M).
$$
The map~$\supp:W(M)\to \mathscr P(I)$ is defined
similarly.
Clearly, $\supp \pi_M(T)\subset \supp T$ for all~$T\in\Br(M)$.
Given a subset~$S$ of~$\Br(M)$ or~$W(M)$, we denote $\supp S=\bigcup_{x\in S} \supp x$. Observe that~$\supp TT'=\supp T\,\cup\,\supp T'$ for $T,T'\in\Br^+(M)$ while $\supp ww'\subset \supp w\,\cup\,\supp w'$ for $w,w'\in W(M)$. In particular, given
any expression $T=T_{i_1}\cdots T_{i_k}$
(respectively, a {\em reduced} expression
$w=s_{i_1}\cdots s_{i_k}$) where~$i_1,\dots,i_k\in I$
we have~$\supp T=\{i_1,\dots,i_k\}$ (respectively,
$\supp w=\{i_1,\dots,i_k\}$). It follows that the map
$\supp$ is surjective.
The following is well-known
(cf.~\cite{Tits}*{Theorem~3}, \cite{Bou}*{Ch. IV, \S1.5}).
\begin{lemma}\label{lem:extend supp}
Let~$w,w'\in W(M)$ with $\supp w\cap \supp w'=\emptyset$. Then
$\ell(ww')=\ell(w)+\ell(w')$ and~$\supp ww'=\supp w\cup\supp w'$.
\end{lemma}

We say that~$J,K\subset I$ are {\em orthogonal}
if $m_{jk}=2$ for all~$j\in J$, $k\in K$. We say that~$J\subset I$
is {\em self-orthogonal} if~$m_{ij}\le 2$ for all~$i,j\in J$.
A Coxeter
matrix~$M$ over~$I$ is said to be {\em irreducible} if~$I$ cannot be written as a
disjoint union of two non-empty orthogonal subsets or, equivalently,
if~$\Gamma(M)$ is connected. We denote~$\Gamma_J(M)$ the full
weighted subgraph
of~$\Gamma(M)$ with vertex set~$J$. Clearly, $\Gamma_J(M)=\Gamma(M_J)$. We say that~$J\subset I$
is {\em connected} if~$\Gamma_J(M)$ is connected as a graph or,
equivalently, if $J$ is not the disjoint union of two non-empty
orthogonal subsets. By abuse of terminology, we say that~$J\subset I$
is a {\em connected component} of~$I$ if~$\Gamma_J(M)$
is a connected component of~$\Gamma(M)$ or, equivalently, if~$J$ is a maximal connected subset of~$I$.

It is well-known (see, e.g.~\cite{Bou}*{Ch. VI, \S4, Thm.~1}) that the Coxeter group~$W(M)$ with irreducible~$M$ is finite if and only if
$\Gamma(M)$ is isomorphic to one of the following graphs 
\begin{alignat}{3}
A_n:&\dynkin[text style/.style={scale=0.8},Coxeter,root radius=0.07,expand labels={1,2,n-1,n},make indefinite edge={2-3},edge length=1.0cm]A4,&&n\ge 1,
\nonumber\\
B_n:&\dynkin[text style/.style={scale=0.8},Coxeter,root radius=0.07,expand labels={1,2,n-1,n},make indefinite edge={2-3},edge length=1.0cm]B4,&&n\ge 2,\nonumber\\
D_{n+1}: &\dynkin[text style/.style={scale=0.8},Coxeter,root radius=0.07,expand labels={1,2,n-1,n,n+1},label directions={,,right,,},make indefinite edge={2-3},edge length=1.0cm]D5,&&n\ge 3,\nonumber\\
E_n:&%
\dynkin[text style/.style={scale=0.8},Coxeter,ordering=Kac,root radius=0.07,expand labels={1,2,3,4,n-1,n},make indefinite edge={4-5},edge length=1.0cm]E6,&\qquad&
n\in\{6,7,8\},\nonumber\\
F_4:&\dynkin[text style/.style={scale=0.8},Coxeter,ordering=Kac,root radius=0.07,expand labels={1,2,3,4},edge length=1.0cm]F4,\nonumber\\
I_2(m):&
\dynkin[text style/.style={scale=0.8},Coxeter,ordering=Kac,root radius=0.07,expand labels={1,2},edge length=1.0cm,gonality=m]I2,&&m\ge 4,\nonumber\\
H_n:&\dynkin[text style/.style={scale=0.8},Coxeter,ordering=Kac,root radius=0.07,ordering=Adams,expand labels={1,2,n-1,n},make indefinite edge={2-3},edge length=1.0cm]H4
,&&n\in\{3,4\}.\label{eq:Coxeter graphs}
\end{alignat}
The labeling shown in~\eqref{eq:Coxeter graphs} will be used throughout the rest of
the paper unless specified otherwise.
Clearly, $I_2(3)$ (respectively,~$I_2(4)$) coincides with $A_2$ (respectively, $B_2$); the graph of type~$I_2(6)$
is traditionally denoted as~$G_2$. We will use~$X_n$ as the notation for the Coxeter matrix of the corresponding graph with the labeling as in~\eqref{eq:Coxeter graphs}. 
We will also denote $I_2(\infty)=\left(\begin{smallmatrix}1&\infty\\\infty&1\end{smallmatrix}\right)$, the corresponding
Artin monoid being just the free monoids with two generators.

An automorphism~$\sigma$ of the weighted graph~$\Gamma(M)$, or, equivalently
a permutation~$\sigma$ of~$I$ such that $m_{\sigma(i)\sigma(j)}=m_{ij}$
for all~$i,j\in I$, induces an automorphism of~$\Br^+(M)$ (respectively,
$\Br(M)$, $W(M)$), called a {\em diagram automorphism} and also denoted
by~$\sigma$, via $\sigma(T_i)=T_{\sigma(i)}$ (respectively,
$\sigma(s_i)=s_{\sigma(i)}$), $i\in I$. If~$W(M)$ is finite and~$\Gamma(M)$ is connected, diagram automorphisms of order~$2$
exist only if $\Gamma(M)$ is of type $A_n$, $n\ge 1$, $D_{n+1}$, $n\ge 3$, $F_4$, $E_6$ or~$I_2(m)$, the corresponding permutation of~$I$ being
\begin{equation}\label{eq:diag aut}
\sigma=\begin{cases}
\prod_{1\le i\le \frac12n} (i,n+1-i),& M=A_n,\, n\ge 2,\\
(n,n+1),&M=D_{n+1},\, n\ge 3,\\
(1,4)(2,3),&M=F_4,\\
(1,5)(2,4),&M=E_6.
\end{cases}
\end{equation}
In type $D_4$, there is also a diagram
automorphism of order~$3$ given by the permutation $(1,3,4)$ of~$[1,4]$ and
so the group of all diagram automorphisms of~$D_4$ is isomorphic to~$S_3$.

If~$J\in\mathscr F(M)$, then $W_J(M)$ contains the unique element~\plink{w0J}
$w_\circ^J$ of maximal length (see, e.g.~\cites{Bou,Tits}), which is obviously an involution. It is well-known (see e.g.~\cite{Bou}*{Ch. IV, Ex. 22} or~\cite{BjBr}*{Proposition~2.3.2}) that
\begin{equation}\label{eq:ell w w0}
\ell(ww_\circ^J)=\ell(w_\circ^J w)=\ell(w_\circ^J)-\ell(w),\qquad
w\in W_J(M).
\end{equation}

\subsection{Hecke monoids}\label{subs:Hecke}
The {\em Hecke monoid} associated with $M$ is the quotient
of~$\Br^+(M)$ by the minimal congruence relation
containing $(T_i^2,T_i)$, $i\in I$.
We denote~\plink{pi*M}$\pi^\star_M$ the canonical homomorphism from $\Br^+(M)$
to the corresponding Hecke monoid.
Thus, the Hecke monoid is generated
by the $s_i:=\pi^\star_M(T_i)$, $i\in I$
subject to relations $s_i\star s_i=s_i$, $i\in I$ and
$$
\brd{s_i\star s_j\star}{m_{ij}}=\brd{s_j\star s_i\star}{m_{ij}},\qquad
i\not=j\in I,\, m_{ij}\not=\infty.
$$
Note that~${}^{op}$ and diagram automorphisms factor through to
the Hecke monoid.
\begin{remark}
In the literature, Hecke monoids are also referred to as Coxeter monoids
(see e.g.~\cite{K14}), $0$-Hecke monoids or Demazure monoids. The latter term is due to the fact that idempotent Demazure operators provide a representation of Hecke monoids.
\end{remark}
\begin{proposition}\label{prop:prod *}
For all~$i\in I$, $w\in W(M)$
\begin{equation}\label{eq:prod *}
s_i\star w=\begin{cases}
s_i w,&\ell(s_i w)>\ell(w),\\
w,&\ell(s_i w)<\ell(w),
\end{cases}\qquad
w\star s_i=\begin{cases}
ws_i,&\ell(ws_i)>\ell(w),\\
w,&\ell(ws_i)<\ell(w),
\end{cases}
\end{equation}
where we abbreviate $w=\pi^\star_M(T_w)$.
In particular, $\pi^\star_M(\Br^+(M))$ identifies with~$W(M)$ as a set,  the restriction
of~$\pi^\star_M$ to~$\SQF^+(M)$ is a bijection onto~$W(M)$ and
$\pi^\star_M|_{\SQF^+(M)}=\pi_M|_{\SQF^+(M)}$.
\end{proposition}
\begin{proof}
By~\cite{Bou}*{Ch.~IV, \S1.5}, if~$\ell(s_i w)>\ell(w)$ then $\ell(s_i w)=\ell(w)+1$ and
so $T_{s_iw}=T_i T_w$ and it remains to apply~$\pi^\star_M$.
Also by~\cite{Bou}*{Ch.~IV, \S1.5}, if $\ell(s_i w)<\ell(w)$ then
$w=s_i u$ for some~$u\in W(M)$ with~$\ell(w)=\ell(u)+1$. Then
$T_w=T_i T_u$ and so applying~$\pi^\star_W$ yields
$s_i\star w=s_i\star (s_i\star u)=s_i\star u=w$. The second identity
follows by using~${}^{op}$.

We now prove by induction on~$\ell(T)$ that for any~$T\in\Br^+(M)$, $\pi^\star_M(T)=w$ for some~$w\in W(M)$. The induction base $\ell(T)=0$ is obvious.
For the inductive step, if~$\ell(T)>0$ then $T=T_i T'$ for some~$i\in I$, $T'\in\Br^+(M)$ with $\ell(T')=\ell(T)-1$.
 Therefore, $\pi^\star_M(T)=s_i\star \pi^\star_M(T')=s_i\star w'$ for some~$w'\in W(M)$ by the induction hypothesis. Then $\pi^\star_M(T)\in \{w',s_i w'\}$ by~\eqref{eq:prod *}, which proves the inductive step. The last assertion
follows by an obvious induction on~$\ell(T)$, $T\in \SQF^+(M)$
since $T\in \SQF^+(M)$ implies
that $T=T_{i_1}\cdots T_{i_k}$, $k=\ell(T)\ge 0$ with $\ell(s_{i_1}
\cdots s_{i_r})=r$ for all~$1\le r\le k$.
\end{proof}
It follows that
$$
\SQF^+(M)=\{ T\in\Br^+(M)\,:\,\ell(\pi^\star_M(T))=\ell(T)\}.
$$
From now on, we identify the Hecke monoid associated with
the Coxeter matrix~$M$ with the Coxeter group~$W(M)$ {\em as a set} and denote it
\plink{HeMon}$(W(M),\star)$. Note that $\supp (w\star w')=\supp w\cup\supp w'$ for all
$w,w'\in W(M)$.
\begin{remark}
Proposition~\ref{prop:prod *} can be regarded as a presentation of the Hecke monoid. Namely,
we can define it as $W(M)$, as a set, equipped with the unique associative operation~$\star$ satisfying the
first property in~\eqref{eq:prod *}.
\end{remark}

The following are immediate.
\begin{lemma}\label{lem:free product Artin}\label{lem:free product}
Let~$M\in\Cox I$, $M'\in\Cox{I'}$. 
\begin{enmalph}
\item\label{lem:free product.a}
Define~$M\times M'\in\Cox{I\sqcup I'}$ by
$(M\times M')_{ij}=(M\times M')_{ji}=\begin{cases}
                                     m_{ij},&i,j\in I,\\
                                     m'_{ij},&i,j\in I',\\
                                     2,&i\in I,\,j\in I'
                                    \end{cases}$.
Then $\Br^+(M)\times \Br^+(M')\cong \Br^+(M\times M')$,
$W(M)\times W(M')\cong W(M\times M')$
and $(W(M),\star)\times (W(M'),\star)\cong 
(W(M\times M'),\star)$;
\item \label{lem:free product.b}
Define~$M\coprod M'\in\Cox{I\sqcup I'}$ by
$(M\textstyle\coprod M')_{ij}=(M\textstyle\coprod M')_{ji}=\begin{cases}
m_{ij},& i,j\in I,\\
m'_{ij},& i,j\in I',\\
\infty,& i\in I,\,j\in I'
\end{cases}
$.
Then $\Br^+(M)\coprod\Br^+(M')\cong \Br^+(M\coprod M')$,
$W(M)\coprod W(M')\cong \Br^+(M\coprod M')$
and $(W(M),\star)\coprod (W(M'),\star)\cong 
(W(M\coprod M'),\star)$, where $\mathsf M\coprod\mathsf M'$ denotes the free product of monoids $\mathsf M$ and~$\mathsf M'$.
\end{enmalph}
\end{lemma}
\begin{lemma}\label{lem:orth factors}
Let~$M\in\Cox I$ and let $J,K\subset I$ be orthogonal.
Then
\begin{enmalph}
    \item\label{lem:orth factors.a}
    $\Br^+_{J\cup K}(M)\cong\Br^+_J(M)\times 
    \Br^+_K(M)$;
    \item\label{lem:orth factors.b}
    $W_{J\cup K}(M)\cong W_J(M)\times 
    W_K(M)$;
    \item\label{lem:orth factors.c}
    $(W_{J\cup K}(M),\star)\cong (W_J(M),\star)\times 
    (W_K(M),\star)$.
\end{enmalph}
In particular, submonoids $\Br^+_J(M)$, $\Br^+_K(M)$ (respectively,
$W_J(M)$, $W_K(M)$ and $(W_J(M),\star)$, $(W_K(M),\star)$) commute element-wise in~$\Br^+(M)$
(respectively, in~$W(M)$, $(W(M),\star)$).
\end{lemma}
\begin{lemma}\label{lem:monoid len}
We have $\ell(u\star v)\ge \max(\ell(u),\ell(v))$ for all $u,v\in W(M)$.
\end{lemma}
\begin{proof}
By Proposition~\ref{prop:prod *}, $\ell(u\star s_i)\ge \ell(u)$ for all $u\in W$, $i\in I$. An obvious induction on the length of~$v$ proves that $
\ell(u\star v)\ge \ell(u)$. Again by Proposition~\ref{prop:prod *}, $\ell(s_i\star v)\ge \ell(v)$ for all $u\in W$ and $i\in I$, and an induction on $\ell(u)$ shows that $\ell(u\star v)\ge \ell(v)$.
\end{proof}
The multiplication in~$(W(M),\star)$ has an additional
characterization which will be important later.

Define a relation $\longrightarrow$ on~$W(M)$
by $u\longrightarrow w$, $u,w\in W(M)$ if and only if
$\ell(u)<\ell(w)$ and $u^{-1}w$ is conjugate to~$s_i$
for some~$i\in I$.
The {\em strong Bruhat order} on~$W(M)$, which we will denote by $\le$, is the transitive
closure of this relation and is easily seen to be
a partial order (see e.g.~\cite{BjBr}*{\S2.1}).
We will need the following properties of the strong
Bruhat order.
\begin{proposition}[see e.g.~\cite{BjBr}*{Theorems~2.2.2 and~2.2.6, Proposition~2.3.4}]\label{prop:Bruhat order}
\begin{enmalph}
    \item\label{prop:Bruhat order.a}
    Let $w=s_{i_1}\cdots s_{i_k}\in W(M)$ be a reduced expression.
    Then $u\le w$ if and only if there exists $J\subset [1,k]$
    such that $|J|=\ell(u)$ and $u=\ascprod_{j\in J} s_{i_j}$.
    In other words, $u\le w$ if and only if any reduced expression for~$w$ contains a reduced expression for~$u$ as a subexpression.
    In particular, the restriction of the strong Bruhat order on~$W(M)$
    to~$W_K(M)$ coincides with the strong Bruhat order on~$W_K(M)$ for any~$K\subset I$.
     \item\label{prop:Bruhat order.b} If~$u\le w\in W(M)$
     then there exists a chain $x_0=u<x_1<\cdots<x_r=w$ in~$W(M)$
     such that $\ell(x_i)=\ell(u)+i$, $0\le i\le r$.
     \item\label{prop:Bruhat order.c} If~$J\in\mathscr F(M)$
     then for any $w,w'\in W_J(M)$, $w<w'$ if and only
     if $w_\circ^J w'<w_\circ^J w$ and if and only if $w'w_\circ^J<
     ww_\circ^J$.
\end{enmalph}
\end{proposition}
Given $u\in W(M)$, denote \plink{dar}$\downarrow u=\{ w\in W(M)\,:\, w\le u\}$ and $\uparrow u=\{ w\in W(M)\,:\, u\le w\}$.

\begin{proposition}[\cite{He09}*{Lemma~1 and Corollary~1}, \cite{K14}*{Lemma~2 and Proposition~8}]\label{prop:Bruhat order *}
Let $w,w'\in W(M)$.
\begin{enmalph}
    \item\label{prop:Bruhat order *.a}
    $u\le w$, $u'\le w'$
    implies that $u\star u'\le w\star w'$;
    \item\label{prop:Bruhat order *.b}
    $(\downarrow w)(\downarrow w'):=
    \{ uu'\,:\, u\in \downarrow w,\,u'\in\downarrow u'\}=
    \downarrow w\star w'$, that is, $w\star w'$ is
    the unique maximal element of $\{ uu'\,:\, u\le w,u'\le w'\}$
    with respect to the strong Bruhat order. Moreover,
    $w\star w'=u\times w'=w\times u'$ for some~$u\le w$, $u'\le w'$.
\end{enmalph}
\end{proposition}
\subsection{Idempotents in Hecke monoids}\label{subs:idemp}
First, note the following characterization of the~$w_\circ^J$, $J\in\mathscr F(M)$, in Hecke monoids.
\begin{lemma}\label{lem:char w_0 monoid}
Suppose that~$J\in\mathscr F(M)$. The following
are equivalent for $w\in W_J(M)$:
\begin{enmroman}
 \item\label{lem:char w_0 monoid.i} $w=w_\circ^J$;
 \item\label{lem:char w_0 monoid.ii} $s_i\star w=w$ for all $i\in J$;
 \item\label{lem:char w_0 monoid.iv} $x\star w=w$ for all $x\in W_J(M)$;
 \item\label{lem:char w_0 monoid.iii} $w\star s_i=w$ for all $i\in J$;
 \item\label{lem:char w_0 monoid.v} $w\star x=w$ for all $x\in W_J(M)$.
\end{enmroman}
In particular, $w_\circ^J$ is an idempotent in~$(W(M),\star)$.
\end{lemma}
\begin{proof}
We may assume, without loss of generality, that~$J=I$. Clearly,
\ref{lem:char w_0 monoid.ii} (respectively, \ref{lem:char w_0 monoid.iii})
is equivalent to~\ref{lem:char w_0 monoid.iv} (respectively, \ref{lem:char w_0 monoid.v}).
Since
$\ell(s_iw_\circ^I),\ell(w_\circ^Is_i)<\ell(w_\circ^I)$ for all~$i\in I$, it follows from
Proposition~\ref{prop:prod *} that~$s_i\star w_\circ^I=w_\circ^I=
w_\circ^I\star s_i$
for all~$i\in I$ and so~\ref{lem:char w_0 monoid.i} implies~\ref{lem:char w_0 monoid.ii} and~\ref{lem:char w_0 monoid.iii}. Finally, if $x\star w=w$ for all~$x\in W(M)$
then $w_\circ^I\star w=w$. Yet since $w_\circ^I\star x=w_\circ^I$
for all~$x\in W(M)$, it follows that $w_\circ^I\star w=w_\circ^I$
hence \ref{lem:char w_0 monoid.iv} implies~\ref{lem:char w_0 monoid.i}.
Similarly, \ref{lem:char w_0 monoid.v} implies~\ref{lem:char w_0 monoid.i}.
\end{proof}
We will now prove that
$$
\{ w_\circ^J\,:\, J\in\mathscr F(M)\}=
\{ x\in (W(M),\star)\,:\, x\star x=x\}.
$$
Note that if $\ell(uv)=\ell(u)+\ell(v)$ then $uv=u\star v$ and, following~\cite{K14}, we will abbreviate
that equality as \plink{times}$u\times v$.

\begin{proposition}[Absorption property]\label{prop:absorb prop}
Suppose that $u,v\in W(M)$ satisfy $\ell(u\star v)=\ell(u)$
(respectively, $\ell(u\star v)=\ell(v)$).
Then $u\star v=u$ and $\supp v\subset \supp u$
(respectively, $u\star v=v$  and $\supp u\subset \supp v$).
In particular, if $\ell(u\star v)=\ell(u)=\ell(v)$ then
$\supp u=\supp v\in\mathscr F(M)$ and
$u=v=w_\circ^{\supp u}$.
\end{proposition}
\begin{proof}
The argument is by induction on length of~$v$.
If $v=1$ then there is nothing to prove.

If~$v=s_j$ and $\ell(u\star s_j)=\ell(u)$ then~$u\star s_j=u$
by~\eqref{eq:prod *}. Also, if~$j\notin\supp u$ then $\ell(u\star s_j)=\ell(us_j)>\ell(u)$
which contradicts~\eqref{eq:prod *}.

For the inductive step, write $v=s_j\times v'$
for some~$j\in\supp v$ and~$v'\in W(M)$. Then
$$\ell(u)=\ell(u\star v)=\ell((u\star s_j)\star v')
\ge \ell(u\star s_j)\ge \ell(u)
$$
by Lemma~\ref{lem:monoid len}.
Thus, $\ell(u\star s_{j})=\ell(u)$ and so $u\star s_j=u$ and $j\in\supp u$
by the induction base. Then
$u\star v=u\star v'=u$ and $\supp v'\subset\supp u$ by the induction hypothesis,
whence~$\supp v=\{j\}\cup\supp v'\subset \supp u$.

The second assertion is proved verbatim.

Finally, if $\ell(u\star v)=\ell(u)=\ell(v)$ then
$u=v$ and $\supp u=\supp v=:J$. We claim that
$u\star s_j=u$ for all~$j\in J$. Suppose not. Write
$u=s_{j_1}\times\cdots\times s_{j_r}$, $r=\ell(u)$.
Then~$J=\{j_1,\dots,j_r\}$. Let~$1\le t\le r$ be minimal such that $\ell(u\star s_{j_t})>\ell(u)$.
Then by Lemma~\ref{lem:monoid len},
$\ell(u)=\ell(u\star u)\ge \ell(u\star s_{j_t})>\ell(u)$, which is a
contradiction.

In particular, for all~$w\in W_J(M)$ we have
$u\star w=u$
and so
$\ell(u)=\ell(u\star w)\ge \ell(w)$ by Lemma~\ref{lem:monoid len}. Thus, $J\in\mathscr F(M)$. It remains to apply Lemma~\ref{lem:char w_0 monoid}.
\end{proof}

\begin{corollary}\label{cor:max elts}
\begin{enmalph}
\item\label{cor:max elts.a} Let~$G$ be
a finite subsemigroup of~$(W(M),\star)$. Then~$\supp G\in\mathscr F(M)$, $w_\circ^{\supp G}\in G$ and is the unique element of~$G$ of
maximal length.
\item\label{cor:max elts.b} $w\in (W(M),\star)$ is an idempotent if
and only if $\supp w\in\mathscr F(M)$ and $w=w_\circ^{\supp w}$.
\item\label{cor:max elts.c} Let~$J\subset I$. Then~$J\in\mathscr F(M)$
if and only if $(W(M),\star)$ contains an idempotent~$w$ with~$\supp w=J$.
\end{enmalph}
\end{corollary}
\begin{proof}
Since~$G$ is finite, it contains an element~$u$ of maximal length.
Suppose that $v\in G$  with~$\ell(v)=\ell(u)$.
By Lemma~\ref{lem:monoid len}, $\ell(u\star v)=\ell(u)=\ell(v)$ and so by Proposition~\ref{prop:absorb prop}, $\supp u\in\mathscr F(M)$ and~$u=v=w_\circ^{\supp u}$. If $w\in G$
then, as $\ell(w\star u)\ge \ell(u)$ by Lemma~\ref{lem:monoid len}, $\ell(w\star u)=\ell(u)$ and hence $\supp w\subset \supp u$ by Proposition~\ref{prop:absorb prop}. It follows that~$\supp G=\supp u$.

Part~\ref{cor:max elts.b}
follows from~\ref{cor:max elts.a} since
the subset consisting of an idempotent
element is a subsemigroup of~$W(M)$.

The forward direction in part~\ref{cor:max elts.c} is established in Lemma~\ref{lem:char w_0 monoid} while the converse is established in Proposition~\ref{prop:absorb prop}.
\end{proof}

Given~$w\in W(M)$, denote\plink{DL(w)}
$$
D_L(w)=\{ i\in I\,:\, \ell(s_iw)<\ell(w)\},\quad
D_R(w)=\{ i\in I\,:\, \ell(ws_i)<\ell(w)\}.
$$
Clearly, $D_R(w)=D_L(w^{-1})$.
The following is apparently well-known (see for example~\cite{BjBr}*{Proof of Lemma~3.2.3}). We
provide a proof here since the argument is quite elegant in the setting of Hecke monoids.
\begin{lemma}\label{lem:left desc}
For any~$w\in W(M)$, $\{
x\in W(M)\,:\, x\star w=w\}$ is a finite
submonoid of~$(W(M),\star)$ and is
equal to $(W_{D_L(w)}(M),\star)$. In particular, $D_L(w)\in\mathscr F(M)$.
\end{lemma}
\begin{proof}
Let $G(w)=\{x\in W(M)\,:\, x\star w=w\}$
which is manifestly a submonoid of $(W(M),\star)$.
Let~$J(w)=\supp G(w)$.
By Lemma~\ref{lem:monoid len}, $\ell(w)=
\ell(x\star w)\ge \ell(x)$ for all~$x\in G(w)$, hence~$G(w)$ is finite. By Corollary~\partref{cor:max elts.a}, $J(w)\in\mathscr F(M)$.
By
Proposition~\ref{prop:prod *}, $\{s_i\,:\, i\in D_L(w)\}\subset G(w)$ whence~$W_{D_L(w)}(M)
\subset G(w)$ and in particular~$D_L(w)\in\mathscr F(M)$.
If $i\in J(w)$ then
$s_i\star w=s_i\star (w_\circ^{J(w)}\star w)=(s_i\star w_\circ^{J(w)})\star w=w_\circ^{J(w)}\star w=w$ by Lemma~\ref{lem:char w_0 monoid},
whence
$i\in D_L(w)$ by Proposition~\ref{prop:prod *}.
Thus, $J(w)\subset  D_L(w)$ and so~$G(w)\subset
W_{D_L(w)}(M)$.
\end{proof}
\begin{remark}
In view of Proposition~\ref{prop:absorb prop},
$\{x\in W(M)\,:\, x\star w=w\}=
\{x \in W(M)\,:\, \ell(x\star w)=\ell(w)\}$.
\end{remark}

\subsection{Divisibility, longest elements and Coxeter elements}\label{subs:w0J}
We say that $X\in\Br^+(M)$ is a left (respectively, right) divisor
of~$Y\in\Br^+(M)$ if $Y=XU$ (respectively, $Y=VX$) for some~$U\in\Br^+(M)$
(respectively, $V\in\Br^+(M)$). Since~$\Br^+(M)$ is cancellative,
such an element~$U$ (respectively, $V$), if exists, is unique
and will be denoted by~$(Y:X)_l$ (respectively, $(Y:X)_r$). The
following classical results will be often used in the sequel.
\begin{proposition}[\cite{BrSa}*{Lemma~5.1, Propositions~5.7,  Theorem~7.1}
and~\cite{Del}*{Theorem~4.21}]
\label{prop:fund elts BrSa}
Let~$J\in\mathscr F(M)$. Then
\begin{enmalph}
\item\label{prop:fund elts BrSa.a} $T_{w_\circ^J}$ is ${}^{op}$-invariant;
  \item\label{prop:fund elts BrSa.b} $T_{w_\circ^J}$ is the left and the right least common multiple
  of the~$T_j$, $j\in J$, that is, $T_{w_\circ^J}$ is left (respectively,
  right) divisible by all the $T_j$, $j\in J$ and is a left (respectively, right)
  divisor of every element of~$\Br^+(M)$ with that property;
\item\label{prop:fund elts BrSa.c'}
$X\in\Br^+_J(M)$ is left divisible by~$T_{w_\circ^J}$ if and
only if it is right divisible by~$T_{w_\circ^J}$;
\item\label{prop:fund elts BrSa.0}
If~$J=J_1\cup J_2$ with~$J_1$, $J_2$ orthogonal
then $T_{w_\circ^{J}}=T_{w_\circ^{J_1}}T_{w_\circ^{J_2}}$;
\item\label{prop:fund elts BrSa.c} There is a unique
involutive diagram automorphism~$\Sigma_J$ of~$\Br^+_J(M)$ such that
$X T_{w_\circ^J}=T_{w_\circ^J} \Sigma_J(X)$ for all~$X\in\Br^+_J(M)$;
\item\label{prop:fund elts BrSa.d}
The center of~$\Br^+_J(M)$ is generated by $T_{w_\circ^J}^k$ where
$k\in\{1,2\}$ is the order of~$\Sigma_J$;
\item\label{prop:fund elts BrSa.e}
$T_{w_\circ^J}$ is the unique element of~$\SQF^+(M)\cap \Br^+_J(M)$ of
maximal length and every square free element of~$\Br^+_J(M)$ is a
left and a right divisor of~$T_{w_\circ^J}$.
\end{enmalph}
\end{proposition}
The involution~$\Sigma_J$ for~$J\in\mathscr F(M)$ connected is
trivial except if~$M_J$ is of type~$A_n$, $n\ge 1$, $D_{n+1}$
with~$n$ even or~$E_6$. Note that~$\Br^+(D_{n+1})$ admits a
non-trivial diagram automorphism for all~$n\ge 3$, yet~$\Sigma$
is trivial if~$n$ is odd; likewise,
$\Br^+(F_4)$ admits a diagram automorphism, yet~$\Sigma$ is also trivial.

Given~$X\in \Br^+(M)$, define \plink{I(X)}$D_L(X)=\{ i\in I\,:\, \text{$T_i$ is
a left divisor of~$X$}\}$.
Let $i\in D_L(X)$. Then $X=T_i X'$ for some~$X'\in \Br^+(M)$
with~$\ell(X')=\ell(X)-1$ and so $$\pi_M^\star(X)=s_i\star \pi_M^\star(X')=
s_i\star s_i\star\pi^\star_M(X')=s_i\star \pi^\star_M(X),
$$
whence
$i\in D_L(\pi^\star_M(X))$ by Proposition~\ref{prop:prod *}.
Thus, $D_L(X)
\subset D_L(\pi^\star_M(X))
\in\mathscr F(M)$ by Lemma~\ref{lem:left desc}.
Then $X$ is left
divisible by $T_{w_\circ^{D_L(X)}}$ by Proposition~\partref{prop:fund elts BrSa.b}.
\begin{remark}
It should be noted that~$D_L(X)$ can be a proper subset of~$D_L(\pi^\star_M(X))$. For example, for~$X=T_1^3T_2^2 T_1 T_3^2T_2T_1\in \Br^+(A_3)$
we have~$D_L(X)=\{1\}$ while~$D_L(\pi^\star_{A_3}(X))=D_L(w_\circ^I)=I$. 
\end{remark}

Given~$X\in \Br^+(M)$, define inductively $D_0(X)=D_L(X)$ and $$
D_j(X)=D_L((X:\ascprod_{0\le k\le j-1} T_{w_\circ^{D_k(X)}})_l),\qquad  j\ge 1.
$$
Clearly, $D_k(X)=\emptyset$ for~$k\gg0$.
\begin{proposition}[\cite{BrSa}*{Theorem~6.3}]\label{prop:Normal form}
Let $X,Y\in\Br^+(M)$. Then
\begin{enmalph}
\item $X=\ascprod_{j\in\ZZ_{\ge 0}} T_{w_\circ^{D_j(X)}}$ (this expression for~$X$ is called its {\em normal form});
\item $X=Y$ if and only if~$D_j(X)=D_j(Y)$ for
all~$j\ge 0$.
\end{enmalph}
\end{proposition}

\begin{definition}\label{defn:weakly orthogonal}
Let~$M\in\Cox I$. We say that~$J,K\subset I$
are {\em weakly orthogonal} if 
$J\setminus K$ is orthogonal to~$K$ and $K\setminus J$ is orthogonal to~$J$. 
\end{definition}
In particular, every subset~$J$ of~$I$ is weakly
orthogonal to itself and if $J$, $K$ are orthogonal then they are weakly orthogonal.
The following Lemma is immediate from Proposition~\partref{prop:fund elts BrSa.0} and
Lemma~\partref{lem:orth factors.a}.
\begin{lemma}\label{lem:weakly orthogonal}
Let~$M\in\Cox I$, $J,K\in\mathscr F(M)$. If~$J$ and~$K$
are weakly orthogonal then $T_{w_\circ^J}T_{w_\circ^K}=
T_{w_\circ^K}T_{w_\circ^J}$. 
\end{lemma}

A {\em Coxeter element} of $\Br^+_J(M)$ or~$W(M)$
is any $C\in \Br^+_J(M)$ (respectively, $c\in W(M)$)
with~$\supp C=J$ (respectively, $\supp c=J$) and~$\ell(C)=|J|$
(respectively, $\ell(c)=|J|$). In the sequel, we will often consider
special Coxeter elements corresponding to an interval~$J=[a,b]\subset I$,
namely \plink{cab}$\cx ab=\ascprod_{a\le i\le b} s_i$, $\cxr ab=\dscprod_{a\le i\le b}s_i$, $\Cx ab=T_{\cx ab}$ and~$\Cxr ab=T_{\cxr ab}=(\Cx ab)^{op}$. We will use the convention that~$\cx ij=\cxr ij=1$ if~$i>j$
and similarly for~$\Cx ij$ and~$\Cxr ij$.

It is well-known (see e.g.~\cite{Bou}*{Ch. V, \S6})
that if~$J\in\mathscr F(M)$ then all Coxeter elements~$c\in W_J(M)$
are conjugate and of the same order~\plink{h(M)}$h(M_J)$, called the {\em Coxeter number}
of~$W_J(M)$. The Coxeter number is even for all finite types
except~$I_2(2m+1)$, $m>0$ and $A_{2m}$.
Note also that if~$J\subset I$ is self-orthogonal then $T_{w_\circ^J}$ is the unique Coxeter element of~$W_J(M)$.
The following is established in~\cite{BrSa}.
\begin{proposition}[\cite{BrSa}*{\S5.8}]\label{prop:Coxeter splitting}
Let~$M$ be a Coxeter matrix and let~$J\in\mathscr F(M)$. Then for any Coxeter element~$C\in\Br^+_J(M)$
\begin{enmalph}
\item\label{prop:Coxeter splitting.a}
$T_{w_\circ^J}^2=C^{h(M_J)}$;
\item\label{prop:Coxeter splitting.b}
If $\Sigma_J$ from Proposition~\partref{prop:fund elts BrSa.c} is trivial then~$h(M_J)$ is even and $T_{w_\circ^J}=
C^{h(M_J)/2}$;
\item\label{prop:Coxeter splitting.c}
If~$M_J$ is irreducible and~$J=J'\cup J''$
is a partition of~$J$ into disjoint non-empty self-orthogonal subsets then
$$T_{w_\circ^J}=\brd{T_{w_\circ^{J'}}T_{w_\circ^{J''}}}{h(M_J)}
=\brd{T_{w_\circ^{J''}}T_{w_\circ^{J'}}}{h(M_J)}.
$$
\end{enmalph}
\end{proposition}
\subsection{Parabolic elements and 
submonoids they generate}\label{subs:parab elts}
Given~$J\subset K\in\mathscr F(M)$, denote
\plink{wJ;K}$w_{J;K}=w_\circ^J w_\circ^K\in W_K(M)$. Such elements are called {\em
$K$-parabolic}. If~$M$ is of finite type, we abbreviate~$w_J:=w_{J;I}$ and 
call such elements parabolic.
Analogously, we call elements $T_{w_{J;K}}
$ of~$\Br^+_K(M)$ $K$-parabolic.
The following is immediate.
\begin{lemma}\label{lem:prod orth parab}
Let~$K,K'\in\mathscr F(M)$ be orthogonal.
Then $w_{J;K}\times w_{J';K'}=w_{J\cup J'; K\cup K'}$
and hence $T_{w_{J;K}}T_{w_{J';K'}}=T_{w_{J\cup J';K\cup K'}}$
for all~$J\subset K$, $J'\subset K'$.
\end{lemma}

Let~$K\in\mathscr F(M)$ and let~$K_1,\dots,K_r$, $r\ge 1$ be its connected components. 
We define~\plink{Psbm}$\mP(W_K(M),\star)$ (respectively,
$\mP(\Br^+_K(M))$ as the submonoid of~$(W_K(M),\star)$
(respectively, $\Br^+_K(M)$) generated by
the~$w_{J;K_i}$ (respectively, $T_{w_{J;K_i}}$),
$1\le i\le r$, $J\subset K_i$. 
It follows from Lemma~\ref{lem:prod orth parab}
that~$\mP(W_K(M),\star)$ coincides with the submonoid
of~$(W_K(M),\star)$ generated by the~$w_{J;K}$,
$J\subset K$. However, if~$K$ is not connected, $\mP(\Br^+_K(M))$
is rather different from the submonoid
of~$\Br^+_K(M)$ generated by the $T_{w_{J;K}}$,
$J\subset K$. For instance,
the former contains $T_{w_\circ^{K_1}}^{a_1}\cdots  T_{w_\circ^{K_r}}^{a_r}$ for all~$a_1,\dots,a_r\in\mathbb Z_{\ge 0}$
while the latter contains only~$T_{w_\circ^{K}}^a=
T_{w_\circ^{K_1}}^a\cdots T_{w_\circ^{K_r}}^a$, $a\in\mathbb Z_{\ge 0}$.

\begin{lemma}\label{lem:wK absorption}
Let~$J\in\mathscr F(M)$ and let~$K\subset J$.
Then
\begin{align*}
&s_i\star w_{K;J}=\begin{cases} w_{K;J},& i\in J\setminus K\\
 s_i\times w_{K;J},& i\in K\cup (I\setminus J).
 \end{cases}\\
 &w_{K;J}{}^{-1}\star s_i=\begin{cases} w_{K;J}{}^{-1},& i\in J\setminus K\\
 w_{K;J}{}^{-1}\times s_i,& i\in K\cup (I\setminus J).
 \end{cases}
 \end{align*}
Moreover, $\ell(ww_{K;J})=\ell(w)+\ell(w_{K;J})$
and $\ell(w_{K;J}{}^{-1}w)=
\ell(w_{K;J}^{-1})+\ell(w)$
for all $w\in W_K(M)$.
\end{lemma}
\begin{proof}
We only prove the assertions for~$w_{K;J}$, as the result for~$w_{K;J}^{-1}$ follows by applying~${}^{op}$.
The first claim is immediate from Proposition~\ref{prop:prod *} and Lemma~\ref{lem:extend supp} for~$i\in I\setminus J$ so we may
assume without loss of generality that~$J=I$.
We have $s_iw_{K}=(s_iw_\circ^K)w_\circ^I$ and
so~$\ell(s_i w_K)=\ell(w_\circ^I)-\ell(s_i w_\circ^K)$ for all~$i\in I$
by~\eqref{eq:ell w w0}.
Since $\ell(s_iw_\circ^K)=\ell(w_\circ^K)-1$ if $i\in K$ and
$\ell(s_i w_\circ^K)=\ell(w_\circ^K)+1$ if $i\in I\setminus K$,
it follows that~$\ell(s_i w_K)=\ell(w_K)+1$ if~$i\in K$ and
$\ell(s_i w_K)=\ell(w_K)-1$ if~$i\in I\setminus K$.
It remains to apply Proposition~\ref{prop:prod *}.
Furthermore, given $w\in W_K(M)$ we have
$\ell(w w_\circ^K)=\ell(w_\circ^K)-\ell(w)$ by~\eqref{eq:ell w w0} hence $\ell(ww_K)=\ell(w_\circ^I)-\ell(ww_\circ^K)
=\ell(w_\circ^I)-\ell(w_\circ^K)+\ell(w)=\ell(w_K)+\ell(w)$.
\end{proof}

Quite surprisingly, it turns out (see~\cite{BK07}*{Proposition~3.20}, \cite{He09} and Proposition~\ref{prop:submonoid *}) that
$\mP(W_K(M),\star)=\{w_{J;K}\,:\, J\subset K\}$.  However,
the submonoid $\mP(\Br^+_K(M))$ of~$\Br^+(M)$
is manifestly infinite.

\section{General properties of homomorphisms
of Hecke and Artin monoids}\label{sec:Gen homs}

Throughout this chapter, we denote
standard generators of $(W(M'),\star)$
(respectively, of $\Br^+(M')$) corresponding
to a Coxeter matrix~$M'$ over~$I'$
by~$s'_i$ (respectively, $T'_i$), $i\in I'$ and so on.
We will often use the obvious fact that a homomorphism of Artin monoids extends to a homomorphism of their ambient Artin groups.

Let~\plink{A C H}$\mathscr A$ (respectively, $\mathscr C$, $\mathscr H$) be the category whose objects are Coxeter matrices
and morphisms are homomorphisms of corresponding Artin monoids (respectively, Coxeter groups, Hecke monoids).
Parabolic submonoids and subgroups are, naturally, subobjects in these categories. By 
Lemma~\ref{lem:free product}
all these categories admit finite products and coproducts via, respectively, $(M,M')\mapsto M\times M'$ and
$(M,M')\mapsto M\coprod M'$, $M\in\Cox I$, $M'\in\Cox{I'}$ (see Remark~\ref{rem:prod coprod}).

\subsection{Homomorphisms of Hecke monoids}\label{subs:Hecke homs}
Given $M'\in\Cox{I'}$, $M\in\Cox I$ and~$\phi\in\Hom_{\mathscr H}(M',M)$, we denote by \plink{[phi]}$[\phi]$
the map $I'\to \mathscr P(I)$ defined by $i\mapsto \supp\phi(s'_i)$, $i\in I'$ and extend it
to a map $[\phi]:\mathscr P(I')\to \mathscr P(I)$ via
$[\phi](J')=\bigcup_{j\in J'} [\phi](j')$, $J'\subset I'$.
We will usually abbreviate~$[\sigma]$ as~$\sigma$ for diagram automorphisms.
\begin{definition}\label{def:types heck hom}
We say that~$\phi\in\Hom_{\mathscr H}(M',M)$ is:
\begin{itemize}
\item[-] {\em disjoint} if $[\phi](i)\cap[\phi](j)=
\emptyset$ for all $i\not=j\in I'$;
\item[-] {\em fully supported} if $[\phi](I')=I$;
\end{itemize}
\end{definition}
\begin{lemma}\label{lem:[phi]comp}
Let~$M\in\Cox I$, $M'\in\Cox{I'}$ and~$M''\in\Cox{I''}$.
\begin{enmalph}   
\item\label{lem:[phi]comp.a}
$\supp\phi(x)=[\phi](\supp x)$
for all~$\phi\in\Hom_{\mathscr H}(M',M)$ and for
all~$x\in W(M')$;
\item\label{lem:[phi]comp.b}
$[\phi\circ\phi']=[\phi]\circ[\phi']$ as maps $\mathscr P(I'')\to
\mathscr P(I)$ for any~$\phi\in\Hom_{\mathscr H}(M',M)$
and~$\phi'\in\Hom_{\mathscr H}(M'',M')$;
\item
   \label{lem:[phi]comp.c}
   If~$\phi\in \Hom_{\mathscr H}(M',M)$ is
    disjoint and $[\phi](i)\not=\emptyset$ for all~$i\in I'$
    then~$[\phi]:\mathscr P(I')\to
    \mathscr P(I)$ is
    injective.
\end{enmalph}
\end{lemma}
\begin{proof}
Since~$\supp x\star y=\supp x\,\cup\,\supp y$ for all $x,y\in (W(M),\star)$, we have for all~$x'\in (W(M'),\star)$ $$\supp\phi(x')=\bigcup_{j\in\supp x'}
\supp \phi(s'_j)=\bigcup_{j\in\supp x'} [\phi](j)=
[\phi](\supp x'),$$
which proves~\ref{lem:[phi]comp.a}.
To prove part~\ref{lem:[phi]comp.b},
note that by part~\ref{lem:[phi]comp.a} we have 
for all~$x''\in W(M'')$
$$
[\phi\circ\phi'](\supp x'')=
\supp(\phi\circ\phi')(x'')
=[\phi](\supp\phi'(x''))
=[\phi]([\phi'](\supp x'')).
$$
Since $\supp: W(M'')\to \mathscr P(I'')$
is surjective, the assertion follows. 

To prove~\ref{lem:[phi]comp.c}, suppose that~$[\phi](J)=[\phi](J')$ for some~$J\not=J'$. We may assume, without loss of generality,
that $J'\not\subset J$. Let~$j\in J\setminus J'$. Then $\emptyset\not=[\phi](j)\subset [\phi](J)=[\phi](J')=\bigcup_{j\in J'} [\phi](j')$ which is a contradiction since~$[\phi](j)\cap[\phi](j')=\emptyset$ for all~$j'\in J'$.
\end{proof}

Our present aim is to describe~$\Hom_{\mathscr H}(M',M)$ for any Coxeter matrices~$M'$ and~$M$.
We begin with the following observation.
\begin{lemma}\label{lem:Hecke hom w0J}
Let $M\in\Cox I$, $M'\in\Cox{I'}$ and~$\phi\in\Hom_{\mathscr H}(M',M)$.
Then~$[\phi](J')\in\mathscr F(M)$ and $\phi(w_\circ^{J'})=w_\circ^{[\phi](J')}$
for all~$J'\in\mathscr F(M')$. In particular,
$[\phi](i)\in \mathscr F(M')$ for all~$i\in I'$.
\end{lemma}
\begin{proof}
Since~$w_\circ^{J'}$ is an idempotent in~$(W(M'),\star)$,
$\phi(w_\circ^{J'})=w_\circ^J$ for some~$J\in\mathscr F(M)$
by Corollary~\partref{cor:max elts.b}. Clearly~$J\subset [\phi](J')$.
Since
$s'_j\star w_\circ^{J'}=w_\circ^{J'}$ for all~$j\in J'$
by Lemma~\ref{lem:char w_0 monoid}, it follows that
$\phi(s'_j)\star w_\circ^J=w_\circ^J$ for all~$j\in J'$ and
so $[\phi](j)=\supp\phi(s'_j)\subset J$ for all~$j\in J'$ by Proposition~\ref{prop:absorb prop}. Thus, $[\phi](J')\subset J$. The last assertion follows
by taking~$J'=\{i\}$, $i\in I'$.
\end{proof}

Let~$J,K\in\mathscr F(M)\setminus\{\emptyset\}$. Note that~$J\cup K\in\mathscr F(M)$ if and only if
$$
G_{J,K}:=\la w_\circ^J,w_\circ^K\ra = \{
\brd{w_\circ^J\star w_\circ^K\star}{t},\brd{w_\circ^K\star w_\circ^J\star}{t}\,:\,
t\in\ZZ_{\ge 0}\}\subset (W_{J\cup K}(M),\star)
$$
is finite. Indeed, the forward direction is obvious while if~$G_{J,K}$
is finite then $\supp G_{J,K}=J\cup K\in\mathscr F(M)$ by Corollary~\partref{cor:max elts.a}.

We will now define a map~\plink{mu M}$\mu_M:\mathscr F(M)\times
\mathscr F(M)\to \mathbb Z_{>0}\cup\{\infty\}$.
If~$J\cup K\notin\mathscr F(M)$, set
$\mu_M(J,K)=\mu_M(K,J)=\infty$. Otherwise,
by Corollary~\partref{cor:max elts.a},
$G_{J,K}$ contains a unique element of maximal length, namely
$w_\circ^{J\cup K}$, and so either
$$w_\circ^{J\cup K}=\brd{w_\circ^J\star w_\circ^K\star}{m}
$$ or
$$
w_\circ^{J\cup K}=\brd{w_\circ^K\star w_\circ^J\star}{m}
$$
for some~$m> 0$.  Note that if~$m$ is even then
both equalities hold at the same time since~$w_\circ^{J\cup K}$,
$w_\circ^K$ and~$w_\circ^J$ are ${}^{op}$-invariant.
If say the first equality holds for~$m$ odd
then Lemma~\ref{lem:char w_0 monoid} implies that
$$
w_\circ^{J\cup K}=w_\circ^K\star w_\circ^{J\cup K}=
w_\circ^K\star
\brd{w_\circ^J\star w_\circ^K\star}{m}
=\brd{w_\circ^{K}\star w_\circ^{J}\star}{m+1}.
$$
Thus, for any~$J,K\in\mathscr F(M)$ such that~$J\cup K\in\mathscr F(M)$,
$$
\mu_M(J,K):=\min\{ m\in\ZZ_{>0}\,:\, \brd{w_\circ^J\star w_\circ^K\star}{m}=
w_\circ^{J\cup K}\}
$$
is well-defined. By the above, $|\mu_M(J,K)-\mu_M(K,J)|\le 1$.
\begin{example}
In~$(W(A_4),\star)$ we have
$$
s_3s_4s_3\star s_2s_4\star s_3s_4s_3=
w_\circ^{\{2,3,4\}},
$$
while
$$
s_2s_4\star s_3s_4s_3\star s_2s_4=
s_2s_3s_4s_3s_2,\qquad
(s_2s_4\star s_3s_4s_3)^{\star 2}=
w_\circ^{\{2,3,4\}}.
$$
Thus, $\mu_{A_4}(\{3,4\},\{2,4\})=3$,
$\mu_{A_4}(\{2,4\},\{3,4\})=4$.
\end{example}

For any~$M\in\Cox I$ and~$M'\in\Cox{I'}$
define\plink{LM'M}
$$
\Lambda(M',M):=\{\xi:I'\to \mathscr F(M)\,:\,
\max(\mu_M(\xi(i),\xi(j)),
\mu_M(\xi(j),\xi(i)))\le
m'_{ij},\,\forall\, i\not=j\in I'\}.
$$
\begin{theorem}\label{thm:Hom Heck Mon}
Let~$M'\in\Cox{I'}$, $M\in\Cox I$.
The assignments~$\phi\mapsto[\phi]$ define a bijection $\Hom_{\mathscr H}(M',M)\to \Lambda(M',M)$. 
\end{theorem}
\begin{proof}
We need the following
\begin{lemma}\label{lem:[phi]Lam M'M}
For any~$\phi\in\Hom_{\mathscr H}(M',M)$,
$[\phi]\in\Lambda(M',M)$.
\end{lemma}
\begin{proof}
Note first that~$[\phi]$ is a map~$I'\to \mathscr F(M)$ by Lemma~\ref{lem:Hecke hom w0J}.

Let~$i\not=j\in I'$.
If~$m'_{ij}=\infty$ or at least one of~$[\phi](i)$, $[\phi](j)$ is empty, then the inequality
\begin{equation}\label{eq:ineq Lam M'M}
\max(\mu_M([\phi](i),[\phi](j)),
\mu_M([\phi](j),[\phi](i)))\le m'_{ij}
\end{equation}
is trivial.
Otherwise, the submonoid $G'_{i,j}=\la s'_i,s'_j\ra$ of~$(W(M'),\star)$ is finite, its
longest element being $\brd{s'_i\star s'_j\star}{m'_{ij}}=
\brd{s'_j\star s'_i\star}{m'_{ij}}$. Then
$G_{i,j}=\phi(G'_{i,j})$ is a finite submonoid of~$(W(M),\star)$ contained
in~$G_{[\phi](i),[\phi](j)}$
and~$\supp G_{i,j}=[\phi](i)\cup [\phi](j)$. By Corollary~\partref{cor:max elts.a}, $[\phi](i)\cup [\phi](j)\in\mathscr F(M)$
and the longest element of~$G_{i,j}$ is
$$
w_\circ^{[\phi](i)\cup [\phi](j)}=\brd{w_\circ^{[\phi](i)}\star
w_\circ^{[\phi](j)}\star}{\mu_M([\phi](i),[\phi](j))}=\brd{w_\circ^{[\phi](j)}\star
w_\circ^{[\phi](i)}\star}{\mu_M([\phi](j),[\phi](i))}.
$$
Also, since $\brd{s'_i\star s'_j\star}{m'_{ij}}=
w_\circ^{\{i,j\}}$, by Lemma~\ref{lem:Hecke hom w0J} we have
$$
\phi(\brd{s'_i\star s'_j\star}{m'_{ij}})=
\phi(\brd{s'_j\star s'_i\star}{m'_{ij}})=
\phi(w_\circ^{\{i,j\}})=w_\circ^{[\phi](i)\cup
[\phi](j)}
$$
whence
$$
w_\circ^{[\phi](i)\cup [\phi](j)}=
\brd{w_\circ^{[\phi](i)}\star
w_\circ^{[\phi](j)}\star}{m'_{ij}}
=\brd{w_\circ^{[\phi](j)}\star
w_\circ^{[\phi](i)}\star}{m'_{ij}}.
$$
The inequality~\eqref{eq:ineq Lam M'M} is now immediate.
\end{proof}
The next step is to construct 
a map~$\Lambda(M',M)\to \Hom_{\mathscr H}(M',M)$.
\begin{lemma}\label{lem:Lam M'M HomM'M} 
For any~$\xi\in\Lambda(M',M)$,
the assignments $s'_i\mapsto w_\circ^{\xi(i)}$, $i\in I'$
define
\plink{Theta xi}$\Theta_\xi\in\Hom_{\mathscr H}(M',M)$.
\end{lemma}
\begin{proof}
Let~$\xi\in \Lambda(M',M)$.
By definition
of~$\mu_M$ we have for all $i\not=j\in I'$ with $m'_{ij}<\infty$
$$
\brd{w_\circ^{\xi(i)}\star w_\circ^{\xi(j)}\star}{m'_{ij}}=w_\circ^{\xi(i)\cup \xi(j)}\star x_{i,j}
$$
where
$$
x_{i,j}=\begin{cases}\brd{w_\circ^{\xi(i)}\star w_\circ^{\xi(j)}}{m'_{ij}-
\mu_M(\xi(i),\xi(j))},& \text{$\mu_M(\xi(i),\xi(j))$ is even,}\\
\brd{w_\circ^{\xi(j)}\star w_\circ^{\xi(i)}}{m'_{ij}-
\mu_M(\xi(i),\xi(j))},& \text{$\mu_M(\xi(i),\xi(j))$ is odd}.
\end{cases}
$$
Since $\supp x_{i,j}\subset \xi(i)\cup \xi(j)$, it follows from Lemma~\ref{lem:char w_0 monoid} that
$$
\brd{w_\circ^{\xi(i)}\star w_\circ^{\xi(j)}\star}{m'_{ij}}=w_\circ^{\xi(i)\cup \xi(j)}.
$$
Thus, $\brd{w_\circ^{\xi(i)}\star w_\circ^{\xi(j)}\star}{m'_{ij}}=
\brd{w_\circ^{\xi(j)}\star w_\circ^{\xi(i)}\star}{m'_{ij}}$.
\end{proof}
To finish the proof of Theorem~\ref{thm:Hom Heck Mon}, it remains to observe that 
$[\Theta_\xi]=\xi$ for all~$\xi\in\Lambda(M',M)$ while
$\Theta_{[\phi]}=\phi$ for all~$\phi\in\Hom_{\mathscr H}(M',M)$.
\end{proof}

\begin{corollary}\label{cor:tautological}
Let $M, M'\in\Cox I$ and suppose that
$m'_{ij}\ge m_{ij}$ for all~$i,j\in I$. Then
the assignments $s'_i\mapsto s_i$, $i\in I$,
define a homomorphism $(W(M'),\star)\to (W(M),\star)$.
\end{corollary}
We call such homomorphisms {\em tautological}.
For example, for any~$m'>m$ there is a
tautological homomorphism~$W(I_2(m'),\star)\to
W(I_2(m),\star)$.

\begin{definition}\label{defn:optial}
We say that~$\phi\in\Hom_{\mathscr H}(M',M)$ is {\em optimal} if for all~$i,j\in I'$ with~$[\phi](i)\not=[\phi](j)$
\begin{equation}\label{eq:optimal cond}
m'_{ij}=
\max(2,\mu_M([\phi](i),[\phi](j)),\mu_M([\phi](j),[\phi](i)).
\end{equation}
\end{definition}
\begin{proposition}\label{prop:optimal taut}
Every homomorphism of Hecke monoids can be written as a composition of a tautological homomorphism with an optimal one.
\end{proposition}
\begin{proof}
Let~$M'=(m'_{ij})_{i,j\in I'}\in\Cox{I'}$, $M\in\Cox I$ and let $
\phi\in\Hom_{\mathscr H}(M',M)$.
Let $M''=(m''_{i,j})_{i,j\in I'}$ with
$$
m''_{ij}=\begin{cases}m'_{ij},&[\phi](i)=[\phi](j),\\
\max(2,\mu_M([\phi](i),[\phi](j)),\mu_M([\phi](j),[\phi](i)))),&i\not=j
\end{cases}
$$
for all~$i,j\in I'$. Clearly, $M''\in\Cox{I'}$.
Since~$[\phi]\in\Lambda(M',M)$
by Theorem~\ref{thm:Hom Heck Mon}, 
$m'_{ij}\ge 
m''_{ij}$ for all~$i,j\in I'$. 
This yields a tautological homomorphism $\phi'\in\Hom_{\mathscr H}(M',M'')$. Since $[\phi]\in\Lambda(M'',M)$ by 
definition of~$M''$, 
by Lemma~\ref{lem:Lam M'M HomM'M} the assignments $s''_i\mapsto w_\circ^{[\phi](i)}$, $i\in I'$ define 
$\phi''\in\Hom_{\mathscr H}(M'',M)$, which
is clearly optimal. Finally, $(\phi''\circ\phi')(s'_i)=\phi''(s''_i)=
w_\circ^{[\phi](i)}
=\phi(s'_i)$ for all~$i\in I'$ and so
$\phi=\phi''\circ\phi'$.
\end{proof}

A very important class of homomorphisms of Hecke monoids are {\em parabolic projections}.
\begin{lemma}\label{lem:p J defn}\label{lem:proj w0}
Let~$J\subset I$. The assignments
$$
s_i\mapsto \begin{cases}1,&i\in I\setminus J,\\
s_i,&i\in J,
\end{cases}
$$
for all~$i\in I$,
define surjective homomorphism \plink{pJ}$p_J:(W(M),\star)\to (W_J(M),\star)$. Moreover, $p_J(w_\circ^K)=w_\circ^{J\cap K}$
for all~$K\in\mathscr F(M)$.
\end{lemma}
\begin{proof}
Define $\xi_J:I\to \mathscr F(I)$
by $\xi(i)=\{i\}$ if~$i\in J$ and~$\xi(i)=\emptyset$
if~$i\in I\setminus J$. Then~$\xi\in\Lambda(M,M_J)$
and~$p_J=\Theta_{\xi_J}$ in the notation of Lemma~\ref{lem:Lam M'M HomM'M}.
The second assertion follows from Lemma~\ref{lem:Hecke hom w0J} since~$[p_J](K)=\bigcup_{k\in K}\xi_J(k)=K\cap J$.
\end{proof}
Sometimes it is convenient to treat~$p_J$ as an endomorphism of~$(W(M),\star)$.
The following is immediate from the definition of~$p_J$.
\begin{lemma}\label{lem:p_J comp}
For any~$J,K\subset I$, $p_J\circ p_K=
p_{J\cap K}$.
\end{lemma}
\begin{lemma}\label{lem:parab prod}
Let~$J\subset I$ and suppose that~$J$ and~$I\setminus J$ are orthogonal. Then
$w=p_J(w)\times p_{I\setminus J}(w)$ for all~$w\in W(M)$.
\end{lemma}
\begin{proof}
We use induction on~$\ell(w)$, $w\in W(M)$,
the case~$\ell(w)=0$ being obvious. For the inductive step, write
$w=s_i\times w'$ with~$i\in I$, $\ell(w')=\ell(w)-1$. Then~$w=s_i\times p_J(w')\times
p_{I\setminus J}(w')$ by the induction hypothesis. If~$i\in J$ then
$s_i\times p_J(w')=p_J(s_i\star w')=p_J(s_iw')
=p_J(w)$ and~$p_{J\setminus I}(w')=p_{J\setminus I}(s_i)\star p_{J\setminus I}(w')=p_{J\setminus I}(s_i\star w')=
p_{J\setminus I}(w)$. If~$i\in I\setminus J$ then, since~$J$ and~$I\setminus J$ are orthogonal, $s_i$ commutes with~$p_J(w')$ and $p_J(w')=p_J(s_i)\star p_J(w')=p_J(s_i\star w')=p_J(w)$ while
$s_i\star p_{I\setminus J}(w')=p_{I\setminus J}(s_i \star w')=p_{I\setminus J}(w)$. Thus, we have 
$w=p_J(w)\star p_{I\setminus J}(w)=p_J(w)\times 
p_{I\setminus J}(w)$ by Lemma~\ref{lem:extend supp}.
\end{proof}

\subsection{Homomorphisms of Artin monoids}\label{subs:Artin homs}
Given $\wh M\in\Cox{\wh I}$, $M\in\Cox I$ and~$\Phi\in\Hom_{\mathscr A}(\wh M,M)$
we define,
as in~\S\ref{subs:Hecke homs}, 
$[\Phi]:\wh I\to\mathscr P(I)$,
$i\mapsto \supp\Phi(\wh T_i)$, $i\in \wh I$
and extend it to a map
\plink{[Phi]}$[\Phi]:\mathscr P(\wh I)\to\mathscr P(I)$
via $[\Phi](\wh J)=\bigcup_{j\in\wh J}[\Phi](j)$, $\wh J\subset \wh I$.
The definitions of disjoint and fully supported 
homomorphisms are the same as for Hecke monoids.
The following is proved exactly as Lemma~\ref{lem:[phi]comp}.
\begin{lemma}\label{lem:[Phi]comp}
Let~$M\in\Cox I$, $M'\in\Cox{I'}$ and~$M''\in\Cox{I''}$.
\begin{enmalph} 
\item\label{lem:[Phi]comp.a}
$\supp\Phi(x)=[\Phi](\supp x)$
for any~$\Phi\in\Hom_{\mathscr A}(M',M)$ and for
all~$x\in \Br^+(M')$;
\item\label{lem:[Phi]comp.b}
$[\Phi\circ\Phi']=[\Phi]\circ[\Phi']$ as maps $\mathscr P(I'')\to
\mathscr P(I)$ for any~$\Phi\in\Hom_{\mathscr A}(M',M)$
and~$\Phi'\in\Hom_{\mathscr A}(M'',M')$;
\item\label{lem:[Phi]comp.c}
\label{lem:elem Artin hom.a}
   If~$\Phi\in \Hom_{\mathscr A}(M',M)$ is
    disjoint and~$[\Phi](i)\not=\emptyset$ for all~$i\in\wh I$ then~$[\Phi]:\mathscr P(I')\to
    \mathscr P(I)$ is
    injective.
\end{enmalph}
\end{lemma}
The following is immediate from definitions.
\begin{lemma}\label{lem:fund hom}
Let~$M=(m)_{i,j\in I}$ be a Coxeter matrix, let~$\mathsf M$
be any multiplicative monoid
and let $X_i$, $i\in I$ be a collection of 
elements in~$\mathsf M$. The assignments 
$T_i\mapsto X_i$, $i\in I$ define 
a homomorphism of monoids $\Br^+(M)\to \mathsf M$
if and only if~$m_{ij}\in B(X_i,X_j)\cup\{\infty\}$
for all~$i\not=j\in I$.
\end{lemma}
\begin{example}\label{ex:char hom}
Let~$\wh M=(\wh m_{ij})_{i,j\in \wh I}\in \Cox{\wh I}$,
$M\in\Cox I$. Let~$\mathbf X=\{X_i\,:\,i\in\wh I\}\subset\Br^+(M)$, be a family of commuting elements satisfying $X_i=X_j$
whenever $\wh m_{ij}$ is odd, $i,j\in\wh I$. Then the assignments 
$\wh T_i\mapsto X_i$, $i\in\wh I$
define \plink{char hom}$\Xi_{\mathbf X}\in\Hom_{\mathscr A}(\wh M,M)$. We call such a homomorphism a {\em character homomorphism}. The most basic
example is the generalization~$\ell_{\mathbf d}$, $\mathbf d=(d_i)_{i\in\wh I}\in\ZZ_{\ge 0}^{\wh I}$ of 
the length homomorphism $\ell:\Br^+(\wh M)\to (\ZZ_{\ge 0},+)\cong \Br^+(A_1)$ defined by $\wh T_i\mapsto d_i$,
$i\in \wh I$ where $d_i=d_j$ whenever~$\wh m_{ij}$ is odd.
\end{example}

\begin{definition}\label{def:types}
Let~$\wh M=(\wh m_{ij})_{i,j\in\wh I}\in\Cox{\wh I}$ and let~$M\in\Cox I$.
We say that~$\Phi:\Hom_{\mathscr A}(\wh M,M)$ is:
\begin{itemize}
\item[-] {\em square free} if $\Phi(\wh T_i)\in \SQF^+(M)$
for all~$i\in\wh I$;
\item[-] {\em strongly square free} if $\Phi(\SQF^+(\wh M))\subset
\SQF^+(M)$.
\item[-] {\em optimal} if
$\wh m_{ij}=\min B(\Phi(\wh T_i),\Phi(\wh T_j))$ for all~$i,j\in\wh I$ such that~$\Phi(\wh T_i)\not=\Phi(\wh T_j)$ and~$B(\Phi(\wh T_i),\Phi(\wh T_j))$ is non-empty.
\end{itemize}
\end{definition}
We now establish some elementary properties of homomorphisms of Artin monoids.
\begin{lemma}\label{lem:elem Artin hom}
Let $M\in\Cox I$, $\wh M=(\wh m_{ij})_{i,j\in\wh I}\in\Cox{\wh I}$ and
let~$\Phi\in\Hom_{\mathscr A}(\wh M,M)$.
\begin{enmalph}
  \item\label{lem:elem Artin hom.b} If~$\wh M$ is of finite type then
    $\Phi$ is strongly square free if and only if $\Phi(\wh T_{w_\circ^{\wh I}})\in\SQF^+(M)$.
    \item\label{lem:elem Artin hom.b'} $\Phi$
    commutes with~${}^{op}$ if and only
    if all the~$\Phi(\wh T_i)$, $i\in\wh I$ are
    ${}^{op}$-invariant.
    \item\label{lem:elem Artin hom.c}
    $\ell(\Phi(\wh T_i))=\ell(\Phi(\wh T_j))$
    for all $i,j\in \wh I$ such that~$\wh m_{ij}$
    is odd. In particular, if $[\Phi](i)=\emptyset$
    for some~$i\in\wh I$ then~$[\Phi](j)=\emptyset$
    for all~$j\in\wh I$ such that~$\wh m_{ij}$ is odd.
\end{enmalph}
\end{lemma}
\begin{proof}
Part~\ref{lem:elem Artin hom.b} is
immediate from Proposition~\partref{prop:fund elts BrSa.e}.
Part~\ref{lem:elem Artin hom.b'} is obvious.
To prove~\ref{lem:elem Artin hom.c}, note that
$$
\brd{\Phi(\wh T_i)\Phi(\wh T_j)}{\wh m_{ij}}=
\Phi(\brd{\wh T_i\wh T_j}{\wh m_{ij}})=\Phi(\brd{\wh T_j\wh T_i}{\wh m_{ij}})=
\brd{\Phi(\wh T_j)\Phi(\wh T_i)}{\wh m_{ij}}
$$
implies, for~$\wh m_{ij}$ odd,
that $$
\tfrac12(\wh m_{ij}-1)(\ell(\Phi(\wh T_i))
+\ell(\Phi(\wh T_j)))+\ell(\Phi(\wh T_i))=
\tfrac12(\wh m_{ij}-1)(\ell(\Phi(\wh T_j))
+\ell(\Phi(\wh T_i)))+\ell(\Phi(\wh T_j)).
$$
The assertion is now immediate.
\end{proof}

\begin{lemma}\label{lem:blowup}
Let~$M\in\Cox I$ and let
$\wh M=(d_{ij}m_{ij})_{i,j\in I}$ where
$d_{ij}=d_{ji}\in\mathbb Z_{>0}\cup\{\infty\}$ and~$d_{ii}=1$, $i,j\in I$. Then~$\wh M\in\Cox I$ and
the assignments $\wh T_i\mapsto T_i$, $i\in I$
define a homomorphism $\Br^+(\wh M)\to \Br^+(M)$.
\end{lemma}
\begin{proof}
This is immediate from Lemma~\partref{lem:taut homs.a}.
\end{proof}
We call such homomorphisms {\em tautological}.
For example, for any~$d>0$, $m\ge 2$ there is a tautological homomorphism~$\Br^+(I_2(dm))\to
\Br^+(I_2(m))$.

\begin{lemma}\label{lem:factor homs}
Every homomorphism of Artin monoids is a
composition of a tautological homomorphism
with an optimal one.
\end{lemma}
\begin{proof}
Let $M=(m_{ij})_{i,j\in I}$, $\wh M=(\wh m)_{i,j\in\wh I}$ be Coxeter matrices
and let $\Phi\in\Hom_{\mathscr A}(\wh M,M)$.
Define~$\wh M'=(\wh m'_{ij})_{i,j\in\wh I}$ as follows.
If~$\Phi(\wh T_i)=\Phi(\wh T_j)$, set 
set~$\wh m'_{ij}=\wh m_{ij}$. Otherwise,
if~$B_{ij}:=B(\Phi(\wh T_i),\Phi(\wh T_j))$ is empty then~$\wh m_{ij}=\infty$ and we set~$\wh m'_{ij}=\infty$. If~$B_{ij}\not=\emptyset$, then, since~$\Br^+(M)$ is cancellative,
$B_{ij}=
\ZZ_{>0}\wh m'_{ij}$ for some~$\wh m'_{ij}\in\ZZ_{>0}$
by Lemma~\partref{lem:taut homs.b}.
Moreover, since~$\Phi(\wh T_i)\not=\Phi(\wh T_j)$, $m'_{ij}>1$.
Since~$\wh m_{ij}\in B_{ij}\cup\{\infty\}$ by Lemma~\ref{lem:fund hom}, we have $\wh m_{ij}=d_{ij}\wh m'_{ij}$ for some~$d_{ij}\in\ZZ_{>0}\cup\{\infty\}$. Thus, $\wh M'\in\Cox{I'}$,
the assignments $\wh T_i\mapsto \wh T'_i$, $i\in \wh I$ define a tautological $\Phi'\in\Hom_{\mathscr A}(\wh M,\wh M')$, and
the assignments $\wh T'_i\mapsto \Phi(\wh T_i)$, $i\in\wh I$, define an optimal
$\Phi''\in\Hom_{\mathscr A}(\wh M',M)$. By construction,
$\Phi=\Phi''\circ\Phi'$.
\end{proof}
Note the following useful and immediate fact.
\begin{lemma}\label{lem:cradicals}
Let~$\mathsf M$ be a monoid and suppose that
one of $X_1,X_2\in\mathsf M$ is
invertible that and~$(X_1X_2)^m$, $m\in\ZZ_{>0}$ is central in the submonoid they generate. Then the assignments~$T_i\mapsto X_i$, $i\in\{1,2\}$ define a homomorphism~$\Br^+(I_2(2m))\to 
\mathsf M$.
\end{lemma}
Since an Artin monoid embeds into the corresponding Artin group, this Lemma shows that factorizations of radicals of central elements give rise to homomorphisms of Artin monoids.

\subsection{Decorating homomorphisms from Artin
monoids}\label{subs:decor hom}
We will now discuss a machinery which allows us to produce new homomorphisms from Artin monoids to other monoids from existing ones. This construction will be used extensively in the sequel.

\begin{definition}\label{def:decoration}
Let~$M=(m_{ij})_{i,j\in I}$ be a Coxeter matrix, let~$\mathsf M$ be a multiplicative monoid and let~$\Phi:\Br^+(M)\to\mathsf M$ be a 
homomorphism of monoids. 
We say that
$\mathbf z=(z_i)_{i\in I}\in\mathsf M^I$ 
is a {\em decoration} of~$\Phi$ if 
the assignments~$T_i\mapsto \Phi(T_i)z_i$, $i\in I$
define a homomorphism~\plink{Phi z}$\Phi_{\mathbf z}:\Br^+(M)\to \mathsf M$.
\end{definition}
We will sometimes refer to~$\Phi_{\boldsymbol z}$
as a decorated companion of~$\Phi$. The following Lemma is immediate.
\begin{lemma}\label{lem:decs are invertible}
Let~$M\in\Cox I$, let $\mathsf M$ be a multiplicative
monoid and let~$\mathbf z=(z_i)_{i\in I}$ with all the~$z_i$ invertible. Then~$\mathbf z$ is a decoration of~$\Phi$ if and only if~$\mathbf z^{-1}=(z_i^{-1})_{i\in I}$ is a decoration of~$\Phi_{\mathbf z}$ and~$(\Phi_{\mathbf z})_{\mathbf z^{-1}}=\Phi$.
\end{lemma}
The following result provides a rather strong sufficient condition for the existence of a decoration of 
a given homomorphism, which will be used in multiple proofs later.
\begin{theorem}\label{thm:decoration sufficient}
Let~$M=(m_{ij})_{i,j\in I}$ be a Coxeter matrix, let~$\mathsf M$ be a multiplicative monoid and let~$\Phi:\Br^+(M)\to\mathsf M$ be a 
homomorphism of monoids. 
Suppose that for any~$i,j\in I$
we are given $z_{i,j}^{(k)}$, $k\in[1,m_{ij}]$
such that $z_{i,j}^{(1)}=z_{i,i}^{(1)}$ for all~$j\in I$
and for all~$i\not=j$ with~$m_{ij}<\infty$
\begin{enumerate}[label={$\arabic*^\circ.$},
ref={$\arabic*^\circ$}]
    \item\label{thm:decoration sufficient.1}
$z_{i,j}^{(k)}\Phi(T_j)=\Phi(T_j)z_{i,j}^{(k+1)}$
if~$k$ is odd while $z_{i,j}^{(k)}\Phi(T_i)=\Phi(T_i)z_{i,j}^{(k+1)}$ if~$k$
is even, $k\in[1,m_{ij}-1]$;
\item\label{thm:decoration sufficient.2} $z_{i,j}^{(m_{ij})}z_{j,i}^{(m_{ij}-1)}\cdots 
=z_{j,i}^{(m_{ij})}z_{i,j}^{(m_{ij}-1)}\cdots$.
\end{enumerate}
Then~$\mathbf z=(z_{i,i}^{(1)})_{i\in I}$ is a
decoration of~$\Phi$. Moreover, if~\ref{thm:decoration sufficient.1} is satisfied and~$\mathsf M$ is left 
cancellative then~$\mathbf z=(z_{i,i}^{(1)})_{i\in I}$
is a decoration of~$\Phi$ if and only if~\ref{thm:decoration sufficient.2} holds.
\end{theorem}
\begin{proof}
It suffices to prove the theorem for~$I$ with~$|I|=2$.
Let~$m=m_{ij}$, $\{i,j\}=I$ and assume that~$m<\infty$.
Abbreviate $t_i=\Phi(T_i)$, $i\in I$ and~$z_i^{(k)}=z_{i,j}^{(k)}$, $\{i,j\}=I$, $k\in[1,m]$. We will
use the convention that~$i+r=i$, $i\in I$ if~$r\in\ZZ$ is even and~$i+r=j$ if~$r$ is odd, $\{i,j\}=I$.
Using this convention,
the condition~\ref{thm:decoration sufficient.1}
can be written as
\begin{equation}\label{eq:decoration-1}
z_i^{(k)}\Phi(T_{i+k})=\Phi(T_{i+k})z_i^{(k+1)},\qquad 
i\in I,\,k\in[1,m-1],
\end{equation}
while~\ref{thm:decoration sufficient.2} becomes
\begin{equation}\label{eq:decoration-2}
    \dscprod_{1\le k\le m} z^{(k)}_{i-k}
    =\dscprod_{1\le k\le m} z^{(k)}_{i+r-k},\qquad i\in I,\,r\in\ZZ.
\end{equation}
\begin{lemma}\label{lem:decoration 1}
We have
$z_i^{(k)}\brd{t_{i+k}t_{i+k+1}}{l}=
\brd{t_{i+k}t_{i+k+1}}{l}z_i^{(k+l)}
$ for any~$i\in I$, $l\in[1,m]$ and $k\in[1,m-l]$.
\end{lemma}
\begin{proof}
We use induction on~$l$, the case~$l=1$ being just~\eqref{eq:decoration-1}.
For the inductive step, since  $\brd{t_jt_{j+1}}{r}=
t_j\,\brd{t_{j+1}t_{j+2}}{r-1}$ for any~$j\in I$, $r\in\ZZ_{>0}$, we have 
by~\eqref{eq:decoration-1} and the induction hypothesis
\begin{align*}
z_i^{(k)}\brd{t_{i+k}t_{i+k+1}}{l+1}&=
z_i^{(k)}t_{i+k}\brd{t_{i+k+1}t_{i+k+2}}{l}\\
&=t_{i+k}z_i^{(k+1)}\brd{t_{i+k+1}t_{i+k+2}}{l}
=\brd{t_{i+k}t_{i+k+1}}{l+1}z_i^{(k+l+1)}.\qedhere
\end{align*}
\end{proof}
\begin{lemma}\label{lem:decoration 2}
Let~$i\in I$. Then for any~$0\le l\le m$
\begin{align*}
\brd{(t_i z_i)(t_{i+1} z_{i+1})}{m}=\brd{(t_iz_i)(t_{i+1}z_{i+1})}{m-l}\,\brd{ t_{i+m-l}t_{i+m-l+1}}{l}\,
\dscprod_{1\le k\le l} z^{(k)}_{i+m-k}.
\end{align*}
\end{lemma}
\begin{proof}
We use induction on~$l$, the case~$l=0$ being trivial. 
For the inductive step we have by Lemma~\ref{lem:decoration 1} and the induction hypothesis
\begin{align*}
\brd{(t_i z_i)(t_{i+1} z_{i+1})}{m}
&=\brd{(t_iz_i)(t_{i+1}z_{i+1})}{m-l-1}
\, t_{i+m-l-1}z_{i+m-l-1}\,\brd{ t_{i+m-l}t_{i+m-l+1}}{l}\,
\dscprod_{1\le k\le l} z^{(k)}_{i+m-k}\\
&=\brd{(t_iz_i)(t_{i+1}z_{i+1})}{m-l-1}
\, \brd{ t_{i+m-l-1}t_{i+m-l}}{l+1}\,
z_{i+m-l-1}^{(l+1)}\dscprod_{1\le k\le l} z^{(k)}_{i+m-k}\\
&=\brd{(t_iz_i)(t_{i+1}z_{i+1})}{m-(l+1)}
\, \brd{ t_{i+m-(l+1)}t_{i+m-l}}{l+1}\,
\dscprod_{1\le k\le l+1} z^{(k)}_{i+m-k}.\qedhere
\end{align*}
\end{proof}
Using Lemma~\ref{lem:decoration 2} with~$m=l$ we obtain,
$$
\brd{(t_iz_i)(t_j z_j)}{m}=
\brd{t_i t_j}{m}\dscprod_{1\le k\le m} z^{(k)}_{i+m-k},
$$
where~$I=\{i,j\}$. Since
$\dscprod_{1\le k\le m} z^{(k)}_{i+m-k}=
\dscprod_{1\le k\le m} z^{(k)}_{i+k}=
\dscprod_{1\le k\le m} z^{(k)}_{j+k}$ by~\eqref{eq:decoration-2} while~$\Phi$ being a homomorphism yields~$\brd{t_it_j}{m}=
\brd{t_jt_i}m$, the first assertion follows.
The second assertion is now immediate.
\end{proof}

Note that if the~$\Phi(T_i)$, $i\in I$ are invertible, 
then the condition~\ref{thm:decoration sufficient.1} determines the $z_{i,j}^{(k)}$ uniquely as
\begin{equation}\label{eq:inv decoration}
z_{i,j}^{(k+1)}=\begin{cases}
    \Phi(T_i)^{-1} z_{i,j}^{(k)}\Phi(T_i),&\bar k=0,\\
    \Phi(T_j)^{-1} z_{i,j}^{(k)}\Phi(T_j),&\bar k=1    
\end{cases}
\end{equation}
for all~$i\not=j\in I$ with~$m_{ij}<\infty$ and
$k\in[1,m_{ij}-1]$ or, in the closed form,
$$
z_{i,j}^{(k)}=(\brd{\Phi(T_j)\Phi(T_i)}{k-1})^{-1}
z_{i,i}^{(1)}(\brd{\Phi(T_j)\Phi(T_i)}{k-1}).
$$
\begin{corollary}\label{cor:dec iff}
Suppose that~$\mathsf M$ is left cancellative and
the $\Phi(T_i)$, $i\in I$ are invertible.
Then~$\mathbf z=(z_{i,i}^{(1)})_{i\in I}$  is a decoration of~$\Phi$ if and only if
the condition~\ref{thm:decoration sufficient.2} of
Theorem~\ref{thm:decoration sufficient} holds
for the~$z_{i,j}^{(k)}$ defined by~\eqref{eq:inv decoration}.
\end{corollary}

\begin{example}\label{ex:embedded std}
Let~$M=B_2$ and let~$\mathsf M=\Br(A_n)$, or~$\Br(D_{n+1})$ with~$n$ even, or~$\Br(E_6)$.
The assignments
$\wh T_1\mapsto 1$, $\wh T_2\mapsto T_{w_\circ^{I}}$ define a homomorphism~$\Phi:\Br^+(M)\to \mathsf M$
(cf. Example~\ref{ex:char hom}). Let~$\sigma$ be the involutive diagram automorphism of~$\mathsf M$ and suppose that~$J$ and~$\sigma(J)$ are weakly orthogonal.
Set
$z_{1,1}^{(1)}=T_{w_\circ^J}$ and~$z_{2,2}^{(1)}=1$. Then using~\eqref{eq:inv decoration} and
Proposition~\partref{prop:fund elts BrSa.c}
we obtain 
$z_{1,2}^{(2)}=T_{w_\circ^{I}}^{-1}T_{w_\circ^J}T_{w_\circ^{I}}=T_{w_\circ^{\sigma(J)}}$,
$z_{1,2}^{(3)}=z_{1,2}^{(2)}$ and
$z_{1,2}^{(4)}=T_{w_\circ^J}$, while~$z_{2,1}^{(k)}=1$,
$1\le k\le 4$. Since~$T_{w_\circ^J}$
commutes with~$T_{w_\circ^{\sigma(J)}}$ by Lemma~\ref{lem:weakly orthogonal} and so
$$
z_{1,2}^{(4)}z_{2,1}^{(3)}z_{1,2}^{(2)}z_{2,1}^{(1)}
=T_{w_\circ^{J}}T_{w_\circ^{\sigma(J)}}=T_{w_\circ^{\sigma(J)}}T_{w_\circ^J}=z_{2,1}^{(4)}z_{1,2}^{(3)}z_{2,1}^{(2)}z_{1,2}^{(1)}.
$$
Thus, $\mathbf z=(z_{1,1}^{(1)},z_{2,2}^{(1)})$ is a decoration of~$\Phi$
and $\Phi_{\mathbf z}:\Br^+(B_2)\to\mathsf M$ is given by $\wh T_1\mapsto T_{w_\circ^J}$, $\wh T_2\mapsto T_{w_\circ^{I}}$ and hence a homomorphism to the respective Artin monoid.

More generally, let~$M\in\Cox I$ be of finite type.
Let~$\wh I$ be a finite set.
Fix a total order~$\prec$ on~$\wh I$ and 
let~$\{J_i\}_{i\in\wh I}$ be any collection of non-empty subsets of~$I$ such $J_i\subset J_k$ whenever~$i\prec k\in \wh I$ and~$\bigcup_{i\in\wh I} J_i=I$. Let~$\sigma_i=\Sigma_{J_i}$, $i\in\wh I$, in the notation of Proposition~\partref{prop:fund elts BrSa.c}.
Define~$\wh M=(\wh m_{ik})_{i,k\in\wh I}$ by $\wh m_{ii}=1$, $i\in\wh I$
and
$$
\wh m_{ik}=\wh m_{ki}=\begin{cases}
2,& \text{$J_i=J_k$ or~$J_i=\sigma_k(J_i)$},\\
4,& \text{$J_i\subsetneq J_k$ and 
$J_i\not=\sigma_k(J_i)$ are weakly orthogonal},\\
\infty,&\text{otherwise}
\end{cases}
$$
for all~$i\prec k\in\wh I$, $i\not=k$. Then~$\wh M\in\Cox{\wh I}$.
By the above, the assignments~$\wh T_i\mapsto T_{w_\circ^{J_i}}$, $i\in\wh I$
define a homomorphism~$\Br^+(\wh M)\to \Br^+(M)$.
\end{example}

The following Lemma provides perhaps the simplest yet important example of a decoration.
\begin{lemma}\label{lem:cent decor}
Let~$M\in\Cox I$, let $\mathsf M$ be a multiplicative monoid and let
$\Phi:\Br^+(M)\to\mathsf M$ be a homomorphism for some multiplicative monoid~$\mathsf M$. Let~$\mathbf z=(z_i)_{i\in I}\in\mathsf M^I$ where 
$z_iz_j=z_jz_i$ for all~$i,j\in I$,
$z_i=z_j$ if~$m_{ij}$ is odd and all the~$z_i$, $i\in I$
are in the centralizer of the image of~$\Phi$.
Then~$\mathbf z$ is a decoration of~$\Phi$.
\end{lemma}
\begin{proof}
Let $z_{i,j}^{(k)}=z_i$
for all~$i,j\in I$, $k\in[1,m_{ij}]$. Suppose that~$i\not=j$ and~$m_{ij}<\infty$.
The condition~\ref{thm:decoration sufficient.1} of Theorem~\ref{thm:decoration sufficient} is obviously satisfied, while $z_{i,j}^{(m_{ij})}z_{j,i}^{(m_{ij}-1)}\cdots=
z_i^{\lceil \frac12 m_{ij}\rceil} z_j^{\lfloor\frac12 m_{ij}\rfloor}$ which is manifestly symmetric in~$i$ and~$j$ if~$m_{ij}$ is even and equals to $z_i^{m_{ij}}=z_j^{m_{ij}}$ if~$m_{ij}$ is odd.
\end{proof}
\begin{example}\label{ex:cent decoration}
Let~$M\in\Cox I$. Then any collection $\mathbf z=(z_i)_{i\in I}$
of central elements of~$\Br^+(M)$ satisfying
$z_i=z_j$ whenever~$m_{ij}$ is odd is a decoration of~$\id\in\Hom_{\mathscr A}(M,M)$ and is manifestly optimal.
\end{example}

\subsection{Hecke and Coxeter type homomorphisms}\label{subs:Heck Cox hom}
Let~$\mathscr M$ be the category of monoids.
Let~$\mathscr{CM}$ be the category whose objects
are pairs
$(\mathsf M,\mathcal C)$ where~$\mathsf M$ is a monoid and~$\mathcal C\subset 
\mathsf M\times\mathsf M$
is a congruence relation (cf.~\S\ref{subs:monoids}). A morphism
$f:(\mathsf M,\mathcal C)\to (\mathsf M',\mathcal C')$
is a homomorphism of monoids~$f:\mathsf M\to\mathsf M'$
satisfying $(f\times f)(\mathcal C)\subset \mathcal C'$.
Then $\id_{(\mathsf M,\mathcal C)}=\id_{\mathsf M}$
for any congruence relation~$\mathcal C$ on~$\mathsf M$.
Clearly, if $f\in\Hom_{\mathscr{CM}}((\mathsf M,\mathcal C),(\mathsf M',\mathcal C'))$ and
$f'\in\Hom_{\mathscr{CM}}((\mathsf M',\mathcal C'),(\mathsf M'',\mathcal C''))$ then~$(f'\circ f
\times f'\circ f)(\mathcal C)=(f'\times f')\circ (f\times f)(\mathcal C)\subset (f'\times f')(\mathcal C')
\subset \mathcal C''$ and so~$\mathscr{CM}$ is indeed a 
category.

\begin{proposition}\label{prop:induced hom of monoids}
The assignments 
\begin{alignat*}{2}
&(\mathsf M,\mathcal C)\mapsto 
\mathsf M/\mathcal C, &\qquad & (\mathsf M,\mathcal C)\in\mathscr{CM},\\ 
&f\mapsto \overline f, && f\in\Hom_{\mathscr{CM}}(
(\mathsf M,\mathcal C),(\mathsf M',\mathcal C')),
\end{alignat*}
where 
$\overline f([x]_{\mathcal C})=[f(x)]_{\mathcal C'}$, $ x\in \mathsf M$,
define a functor~$\mathsf F:\mathscr{CM}\to \mathscr M$.
Furthermore,  canonical homomorphisms $\pi_{\mathcal C}:\mathsf M\to \mathsf M/\mathcal C$ provide a natural transformation from the forgetful functor
$\mathscr{CM}\to \mathscr M$
to~$\mathsf F$.
\end{proposition}
\begin{proof}
We need the following
\begin{lemma}\label{lem:bar f well def}
Let~$\mathsf M$, $\mathsf M'$ be monoids,
$f\in\Hom_{\mathscr M}(\mathsf M,\mathsf M')$
and let~$\mathcal C\subset\mathsf M\times \mathsf M$,
$\mathcal C'\subset \mathsf M'\times\mathsf M'$ be congruence relations. Then~$\overline f$ is a well-defined map $\mathsf M/\mathcal C\to\mathsf M'/\mathcal C'$ if and only if~$f\in\Hom_{\mathscr{CM}}((\mathsf M,\mathcal C),(\mathsf M',\mathcal C'))$.
\end{lemma}
\begin{proof}
It suffices to note that $\overline f$ is
well-defined if and only if for any $x,x'\in\mathsf M$,
$(x,x')\in\mathcal C$
implies that $(f(x),f(x'))\in\mathcal C'$.
\end{proof}
Let~$f\in\Hom_{\mathscr{CM}}((\mathsf M,\mathcal C),
(\mathsf M',\mathcal C'))$. Then~$\bar f$
is well-defined by Lemma~\ref{lem:bar f well def} and
for all $x,x'\in\mathsf M$
\begin{align*}
\overline f([x]_{\mathcal C}[x']_{\mathcal C})&=\overline f([xx']_{\mathcal C})=[f(xx')]_{\mathcal C'}
=[f(x)f(x')]_{\mathcal C'}
=[f(x)]_{\mathcal C'}[f(x')]_{\mathcal C'}
=\overline f([x]_{\mathcal C})
\overline f([x']_{\mathcal C}),
\end{align*}
that is~$\overline f\in\Hom_{\mathscr M}(\mathsf M/\mathcal C,\mathsf M'/\mathcal C')$. Clearly,
$\overline{\id_{(\mathsf M,\mathcal C)}}=\id_{\mathsf M/\mathcal C}$. 

Let~$f'\in\Hom_{\mathscr{CM}}((\mathsf M',\mathcal C'),
(\mathsf M'',\mathcal C''))$. By Lemma~\ref{lem:bar f well def}, 
$\overline{f'\circ f}$
is a well-defined homomorphism of monoids $\mathsf M/\mathcal C\to\mathsf M''/\mathcal C''$ and for all~$x\in\mathsf M$
\begin{equation*}
\overline{f'\circ f}([x]_{\mathcal C})=[(f'\circ f)(x)]_{\mathcal C''}
=[f'(f(x))]_{\mathcal C''}=
\overline{f'}([f(x)]_{\mathcal C'})=
\overline{f'}(
\overline f([x]_{\mathcal C})),
\end{equation*}
and so~$\overline{f'\circ f}=\overline{f'}\circ \overline f$.
Finally, to prove the last assertion it suffices to note that the diagram 
$$
    \begin{tikzcd}[ampersand replacement=\&]
	{\mathsf M} \&\& {\mathsf M'}\\ 
	{\mathsf M/\mathcal C} \&\& {\mathsf M'/\mathcal C'}
	\arrow["f", from=1-1, to=1-3]
	\arrow["{\pi_{\mathcal C}}", from=1-1, to=2-1]
	\arrow["{\pi_{\mathcal C'}}", from=1-3, to=2-3]
	\arrow["\overline f", from=2-1, to=2-3]
\end{tikzcd}
    $$
commutes for every~$f\in\Hom_{\mathscr{CM}}((\mathsf M,
\mathcal C),(\mathsf M',\mathcal C'))$
by definition of~$\bar f$.
\end{proof}

Given monoids~$\mathsf M$, $\mathsf M'$ with respective congruence relations~$\mathcal C$, $\mathcal C'$ and~$f\in\Hom_{\mathscr M}(\mathsf M,\mathsf M')$
we say that~$f$ induces $g\in\Hom_{\mathscr M}(\mathsf M/\mathcal C,\mathsf M'/\mathcal C')$ if 
$f\in\Hom_{\mathscr{CM}}((\mathsf M,\mathcal C),
(\mathsf M',\mathcal C'))$ and~$g=\mathsf F(f)$.

Recall that~$W(M)$ (respectively, $(W(M),\star)$)
is the quotient of~$\Br^+(M)$ by the minimal congruence relation containing the $(T_i^2,1)$ (respectively, 
the $(T_i^2,T_i)$) for all~$i\in I$. 
\begin{definition}\label{defn:Heck Coxeter}
We say that~$\Phi\in\Hom_{\mathscr A}(\wh M,M)$ is
\begin{itemize}
    \item[-] a {\em Coxeter type} homomorphism if
    $\Phi$ induces a homomorphism of groups \plink{barPhi}$\overline\Phi:W(\wh M)\to W(M)$;
    \item[-] a {\em Hecke type} homomorphism if 
    it induces $\overline\Phi_\star\in\Hom_{\mathscr H}(\wh M,M)$;
    \item[-] a {\em Coxeter-Hecke type} homomorphism if it is both a Coxeter and a
    Hecke type homomorphism;
    \item[-] {\em standard} if it is a square free Hecke type homomorphism.
\end{itemize}
\end{definition}
The following is an immediate consequence of Proposition~\ref{prop:induced hom of monoids}.
\begin{corollary}\label{cor:cat AH AC}
\begin{enmalph}
    \item \label{cor:cat AH AC.a}
    Coxeter matrices together with Coxeter (respectively, Hecke) type homomorphisms of 
    the corresponding Artin monoids form a subcategory~\plink{AH AC}$\mathscr{AC}$ (respectively, $\mathscr{AH}$) of~$\mathscr A$.
    \item \label{cor:cat AH AC.b}
    The assignments $\Phi\mapsto \overline\Phi$ for all
    $\Phi\in\Hom_{\mathscr{AC}}(\wh M,M)$
    (respectively, $\Psi\mapsto \overline\Psi_\star$ for all $\Psi\in\Hom_{\mathscr{AH}}(\wh M,M)$),
    $\wh M,M\in\mathscr A$
define a functor~\plink{H C}$\mathsf H:\mathscr{AH}\to\mathscr H$ (respectively, $\mathsf C:\mathscr{AC}\to\mathscr C$). 
\item \label{cor:cat AH AC.c}
$\mathsf H(\Phi)\circ\pi^\star_{\wh M}=
\pi^\star_M\circ \Phi$ (respectively, 
$\mathsf C(\Psi)\circ\pi_{\wh M}=
\pi_M\circ \Psi$)
for every~$\Phi\in\Hom_{\mathscr{AH}}(\wh M,M)$, $\Psi\in\Hom_{\mathscr{AC}}(\wh M,M)$, $\wh M,M\in\mathscr A$.
\end{enmalph} 
\end{corollary}

Before providing examples, we establish a simple criterion for a homomorphism of Artin monoids to be of Coxeter or Hecke type. 
Let~$\wh M\in\Cox{\wh I}$ and~$M\in\Cox I$. Since $\pi_{\wh M}|_{\SQF^+(\wh M)}=\pi^\star_{\wh M}|_{\SQF^+(\wh M)}$ is a bijection~$\SQF^+(\wh M)\to W(\wh M)$, given
{\em any} map~$\Phi:\Br^+(\wh M)\to \Br^+(M)$ we obtain maps
$\tilde\Phi=\pi_M\circ\Phi\circ \pi_{\wh M}^{-1}:W(\wh M)\to W(M)$ and~$\tilde\Phi_\star=\pi^\star_M\circ 
\Phi\circ\pi_{\wh M}^{-1}:(W(\wh M),\star)\to 
(W(M),\star)$.
\begin{proposition}\label{prop:elem prop Coxeter Hecke}
Let~$\wh M\in\Cox{\wh I}$, $M\in\Cox I$ and let~$\Phi\in\Hom_{\mathscr A}(\wh M,M)$. 
\begin{enmalph}
    \item\label{prop:elem prop Coxeter Hecke.a} $\Phi$ is a Coxeter type homomorphism if and only if
    $\pi_M(\Phi(\wh T_i))$ is an involution for every~$i\in\wh I$. In particular, if~$\Phi$
    is square free, then $\Phi$ is a Coxeter type homomorphism
    if and only if
    $\Phi(\wh T_i)=T_{w_i}$ where each~$w_i\in W(M)$,
    $i\in\wh I$
    is an involution.
    \item\label{prop:elem prop Coxeter Hecke.b} $\Phi$ is a Hecke type homomorphism if and only if $[\Phi](i)\in\mathscr F(M)$ and
    $\pi^\star_M(\Phi(\wh T_i))=w_\circ^{[\Phi](i)}$ for each~$i\in\wh I$. In particular, if~$\Phi$
    is square free then~$\Phi$ is a Hecke type
    homomorphism if and only if~$[\Phi](i)\in\mathscr F(M)$ and~$\Phi(\wh T_i)=
    T_{w_\circ^{[\Phi](i)}}$ for each~$i\in \wh I$. 
    \item \label{prop:elem prop Coxeter Hecke.b'}
    All standard homomorphisms are of
    Coxeter-Hecke type.
    \item\label{prop:elem prop Coxeter Hecke.c}
    Square free Coxeter or Hecke type homomorphisms
    commute with~${}^{op}$.
 \end{enmalph}
 \end{proposition}
\begin{proof}
Let~$\Phi\in\Hom_{\mathscr A}(\wh M,M)$ be of Coxeter type and let~$\overline\Phi$
be the induced homomorphism~$W(\wh M)\to W(M)$. By Lemma~\ref{prop:induced hom of monoids},
$\overline\Phi\circ\pi_{\wh M}=\pi_{M}\circ\Phi$ and so
$$
\overline\Phi(w)=\pi_M(\Phi(\wh T_w)),\qquad w\in W(\wh M).
$$
Likewise, if~$\Phi$ is of Hecke type,
$$
\overline\Phi(w)=\pi^\star_M(\Phi(\wh T_w)),\qquad w\in W(\wh M).
$$
Since~$\overline\Phi$ (respectively, $\overline\Phi_\star$) is a homomorphism, it follows that~$\overline\Phi(\wh s_i)$ (respectively, $\overline\Phi_\star(\wh s_i)$) is an involution (respectively, an idempotent) for every~$i\in\wh I$.
The forward direction in part~\ref{prop:elem prop Coxeter Hecke.a} is now obvious, while for part~\ref{prop:elem prop Coxeter Hecke.b} it remains 
to apply Corollary~\partref{cor:max elts.b}.

To prove the converse in parts~\ref{prop:elem prop Coxeter Hecke.a} and~\ref{prop:elem prop Coxeter Hecke.b}, we show that~$\tilde\Phi$ and~$\tilde\Phi_\star$ defined above are
homomorphisms
$W(\wh M)\to W(M)$ (respectively, $(W(\wh M),\star)\to (W(M),\star)$) provided that the~$\pi_M(\Phi(\wh T_i))$
(respectively, $\pi^\star_M(\Phi(\wh T_i))$), $i\in\wh I$ are involutions (respectively, idempotents).
For, note
the following obvious
\begin{lemma}\label{lem:hom on gens}
Let~$\mathsf M$, $\mathsf M'$ be monoids and let~$S$
be a set of generators for~$\mathsf M$. The following
are equivalent for a map
$f:\mathsf M\to \mathsf M'$.
\begin{enmroman}
\item $f$ is a homomorphism of monoids;
\item\label{lem:hom on gens.b} $f(sx)=f(s)f(x)$ for all~$s\in S$, $x\in\mathsf M$.
\end{enmroman}
\end{lemma}
Let~$w\in W(\wh M)$, $i\in\wh I$. If~$\wh s_iw=\wh s_i\times w=\wh s_i\star w$ then by Theorem~\partref{thm:Tits.b}
$$
\tilde\Phi(\wh s_i w)=\pi_M(\Phi(\wh T_{\wh s_iw}))
=\pi_M(\Phi(\wh T_i\wh T_w))
=\pi_M(\Phi(\wh T_{i}))\pi_M(\Phi(\wh T_w))
=\tilde\Phi(\wh s_i)\tilde\Phi(w),
$$
and similarly for~$\tilde\Phi_\star$. Otherwise,
write~$w=\wh s_i\times w'$ for some~$w'\in W(M')$. By the above,
$\tilde\Phi(w)=\tilde\Phi(\wh s_i)\tilde\Phi(w')$ and so
$$
\tilde\Phi(\wh s_iw)=\tilde\Phi(w')=\tilde\Phi(\wh s_i)^2\tilde\Phi(w')=\tilde\Phi(\wh s_i)\tilde\Phi(w).
$$
Likewise,
$\tilde \Phi_\star(w)=\tilde\Phi_\star(\wh s_i)\star \tilde\Phi_\star(w')$ whence
$$
\tilde\Phi_\star(\wh s_i\star w)=\tilde\Phi_\star(w)=\tilde\Phi_\star(\wh s_i)\star \tilde\Phi_\star(w')=\tilde\Phi_\star(\wh s_i)^{\star 2}\star\tilde\Phi_\star(w
')=\tilde\Phi_\star(\wh s_i)\star\tilde\Phi_\star(w).
$$
Thus, the condition~\ref{lem:hom on gens.b} from Lemma~\ref{lem:hom on gens} is satisfied.

Since~$\tilde\Phi$ is a homomorphism~$W(\wh M)\to W(M)$ and~$\tilde\Phi\circ\pi_{\wh M}=\pi_M\circ\Phi$ on
generators of~$\Br^+(\wh M)$,
it follows that $\tilde\Phi\circ\pi_{\wh M}=\pi_M\circ\Phi$. Therefore,
the assignments $\pi_{\wh M}(T)\mapsto \pi_M(\Phi(T))$,
$T\in\Br^+(\wh M)$ yield a well-defined homomorphism
$W(\wh M)\to W(M)$ and so~$\Phi$ is of Coxeter type.
The argument for~$\tilde\Phi_\star$ is similar.

Since a standard homomorphism is square free and satisfies~$\Phi(\wh T_i)=T_{w_\circ^{[\Phi](i)}}$ for all~$i\in\wh I$, part~\ref{prop:elem prop Coxeter Hecke.b'} follows from parts~\ref{prop:elem prop Coxeter Hecke.a} and~\ref{prop:elem prop Coxeter Hecke.b}. 
Part~\ref{prop:elem prop Coxeter Hecke.c} is immediate from parts~\ref{prop:elem prop Coxeter Hecke.a} and~\ref{prop:elem prop Coxeter Hecke.b} together with Lemmata~\partref{lem:elem Artin hom.b'} and~\ref{lem:can image op inv}.
\end{proof}
\begin{remark}
As a byproduct, we conclude that $\Phi\in\Hom_{\mathscr A}(\wh M,M)$ is of Coxeter (respectively, Hecke) type if and only if~$\tilde\Phi$ (respectively, $\tilde\Phi_\star$)
is a homomorphism of Coxeter groups (respectively, Hecke monoids) and in that case it identifies 
with the induced homomorphism~$\overline\Phi$ (respectively, $\overline\Phi_\star$).
\end{remark}

\begin{example}\label{ex:non sqf non Cox}
One checks (it requires approximately 270 applications of braid relations) that
the assignments
$\wh T_1\mapsto T_1T_2T_1T_4$, $\wh T_2\mapsto (T_2T_3)^2$
define a homomorphism $\Br^+(I_2(10))\to \Br^+(A_4)$ which
is of Hecke type since~$\pi^\star_{A_4}(T_1T_2T_1T_4)=w_{\circ}^{\{1,2,4\}}$
and $\pi^\star_{A_4}((T_2T_3)^2)=s_3\star s_2\star s_3\star s_3=w_\circ^{\{2,3\}}$ but not of Coxeter type since $\pi_{M}((T_2T_3)^2)=s_3 s_2$
and hence is not an involution.
\end{example}

\begin{example}\label{ex: sqf cox not Hecke}
The assignments
$\wh T_i\mapsto T_{2i}T_{2i-1}T_{2i+1}T_{2i}$,
$i\in [1,n]$ define a strongly square free Coxeter type
homomorphism $\Br^+(A_n)\to \Br^+(A_{2n+1})$
(see~Theorem~\ref{thm:monomial brd}). 
However, this homomorphism is not of Hecke type since
$\pi^\star_{A_{2n+1}}(T_2 T_1 T_3 T_2)=s_2s_1s_3 s_2\not=w_\circ^{[1,3]}$.
\end{example}

\begin{example}
Let~$M$ be of finite irreducible type and let~$I=I_1\sqcup I_2$ be any partition of~$I$
into disjoint non-empty subsets.
Choose Coxeter elements~$C_j$, $j\in\{1,2\}$ in~$\Br^+_{I_j}(M)$. Then $C_1 C_2$, $C_2 C_1$ are Coxeter elements in~$\Br^+(M)$. Let~$m=h(M)/2$ if~$T_{w_\circ^I}$ is central in~$\Br^+(M)$ (in which case $h(M)$ is even) and~$m=h(M)$
otherwise. By Proposition~\ref{prop:Coxeter splitting},
$(C_1C_2)^{2m}=(C_2C_1)^{2m}=T_{w_\circ}^{2m/h(M)}$,
whence the assignments $\wh T_i\mapsto C_i$,
$i\in \{1,2\}$ define $\Phi\in\Hom_{\mathscr A}(I_2(2m),M)$ which 
is manifestly square free and, unless~$I_1$ and~$I_2$ are
self-orthogonal, is neither of Hecke nor of Coxeter type. If~$h(M)$ is even, then~$\Phi$ is also strongly square free.
\end{example}
\begin{example}
For all~$n\ge 2$, $0\le k\le n-2$ the
assignments~$\wh T_1\mapsto T_{w_{[1,k];[1,n+k]}}$,
$\wh T_2\mapsto T_{w_{[k+2,n-1];[k+2,2n-1]}}$ define 
a homomorphism~$\Br^+(B_2)\to \Br^+(A_{2n-1})$  which is square free but neither of
Hecke nor of Coxeter type (see Corollary~\ref{cor:strange homs}).
\end{example}

\begin{example}\label{ex:not sqf not cox not Hecke}
The assignments $\wh T_1\mapsto T_1T_2T_3T_2T_1T_3$,
$\wh T_2\mapsto T_2T_3T_4T_3T_2T_4$ define a homomorphism
$\Phi:\Br^+(I_2(5))\to\Br^+(A_4)$ which is not square free and is neither Hecke,
nor Coxeter type. Indeed, $\pi_M(\Phi(\wh T_1))=
s_1s_2s_3s_2s_1s_3=s_1s_3s_2s_1$ which is not an involution
since its square equals~$s_1s_2s_1s_3$,
while $\pi^\star_M(\Phi(\wh T_1))=s_1\star s_2\star s_3\star s_2\star s_1
\star s_3=
s_1s_2s_3s_2s_1\not=w_\circ^{[1,3]}$.
\end{example}
\begin{example}\label{ex:char hom heck}
The character homomorphism~$\Xi_{\mathbf X} \in\Hom_{\mathscr A}(\wh M,M)$
from Example~\ref{ex:char hom} is of Hecke type if and only if~$\pi^\star_M(X_i)$, $i\in\wh I$
is a family of commuting idempotents in $(W(M),\star)$. In particular, this happens if~$\pi^\star_M(X_i)=w_\circ^J$ for all~$i\in\wh I$ where $J\subset I$. 
\end{example}
\begin{example}\label{ex:shift by center}
Let~$M\in\Cox I$ be of finite type and 
let~$\mathbf z=(z_i)_{i\in I}$ be
as in Example~\ref{ex:cent decoration}.
Then~$\id_{\mathbf z}$ is an endomorphism of~$\Br^+(M)$ of Hecke type. 
Indeed, it suffices to verify that for~$M$ irreducible in which case $z_i=T_{w_\circ^I}^{a_i}$ for some~$a_i\in\ZZ_{\ge 0}$ by Proposition~\partref{prop:fund elts BrSa.d}. Then
$\pi^\star_M(z_i T_i)=w_\circ^{I}$ for all~$i\in I$ such that~$a_i>0$ and~$\pi^\star_M(z_iT_i)=s_i$ otherwise. 
\end{example}
\begin{example}
Homomorphisms from Example~\ref{ex:embedded std} are
of Coxeter-Hecke type.
\end{example}
\begin{lemma}\label{lem:T_w0 divides image}
Let $\Phi\in\Hom_{\mathscr{AH}}(\wh M,M)$ and let~$J\in\mathscr F(\wh M)$. Then
\begin{enmalph}
    \item\label{lem:T_w0 divides image.b}
    $[\Phi](J)\in\mathscr F(M)$ and $\overline\Phi_\star(w_\circ^J)
    =w_\circ^{[\Phi](J)}$.
    \item\label{lem:T_w0 divides image.a} If $\Phi$ is square free then
    $\Phi(\wh T_{w_\circ^J})=
    T_{w_\circ^{[\Phi](J)}}u=\Sigma_{[\Phi](J)}(u) T_{w_\circ^{[\Phi](J)}}$ for some~$u\in\Br^+_{[\Phi](J)}(M)$.
\end{enmalph}
\end{lemma}
\begin{proof}
Denote~$w=\overline\Phi_\star(w_\circ^J)$.
By Corollary~\partref{cor:cat AH AC.c}, $w=\pi^\star_M(\Phi(\wh T_{w_\circ^J}))$ whence
$\supp w=[\Phi](J)$. Since~$w_\circ^J$
is an idempotent in~$(W(\wh M),\star)$, it follows
that~$w$ is an idempotent in~$(W(M),\star)$ and
so $\supp w=[\Phi](J)\in\mathscr F(M)$ and~$w=w_\circ^{[\Phi](J)}$ by Corollary~\partref{cor:max elts.b}.

To prove part~\ref{lem:T_w0 divides image.a}, let~$j\in J$. Then~$\wh T_{w_\circ^J}$ is left divisible by~$\wh T_j$ by Proposition~\partref{prop:fund elts BrSa.b} whence
$\Phi(\wh T_{w_\circ^J})\in\Br^+_{[\Phi](J)}(M)$ is left divisible
by $\Phi(\wh T_j)$. Since~$\Phi$ is square free, $\Phi(\wh T_j)=T_{w_\circ^{[\Phi](j)}}$ by Proposition~\partref{prop:elem prop Coxeter Hecke.b}
and so,
again by Proposition~\partref{prop:fund elts BrSa.b},
is left divisible by all the~$T_i$, $i\in[\Phi](j)$.
Therefore, $\Phi(\wh T_{w_\circ^J})$ is left divisible by all the~$T_i$ with~$i\in[\Phi](J)=\bigcup_{j\in J}[\Phi](j)$. Since~$[\Phi](J)\in\mathscr F(M)$ by part~\ref{lem:T_w0 divides image.b}, $\Phi(\wh T_{w_\circ^J})$
is
left divisible by~$T_{w_\circ^{[\Phi](J)}}$ by Proposition~\partref{prop:fund elts BrSa.b}. The second equality follows
by Proposition~\partref{prop:fund elts BrSa.c}.
\end{proof}

Note some additional properties of
standard homomorphisms.
\begin{lemma}\label{lem:sqf Hecke hom}
Let~$\wh M=(\wh m_{ij})_{i,j\in\wh I}\in\Cox{\wh I}$, $M\in\Cox I$ and let~$\Phi\in\Hom_{\mathscr A}(\wh M,M)$ be
standard. Then
\begin{enmalph}
   \item\label{lem:sqf Hecke hom.a} $\Phi$ is uniquely determined by~$[\Phi]:\wh I\to \mathscr F(M)$;
    \item\label{lem:sqf Hecke hom.b}
    If~$\wh m_{ij}$, $i\not=j\in\wh I$ is odd then
    $\ell(w_\circ^{[\Phi](i)})=\ell(w_\circ^{[\Phi](j)})$;
    \item\label{lem:sqf Hecke hom.c}
    If~$\wh m_{ij}$, $i,j\in\wh I$ is even then
$(T_{w_\circ^{[\Phi](i)}}T_{w_\circ^{[\Phi](j)}})^{\wh m_{ij}}$ is ${}^{op}$-invariant;
    \item\label{lem:sqf Hecke hom.d}
    If~$\Phi$ is strongly square free and injective then~$\overline\Phi_\star$ is injective.
\end{enmalph}
\end{lemma}
\begin{proof}
Part~\ref{lem:sqf Hecke hom.a} is immediate from Proposition~\partref{prop:elem prop Coxeter Hecke.b}. Part~\ref{lem:sqf Hecke hom.b}
follows from Lemma~\partref{lem:elem Artin hom.c}
and Proposition~\partref{prop:elem prop Coxeter Hecke.b}.
Part~\ref{lem:sqf Hecke hom.c} follows from Proposition~\partref{prop:elem prop Coxeter Hecke.c}.
To prove part~\ref{lem:sqf Hecke hom.d}, suppose that $
\overline\Phi_\star(\wh w_1)=\overline\Phi_\star(\wh w_2)$ for
some~$\wh w_1,\wh w_2\in W(\wh M)$. Then $\pi^\star_M(\Phi(\wh T_{\wh w_1}))=\pi^\star_M(\Phi(\wh T_{\wh w_2}))$. Since~$\Phi$
is strongly square free, $\Phi(\wh T_{\wh w_i})=T_{w_i}$ for some~$w_i\in W(M)$, $i\in\{1,2\}$.
Since $\pi^\star_M|_{\SQF^+(M)}$ is a bijection,
it follows that $w_1=w_2$. Then~$\wh w_1=\wh w_2$ by
injectivity of~$\Phi$.
\end{proof}
In some cases, we can reconstruct a Coxeter or Hecke type homomorphism of Artin monoids from its ``shadow''.
\begin{lemma}\label{lem:lifting to Cox-Hecke}
Let~$\wh M=(\wh m_{ij})_{i,j\in\wh I}\in\Cox{\wh I}$, $M\in\Cox I$ and 
suppose that the $X_i\in\Br^+(M)$, $i\in\wh I$
satisfy 
\begin{enumerate}[label={$\arabic*^\circ.$},ref=
{$\arabic*^\circ$}]
    \item\label{lem:lifting to Cox-Hecke.1} $\brd{X_iX_j}{\wh m_{ij}}\in\SQF^+(M)$ for 
    all $i,j\in\wh I$ with~$\wh m_{ij}<\infty$;
    \item\label{lem:lifting to Cox-Hecke.2} The assignments $\wh s_i\mapsto\pi_M(X_i)$, $i\in \wh I$ define a homomorphism of Coxeter groups $W(\wh M)\to W(M)$
    
    or

\noindent
    the assignments $\wh s_i\mapsto\pi^\star_M(X_i)$, $i\in\wh I$ define a homomorphism of Hecke monoids $(W(\wh M),\star)\to (W(M),\star)$.
\end{enumerate}
Then the assignments $\wh T_i\mapsto X_i$, $i\in\wh I$
define a homomorphism $\Br^+(\wh M)\to\Br^+(M)$.
\end{lemma}
\begin{proof}
Let~$i\not=j\in\wh I$ with~$\wh m_{ij}<\infty$. Since 
$$\brd{\pi_M(X_i)\pi_M(X_j)}{\wh m_{ij}}=\pi_M(\brd{X_iX_j}{\wh m_{ij}})
$$ 
and~$\pi_M|_{SQF^+(M)}$ is a bijection onto~$W(M)$, the assumption~\ref{lem:lifting to Cox-Hecke.2} implies that~$\wh m_{ij}\in B(X_i,X_j)$. It remains to apply Lemma~\ref{lem:fund hom}. The argument in the Hecke version is identical.
\end{proof}

Let~$M$ be a Coxeter matrix over~$I$.
While for Hecke monoids the parabolic
projection
$p_J:(W(M),\star)\to (W_J(M),\star)$ is well-defined for any~$J\subset I$, in the framework
of Artin monoids and Coxeter groups analogous homomorphisms exist only in special cases.
\begin{proposition}\label{prop:parab proj Artin}
Let~$M=(m_{ij})_{i,j\in I}$ be a Coxeter matrix and
let~$J\subset I$. The assignments
$$
T_i\mapsto \begin{cases}
T_i,&i\in J,\\
1,&i\in I\setminus J,
\end{cases}
$$
define a surjective Hecke type homomorphism \plink{PJ}$P_J:\Br^+(M)\to\Br_J^+(M)$ if and only
if $m_{ij}$ is even for all $j\in J$, $i\in I\setminus J$.
In particular, if $J$ and~$I\setminus J$ are orthogonal
then both~$P_J$ and~$P_{I\setminus J}$ are homomorphisms
of respective Artin monoids and $T=P_J(T)P_{I\setminus J}(T)$ for
all~$T\in\Br^+(M)$. Moreover, the same assertions hold for Coxeter groups.
\end{proposition}
\begin{proof}
It follows from Lemma~\partref{lem:elem Artin hom.c} that if~$P_J$
is a homomorphism then~$m_{ij}$ is even for all~$i\in I\setminus J$, $j\in J$. 
For the converse, note that, for all~$i\in I\setminus J$, $j\in J$ 
we have $\brd{T_i\cdot 1}{m_{ij}}=\brd{1\cdot T_i}{m_{ij}}=
T_i^{\frac12m_{ij}}$ since~$m_{ij}$ is 
even and so~$P_J$ is a homomorphism.

If~$J$ and~$I\setminus J$ are orthogonal then $m_{ij}=2$ for all $i\in I\setminus J$, $j\in J$ and so both~$P_J$ and~$P_{I\setminus J}$ are well-defined homomorphisms.
The last statement is proved exactly as Lemma~\ref{lem:parab prod}.
\end{proof}

For irreducible Coxeter matrices of finite type, the only non-trivial examples of parabolic projections of Artin monoids are $P_{[1,n-1]}\in\Hom_{\mathscr A}(B_n,A_{n-1})$, 
$P_{\{1,2\}}\in\Hom_{\mathscr A}(F_4,A_2)$,
$P_{\{n\}}\in \Hom_{\mathscr A}(B_n,A_1)$
and $P_{\{i\}}\in\Hom_{\mathscr A}(I_2(2m),A_1)$, $m\ge 2$,
$i\in\{1,2\}$.
We establish properties of the first two in Propositions~\ref{prop:PJ Bn An-1} and~\ref{prop:PJ F4 A2}.

\begin{remark}\label{rem:prod coprod}
For any Coxeter matrices $M$, $M'$,
canonical morphisms $M\times M'\to M$
and~$M\times M'\to M'$ in either of categories~$\mathscr A$, $\mathscr C$ and~$\mathscr H$ involve parabolic projections in
an obvious way. 
\end{remark}

The following Lemma allows us to reduce the study of Hecke type homomorphisms of Artin monoids to that of fully supported ones with a connected codomain.
\begin{lemma}\label{lem:diagonal}
Let $M\in\Cox I$, $\wh M\in\Cox{\wh I}$ and let $\Phi\in\Hom_{\mathscr A}(\wh M,M)$
be of Hecke type (respectively, of Coxeter type, square free). Let~$J\subset
\Phi(\wh I)$ and suppose that~$J$ and~$\Phi(\wh I)\setminus J$ are orthogonal. Then $P_J\circ\Phi$ is of Hecke type (respectively, of Coxeter type, square free).

Conversely, given pairwise orthogonal $J_1,\dots,J_k\subset I$ and
Hecke (respectively, Coxeter, square free) homomorphisms
$\Phi_t:\Br^+(\wh M)\to \Br^+_{J_t}(M)$, the map
$\Phi:\Br^+(\wh M)\to\Br^+(M)$,
$T\mapsto \Phi_1(T)\cdots\Phi_k(T)$, $T\in\Br^+(\wh M)$, is a Hecke (respectively,
Coxeter, square free) homomorphism~$\Br^+(\wh M)\to\Br^+(M)$.
\end{lemma}
\begin{proof}
Observe that~$\pi^\star_{M_J}\circ P_J=p_J\circ \pi^\star_M$ and
$\pi_{M_J}\circ P_J=\tilde p_J\circ \pi_M$ where~$\tilde p_J:W(M)\to W_J(M)$ is the Coxeter group counterpart of~$P_J$. Thus, if $\pi^\star_M(\Phi(\wh T_i))$
(respectively, $\pi_M(\Phi(\wh T_i))$)
is
an idempotent (respectively, an involution) then so is~$\pi^\star_{M_J}(
P_J\circ \Phi(\wh T_i))$ (respectively, $\pi_{M_J}(P_J\circ \Phi(\wh T_i))$).
Furthermore, if $X_i=\Phi(\wh T_i)\in \SQF^+(M)$
then, since
$X_i=P_J(X_i)P_{I\setminus J}(X_i)$ by Proposition~\ref{prop:parab proj Artin}, it follows from Lemma~\ref{lem:sq free fact} that~$P_J(X_i)=P_J\circ\Phi(\wh T_i)$ is square free.

For the converse, since images of the~$\Phi_t$, $1\le t\le k$ commute in~$\Br^+(M)$, it follows that $\Phi(TT')=\prod_{1\le t\le k}\Phi_t(TT')
=\prod_{1\le t\le k}\Phi_t(T)\prod_{1\le k\le t}\Phi_t(T')=
\Phi(T)\Phi(T')$ and so~$\Phi$ is indeed a homomorphism. Since
the product of commuting idempotents (respectively, involutions)
is again an idempotent (respectively, an involution), it follows that
if all the~$\Phi_t$ are Hecke (respectively, Coxeter) then so is~$\Phi$.
Finally, since the product of square free elements
from~$\Br^+_{J}(M)$ and~$\Br^+_{K}(M)$ where~$J$ and~$K$ are orthogonal is obviously square free, if $\Phi_t(T_i)\in \SQF^+(M)\cap \Br^+_{J_t}(M)$ for all~$1\le t\le k$ then $\Phi(T_i)\in \SQF^+(M)$.
\end{proof}

We now discuss faithfulness and fullness of functors~$\mathsf H$
and~$\mathsf C$.

\begin{example}\label{ex:non-faithful}
Let~$M\in\Cox I$, $|I|>1$ be of finite type and irreducible. Let~$\mathbf z=(z_i)_{i\in I}$
where all~$z_i$ are central and not equal to~$1$ with
$z_i=z_j$ whenever~$m_{ij}$ is odd,
and let~$\mathbf z^2=(z_i^2)_{i\in I}$.
Then both~$\mathbf z$, $\mathbf z^2$ are decorations
of~$\id\in\Hom_{\mathscr A}(M,M)$ and
$\id_{\mathbf z}\not=\id_{\mathbf z^2}$. 
Yet 
$\overline{\id_{\mathbf z}}_\star(s_i)=w_\circ^I=
\overline{\id_{\mathbf z^3}}_\star(s_i)$ for all~$i\in I$. Thus, $\mathsf H$ is not faithful.

Similarly, if~$M$ is as above and not of type
$A_n$, $D_{n+1}$ with~$n$ odd or~$E_6$, let
$\mathbf z=(z_i)_{i\in I}$ where~$z_i=T_{w_\circ^I}$
and let~$\mathbf z^3=(z_i^3)_{i\in I}$. Then
$\overline{\id_{\mathbf z}}(s_i)=w_\circ^I=
\overline{\id_{\mathbf z^3}}(s_i)$,
yet~$\id_{\mathbf z}\not=\id_{\mathbf z^3}$.
Thus, $\mathsf C$ is not faithful either.
\end{example}

We say that~$\phi\in\Hom_{\mathscr H}(\wh M,M)$
is (square free) {\em liftable}
if there exists a (square free)
$\Phi\in\Hom_{\mathscr A}(\wh M,M)$ of Hecke type such that
$\overline\Phi_\star=\phi$. The corresponding notions
for homomorphisms of Coxeter groups are defined similarly.
Note that if~$\phi$
is square free liftable then~$\Phi(\wh T_{ w})=T_{\phi(w)}$
for all~$w\in W(\wh M)$.
\begin{lemma}\label{lem:non-lift-parab}
Let~$M\in\Cox I$.
For any~$J\subset I$, the parabolic projection $p_J:(W(M),\star)\to (W_J(M),\star)$ is liftable if and only if~$P_J$ is a well-defined homomorphism $\Br^+(M)\to\Br^+_J(M)$. 
\end{lemma}
\begin{proof}
If~$P_J$ is well-defined then, clearly, $p_J=\overline{(P_J)}_\star$. 
Suppose that~$P_J$ is not well-defined and $p_J=\overline{\Phi}_\star$ for some homomorphism $\Phi:\Br^+(M)\to \Br^+_J(M)$. By Proposition~\ref{prop:parab proj Artin}, there exist $i\in I\setminus J$, $j\in J$ such that~$m_{ij}\ge 3$ is odd. Since~$1=p_J(s_i)=\pi^\star_M(\Phi(T_i))$, it follows that~$[\Phi](i)=\emptyset$. Then by Lemma~\partref{lem:elem Artin hom.c}
we have~$[\Phi](j)=\emptyset$, whence 
$\phi(s_j)=\pi^\star(\Phi(T_j))=1$ which is a contradiction.
\end{proof}
\begin{example}\label{ex:non-liftable}
By Lemma~\ref{lem:non-lift-parab},
a parabolic projection~$p_J:(W(A_n),\star)\to (W_J(A_n),\star)$ with~$J\not=\emptyset,[1,n]$ is non-liftable.
\end{example}
\begin{example}\label{ex:non-liftable-1}
\label{ex:H3 D5}
The assignments
$$
s'_1\mapsto s_4,\quad s'_2\mapsto s_1s_3,\quad s'_3\mapsto s_2 s_5
$$
define a homomorphism of Hecke monoids $\phi:(W(H_3),\star)\to (W(D_5),\star)$. Indeed, $\phi$ is the composition of the homomorphism $(W(H_3),\star)\to
(W(D_6),\star)$ induced by the standard
homomorphism~\eqref{eq:unfold F4 E6} of respective Artin monoids with the
parabolic projection~$p_{[2,6]}:(W(D_6),\star)\to (W_{[2,6]}(D_6),\star)\cong (W(D_5),\star)$.
This homomorphism is easily seen to be
injective on every parabolic submonoid~$W_J(H_3)$ with~$|J|=2$. However,
$$
\phi(s'_2s'_1s'_3s'_2s'_3s'_2s'_1s'_2s'_3s'_2s'_3)=
\phi(s'_2s'_1s'_3s'_2s'_3s'_2s'_3s'_1s'_2s'_3s'_2s'_3)
$$
while both words are reduced in $W(H_3)$ and have different lengths.

The homomorphism~$\phi$ is not liftable. Indeed, if~$\phi=\overline\Phi_\star$ for some~$\Phi\in\Hom_{\mathscr A}(H_3,D_5)$ then $\Phi(T'_1)=T_4^{a_4}$,
$\Phi(T'_2)=T_1^{a_1}T_3^{a_3}$ and~$\Phi(T'_3)=T_2^{a_2}T_5^{a_5}$ where
$a_i\in\mathbb Z_{>0}$, $1\le i\le 5$
and $a_4=a_1+a_3=a_2+a_5$ by Lemma~\partref{lem:elem Artin hom.c}. Then the relation~$T'_1T'_2T'_1=T'_2T'_1T'_2$ in~$\Br^+(H_3)$
yields $T_4^{a_1+a_3}T_1^{a_1}T_3^{a_3}T_4^{a_1+a_3}
=T_1^{a_1}T_3^{a_3}T_4^{a_1+a_3}T_1^{a_1}T_3^{a_3}$
whence, since $\Br^+(D_5)$ is cancellative,
$$T_4^{a_1+a_3}T_3^{a_3}T_4^{a_1+a_3}
=T_3^{a_3}T_4^{a_1+a_3}T_3^{a_3}T_1^{a_1}.
$$
This forces~$a_1=0$ which is a contradiction.
\end{example}
\begin{example}
Let~$M=\tilde C_2$ (see Example~\ref{ex:affine}). Then~$p_{\{1,3\}}$ is liftable by Lemma~\ref{lem:non-lift-parab}. However, for any~$d_1,d_3\in\ZZ_{>0}$,
the assignments $T_1\mapsto T_1^{d_1}$,
$T_2\mapsto 1$, $T_3\mapsto T_3^{d_2}$
define a homomorphism $\Phi:\Br^+(M)\to \Br^+_{\{1,3\}}(M)$ such that~$\overline\Phi_\star=p_{\{1,3\}}$.
\end{example}
Thus, $\mathsf H$ is neither full nor faithful.

\begin{example}
The homomorphism of Artin monoids from Example~\ref{ex:non sqf non Cox} induces a homomorphism
of Hecke monoids~$(W(I_2(10)),\star)\to
(W(A_4),\star)$, $s'_1\mapsto s_1s_2s_1 s_4$,
$s'_2\mapsto s_2s_3s_2$
which is, therefore, liftable. However, it
is not square free liftable
since the canonical image of $X=((T_1T_2T_1T_4)(T_2T_3T_2))^5$
in~$W(A_4)$ is equal to $s_2s_1s_3s_2s_4$ which is not an involution. Then by Lemma~\ref{lem:can image op inv}, $X$
is not ${}^{op}$-invariant which contradicts Lemma~\partref{lem:sqf Hecke hom.c}.
\end{example}

\begin{remark}\label{rem:C not full}
Since every Coxeter group embeds into a symmetric group
in many different ways, it is highly unlikely that~$\mathsf C$ is full. For example, the left multiplication in~$S_3$ defines a homomorphism
$S_3=W(A_2)\to W(A_5)=S_6$ given by $\wh s_1\mapsto 
(1,2)(3,4)(5,6)=s_1s_3s_5$, $\wh s_2\mapsto(1,3)(2,5)(4,6)=s_2s_1s_4s_3s_2s_5s_4$. This
homomorphism does not have a straightforward
lifting to a homomorphism~$\Br^+(A_2)\to\Br^+(A_5)$,
and there are no reasons to expect it to have a more
sophisticated one.
\end{remark}

\begin{remark}
It is easy to see that strongly square free homomorphisms of Artin monoids also form a category.
\end{remark}

\subsection{Parabolic homomorphisms of Artin monoids}\label{subs:parab hom Artin}
We now introduce parabolic
homomorphism for Artin monoids and study their general properties.
\begin{definition}\label{defn:parabhom}
Let~$\wh M$, $M$ be Coxeter matrices.
We say that~$\Phi\in\Hom_{\mathscr A}(\wh M,M)$
is {\em parabolic} if $\Phi(\mP(\Br^+_J(\wh M)))\subset \mP(\Br^+_{[\Phi](J)}(M))$ for 
every maximal~$J\in\mathscr F(\wh M)$.
We say that~$\Phi$ is 
{\em strict parabolic} if~$\Phi(\wh T_{w_{\wh J;\wh K}})=T_{w_{J;K}}$ for any
maximal~$\wh K\in\mathscr F(\wh M)$
where~$K=[\Phi](\wh K)$ and~$J\subset K$ depends on~$\wh J$.
\end{definition}
It is immediate that homomorphism
from Example~\ref{ex:shift by center} are non-strict parabolic. 

\begin{lemma}\label{lem:parab sqf}
A strict parabolic homomorphism of Artin monoids of finite type is
strongly square free.
\end{lemma}
\begin{proof}
Since~$\wh T_{w_\circ^{\wh I}}$ is parabolic,
it follows that its image is a square
free element $T_{w_J}$ for some~$J\subset I$.
The assertion follows from Lemma~\partref{lem:elem Artin hom.b}.
\end{proof}
\begin{lemma}\label{lem:non-empty}
Let~$\wh M\in\Cox{\wh I}$, $M\in\Cox I$ be of finite type. Suppose that~$\wh M$ is irreducible and let~$\Phi\in\Hom_{\mathscr A}(\wh M,M)$ be strict parabolic. Then either~$\Phi$ is trivial or~$[\Phi](i)\not=\emptyset$ for all~$i\in\wh I$.
\end{lemma}
\begin{proof}
Suppose that~$[\Phi](i)=\emptyset$ for 
some~$i\in\wh I$. Since~$\wh M$ is irreducible, there is~$j\in\wh I$ such that~$\wh m_{ij}>2$.
If~$\wh m_{ij}$ is odd then~$[\Phi](j)=\emptyset$ by Lemma~\partref{lem:elem Artin hom.c}.
Otherwise, 
$$
\Phi(\wh T_{w_\circ^{\{i,j\}}})=
\Phi(\brd{\wh T_i\wh T_j}{\wh m_{ij}})
=\brd{\Phi(\wh T_i)\Phi(\wh T_j)}{\wh m_{ij}}=
\Phi(\wh T_j)^{\frac12 \wh m_{ij}}.
$$
Since~$\Phi$ is strongly square free by Lemma~\ref{lem:parab sqf}, it follows that~$[\Phi](j)=\emptyset$. Thus, $[\Phi](j)=\emptyset$ for all~$j\in\wh I$ 
in the connected component of~$i$ in~$\Gamma(\wh M)$, which is all of~$\Gamma(\wh M)$ since~$\wh M$ is irreducible. Therefore, $\Phi$ is trivial.
\end{proof}

\begin{proposition}\label{prop:parabolic<->w0J}
Let~$\wh M\in\Cox{\wh I}$, $M\in\Cox I$ and let~$\Phi\in\Hom_{\mathscr A}(\wh M,M)$ be
of Hecke type. Then~$\Phi$ is
strict parabolic if and only if~$\Phi$ is disjoint square free and~$\Phi(\wh T_{w_\circ^J})=T_{w_\circ^{[\Phi](J)}}$ for all~$J\subset\wh I$. In particular, a strict parabolic
Hecke type homomorphism
maps multiparabolic elements to multiparabolic elements.
\end{proposition}
\begin{remark}
This fails for parabolic homomorphisms which are not Hecke. For instance, for the homomorphism~$\Phi^{(2)}_n$ from Theorem~\partref{thm:monomial brd.a} we have $\Phi^{(2)}_n(\wh T_{w_\circ^{[1,n]}})=T_{w_{[1,2n+1]_2}}\not=T_{w_\circ^{[1,2n+1]}}$.
\end{remark}
\begin{proof}
By Lemma~\ref{lem:diagonal}, it suffices to consider the case when~$\Phi$ is fully supported.
We need the following
\begin{lemma}\label{lem:parabolic->w0J}
Let~$\Phi$ be a fully supported strict parabolic Hecke type homomorphism. Then $
\Phi(\wh T_{w_\circ^J})=T_{w_\circ^{[\Phi](J)}}$ for all~$J\subset I$.
\end{lemma}
\begin{proof}
Note first that, since~$w_{\emptyset}=w_\circ^{\wh I}$, we must
have~$\Phi(\wh T_{w_\circ^{\wh I}})=T_{w_J}$ for some~$J\subset I$.
Since~$\Phi$ is strongly square free by Lemma~\ref{lem:parab sqf},
by Lemma~\partref{lem:T_w0 divides image.a}, $\Phi(\wh T_{w_\circ^{\wh I}})=
T_{w_\circ^{[\Phi](\wh I)}}u=
T_{w_\circ^I}u$ for some~$u\in\Br^+(M)$; moreover,
we conclude that~$u=1$.
Since~$\ell(T_{w_J})<\ell(T_{w_\circ^I})$ for all~$J\not=\emptyset$, it follows
that~$\Phi(\wh T_{w_\circ^{\wh I}})=T_{w_\emptyset}=T_{w_\circ^I}$.

Let~$K\subset I$.
Since $\wh T_{w_\circ^I}=\wh T_{w_\circ^K}\wh T_{w_K}$ we have
$T_{w_\circ^I}=\Phi(\wh T_{w_\circ^K})T_{w_{K'}}$
for some~$K'\subset I$. Since~$T_{w_\circ^I}=T_{w_\circ^{K'}}T_{w_{K'}}$
and Artin monoids are cancellative, it follows that
$\Phi(\wh T_{w_\circ^K})=T_{w_\circ^{K'}}$ which
is contained in~$\Br^+_{[\Phi](K)}(M)$ and
is left divisible by~$T_{w_\circ^{[\Phi](K)}}$
by Lemma~\partref{lem:T_w0 divides image.a}. By Proposition~\ref{prop:fund elts BrSa} this forces~$K'=[\Phi](K)$.
\end{proof}
\begin{lemma}\label{lem:parab ->disjoint}
Suppose that~$\Phi\in\Hom_{\mathscr A}(\wh M,M)$ satisfies~$\Phi(\wh T_{w_\circ^J})=T_{w_\circ^{[\Phi](J)}}$ for all~$J\in\mathscr F(M)$,
$|J|\le 2$.
Then~$\Phi$ is disjoint.
\end{lemma}
\begin{proof}
Suppose that~$J=[\Phi](i)\cap [\Phi](j)\not=\emptyset$ for some~$i\not=j\in\wh I$. Since $w_\circ^{[\Phi](i)}=
w_i^{-1}\times w_\circ^J$ and
$w_\circ^{[\Phi](j)}=w_\circ^J\times w_j$ where
$w_k=w_{J;[\Phi](k)}$, $k\in\{i,j\}$, we have
$$
T_{w_\circ^{[\Phi](i)}}T_{w_\circ^{[\Phi](j)}}=
T_{w_i}^{op} T_{w_\circ^J}^2 T_{w_j}.
$$
It follows that~$\Phi(\wh T_{w_\circ^{\{i,j\}}})=\Phi(\brd{\wh T_i\wh T_j}{\wh m_{ij}})$
is not square free and hence cannot be equal to~$T_{w_\circ^K}$
for any~$K\in\mathscr F(M)$.
\end{proof}
The forward
direction Proposition~\ref{prop:parabolic<->w0J} follows from Lemmata~\ref{lem:parabolic->w0J} and~\ref{lem:parab ->disjoint}.
To establish the converse, we need the following
\begin{lemma}\label{lem:parab preserving}
Suppose that~$\Phi\in\Hom_{\mathscr A}(\wh M,M)$ satisfies $\Phi(\wh T_{w_\circ^J})=T_{w_\circ^{[\Phi](J)}}$
for all~$J\in\mathscr F(\wh M)$. Then~$\Phi$
is disjoint, of Coxeter-Hecke type,  
$\Phi(\wh T_{w_{J; K}})=
T_{w_{[\Phi](J);[\Phi](K)}}$
and also $\bar\Phi(w_{J; K})
=\bar\Phi_\star(w_{J; K})=w_{[\Phi](J);[\Phi](K)}$
for all~$J\subset K\in\mathscr F(\wh M)$.
\end{lemma}
\begin{proof}
The assumption that~$\Phi(\wh T_{w_\circ^J})=T_{w_\circ^{[\Phi](J)}}$
for any~$J\in\mathscr F(M)$ implies that~$\Phi$ square
free and of Hecke type and hence of Coxeter type. By Lemma~\ref{lem:parab ->disjoint}, $\Phi$ is disjoint.
Since~$\Phi(\wh T_{w_\circ^J})=T_{w_\circ^{[\Phi](J)}}$ for any~$J
\in\mathscr F(\wh M)$ and~$\wh T_{w_{J;K}}=\wh T_{w_\circ^J}^{-1}\wh T_{w_\circ^K}$ in~$\Br(\wh M)$,
it follows that~$\Phi(\wh T_{w_{J;K}})=
T_{w_\circ^{[\Phi](J)}}^{-1}T_{w_\circ^{[\Phi](K)}}=
T_{w_{[\Phi](J);[\Phi](K)}}$. Then by Corollary~\partref{cor:cat AH AC.c}, $\overline\Phi(w_{J;K})=\pi_M(\Phi(\wh T_{w_{J;K}}))
=\pi_M(T_{w_{[\Phi](J);[\Phi](K)}})=
w_{[\Phi](J);[\Phi](K)}$. The identity for~$\overline\Phi_\star$ is proved similarly.
\end{proof}
This completes the proof of the converse.
The last assertion in Proposition~\ref{prop:parabolic<->w0J} is immediate from Lemma~\ref{lem:parab preserving}.
\end{proof}

We now classify all strict parabolic Hecke type
homomorphisms of Artin monoids of finite
type.
\begin{theorem}\label{thm:artin parab}\label{thm:adm finite class}
Let~$\wh M\in\Cox{\wh I}$, $M\in\Cox I$ be irreducible and 
of finite type. The following homomorphisms
$\Phi:\Br^+(\wh M)\to \Br^+(M)$
are strict parabolic, of Hecke type and injective: 
\begin{enmalph}
\item\label{thm:adm finite class.unfold}
For~$\wh M=B_n$, $n\ge 2$, 
\begin{alignat}{4}
&M=A_{2n-1}:\label{eq:unfold Bn A2n-1}
&\qquad&\Phi(\wh T_i)=T_i T_{2n-i},&\quad &i\in [1,n-1],&\quad& 
\Phi(\wh T_n)=T_n,
\\
&M=A_{2n}:\label{eq:unfold Bn A2n}
&& \Phi(\wh T_i)=T_i T_{2n+1-i},&& i\in [1,n-1],&& \Phi(\wh T_n)=T_n T_{n+1}T_n,
\\
&M=D_{n+1}:\label{eq:unfold Bn Dn+1}
&& \Phi(\wh T_i)=T_i,&& i\in [1,n-1],&& \Phi(\wh T_n)=T_n T_{n+1};
\end{alignat}
\item\label{thm:adm finite class.unfold F4}
For~$\wh M=F_4$, $M=E_6$ 
\begin{align}
\label{eq:unfold F4 E6}
\Phi(\wh T_1)=T_1 T_5,\quad \Phi(\wh T_2)
=T_2 T_4,\quad \Phi(\wh T_3)=T_3,\quad 
\Phi(\wh T_4)=T_6
\end{align}
\item \label{thm:adm finite class.odd}
For~$\wh M=I_2(2m+1)$, $m>0$, $M=A_{2m}$
$$\Phi(\wh T_i)=\prod_{j\in [1,2m+1-i]_2} T_j, \qquad i\in\{1,2\};
$$
\item\label{thm:adm finite class.even}
For $\wh M=I_2(2m)$, $m>1$,
any~$M$ with~$h(M)=2m$ and any
partition $I=I_1\sqcup I_2$ of~$I$ into
non-empty self-orthogonal subsets 
$$
\Phi(\wh T_j)=T_{w_\circ^{I_j}}=
\prod_{i\in I_j} T_i, \qquad j\in\{1,2\};
$$
\item\label{thm:adm finite class.I8 F4} For~$\wh M=I_2(8)$, $M=F_4$,
$$
\Phi(\wh T_1)=T_1 T_4,\quad \Phi(\wh T_2)=T_2 T_3 T_2;
$$

\item\label{thm:adm finite class.H3}
For~$\wh M=H_3$, $M=D_6$, 
\begin{equation}\label{eq:unfold H3D6}
\Phi(\wh T_1)=T_1T_5,\quad \Phi(\wh T_2)=T_2T_4,\quad \Phi(\wh T_3)=T_3T_6;
\end{equation}

\item\label{thm:adm finite class.H4}
For~$\wh M=H_4$, $M=E_8$,
\begin{equation}\label{eq:unfold H4E8}
\Phi(\wh T_1)=T_1T_7,\quad \Phi(\wh T_2)=T_2T_6,\quad 
\Phi(\wh T_3)=T_3T_5,\quad \Phi(\wh T_4)=
T_4T_8.
\end{equation}
\end{enmalph}
The homomorphisms from parts~\ref{thm:adm finite class.unfold} and~\ref{thm:adm finite class.unfold F4} are isomorphisms onto submonoids of~$\Br^+(M)$ fixed by respective diagram automorphisms.
Moreover, every non-trivial strict parabolic Hecke type homomorphism is a composition of the above ones, diagram automorphisms and natural
inclusions of parabolic submonoids.  
\end{theorem}
\begin{proof}
In view of Lemma~\ref{lem:diagonal}, it 
suffices to classify fully supported
strict parabolic Hecke type homomorphisms for irreducible~$\wh M$ and~$M$. By
Proposition~\ref{prop:parabolic<->w0J}, 
if~$\Phi$ is such a homomorphism then 
it is square free, disjoint and satisfies 
$\Phi(\wh T_i)=T_{w_\circ^{[\Phi](i)}}$
for all~$i\in \wh I$ and 
$\Phi(\wh T_{w_\circ^{\{i,j\}}})=T_{w_\circ^{[\Phi](\{i,j\})}}$
for all~$i\not=j\in\wh I$. If~$\Phi$
is non-trivial then~$[\Phi](i)\not=\emptyset$ for all~$i\in\wh I$
by Lemma~\ref{lem:non-empty}. Such homomorphisms 
were classified in~\cites{Cri,God,Cas,Mue} where they are called LCM homomorphisms; in particular,
the list of such homomorphisms, up to
compositions with diagram automorphisms,
coincides with the one provided in the Theorem. By~\cite{Cri}, all LCM
homomorphisms are injective. The image of homomorphisms from parts~\ref{thm:adm finite class.unfold} and~\ref{thm:adm finite class.unfold F4} is clearly contained in the submonoid of~$\Br^+(M)$ fixed by the respective diagram automorphism, which is isomorphic to the respective~$\Br^+(\wh M)$ by~\cite{Cri1}.
\end{proof}
\begin{remark}
Curiously, homomorphisms~\eqref{eq:unfold H3D6} and~\eqref{eq:unfold H4E8} are liftings of homomorphisms of respective Coxeter groups studied in~\cite{BBO'C} in the framework of continuous crystals.
\end{remark}

\subsection{Submonoids of Hecke monoids generated by parabolic elements}
Note the following result (\cites{BK07,He09}).
\begin{proposition}\label{prop:submonoid *}
Let~$J\in\mathscr F(M)$. Then $\mP(W_J(M),\star)=\{ w_{J';J}\,:\, J'\subset J\}$ and is 
abelian. More precisely, for any
$J',J''\subset J$
there exists a unique \plink{*J}$J'\star_J J''=J''\star_J J'\subset J'\cap J''$ such that $w_{J';J}
\star w_{J'';J}=w_{J'\star_J J'';J}$.
\end{proposition}
\begin{proof}
We may assume, without loss of generality, that~$J=I$. For~$M$ of types
$A$ through $G$ this result was proven in~\cite{BK07}*{Proposition~2.30}. For~$M$ of type~$I_2(m)$,
the $\star$ product of any two parabolic elements
is easily seen to be equal to~$w_\circ^I=w_{\emptyset}$.
Finally, if~$M$ is of type~$H_3$ or~$H_4$
then using the injective fully supported Hecke type homomorphism
$\Phi:\Br^+(H_3)\to\Br^+(D_6)$ (respectively,
$\Phi:\Br^+(H_4)\to\Br^+(E_8)$) which is parabolic
by Theorem~\ref{thm:artin parab},
we have $\overline{\Phi}_\star(w_J\star w_{J'})=
\overline{\Phi}_\star(w_J)\star
\overline{\Phi}_\star(w_{J'})=w_{[\Phi](J)}\star
w_{[\Phi](J')}=w_{[\Phi](J)\star_I [\Phi](J')}$
by Lemma~\ref{lem:parab preserving}. It is easy
to check, for example using our Python program for
computations in Hecke monoids, that
$[\Phi](J)\star_I [\Phi](J')=[\Phi](J'')$
for some~$J''\subset \wh I$. Since~$\Phi$
is strict parabolic and injective, $\overline\Phi_\star$
is injective by Lemma~\partref{lem:sqf Hecke hom.d} and
$w_{[\Phi](J'')}=
\overline\Phi_\star(w_{J''})$ by Lemma~\ref{lem:parab preserving} whence $w_J\star w_{J'}=
w_{J''}$.
\end{proof}
Note that, in general $J\star_I J'\not=J\star_K J'$ for $J,J'\subset K\subsetneq I$. For example, if~$M=A_n$, $I=[1,n]$, $K=[1,m]$, $1\le m<n$, $J=[a,b],
J'=[a',b']\subset K$ then
 $[a,b]\star_I [a',b']=[a+a'-1,b+b'-n]$
 while $[a,b]\star_K [a',b']=[a+a'-1,b+b'-m]$ (see Corollary~\ref{cor:A J*K})
 which are equal if and only if~$b-a+b'-a'<m-1$ in which case both
 $J\star_I J'$ and~$J\star_K J'$ are empty sets.

\begin{lemma}\label{lem:hom parab submonoid}
Let~$\Phi:\Br^+(\wh M)\to\Br^+(M)$ be a strict parabolic
Hecke type homomorphism. Let~$J',J''\subset J\in\mathscr F(\wh M)$. Then $[\Phi](J'\star_J J'')
=[\Phi](J')\star_{[\Phi](J)}[\Phi](J'')$.
\end{lemma}
\begin{proof}
By Lemma~\ref{lem:parab preserving}, $\overline\Phi_\star$ is
a homomorphism of Hecke monoids and $\overline\Phi_\star(w_{K;J})=
w_{[\Phi](K);[\Phi](J)}$ for any~$K\subset J\in\mathscr F(\wh M)$. Then by Proposition~\ref{prop:submonoid *}
\begin{align*}
\overline\Phi_\star(w_{J';J}\star w_{J'';J})&=
\overline\Phi_\star(w_{J';J})\star\overline\Phi_\star(w_{J'';J})
\\
&=w_{[\Phi](J');[\Phi](J)}\star w_{[\Phi](J'');[\Phi](J)}\\
&=w_{[\Phi](J')\star_{[\Phi](J)}[\Phi](J'');[\Phi](J)}.
\end{align*}
On the other hand, since $w_{J';J}\star w_{J'';J}=
w_{J'\star_J J'';J}$ by Proposition~\ref{prop:submonoid *},
$$
\overline\Phi_\star(w_{J';J}\star w_{J'';J})=
w_{[\Phi](J'\star_J J''); [\Phi](J)}.
$$
The assertion is now immediate.
\end{proof}
\begin{lemma}\label{lem:hecke parab induces parab}
Let~$\wh M\in\Cox{\wh I}$, $M\in\Cox I$.
If~$\Phi\in\Hom_{\mathscr A}(\wh M,M)$
is parabolic and of Hecke type then~$\overline\Phi_\star\in\Hom_{\mathscr H}(\wh M,M)$
is also parabolic.
\end{lemma}
\begin{proof}
Assume, for simplicity, that~$\Phi$
is fully supported and~$\wh I$, $I$
are irreducible and of finite type.
Let~$\wh J\subset \wh I$. Since~$\Phi$
is parabolic,
$\Phi(T_{w_{\wh J}})=T_{w_{J_1}}\cdots T_{w_{J_k}}$ for some~$J_1,\dots,J_k\subset I$.
Then $\overline\Phi_\star(w_{\wh J})
=w_{J_1}\star\cdots\star w_{J_k}
=w_{J_1\star_I\cdots\star_I J_k}$
by Corollary~\partref{cor:cat AH AC.c}
and Proposition~\ref{prop:submonoid *}.
\end{proof}
\subsection{Multiparabolic elements}\label{subs:multiparab}
Given $J\in\mathscr F(M)$ and a collection
$J_1,\dots,J_k$ of pairwise disjoint subsets of~$J$,
define \plink{multiparab}$w_{J_1,\dots,J_k;J}:=w_\circ^{J_1}\cdots w_\circ^{J_k}w_\circ^J
\in W_J(M)$.
We call such elements {\em $J$-multiparabolic}.
Clearly, if $J_t$ and~$J_{t+1}$ are orthogonal for
some $1\le t<k$ then
$w_{J_1,\dots,J_k;J}=w_{J_1,\dots,J_t\cup J_{t+1},\dots, J_k;J}$. In particular, 
if $w_{J;K}$, $J\subset K\in\mathscr F(M)$ is a $K$-parabolic element then
$w_{J;K}=w_{J_1,\dots J_k;K}$ where $J_1,\dots,J_k$ are connected components of~$J$. 
If~$M$ is of finite type, we abbreviate $w_{J_1,\dots,J_k}:=w_{J_1,\dots,J_k;I}$
and call them multiparabolic elements.

Since the~$J_t$, $1\le t\le k$ are pairwise disjoint, $\ell(w_\circ^{J_1}\cdots w_\circ^{J_k})=
\sum_{1\le t\le k} \ell(w_\circ^{J_t})$ by Lemma~\ref{lem:extend supp} and so
$\ell(w_{J_1,\dots,J_k;J})=\ell(w_\circ^J)-\sum_{1\le t\le k}
\ell(w_\circ^{J_t})$ by~\eqref{eq:ell w w0} whence
$$
w_\circ^J=w_\circ^{J_k}\times\cdots\times w_\circ^{J_1}\times w_{J_1,\dots,J_k;J}
$$
and
$$
T_{w_\circ^J}=\Big(\dscprod_{1\le t\le k}T_{w_\circ^{J_t}} \Big) T_{w_{J_1,\dots,J_k;J}}=
T_{w_\circ^{J_k}\times\cdots\times w_\circ^{J_1}}T_{w_{J_1,\dots,J_k;J}}.
$$
In particular, $T_{w_{J_1,\dots,J_k;J}}=
T_{w_\circ^{J_k}\times\cdots\times w_\circ^{J_1}}{}^{-1}
T_{w_\circ^J}=\big(\ascprod_{1\le s\le k}T_{w_\circ^{J_s}}^{-1}\big)T_{w_\circ^J}$
in~$\Br(M)$.
The proof of the following Lemma is thus similar to that of Lemma~\ref{lem:parab preserving} and is omitted.
\begin{lemma}\label{lem:multiparab preserving}
Suppose that~$\Phi\in\Hom_{\mathscr A}(\wh M,M)$ satisfies $\Phi(\wh T_{w_\circ^J})=T_{w_\circ^{[\Phi](J)}}$
for all~$J\in\mathscr F(\wh M)$. Then
for any~$J\in\mathscr F(\wh M)$ and any collection $J_1,\dots,J_k\subset J$ of pairwise disjoint
subsets of~$J$, $\Phi(\wh T_{w_{J_1,\dots,J_k; J}})=
T_{w_{[\Phi](J_1),\dots,[\Phi](J_k);[\Phi](J)}}$
and also $\bar\Phi(w_{J_1,\dots,J_k; J})
=\bar\Phi_\star(w_{J_1,\dots,J_k; J})=w_{[\Phi](J_1),\dots,[\Phi](J_k);[\Phi](J)}$.
\end{lemma}

We will now enumerate multiparabolic elements
in all finite types.
First, denote $$
\mathbf P_{k}(M)=\{
w_\circ^{J_1}\cdots w_\circ^{J_k}\,:\,
\text{the $J_i$, $1\le i\le k$ are pairwise disjoint
and connected}\}
$$
and set~$\mathbf P_0(M)=\{1\}$, $\mathbf P_{-1}(M)=
\emptyset$.
Clearly, $\mathbf P_k(M)=
\emptyset$ if~$k>|I|$. Let~$\mathbf P(M)=\bigcup_{k\ge 0} \mathbf P_k(M)$. Evidently, the assignment
$w\mapsto ww_\circ^I$, $w\in\mathbf P(M)$ is a bijection
onto the set of multiparabolic elements in~$W(M)$.

Denote $\mathbf p_M(t)=\sum_{k\ge 0}
|\mathbf P_k(M)|t^k$. Clearly, if $I=I_1\cup\cdots\cup I_r$ where the $I_t$, $1\le j\le r$ are pairwise orthogonal then $\mathbf p_M(t)=\prod_{1\le j\le r}
\mathbf p_{M_{I_j}}(t)$.
\begin{theorem}
\label{thm:count multipar}
For~$M$ irreducible and of finite type,
$$
\mathbf p_M(t)=(1+t)(2\mathbf p_{M_J}(t)-\mathbf p_{M_{J'}}(t)),
$$
where~$J=I\setminus \{i\}$, $J'=I\setminus \{i,j\}$
with
$$
(i,j)=
\begin{cases}(1,2),&
\text{$M$ is not of type~$E$}\\
(n-1,n-2),&
\text{$M$ is of type~$E_n$, $n\in\{6,7,8\}$}
\end{cases}
$$
and
$$
\mathbf p_{A_0}(t)=1,\qquad \mathbf p_{A_1}(t)=1+t.
$$
In particular, $$
|\mathbf P(M)|=\mathbf p_M(1)=\begin{cases}
a_{|I|}, & \text{$M$ is not of type~$D$, $E$},\\
d_{|I|}, & \text{$M$ is of type~$D$},\\
e_6=4d_5-2a_4=856,& \text{$M$ is of type~$E_6$},\\
e_7=4e_6-2d_5=2928,&\text{$M$ is of type~$E_7$},\\
e_8=4e_7-2e_6=10000,&\text{$M$ is of type~$E_8$},
\end{cases}
$$
where
\begin{align}\label{eq:a_n binom}
a_n&=\sum_{0\le k\le \lfloor\frac12 n\rfloor}\binom{n}{2k}2^{n-k},\qquad n\ge 0,\\
\intertext{and}
d_n&=2 a_{n-1}+\sum_{0\le k\le \lfloor\frac12(n-1)\rfloor-1}\binom {n-2}{2k+1} 2^{n-k},\qquad n\ge 2.\label{eq:d_n binom}
\end{align}

\end{theorem}
\begin{proof}
Set~$\mathbf p_{A_0}(t)=1$.
For~$M=A_1$ we have $\mathbf P_0(M)=\{1\}$,
$\mathbf P_1(M)=\{s_1\}$ and so $\mathbf p_{A_1}(t)=1+t$. If~$|I|=2$ and~$m_{12}>2$ then~$\mathbf P_0(M)=\{1\}$,
$\mathbf P_1(M)=\{s_1,s_2,w_\circ^I\}$, $\mathbf P_2(M)=\{s_2s_1,
s_1s_2\}$ and so~$\mathbf p_M(t)=1+3t+2t^2=
(1+t)(1+2t)=(1+t)(2\mathbf p_{A_1}(t)-\mathbf p_{A_0}(t))$. Therefore, the recursion starts.

For the inductive step, note
that, by the choice of~$J$ and~$J'$,
$\mathbf P_k(M)$, $k> 0$
is the disjoint union of the following sets
\begin{align*}
\mathbf P_k(M)^{(0)}&=\{ w\in \mathbf P_k(M)\,:\,
\supp w\subset J\},\\
\mathbf P_k(M)^{(1)}&=\{ w s_i\,:\, w\in \mathbf P_{k-1}(M),\, \supp w\subset J\},\\
\mathbf P_k(M)^{(2)}&=\{ s_i w\,:\, w\in
\mathbf P_{k-1}(M),\, \{j\}\subset \supp w\subset J\},\\
\mathbf P_k(M)^{(3)}&=\{ w=w_\circ^{J_1}\cdots
w_\circ^{J_k}\in \mathbf P_k(M)\,:\,
\text{$\{i,j\}\subset J_{r(i)}$ for some~$1\le r(i)\le k$}\}.
\end{align*}
Note that~$r(i)$ in the definition of~$\mathbf P_k(M)^{(3)}$ is unique since all the~$J_p$, $1\le p
\le k$ are disjoint.
Clearly, $|\mathbf P_k(M)^{(0)}|=|\mathbf P_{k}(M_J)|$,
$|\mathbf P_k(M)^{(1)}|=|\mathbf P_{k-1}(M_J)|$ and~$|\mathbf P_k(M)^{(2)}|
=|\mathbf P_{k-1}(M_J)\setminus \mathbf P_{k-1}(M_{J'})|$. Finally, $|\mathbf P_k(M)^{(3)}|
=|\mathbf P_k(M_J)\setminus \mathbf P_k(M_{J'})|$. Indeed, define $f:\mathbf P_k(M)^{(3)}
\to \mathbf P_k(M_J)\setminus \mathbf P_k(M_{J'})$
by
$$
w_\circ^{J_1}\cdots w_\circ^{J_k}\mapsto
w_\circ^{J'_1}\cdots w_\circ^{J'_k},
\qquad J'_p=J_p\setminus \{i\},\,1\le p\le k
$$
and $g:\mathbf P_k(M_J)\setminus \mathbf P_k(M_{J'})\to \mathbf P_k(M)$ by
$$
w_\circ^{J'_1}\cdots w_\circ^{J'_k}\mapsto
w_\circ^{J_1}\cdots w_\circ^{J_k},
$$
where~$J_p=J'_p$ if~$j\notin J'_p$ and~$J_p=J'_p\cup\{i\}$ for the unique~$1\le p\le k$ such that~$j\in J_p$. Then, clearly, the image of~$g$ is contained
in~$\mathbf P_k(M)^{(3)}$ and
$f$ and~$g$ are inverses of each other. Thus,
$$
|\mathbf P_k(M)|=2(|\mathbf P_k(M_J)|+
|\mathbf P_{k-1}(M_J)|)-(|\mathbf P_k(M_{J'})|+
|\mathbf P_{k-1}(M_{J'})|)
$$
and so
$$
\mathbf p_M(t)=(1+t)(2 \mathbf p_{M_J}(t)-\mathbf p_
{M_{J'}}(t)).
$$
For all types but~$D$ and~$E$ this gives~$\mathbf p_M(t)=\mathbf p_{A_n}(t)$, $n=|I|$.
In particular, $|\mathbf P(M)|=a_n$ where~$a_n=|\mathbf P(A_n)|$ satisfies the recursion
\begin{equation}\label{eq:rec a_n}
a_0=1,\quad a_1=2,\quad a_n=4a_{n-1}-2a_{n-2},\qquad n\ge 2.
\end{equation}
Using elementary linear algebra we obtain from~\eqref{eq:rec a_n} that
$$
a_n=\tfrac12( (2+\sqrt2)^n+(2-\sqrt2)^n),\qquad n\ge 0.
$$
The binomial formula~\eqref{eq:a_n binom} is now immediate.

For type~$D_{n+1}$, the recursion yields
\begin{align*}
\mathbf p_{D_3}(t)&=\mathbf p_{A_3}(t)=(1+t)(2(1+t)(1+2t)-(1+t))=(1+t)^2(1+4t),
\\
\mathbf p_{D_4}(t)&=(1+t)(2\mathbf p_{D_3}(t)-
\mathbf p_{A_1\times A_1}(t))=
(1+t)^3(2(1+4t)-1))=(1+t)^3(1+8t)
\end{align*}
and $\mathbf p_{D_{n+1}}(t)=(1+t)(2\mathbf p_{D_n}(t)-
\mathbf p_{D_{n-1}}(t))$, $n\ge 4$. It follows that
$d_{n+1}=|\mathbf P(D_{n+1})|$ satisfies the recursion
$$
d_2=4,\quad d_3=20,\quad d_n = 4 d_{n-1}-2d_{n-2},\qquad n\ge 4.
$$
Therefore,
\begin{align*}
d_n&=(2+3\sqrt2)(2+\sqrt 2)^{n-2}+(2-3\sqrt 2)(2-\sqrt 2)^{n-2}\\
&=2 a_{n-1}+2\sqrt2((2+\sqrt 2)^{n-2}-(2-2\sqrt 2)^{n-2}),
\end{align*}
which immediately yields~\eqref{eq:d_n binom}. Finally,
since $\mathbf p_{E_6}(t)=
(1+t)(2\mathbf p_{D_5}(t)-\mathbf p_{A_4}(t))$,
$\mathbf p_{E_7}(t)=(1+t)(2\mathbf p_{E_6}(t)-
\mathbf p_{D_5}(t))$ and~$\mathbf p_{E_8}(t)=
(1+t)(2\mathbf p_{E_7}(t)-\mathbf p_{E_6}(t))$, the remaining assertions follow.
\end{proof}
\begin{remark}\label{rem:multipar combinatorics}
The sequence~$a_n$, $n\ge 0$ coincides, up to a shift,
with the sequence \OEIS{A006012} of the number of evil-avoiding permutations, which are the permutations avoiding patterns $2413$, $4132$, $4213$
and $3214$, in $S_{n}$ (\cite{KW22}). It is also the sequence of the numbers of {\em rectangular permutations}, that is, permutations that avoid patterns 2413, 2431, 4213, and 4231, in $S_{n}$, (see \cite{Bie17}*{Theorem 8}, \cite{CFF21}*{Corollary 8}). 
Half of the numbers of multiparabolic elements in $W(A_n)$ relates to the number of order-consecutive partitions of $n$ (\cites{CR15, HM95}) and the sequence~\OEIS{A007052}. However, none of these sets is directly connected with multiparabolic elements.
\end{remark}

Recall (see e.g.~\cite{Cheby}*{\S\S1.2.1,1.2.2}) that {\em Chebyshev polynomials} $T_n(x)$ (respectively, $U_n(x)$) of the
first (respectively second) kind satisfy the
same recursion $R_n(x)=2x R_n(x)-R_{n-1}(x)$, $n\ge 2$
with initial conditions~$T_0(x)=U_0(x)=1$ and
$T_1(x)=x$, $U_1(x)=2x$. We set~$U_{-1}(x)=0=T_{-1}(x)$.
We will now express both infinite families of polynomials~$\mathbf p_M(t)$
in terms of Chebyshev polynomials.
\begin{proposition}\label{prop:Chebyshev}
For all~$n\ge 0$
$$
\mathbf p_{A_n}(t)=(1+t)^{\lceil\frac n2\rceil} %
\begin{cases}
T_{\frac n2}(2t+1),&\text{$n$ is even},\\
U_{\lfloor \frac n2\rfloor}(2t+1)-U_{\lfloor\frac n2\rfloor-1}(2t+1),&\text{$n$ is odd}
\end{cases}
$$
while for~$n\ge 2$
$$
\mathbf p_{D_{n+1}}(t)=(1+t)^{\lceil \frac n2\rceil+1}\begin{cases}
2 T_{\frac n2}(2t+1)-T_{\frac n2-1}(2t+1),&\text{$n$ is even}\\
2 U_{\lfloor\frac n2\rfloor}(2t+1)-3 U_{
\lfloor\frac n2\rfloor-1}(2t+1)+
U_{\lfloor \frac n2\rfloor-2}(2t+1),&\text{$n$ is odd},
\end{cases}
$$

\end{proposition}
\begin{proof}
It is easy to see from the recursion in Theorem~\ref{thm:count multipar} that
$\mathbf p_{A_n}(t)=(1+t)^{\lceil \frac n2\rceil}Q_n(t)$
where~$Q_n(t)\in\mathbb Z[t]$ and satisfies
$Q_0(t)=Q_1(t)=1$, $Q_2(t)=1+2t$, $Q_3(t)=1+4t$ and for~$n\ge 4$
\begin{align*}
Q_n(t)&=2(1+t)^{1-\overline n} Q_{n-1}(t)-Q_{n-2}(t)\\
&=2(1+t)^{1-\overline n}(2(1+t)^{\overline n} Q_{n-2}(t)
-Q_{n-3}(t))-Q_{n-2}(t)
\end{align*}
Since $2(1+t)^{1-\overline n}Q_{n-3}(t)=Q_{n-2}(t)+Q_{n-4}(t)$,
it follows that
\begin{equation}\label{eq:same parity rec}
Q_n(t)=2(1+2t)Q_{n-2}(t)-Q_{n-4}(t).
\end{equation}
Let~$\tilde Q_n(x)=Q_{n}(\frac12(x-1))$, $n\ge 0$. Then $\tilde Q_0(x)=1=\tilde Q_1$, $\tilde Q_2(x)=x$,
$\tilde Q_3(x)=2x-1$, and
$$
\tilde Q_n(x)=2 x \tilde Q_{n-2}(x)-\tilde Q_{n-4}(x),\qquad n\ge 4.
$$
It follows immediately that $\tilde Q_{2k}(x)=T_k(x)$
for all~$k\ge 0$. To prove the second identity note
that $U_k(x)-U_{k-1}(x)$ satisfy the same recursion as
the~$U_k$, $k\ge 0$ but with the initial data
$U_0(x)-U_{-1}(x)=1$ and~$U_1(x)-U_0(x)=2x-1$ and
thus $\tilde Q_{2k+1}(x)=U_k(x)-U_{k-1}(x)$, $k\ge 0$.

For the~$D$ series, it follows from Theorem~\ref{thm:count multipar} that
$\mathbf p_{D_{n+1}}(t)=(1+t)\mathbf p_{A_n}(t)+(1+t)^{\lceil\frac n2\rceil+1}R_{n}(t)$, $n\ge 1$ (here we regard~$W(D_2)$
as $W(A_1)\times W(A_1)$)
where the $R_n$, $n\ge1$ satisfy the same recursion~\eqref{eq:same parity rec} with~$R_1(t)=0$, $R_2(t)=2t$, $R_3(t)=4t$ and~$R_4(t)=6t+8t^2$.
It remains to express these initial conditions
in terms of Chebyshev polynomials in~$2t+1$.
\end{proof}

\begin{lemma}\label{lem:len prop wJ K}
Let $K\in \mathscr F(M)$, $J\subset K$ and
let~$J_1,\dots,J_k\subset J$ be pairwise disjoint. Then
$$
w_{J_1,\dots,J_k;K}=w_{J_1,\dots,J_k;J}\times w_{J;K}.
$$
\end{lemma}
\begin{proof}
Note that
\begin{align*}
w_{J_1,\dots,J_k;K}&=w_\circ^{J_1}
\cdots w_\circ^{J_k} w_\circ^K =
(w_\circ^{J_1}
\cdots w_\circ^{J_k}w_\circ^J)(w_\circ^Jw_\circ^K)=w_{J_1,\dots,J_k;J}w_{J;K}
\end{align*}
while by~\eqref{eq:ell w w0}
\begin{align*}
\ell(w_{J_1,\dots,J_k;K})&=\ell(w_\circ^K)-\sum_{1\le t\le k}\ell(w_\circ^{J_t})
=\ell(w_\circ^K)-\ell(w_\circ^J)+\ell(w_\circ^J)-\sum_{1\le t\le k}\ell(w_\circ^{J_t})\\
&=\ell(w_{J;K})+\ell(w_{J_1,\dots,J_k;J}).\qedhere
\end{align*}
\end{proof}

\section{Light homomorphisms of Hecke and Artin monoids}
\label{sec:Parab proj}

In this section we describe a category of homomorphisms which unifies parabolic projections,
natural inclusions of parabolic submonoids and
tautological homomorphisms. We also prove that
all such homomorphisms in finite types are parabolic.

\subsection{Light homomorphisms of Hecke monoids}\label{subs:parab proj}
Tautological homomorphisms, parabolic projections and natural inclusions of parabolic submonoids belong to a larger class of homomorphisms.
\begin{definition}\label{defn:light}
Let~$M'$ and~$M$ be Coxeter matrices
over respective index sets~$I'$, $I$.
We say that~$\phi\in\Hom_{\mathscr H}(M',M)$
is {\em light} if
$|[\phi](i)|\le 1$ for all~$i\in I'$.
\end{definition}
The following is immediate.
\begin{lemma}\label{lem:light cat}
A composition of light homomorphisms of Hecke monoids is again light. In other words, Coxeter matrices and light homomorphisms of respective Hecke monoids form a subcategory of~$\mathscr H$.
\end{lemma}

Clearly, parabolic projections, tautological homomorphisms and natural inclusions are light.
We now describe another class of surjective light homomorphisms.
\begin{definition}\label{defn:foldable}
Let~$\varpi:I\to J$ be a surjective map.
We say that a Coxeter matrix~$M$ over~$I$ is {\em foldable along~$\varpi$} if
 $m_{ii'}=m_{ii''}$ for all
 $i,i',i''\in I$ with $\varpi(i')=\varpi(i'')\not=
\varpi(i)$.
\end{definition}
Note that any group~$G$ of automorphisms of~$\Gamma(M)$ 
gives rise to a map~$\varpi_G:I\to I/G$
such that~$M$ is foldable along~$\varpi_G$.

If~$M$ is foldable along~$\varpi$,
define~\plink{MII}$M^{\varpi}$ to be the matrix over~$J$ with~$(M^\varpi)_{jj}=1$, $j\in J$ and
$(M^{\varpi})_{jj'}=m_{ii'}$
for any~$i\in \varpi^{-1}(j)$, $i'\in \varpi^{-1}(j')$, $j\not=j'\in J$.
Clearly, $M^{\varpi}$
is a Coxeter matrix.
\begin{lemma}\label{lem:Heck orb fold}
Let~$\varpi:I\to J$ be surjective and let~$M\in\Cox I$
be foldable along~$\varpi$.
The
assignments~$s_i\mapsto s^{\varpi}_{\varpi(i)}$, $i\in I$, where the~$s^\varpi_j$, $j\in J$ are the
generators of~$(W(M^{\varpi}),\star)$,
define a surjective light optimal \plink{phi fold}$\mathbf f_{\varpi}\in\Hom_{\mathscr H}(M,M^{\varpi})$.
\end{lemma}
\begin{proof}
We may regard~$\varpi$
as a map~$I\to\mathscr F(M^{\varpi})$ in an obvious way. Then
$\varpi\in \Lambda(M,M^{\varpi})$ by definition of~$M^\varpi$ and so
$\mathbf f_{\varpi}:=\Theta_{\varpi}
\in\Hom_{\mathscr H}(M,M^\varpi)$.
The condition~\eqref{eq:optimal cond} is evidently satisfied.
\end{proof}
We refer to~$\mathbf f_{\varpi}$ as
the {\em folding along~$\varpi$}.
\begin{example}
Let~$M\in\Cox I$ and let~$\varpi$ be 
the unique map~$I\to\{1\}$. Then~$M$
is foldable along~$\varpi$, $M^\varpi=A_1$ and $\mathbf f_\varpi(s_i)=s^\varpi_1$ for all~$i\in I$.
\end{example}
\begin{example}\label{ex:light projection}
Let~$M=D_{n+1}$ and define~\plink{varpi (n,n+1)}$\varpi_{(n,n+1)}:[1,n+1]\to [1,n]$ by $\varpi_{(n,n+1)}(i)=i-\delta_{i,n+1}$,
$i\in [1,n+1]$.
Then~$M$ is foldable along~$\varpi=\varpi_{(n,n+1)}$,
$M^{\varpi}=A_n$
and, identifying~$W(M^\varpi)$ with~$W_{[1,n]}(M)$, we have $\mathbf f_{\varpi}(s_i)=
s_{i-\delta_{i,n+1}}$,
$i\in [1,n+1]$.

Similarly, if~$M=D_4$, define \plink{varpi (1,3,4)}$\varpi_{(1,3,4)}:[1,4]\to\{1,2\}$ by
$\varpi_{(1,3,4)}(i)=1$, $i\in\{1,3,4\}$ and~$\varpi_{(1,3,4)}(2)=2$. Then~$M$ is foldable along~$\varpi=\varpi_{(1,3,4)}$, $M^{\varpi}=A_2$ and, identifying~$W(M^\varpi)$ with~$W_{\{1,2\}}(M)$ we have $\mathbf f_\varpi(s_i)=s_1$, $i\in\{1,3,4\}$, $\mathbf f_\varpi(s_2)=s_2$.
\end{example}
\begin{example}
Let~$M\in\Cox I$
and suppose that~$I=I_1\sqcup I_2$
where 
$m_{ij}=m\ge 3$ for all~$i\in I_1$,
$j\in I_2$. Define $\varpi:I\to\{1,2\}$
by $\varpi(i)=j$ provided that~$i\in I_j$.
Then~$M$ is foldable
along $\varpi$ and
$M^{\varpi}=I_2(m)$.
\end{example}

Since tautological homomorphisms are light 
and for any light homomorphism~$\phi$ of Hecke monoids $\phi_{op}$ is also light, 
to describe all light homomorphisms it suffices to describe all optimal ones.
\begin{proposition}\label{prop:classify light}
Every optimal light homomorphism of Hecke monoids can be canonically presented as a composition of a parabolic projection, the
folding along a surjective map and a natural inclusion.
\end{proposition}
\begin{proof}
Let~$M\in\Cox I$, $M'\in\Cox{I'}$ and let $\phi:\Hom_{\mathscr H}(M',M)$
be light. Let~$I'_s=\{
i\in I'\,:\, |[\phi](i)|=s\}$, $s\in \{0,1\}$. Then
$\phi=\phi|_{W_{I'_1}(M)}\circ p_{I'_1}$. Furthermore,
if~$\phi$ is not surjective then, since~$\phi$ is light, its image
is~$(W_J(M),\star)$ for some~$J\subset I$ and so we can write
$\phi$ as a composition of a surjective
light homomorphism with the natural
inclusion~$\iota_J$. Therefore, it remains to describe light
homomorphisms with~$I'_0=\emptyset$
which are optimal and surjective.

For such a homomorphism,
$|[\phi](i)|=1$ for all~$i\in I'$ and so
we can regard~$[\phi]$ as a surjective map~$I'\to I$.
By Theorem~\partref{thm:Hom Heck Mon},
$$
m'_{ij}\ge \max(\mu_M([\phi](i),[\phi](j)),\mu_M([\phi](j),[\phi](i)))
=m_{[\phi](i)[\phi](j)},\qquad i\not=j\in I'.
$$
Since~$\phi$ is optimal, it follows
that~$m'_{ij}=m_{[\phi](i)[\phi](j)}$
for all~$i,j\in\wh I'$ such that~$[\phi](i)\not=[\phi](j)$.
Thus, $M'$ is foldable along~$[\phi]$, $M'{}^{[\phi]}=M$
and~$\phi=\mathbf f_{[\phi]}$.
\end{proof}
It turns out that every surjective homomorphism of Hecke monoids has a ``light core'', that is,
restricts to a surjective light homomorphism from a maximal parabolic submonoid.
\begin{lemma}\label{lem:surj hom Heck bas}
Let~$M\in\Cox I$, $M'\in\Cox{I'}$ and
let~$\phi\in\Hom_{\mathscr H}(M',M)$ be surjective. Then
there exists a unique maximal subset $J'=J'(\phi)$ of~$I'$ such that
$\phi|_{W_{J'}(M')}$ is surjective and light.
\end{lemma}
\begin{proof}
Let~$i\in I$. Since~$\phi$ is surjective, $G_i:=\phi^{-1}(s_i)\not=\emptyset$ and
is a subsemigroup of $(W(M'),\star)$.
Then for any~$x\in G_i$ we have
$\{i\}=\supp\phi(x)=[\phi](\supp x)$
by Lemma~\partref{lem:[phi]comp.a}
and so~$[\phi](j)\in \{\{i\},\emptyset\}$ and
$\{j\in \supp G_i\,:\, [\phi](j)=\{i\}\}$ is non-empty. Let~$J'=\bigcup_{i\in I}\supp G_i$.
It follows that~$\phi|_{W_{J'}(M')}$
is surjective and~$\phi(s'_j)\in\{s_i\,:\, i\in I\}\cup\{1\}$ for all~$j\in J'$. It remains to observe that every~$S\subset I'$ with the same properties is contained in~$J'$.
\end{proof}
\begin{example}
Let~$M'$ be a Coxeter matrix over~$I'$
with $m'_{ij}=m>2$ for some~$i\not=j\in I'$.
Then the assignments $s'_i\mapsto s_1$, $s'_j\mapsto s_2$ and $s'_k\mapsto w_\circ^{\{1,2\}}=\brd{s_1s_2}m$, $k\in I'\setminus\{i,j\}$ define a surjective~$\phi\in\Hom_{\mathscr H}(M',I_2(m))$
with~$J'=J'(\phi)=\{i,j\}$. Indeed,
define~$\xi:I'\to \mathscr P(\{1,2\})$ by
$\xi(i)=\{1\}$, $\xi(j)=\{2\}$ and~$\xi(k)=\{1,2\}$,
$k\in I'\setminus\{i,j\}$. Then
$\mu_{I_2(m)}(\{s\},\{1,2\})=
\mu_{I_2(m)}(\{1,2\},\{s\})=2$ for $s\in\{1,2\}$
by Lemma~\ref{lem:char w_0 monoid},
and so~$\xi\in\Lambda(M',I_2(m))$.
Clearly, $\phi=\Theta_\xi$.
\end{example}

Note that every homomorphism of Hecke monoids canonically extends to a surjective one.
Indeed, the free product of
light homomorphisms is again light
by Lemma~\ref{lem:free product}, that is the light category admits coproducts. The following Corollary
is an immediate consequence of Lemma~\ref{lem:free product} and
general properties of free products.
\begin{corollary}\label{cor:surject lift}
For any homomorphism~$\phi:(W(M'),\star)\to (W(M),\star)$, the canonical homomorphism
$\wh\phi:=\phi\textstyle\coprod\iota_{\tilde I}:(W(M\textstyle\coprod M_{\tilde I}),\star)\to (W(M),\star)$ is surjective,
where~$\tilde I=\{ i\in I\,:\,
s_i\notin \phi(W(M'))\}$.
\end{corollary}

Clearly, $\wh\phi=\phi$ if~$\phi$ is surjective.
If~$\phi$ is optimal then it is easy
to see that
\begin{equation}
((M'\textstyle\coprod M_{\tilde I})_{opt})_{ij}
=\begin{cases}
m'_{ij},& \{i,j\}\subset I',\\
m_{ij},&\{i,j\}\subset \tilde I,\\
\max(2,\mu_M([\phi](i),\{j\}),
\mu_M(\{j\},[\phi](i))),& i\in I',\,j\in \tilde I.
\end{cases}\label{eq:ext to surj}
\end{equation}

\subsection{Light homomorphisms of Artin monoids}\label{subs:parab proj Artin}
Similarly to the case of Hecke monoids,
we now introduce the notion of light
homomorphisms of Artin monoids.
\begin{definition}\label{defn:light Artin}
Let~$\wh M\in\Cox{\wh I}$, $M\in\Cox I$.
We say that~$\Phi\in\Hom_{\mathscr A}(\wh M,M)$ is {\em light}
if~$|[\Phi](i)|\le 1$ for all~$i\in \wh I$.
\end{definition}
A light homomorphism is manifestly of Coxeter-Hecke type.
As before, a composition of light homomorphisms is again light, as well as their free product.
Clearly, natural inclusions of parabolic submonoids,
parabolic projections,
tautological homomorphisms and diagram automorphisms are light. Also, if we write a light homomorphism
$\Phi$ as $\Phi''\circ\Phi'$ where~$\Phi''$ is optimal and~$\Phi'$ is tautological (Lemma~\ref{lem:factor homs}) then~$\Phi''$ is also light.

The following is immediate.
\begin{lemma}\label{lem:H functor light}
Let~$\wh M\in\Cox{\wh I}$, $M\in\Cox I$ and let
$\Phi\in\Hom_{\mathscr A}(\wh M,M)$ be light. Then~$\overline\Phi_\star$ is light.
In other words, the functor~$\mathsf H$ restricts to the functor~$\mathsf H_{light}$ from the light Artin category to the light Hecke category.
\end{lemma}

Like homomorphisms of Hecke monoids,
surjective homomorphisms of Artin
monoids always have a ``light core''.
\begin{lemma}\label{lem:surj hom Artin bas}
Let~$\wh M\in\Cox{\wh I}$, $M\in\Cox I$ and
let~$\Phi:\Br^+(\wh M)\to \Br^+(M)$ be a surjective homomorphism of monoids. Then
there exists a unique maximal~$\wh J=\wh J(\Phi)\subset \wh I$ such that
$\Phi|_{\Br^+_{\wh J}(\wh M)}$ is surjective and light.
\end{lemma}
The proof is exactly the same
as that of Lemma~\ref{lem:surj hom Heck bas} and is omitted.

Another class of light homomorphisms is obtained using the construction along
the lines of Lemma~\ref{lem:Heck orb fold}.
\begin{lemma}\label{lem:orbits fold}
Let~$\varpi:I\to J$ be surjective and suppose that~$M$ is foldable along~$\varpi$. The
assignments~$T_i\mapsto T^{\varpi}_{\varpi(j)}$, $i\in I$, where the~$T^\varpi_j$, $j\in J$
are generators of~$\Br^+(M^\varpi)$,
define a surjective light homomorphism of monoids \plink{F wpi}$\mathbf F_{\varpi}:\Br^+(M)\to \Br^+(M^{\varpi})$. Moreover, $\overline{(\mathbf F_\varpi)_\star}=\mathbf f_\varpi$.
\end{lemma}
\begin{proof}
Let~$i\not=i'\in I$ with~$m_{ii'}<\infty$. If~$\varpi(i)=\varpi(i')$
then $\brd{T^\varpi_{\varpi(i)}T^\varpi_{\varpi(i')}}{m_{ii'}}=(T^\varpi_{\varpi(i)})^{m_{ii'}}=
\brd{T^\varpi_{\varpi(i')}T^\varpi_{\varpi(i)}}{m_{ii'}}.
$
Otherwise, $(M^\varpi)_{\varpi(i)\varpi(i')}=m_{ii'}$
and so $\brd{T^{\varpi}_{\varpi(i)}T^{\varpi}_{
\varpi(i')}}{m_{ii'}}
=\brd{T^{\varpi}_{\varpi(i')}T^{\varpi}_{\varpi(i)}}{m_{ii'}}$. The last assertion is obvious.
\end{proof}
\begin{example}
For any~$M\in\Cox I$, $\ell:\Br^+(M)\to (\ZZ_{\ge 0},+)\cong\Br^+(A_1)$
identifies with~$\mathbf F_\varpi$ where~$\varpi$ is the unique map~$I\to\{1\}$.
\end{example}
\begin{example}
Homomorphisms from Example~\ref{ex:light projection} lift
to homomorphisms of the corresponding Artin monoids.
\end{example}

\begin{example}\label{ex:affine}
Using Lemma~\ref{lem:orbits fold} we can obtain more homomorphisms similar to those discussed in Example~\ref{ex:1.6}. 
Consider affine Coxeter graphs labeled as follows:
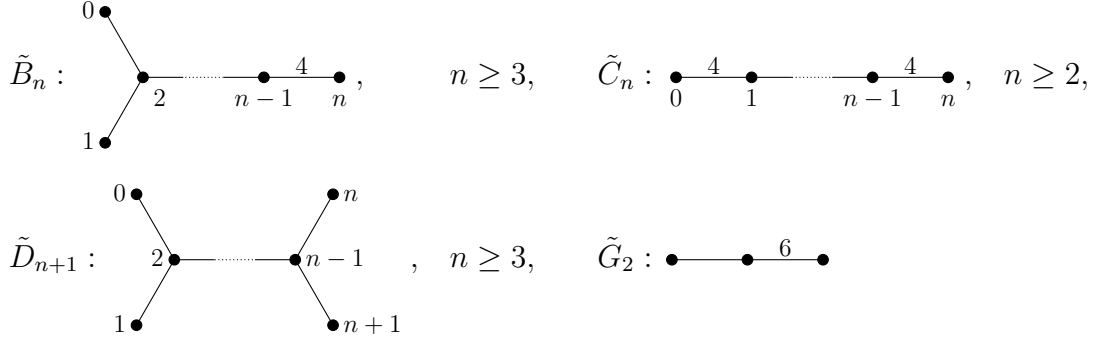
\begin{alignat*}{8}
&\tilde B_n: \dynkin[extended,o/.append style={fill=black},labels={0,1,2,n-1,n},Coxeter,edge length=1cm] B{**...**},&\quad & n\ge 3, %
&\qquad &\tilde C_n: \dynkin[extended,o/.append style={fill=black},labels={0,1,n-1,n},Coxeter,edge length=1cm] C{*...**},&\quad & n\ge 2,&\\
&\tilde D_{n+1}: \dynkin[extended,o/.append style={fill=black},labels={0,1,2,n-1,n,n+1},label directions={,,left,right,,},edge length=1cm] D{**...***}, &&n\ge 3,
&
&\tilde G_2: \dynkin[extended,o/.append style={fill=black},edge length=1cm] G[1]2
\end{alignat*}
We have the following unfolding homomorphisms
$$
\begin{array}{c|c|l}
\wh M&M&\multicolumn{1}{c}{\Phi}\\
\hline
\hline
\vphantom{\tilde{\tilde B}}\tilde B_n& \tilde D_{n+1}&
\wh T_i\mapsto T_i,\,i\in[0,n-1],\, \wh T_n\mapsto T_nT_{n+1}\\
\hline
\vphantom{\tilde{\tilde C}}\tilde C_n& \tilde B_{n+1}& \wh T_0\mapsto T_0T_1,\,T_i\mapsto T_{i+1},\, i\in[1,n]\\
\hline
\vphantom{\tilde{\tilde C}}\tilde G_2& \tilde D_{4}& \wh T_0\mapsto T_0,\,\wh T_1\mapsto T_2,\, \wh T_2\mapsto T_1T_3T_4\\
\hline
\end{array}
$$
while Lemma~\ref{lem:orbits fold} yields the following homomorphisms
$$
\begin{array}{c|c|l}
M&M^\varpi&\multicolumn{1}{c}{\mathbf F_\varpi}\\
\hline
\hline
\vphantom{\tilde{\tilde B}}\tilde B_n& B_n & T_i\mapsto T^{\varpi}_{i+\delta_{i,0}},\,i\in[0,n]\\
\hline
\vphantom{\tilde{\tilde B}}\tilde D_{n+1}& D_{n+1} &
T_i\mapsto T^{\varpi}_{n+1-i+\delta_{i,n+1}},\,i\in[0,n+1]\\
\hline
\vphantom{\tilde{\tilde B}}\tilde D_4& A_3 &  T_0\mapsto T^{\varpi}_1,\,  T_2\mapsto T^{\varpi}_2,\,T_i\mapsto T^{\varpi}_3,\, i\in\{1,3,4\}\\
\hline
\end{array}
$$
Their compositions yield non-standard homomorphisms of the type discussed in Example~\ref{ex:1.6}, namely
\begin{alignat}{3}
&\Br^+(\tilde C_n)\to
\Br^+(\tilde B_{n+1})\to \Br^+(B_{n+1}),&\qquad  &\wh T_i\mapsto T_{i+1}^{i+\delta_{i,0}},&\quad  &i\in [0,n],\nonumber\\
&\Br^+(\tilde B_n)\to \Br^+(\tilde D_{n+1})\to \Br^+(D_{n+1}),& &\wh T_i\mapsto T_{n+1-i}^{1+\delta_{i,n}},&&  i\in [0,n],\nonumber\\
&\Br^+(\tilde G_2)\to \Br^+(\tilde D_4)\to
\Br^+(A_3),& &\wh T_i\mapsto T_{i+1}^{1+2\delta_{i,3}},&& i\in [0,3].\label{eq:elem Tits affine}
\end{alignat}
In addition, the non-standard
homomorphism  $\Br^+(\tilde C_2)\to \Br^+(A_3)$
defined by $T'_i\mapsto T_i^{1+\delta_{i,2}}$, $i\in[1,3]$ is the composition of a standard
homomorphism $\Br^+(\tilde C_2)\to \Br^+(\tilde A_3)$,
$T'_0\mapsto T_0$, $T'_1\mapsto T_1T_3$, $T'_2\mapsto T_2$, and the standard homomorphism $
\Br^+(\tilde A_3)\to\Br^+(A_3)$, $T'_0\mapsto T_1$,
$T'_1\mapsto T_2$, $T'_2\mapsto T_3$, $T'_3\mapsto T_2$,
where Coxeter graph %
of type~$\tilde A_3$ are labeled as follows
$$
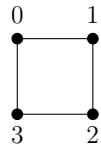
\qquad
\begin{dynkinDiagram}[name=upper,
text style/.style={scale=0.8},Coxeter,root radius=0.07,expand labels={0,1},label directions={above,above},edge length=1cm]A2
\node (current) at ($(upper root 1)+(0,-1cm)$) {};
\dynkin[at=(current),name=lower,expand labels={3,2},text style/.style={scale=0.8},Coxeter,root radius=0.07,edge length=1cm,label directions={below,below}]A2
\begin{pgfonlayer}{Dynkin behind}
\foreach \i in {1,...,2}%
{%
\draw[/Dynkin diagram]
($(upper root \i)$)
-- ($(lower root \i)$);%
}%
\end{pgfonlayer}
\end{dynkinDiagram}
$$
\end{example}

\subsection{Tits homomorphisms}\label{subs:Tits homs}
We now study a particular class of 
light homomorphisms which includes those 
discussed in Examples~\ref{ex:1.6} and~\ref{ex:affine}. 

Given a Coxeter matrix~$M=(m_{ij})_{i,j\in I}$
and~$\mathbf d\in \ZZ_{>0}^I$, define~\plink{M(d)}$M(\mathbf d)=(m_{ij}(\mathbf d))_{i,j\in I}$ by
\begin{equation}
m_{ij}(\mathbf d)=\begin{cases}
m_{ij},&\text{$m_{ij}\le 2$ or~$d_id_j=1$},\\
2d_id_j,&m_{ij}=3,\,d_id_j\in\{2,3\},\\
\infty,&\text{otherwise}.\label{eq:Tits cover}
\end{cases}
\end{equation}
Clearly~$M(\mathbf d)\in\Cox I$.
\begin{theorem}\label{thm:Tits homomorphisms}
Let~$M,\wh M\in\Cox I$ let~$\mathbf d\in \mathbb Z_{>0}^I$. The
assignments $\wh T_i\mapsto T_i^{d_i}$, $i\in I$ define an optimal
\plink{T d}$\mathbf T_{\mathbf d}\in\Hom_{\mathscr A}(\wh M,M)$
if and only if~$\wh M=M(\mathbf d)$.
\end{theorem}
\begin{proof}
We may assume, without loss of generality,
that~$I=\{1,2\}$, and so $\wh M=I_2(\wh m)$, $M=I_2(m)$, $m,\wh m\in\mathbb Z_{\ge 2}\cup\{\infty\}$, 
and that~$d_1\ge d_2$.

Suppose that~$\wh M=M(\mathbf d)$.
First we prove that the assignments $\wh T_i\mapsto
T_i^{d_i}$, $i\in \{1,2\}$ define a homomorphism
$\Br^+(I_2(\wh m))\to \Br^+(I_2(m))$.
We only need to consider the case when
$m=3$ and~$d_1=d\in \{2,3\}$, $d_2=1$,
the other cases being obvious.
\begin{lemma}\label{lem:finite elem Tits}
For~$d\in\{2,3\}$, the assignments $\wh T_1\mapsto T_1^d$, $\wh T_2\mapsto T_2$ define 
an optimal homomorphism~$\Br^+(I_2(2d))\to\Br^+(A_2)$.
\end{lemma}
\begin{proof}
Let $M'=A_3$ if~$d=2$ and~$M'=D_4$ if~$d=3$ and let~$I'=[1,d+1]$ be its index set.
Then $h(M')=2d$
and $I'=(I'\setminus\{2\})\cup\{2\}$ is
a partition of~$I'$ into (self-orthogonal) orbits of a suitable diagram automorphism. By Theorem~\partref{thm:adm finite class.even}, the assignments
$\wh T_1\mapsto \prod_{k\in I'\setminus \{2\}}T_k$,
$\wh T_2\mapsto T_2$ define a homomorphism
$\Br^+(I_2(2d))\to \Br^+(M')$. Furthermore, $M'$
is foldable along the map $\varpi:I'\to\{1,2\}$,
$\varpi(i)=1$, $i\in I'\setminus\{2\}$, $\varpi(2)=2$,
and the corresponding homomorphism
$\Br^+(M')\to\Br^+_{\{1,2\}}(M')\cong \Br^+(A_2)$
maps $T_i$, $i\in I'\setminus\{2\}$ to~$T_1$
and $T_2$ to itself. Their composition
is the desired homomorphism $\Br^+(I_2(2d))\to
\Br^+(A_2)$.
To prove that it is optimal, note that~$2d\in B(T_1^d,T_2)$ and so~$k=\min B(T_1^d,T_2)$ must
divide~$2d$ by Lemma~\partref{lem:taut homs.b}. Yet $T_1^d$
and~$T_2$ do not commute for otherwise
we would have~$T_2T_1^{d+1}=T_1^d T_2 T_1=T_2 T_1 T_2^d$
which by the cancellativity of~$\Br^+(A_2)$ yields a non-existent relation $T_1^d=T_2^d$. Thus, $k>2$. Since~$k$
is even by Lemma~\partref{lem:elem Artin hom.c} and~$2d\in\{4,6\}$ it follows that~$k=2d$.
\end{proof}

It remains to show that~$\mathbf T_{\mathbf d}$
is optimal, which also amounts
to proving the converse. This is obvious for the
case when~$d_1=d_2=1$ or~$m_{ij}=2$ 
and has already been proven in Lemma~\ref{lem:finite elem Tits} for~$d=d_1d_2\in\{2,3\}$ and~$m=3$.
If~$d_2>1$, then
by~\cite{CP} the submonoid of~$\Br^+(I_2(m))$
generated by $T_1^{d_1}$, $T_2^{d_2}$ is free
and so~$B(T_1^{d_1},T_2^{d_2})=\emptyset$.

It remains to consider the case when~$d=d_1>d_2=1$ and either~$m_{ij}>3$ or~$d>3$.
\begin{proposition}\label{prop:key converse Tits}
Let~$m>2$ and suppose that~$\mathbf T_{(d,1)}
\in\Hom_{\mathscr A}(I_2(\wh m),I_2(m))$ with~$d\ge 2(1+\delta_{m,3})$.
Then~$\wh m=\infty$.
\end{proposition}
\begin{proof}
Let $q\in\mathbb R\setminus\{0\}$,
$z\in\mathbb C\setminus\{0\}$.
Define
$$
\check T_1=\begin{pmatrix}1-q^2&q\\q&0\end{pmatrix},\qquad \check T_2(z)=\begin{pmatrix}0&q z\\q z^*&1-q^2\end{pmatrix},
$$
where~$z^*$ denotes the complex conjugate of~$z$.
\begin{lemma}
Let~$\zeta_l$, $l\ge 3$ be an $l$th primitive complex root of unity.
Then for any~$q\in\mathbb R\setminus\{0\}$,
the assignments $T_1\mapsto \check T_1$,
$T_2\mapsto \check T_2(\zeta_l)$, $i\in\{1,2\}$ define a representation of~$\Br^+(I_2(l))$ on~$\mathbb C^2$. Moreover, in this representation
the matrix of $T^{op}$ is the adjoint of that of~$T\in\Br^+(I_2(l))$ with respect to the
standard hermitian product on~$\mathbb C^2$.
\end{lemma}
\begin{proof}
We have
\begin{align}
\brd{\check T_1\check T_2(z)}n&=q^{n-1}
\begin{cases}
\begin{pmatrix}
q z^*{}^{\frac n2}&(1-q^2)(1+z)r_{\frac n2-1}(z)\\
0&q z^{\frac n2}
\end{pmatrix},& \bar n=0,\\
\\
\begin{pmatrix}
(1-q^2)(r_{\lfloor \frac n2\rfloor}(z)+r_{\lfloor \frac n2\rfloor-1}(z))&q z^*{}^{\lfloor \frac n2\rfloor}\\
q z^{\lfloor \frac n2\rfloor}&0
\end{pmatrix},&\bar n=1,
\end{cases}\label{eq:T1T2(z)}\\
\intertext{and}
\brd{\check T_2(z)\check T_1}n&=q^{n-1}
\begin{cases}
\begin{pmatrix}
q z^{\frac n2}&0\\
(1-q^2)(1+z^*)r_{\frac n2-1}(z)&q z^*{}^{\frac n2}
\end{pmatrix},& \bar n=0,\\
\\
\begin{pmatrix}
0&q z{}^{\lceil \frac n2\rceil}\\
q z^*{}^{\lceil \frac n2\rceil}&
(1-q^2)(r_{\lfloor \frac n2\rfloor}(z)+z z^*r_{\lfloor \frac n2\rfloor-1}(z))
\end{pmatrix},&\bar n=1,
\end{cases}\label{eq:T2(z)T1}
\end{align}
where
\begin{equation}\label{eq: r_k defn}
r_k(z)=\sum_{0\le j\le k}z^{k-j}z^*{}^j=
\begin{cases}
(k+1) z^k,&z\in\mathbb R,\\
\dfrac{z^{k+1}-z^*{}^{k+1}}{z-z^*},&z\in\mathbb C\setminus\mathbb R.
\end{cases}
\end{equation}
Note that $r_k(z^*)=r_k(z)=r_k(z)^*$ and
\begin{equation}\label{eq:rec r_k}
z r_k(z)+z^*{}^{k+1}=r_{k+1}(z).
\end{equation}
Indeed, \eqref{eq:T1T2(z)}
clearly holds for~$n=1$.
If~$n=2k\ge 2$ then by the induction hypothesis
\begin{align*}
\brd{\check T_1\check T_2(z)}{n}&=q^{2k-2}
\begin{pmatrix}
(1-q^2)(r_{k-1}(z)+r_{k-2}(z))&q z^*{}^{k-1}\\
q z^{k-1}&0
\end{pmatrix}\check T_2(z)\\
&=q^{2k-1}
\begin{pmatrix}
q z^*{}^k &(1-q^2)(z (r_{k-1}(z)+r_{k-2}(z))+z^*{}^{k-1})\\0&q z^k
\end{pmatrix}
\\
&=q^{2k-1}
\begin{pmatrix}
q z^*{}^k &(1-q^2)(1+z)r_{k-1}(z)\\0&q z^k
\end{pmatrix},
\end{align*}
where we used~\eqref{eq:rec r_k}. If~$n=2k+1$, $k>0$ then by the induction hypothesis
and~\eqref{eq:rec r_k}
\begin{align*}
\brd{\check T_1\check T_2(z)}n&=q^{2k-1}
\begin{pmatrix}
q z^*{}^{k}&(1-q^2)(1+z)r_k(z)\\
0 &qz^k
\end{pmatrix}\check T_1\\
&=q^{2k}\begin{pmatrix}
(1-q^2)(z^*{}^k+(1+z)r_{k-1}(z)&q z^*{}^k\\
q z^k&0
\end{pmatrix}\\
&=q^{2k}\begin{pmatrix}(1-q^2)(r_k(z)+r_{k-1}(z))&q z^*{}^k\\
q z^k&0
\end{pmatrix}.
\end{align*}
The identity~\eqref{eq:T2(z)T1}
is proved similarly.

Suppose now that~$z=\zeta_l$. Then~$z^*=z^{-1}$ and
$z^*{}^{\lceil \frac l2\rceil}=z^{\lfloor \frac l2\rfloor}$. Also, if~$l=2s$, $s\ge 2$ then~$r_{s-1}(z)=0$ by~\eqref{eq: r_k defn},
while for~$l=2s+1$, $s\ge 1$,
$$
r_{s}(z)+r_{s-1}(z)=\frac{z^{s+1}-z^{-s-1}+z^s-z^{-s}}{z-z^{-1}}=\frac{z^{-s}(z^{l}-1)+z^s(1-z^{-l})}{z-z^{-1}}=0.
$$
The identity~$\brd{\check T_1\check T_2(z)}{l}=
\brd{\check T_2(z)\check T_1}{l}$ is now immediate
from~\eqref{eq:T1T2(z)} and~\eqref{eq:T2(z)T1}.
The last assertion follows since $\check T_1$ is symmetric and real, while the complex conjugate of the transpose of~$\check T_2(z)$
equals~$\check T_2(z)$
for any~$z\in\mathbb C$.
\end{proof}
Denote $p_d=(1-(-q^2)^d)/(1+q^2)$, $d\ge 0$. Then $p_0=0$, $p_1=1$ and
\begin{equation}\label{eq:p_d rec}
p_{d+1}=(1-q^2)p_d+q^2 p_{d-1},\quad d\ge 1.
\end{equation}
We set~$p_{-1}=q^{-2}$ so that the above recursion holds for all~$d\ge 0$.
Note that all the~$p_d$, $d\ge 1$ are polynomials in~$q$ with
the leading term~$(-q^2)^{d-1}$. An easy
induction on~$d$ yields, using~\eqref{eq:p_d rec},
$$
\check T_1^d=\begin{pmatrix}
                     p_{d+1}&q p_d\\
                     q p_d &q^2 p_{d-1},
\end{pmatrix}
$$
whence
\begin{align*}
\check T_1^d \check T_2(\zeta_l)&=\begin{pmatrix}
                          q^2 \zeta_l^{-1}p_d& q((1-q^2)p_d+\zeta_l p_{d+1})\\ q^3 \zeta_l^{-1} p_{d-1}&q^2((1-q^2)p_{d-1}+\zeta_l p_d)
                         \end{pmatrix}\\
                         &=\begin{pmatrix}
                            q^2 \zeta_l^{-1}p_d&q ((1+\zeta_l)p_{d+1}-q^2 p_{d-1})\\ q^3 \zeta_l^{-1}p_{d-1}&q^2((1+\zeta_l)p_d-q^2 p_{d-2})
                           \end{pmatrix}.
\end{align*}
Since $\det \check T_1=\det\check T_2(\zeta_l)=-q^2$,
the characteristic polynomial of~$\check T_1^d \check T_2(\zeta_l)$ is
$$
t^2-q^2\tau_d t+(-q^2)^{d+1},\qquad \tau_d:=(\zeta_l+1+\zeta_l^{-1})p_{d-1}-q^2 p_{d-2},
$$
and so the eigenvalues of~$\check T_1^d \check T_2(\zeta_l)$ are
\begin{equation}\label{eq:lam T_1^d T_2}
\lambda_\pm(q)=\tfrac12 q^2\Big(\tau_d\pm\sqrt{\tau_d^2-4(-q^2)^{d-1}}\Big).
\end{equation}
Note that~$\zeta_l+\zeta_l^{-1}\in\mathbb R$ and so~$\tau_d\in\mathbb R$ for all~$q\in\mathbb R\setminus\{0\}$.
\begin{lemma}\label{lem:generic eigenvalues}
Let~$l\ge 3$ and~$d\ge 2(1+\delta_{l,3})$. Then for generic~$q\in\mathbb R\setminus\{0,\pm1\}$, $\lambda_+(q)^n\not=\lambda_-(q)^n$ for all~$n\ge 1$.
\end{lemma}
\begin{proof}
Abbreviate $\zeta=\zeta_l$.
Suppose first that~$l=3$ and so $\zeta+1+\zeta^{-1}=0$. Then $\tau_d=-q^2 p_{d-2}$ and, since~$d>3$, we can write
$$
\lambda_\pm(q)=-\tfrac12 q^4\Big(p_{d-2}\mp\sqrt{p_{d-2}^2-4(-q^2)^{d-3}}\Big).
$$
Since~$p_{d-2}^2-4(-q^2)^{d-3}$ is a polynomial in~$q$ with the leading term $q^{4(d-3)}$, $\lambda_+(q)\not=\lambda_-(q)$
for a generic~$q$.

Suppose that $\lambda_+(q)\not=\lambda_-(q)$ and that
$\lambda_+(q)^n=\lambda_-(q)^n$ for some~$n>1$. Then
\begin{equation}\label{eq:generic eigenvalues 1}
\sum_{k\ge 0} \binom{n}{2k+1} p_{d-2}^{n-2k-1} (p_{d-2}^2-4(-q^2)^{d-3})^k=0.
\end{equation}
Each summand in the left hand side of~\eqref{eq:generic eigenvalues 1} is a
polynomial in~$q$ with the leading term
$\binom{n}{2k+1}(-q^2)^{(d-3)(n-1)}$. Therefore, the sum in~\eqref{eq:generic eigenvalues 1}
is a polynomial in~$q$ with the leading term $(2(-q^2)^{d-3})^{n-1}$. Thus, $\lambda_+(q)^n\not=\lambda_-(q)^n$, $n\ge 1$ for all but countably many~$q\in\mathbb R\setminus\{0,\pm1\}$.

Suppose that~$l>3$ and so~$\zeta+1+\zeta^{-1}\not=0$. If~$d=2$, $\tau_d=\zeta+1+\zeta^{-1}$, $\tau_d^2-4(-q^2)^{d-1}=(\zeta+1+\zeta^{-1})^2+4q^2>0$ for all real~$q$.
If~$d=3$, $\tau_d=(\zeta+1+\zeta^{-1})1-q^2(\zeta+2+\zeta^{-1})$ and so
$$
\tau^2_d-4(-q^2)^{d-1}=(z+4)z q^4-2 (z+1)(z+2)q^2+(z+1)^2,\qquad z=\zeta+\zeta^{-1},
$$
which is clearly a non-zero polynomial in~$q$. If~$d>3$ then the degree of~$\tau_d^2$ as a polynomial in~$q$ is $4(d-2)>2(d-1)$ hence the degree of~$\tau_d^2-4(-q^2)^{d-1}$ is $4(d-2)$.
Thus, $\lambda_+(q)\not=\lambda_-(q)$ for generic~$q\in\mathbb R\setminus\{0,\pm1\}$.

Now, suppose that~$\lambda_+(q)\not=\lambda_-(q)$.
Then~$\lambda_+(q)^n=\lambda_-(q)^n$ for some~$n\ge 2$ if and only if
\begin{equation}\label{eq:pwrs eig}
\sum_{k\ge 0}\binom{n}{2k+1} \tau_d^{n-2k-1}(\tau_d^2-4(-q^2)^{d-1})^k=0.
\end{equation}
If~$d=2$ then the left hand side is equal to
$$
\sum_{k\ge 0}\binom{n}{2k+1} (\zeta+1+\zeta^{-1})^{n-2k-1}(\zeta+1+\zeta^{-1}+4q^2)^k
$$
which is equal to $2(\zeta+1+\zeta^{-1})$ if~$n=2$ and is a polynomial in~$q$ otherwise, with the leading term~$2^{n-1}q^{n-1}$ if~$n$ is odd and~$n(\zeta+1+\zeta^{-1})2^{n-2}q^{n-2}$
if~$n>2$ is even.

If~$d=3$ then the left hand side of~\eqref{eq:pwrs eig} becomes
\begin{equation}\label{eq:lhs pwrs eig}
\sum_{k\ge 0}\binom{n}{2k+1} (z+1)^{n-2k-1}((z+4)z q^4-2 (z+1)(z+2)q^2+(z+1)^2)^k,\qquad z=\zeta+\zeta^{-1}.
\end{equation}
For~$n=2$, this equals~$2(z+1)\not=0$. Suppose that~$n>2$. If~$l=4$ and so~$z=0$,
\eqref{eq:lhs pwrs eig} is a polynomial in~$q$ with the leading term $(2(z+1)(z+2))^{\frac12(n-1)}q^{n-1}$ if~$n$ is odd
and $-n (z+1)^{\frac12 n}(2(z+2))^{\frac12 n-1}q^{n-2}$ if~$n$ is even. If~$l>4$ then
\eqref{eq:lhs pwrs eig} is a polynomial in~$q$ with the leading term $(z(z+4))^{\frac12(n-1)}q^{2(n-1)}$ if~$n$ is odd
and~$n(z+1)(z(z+4))^{\frac12n-1} q^{2(n-2)}$ if~$n$ is even. In either case, it is a non-zero polynomial in~$q$.

Finally, if~$d>3$, each summand in the left hand side of~\eqref{eq:pwrs eig} is a polynomial in~$q$ with the leading term $\binom{n}{2k+1}(-q^2)^{2(d-2)(n-1)}$.
Thus, the leading term of the sum is $(2(-q^2)^{d-2})^{n-1}$ and so the sum is a non-zero polynomial in~$q$.

Thus, for~$l>3$ and~$d\ge 2$, the left hand side of~\eqref{eq:pwrs eig} is a non-zero polynomial in~$q$, whence 
$\lambda_+(q)^n\not=\lambda_-(q)^n$ for all~$n\ge 1$ for all but countably many~$q\in\mathbb R\setminus\{0,\pm1\}$. 
\end{proof}

Suppose that~$\wh m$ is even. Then~$\wh T_2$
commutes with $(\wh T_1\wh T_2)^{\frac12\wh m}$ and so
$T_2=\mathbf T_{(d,1)}(\wh T_2)$ must commute with~$(T_1^d T_2)^{\frac12\wh m}=\mathbf T_{(d,1)}(
(\wh T_1\wh T_2)^{\frac12\wh m})$ in~$\Br^+(I_2(m))$.
Therefore, for any~$q\in\mathbb R\setminus\{0\}$,
$\check T_2(\zeta_m)$ must commute with
$(\check T_1^d T_2(\zeta_m))^{\frac12\wh m}$

 By Lemma~\ref{lem:generic eigenvalues}, we can choose~$q\in\mathbb R\setminus\{0,\pm1\}$ so that~$(\check T_1^d \check T_2(\zeta_m))^n$ has two distinct eigenvalues~$\lambda_\pm(q)^n$ for all~$n\ge 1$. Then
 $\ker((\check T_1^d\check T_2(\zeta_m))^n-
 \lambda_\pm(q)^n\id_{\mathbb C^2})=\ker(\check T_1^d\check T_2(\zeta_m)-
 \lambda_\pm(q)\id_{\mathbb C^2})$ for all~$n\ge 1$.
 Since both~$\check T_2(\zeta_m)$ and~$(\check T_1^d\check T_2(\zeta_m))^{\frac12\wh m}$ are diagonalizable and
 commute, they are simultaneously diagonalizable, which by the above also implies that $\check T_2(\zeta_m)$
 and $\check T_1^d \check T_2(\zeta_m)$ are
 simultaneously diagonalizable and, therefore, must commute. Yet by~\eqref{eq:p_d rec}
\begin{align*}
 \check T_2(\zeta_m)&(\check T_1^d\check T_2(\zeta_m))-(\check T_1^d\check T_2(\zeta_m))\check T_2(\zeta_m)
 =
 q^2 p_d\begin{pmatrix}
  \zeta_m-\zeta_m^{-1} & (q-q^{-1}) (1+\zeta_m) \\
 (q^{-1}-q) (1+\zeta_m^{-1}) & \zeta_m^{-1}-\zeta_m
 \end{pmatrix},
 \end{align*}
which is manifestly non-zero for~$q\in\mathbb R\setminus\{0,\pm1\}$,
as $\zeta_m$ is an $m$th primitive root
of unity with~$m\ge 3$.

Finally, suppose that~$\wh m$ is odd.
Then the composition of~$\mathbf T_{(d,1)}$ with the tautological 
homomorphism~$\Br^+(I_2(2\wh m))\to\Br^+(I_2(\wh m))$
yields~$\mathbf T_{(d,1)}\in\Hom_{\mathscr A}(I_2(2\wh m),I_2(m))$
which, as we just proved, does not exist.

Thus, $\wh m=\infty$.
\end{proof}
This completes the proof of Theorem~\ref{thm:Tits homomorphisms}.
\end{proof}
\begin{lemma}\label{lem:lifts of parab projections}
Let~$M=(m_{ij})_{i,j\in I}$ be a Coxeter matrix, $J\subset I$.
Assume that~$p_J$ is liftable, that each connected component of~$\Gamma_J(M)$ has at least two vertices and that~$m_{ij}<\infty$ for all~$i,j\in J$. Then~$P_J$
is the only~$\Phi\in\Hom_{\mathscr A}(M,M_J)$
such that~$\overline\Phi_\star=p_J$.
\end{lemma}
\begin{proof}
We may assume, without loss of generality, that~$J$ is connected and~$|J|>1$. Let~$\Phi$ be a homomorphism $\Br^+(M)\to\Br^+_J(M)$ and with~$\overline\Phi_\star=p_J$.
Then~$[\Phi](i)=\emptyset$ for all~$i\in I\setminus J$ and~$[\Phi](j)=\{j\}$ for all~$j\in J$. Thus,
$\Phi(T_i)=1$ if~$i\in I\setminus J$ and~$\Phi(T_i)=T_i^{d_i}$, $d_i\in\ZZ_{>0}$ if~$i\in J$. 

Let~$i\not=j\in J$ and suppose that~$\max(d_i,d_j)>1$.
Note that if~$m_{ij}$ is odd then~$d_i=d_j$ by Lemma~\partref{lem:elem Artin hom.c}
and so~$\min(d_i,d_j)>1$. Suppose that~$m_{ij}>2$. 
If~$\min(d_i,d_j)>1$ then~$B(T_i^{d_i},T_j^{d_j})=\emptyset$
by~\cite{CP} which is a contradiction since~$m_{ij}<\infty$. In particular,
$m_{ij}\not=3$.
Similarly,
if~$m_{ij}>3$ and~$\min(d_i,d_j)=1$ then
$B(T_i^{d_i},T_j^{d_j})=\emptyset$
by Proposition~\ref{prop:key converse Tits} which again contradicts the 
assumption that~$m_{ij}<\infty$.
Thus, if~$m_{ij}>2$ then~$d_i=d_j=1$.

Finally, if~$m_{ij}=2$ and, say, $d_i>1$ then, since~$J$ is connected,
there exists~$i'\in J$ with~$m_{ii'}\ge 3$ which leads to 
a contradiction by the above. Thus, $d_j=1$ for all~$j\in J$, that is, $\Phi=P_J$.
\end{proof}
Given~$\mathbf d\in \ZZ_{>0}^I$, let~$I_{>1}(\mathbf d)=\{i\in I\,:\, d_i>1\}$.
In reference to Tits conjecture, we call
an optimal $\mathbf T_{\mathbf d}\in\Hom_{\mathscr A}(M(\mathbf d),M)$, $\mathbf d=(d_i)_{i\in I}\in\ZZ_{>0}^I$ satisfying $\mathbf T_{\mathbf d}(\wh T_i)=T_i^{d_i}$, $i\in I$
a {\em Tits homomorphism}.
We call a Tits homomorphism~$\mathbf T_{\mathbf d}$ {\em elementary} if~$I_{>1}(\mathbf d)=\{i\}$
for some~$i\in I$
and denote such a homomorphism by~\plink{T i,d}$\mathbf T_{i,d}$, $i\in I$, $d\in\mathbb Z_{>1}$.

\begin{proposition}\label{prop:elementary Tits}
Let~$M\in\Cox I$, $\mathbf d=(d_i)_{i\in I}\in\ZZ_{>0}^I$
and let $\mathbf T_{\mathbf d}\in\Hom_{\mathscr A}(M(\mathbf d),M)$ be a Tits homomorphism.
Then for any total order $i_1<\cdots<i_n$ on~$I_{>1}(\mathbf d)$ there is a unique sequence of Coxeter matrices $M_0=M, M_1,\dots, M_n, M_{n+1}=
M(\mathbf d)$ over~$I$ such that
$\mathbf T_{\mathbf d}=\mathbf T_{i_1,d_{i_1}}\circ \cdots\circ \mathbf T_{i_n, d_{i_n}}$ where $\mathbf T_{i_t,d_{i_t}}\in\Hom_{\mathscr A}(M_{t+1},M_t)$.
Thus, every Tits homomorphism is a composition of elementary Tits homomorphisms.
\end{proposition}
\begin{proof}
The argument is by induction on~$|I_{>1}(\mathbf d)|$, the case~$|I_{>1}(\mathbf d)|=1$ being trivial. For the inductive step, we need the following 
\begin{lemma}\label{lem:Tits factorization}
Let~$i\in I_{>1}(\mathbf d)$, 
$\mathbf d_i=(d_j^{\delta_{i,j}})_{j\in I},
\mathbf d^{(i)}=(d_j^{1-\delta_{i,j}})_{j\in I}
\in\ZZ_{>0}^I$ and let~$M^{(i)}=M(\mathbf d_i)$.  
Then $M^{(i)}(\mathbf d^{(i)})=
M(\mathbf d)$.
\end{lemma}
\begin{proof}
Write $M^{(i)}=(m^{(i)}_{jk})_{j,k\in I}$. Then 
$m^{(i)}_{jk}=m_{jk}$, $j,k\not=i$ and 
$$
m^{(i)}_{ij}=m^{(i)}_{ji}=\begin{cases}
             m_{ij},& m_{ij}\le 2,\\
             2d_i,& m_{ij}=3,\, d_i\in\{2,3\},\,d_j=1,\\
             \infty,& \text{otherwise}.
            \end{cases}
$$
Note that, since~$m_{jk}=m^{(i)}_{jk}$,
$j,k\not=i$, we only need to show that
$m_{ij}(\mathbf d)=m^{(i)}_{ij}(\mathbf d^{(i)})$
for~$j\not=i\in I$. By~\eqref{eq:Tits cover}, we have 
$m^{(i)}_{ij}=m_{ij}(\mathbf d)=2d_i$ if~$d_j=1$, $d_i\in\{2,3\}$ and~$m_{ij}=3$. Since~$d^{(i)}_j=
d^{(i)}_i=1$, we have~$m^{(i)}_{ij}(\mathbf d^{(i)})=m^{(i)}_{ij}$ by~\eqref{eq:Tits cover} and so $m_{ij}(\mathbf d)=m^{(i)}_{ij}(\mathbf d^{(i)})$.
Likewise, $m^{(i)}_{ij}=m_{ij}(\mathbf d)=2$ if~$m_{ij}=2$ and so~$m^{(i)}_{ij}(\mathbf d^{(i)})=2=m_{ij}(\mathbf d)$. Finally, if~$d_j=d^{(i)}_j>1$ then~$m^{(i)}_{ij}=\infty$ and
so~$m^{(i)}_{ij}(\mathbf d^{(i)})=\infty$. But in that case, as 
$d_id_j>3$, 
$m_{ij}(\mathbf d)=\infty$ by~\eqref{eq:Tits cover}. 
\end{proof}
It follows from Lemma~\ref{lem:Tits factorization}
and Theorem~\ref{thm:Tits homomorphisms}
that~$\mathbf T_{i,d_i}\in\Hom_{\mathscr A}(M^{(i)},M)$,
$\mathbf T_{\mathbf d^{(i)}}\in\Hom_{\mathscr A}(M(\mathbf d),M^{(i)})$
are Tits homomorphisms. By 
construction, $\mathbf T_{\mathbf d}=
\mathbf T_{i,d_i}\circ \mathbf T_{\mathbf d^{(i)}}$.
Since $|I_{>1}(\mathbf d^{(i)})|=|I_{>1}(\mathbf d)|-1$,
the induction hypothesis applies to~$\mathbf T_{\mathbf d^{(i)}}$ and so it admits
the desired factorization. The uniqueness follows 
from Theorem~\ref{thm:Tits homomorphisms}.
\end{proof}

The following extends the famous {\em Tits conjecture} and its generalization proved in~\cite{CP}.
\begin{conjecture}\label{conj:gen Tits}
All Tits homomorphisms~$\mathbf T_{\mathbf d}$ 
are injective.
\end{conjecture}
By Proposition~\ref{prop:elementary Tits}, it suffices to prove the Conjecture for elementary Tits homomorphisms. The main result of~\cite{CP}
proves the above for all~$\mathbf d\in\mathbb Z_{>1}^I$.
We now provide some supporting evidence.
The following is immediate.
\begin{corollary}\label{cor:left factor Tits}
Let~$M\in\Cox I$, $\mathbf d\in\ZZ_{>0}^I$
Let
$\mathbf T_{\mathbf d}\in\Hom_{\mathscr A}(M(\mathbf d),M)$
be a Tits homomorphism
and let~$\mathbf T_{\mathbf d}=\mathbf T_{i_1,d_{i_1}}\circ\cdots \circ \mathbf T_{i_n,d_{i_n}}$ be any factorization as in Proposition~\ref{prop:elementary Tits}. If~$\mathbf T_{\mathbf d}$ is injective then so is
$\mathbf T_{i_r,d_{i_r}}\circ 
\cdots \circ \mathbf T_{i_n,d_{i_n}}$ 
for any $2\le r\le n$.
\end{corollary}
\begin{example}
Let~$M=A_3$, $\mathbf d=(d,d,d)$, $d\in\{2,3\}$. Then~$M(\mathbf d)=\left(\begin{smallmatrix}
    1&\infty&2\\
    \infty&1&\infty\\
    2&\infty&1
\end{smallmatrix}\right)$
and $\mathbf T_{\mathbf d}\in\Hom_{\mathscr A}(M(\mathbf d),M)$ is injective by~\cite{CP}. It
factorizes into 
the following chain of elementary Tits homomorphisms
$$
\begin{dynkinDiagram}[expand labels={1,2,3}]A3
\dynkinEdgeLabel{1}{2}{\infty}
\dynkinEdgeLabel{2}{3}{\infty}
\end{dynkinDiagram}
\xrightarrow{\mathbf T_{2,d}}
\begin{dynkinDiagram}[expand labels={1,2,3}]A3
\dynkinEdgeLabel{1}{2}{2d}
\dynkinEdgeLabel{2}{3}{2d}
\end{dynkinDiagram}
\xrightarrow{\mathbf T_{1,d}}
\begin{dynkinDiagram}[expand labels={1,2,3}]A3
\dynkinEdgeLabel{2}{3}{2d}
\end{dynkinDiagram}
\xrightarrow{\mathbf T_{3,d}}\\
\dynkin A3
$$
whence~$\mathbf T_{2,d}$ and~$\mathbf T_{1,d}\circ
\mathbf T_{2,d}=\mathbf T_{(d,d,1)}$ are injective. 
\end{example}

\begin{proposition}\label{prop:partial Tits}
Let~$m\in\{3,4,6\}$. Then~$\Phi_d\in\Hom_{\mathscr A}(I_2(\infty),
I_2(m))$ defined by $\Phi_d(\wh T_1)=T_1^d$, $
\Phi(\wh T_2)=T_2$ is injective
if and only if~$d\ge 2(1+\delta_{3,m})$, that is,
if and only if~$\Phi_d=\mathbf T_{1,d}=
\mathbf T_{(d,1)}$ is an (elementary) Tits homomorphism.
\end{proposition}
\begin{proof}
Suppose first that~$m=3$ and so~$I_2(m)=A_2$.
It is well-known (and easy to verify) that 
the assignments 
$$
T_1\mapsto \begin{pmatrix}1&u\\
0&1
\end{pmatrix},\qquad T_2\mapsto \begin{pmatrix}1&0\\
-u^{-1}&1
\end{pmatrix},\qquad u\in\mathbb C\setminus\{0\}
$$
define a homomorphism~$\rho_u:\Br(A_2)\to SL(2,\mathbb C)$.
Then~$\rho_u(T_1^d)=\left(\begin{smallmatrix}1&du\\0&1\end{smallmatrix}\right)$. By~\cite{CJR}*{Theorem~2} the 
subgroup of~$SL(2,\CC)$ generated
by $\left(\begin{smallmatrix}
1&\alpha\\
0&1
\end{smallmatrix}\right)$ and
$\left(\begin{smallmatrix}
1&0\\
\beta&1
\end{smallmatrix}\right)$, $\alpha,\beta\in\CC$
is free if and only if~$|\alpha\beta|,
|\alpha\beta\pm 2|\ge 2$. Since~$\alpha=du$ with $d\ge 4$ and~$\beta=-u^{-1}$
obviously satisfy these conditions, 
the subgroup of~$SL(2,\mathbb C)$ generated 
by $\rho_u(T_1^d)$, $d\ge 4$, and~$\rho_u(T_2)$ is free. Thus,
the subgroup of~$\Br(A_2)$ and hence the submonoid of~$\Br^+(A_2)$ generated by~$T_1^d$, $d\ge 4$
and~$T_2$ are free.

If~$m=4$ or~$m=6$ then the 
factorization of~$\mathbf T_{(d,\frac m2)}\in\Hom_{\mathscr A}(I_2(\infty),A_2)$
as~$\mathbf T_{(d,\frac m2)}=\mathbf T_{2,\frac m2}\circ \mathbf T_{1,d}$ yields the injectivity
of~$\mathbf T_{1,d}$ by Corollary~\ref{cor:left factor Tits}.

Conversely, if~$d=1$ then the homomorphism~$\Phi_d$ is tautological and hence not injective as 
$\brd{T_1T_2}m=\brd{T_2T_1}m$. 
If~$m=3$ and~$d\in\{2,3\}$ then the homomorphism
$\Br^+(I_2(\infty))\to \Br^+(I_2(m))$ is 
again not injective as $\brd{T_1^d T_2}{2d}=
\brd{T_2T_1^{d}}{2d}$. 
\end{proof}

The following Proposition generalizes Examples~\ref{ex:1.6} and~\ref{ex:affine}.
\begin{proposition}\label{prop:Tits standard}
Every Tits homomorphism is a composition of
standard homomorphisms.
\end{proposition}
\begin{proof}
By Proposition~\ref{prop:elementary Tits}, it suffices to prove the assertion for elementary Tits homomorphisms. We need the following 
\begin{lemma}\label{lem:elem Tits std}
Let~$d\in\ZZ_{>1}$, $M\in\Cox I$, $i\in I$
and let~$\mathbf d=(d^{\delta_{i,j}})_{j\in I}$. Let~$\wh I=(I\setminus\{i\})
\sqcup S$ where~$S$ is any set with~$|S|=d$. Define~$\varpi:\wh I\to I$ 
and~$\wh M=(\wh m_{jk})_{j,k\in\wh I}$ by
\begin{align*}
&\varpi(j)=\begin{cases}
j,&j\in I\setminus\{i\},\\
i,&j\in S,
\end{cases}\\
&\wh m_{jk}=\begin{cases}
2-\delta_{j,k},&\varpi(j)=\varpi(k),\\
m_{\varpi(j)\varpi(k)},&\text{otherwise}.
\end{cases}
\end{align*}
for all $j,k\in\wh I$. 
Then
\begin{enmalph}
    \item\label{lem:elem Tits std.a} $\wh M\in\Cox{\wh I}$ is foldable along~$\varpi$ and~$\wh M^{\varpi}=M$;
    \item\label{lem:elem Tits std.b} 
    the assignments $T_k\mapsto \prod_{j\in\varpi^{-1}(k)}\wh T_j$, $k\in I$, define a standard homomorphism
    $\Phi\in\Hom_{\mathscr A}(M(\mathbf d),\wh M)$;
    \item\label{lem:elem Tits std.c} $\mathbf T_{i,d}=\mathbf F_\varpi\circ \Phi\in\Hom_{\mathscr A}(M(\mathbf d),M)$.
\end{enmalph}
\end{lemma}
\begin{proof}
Part~\ref{lem:elem Tits std.a} is obvious. 
To prove~\ref{lem:elem Tits std.b}, note first that the product~$\prod_{s\in S}\wh T_s$ is well-defined since $\wh m_{st}=2$ for all~$s\not=t\in S$. Since~$m(\mathbf d)_{jk}=m_{jk}=\wh m_{jk}$ if $j,k\in I\setminus\{i\}$,
it suffices
to prove that the assignments
\begin{equation}\label{eq:restr}
T_i\mapsto \wh T_{w_\circ^S}=\prod_{s\in S}\wh T_s,\quad T_j\mapsto \wh T_j
\end{equation}
define 
a homomorphism $\Br^+_{\{i,j\}}(M(\mathbf d))\cong \Br^+(I_2(m))\to 
\Br^+(\wh M)$ where
$$
m=m(\mathbf d)_{ij}=\begin{cases}
2,&m_{ij}=2,\\
2d,&m_{ij}=3,\,d\in\{2,3\},\\
\infty,&\text{otherwise}.
\end{cases}
$$
If~$m_{ij}=2$ then $\wh m_{sj}=m_{ij}=2$ for all~$s\in S$
and so
$\prod_{s\in S}\wh T_s$ commutes with~$\wh T_j$. 
If~$m_{ij}=3$ and~$d\in\{2,3\}$ then the assignments~\eqref{eq:restr} define 
a homomorphism $\Br^+(I_2(2d))\to \Br^+(D_{d+1})$
(see the proof of Lemma~\ref{lem:finite elem Tits}). The remaining case is obvious. Finally, $\Phi$ is
standard as~$\prod_{s\in S}\wh T_s=\wh T_{w_\circ^S}$
since all the $\wh T_s$, $s\in S$ commute.

Finally, we have~$\mathbf F_\varpi\circ\Phi(T_k)=T_k$, $k\in I\setminus\{i\}$
and~$\mathbf F_\varpi\circ\Phi(T_i)=\mathbf F_\varpi(\prod_{s\in S}\wh T_s)=T_i^d$
which proves~\ref{lem:elem Tits std.c}.
\end{proof}
Since both~$\Phi$ and~$\mathbf F_\varpi$ are standard,
the assertion follows for elementary Tits homomorphisms.
\end{proof}
\begin{remark}
The elementary Tits homomorphisms listed in Example~\ref{ex:affine} exhaust fully supported elementary Tits homomorphisms in
$\Hom_{\mathscr A}(\wh M,M)$ where~$\wh M$ is of affine type and~$M$ is of finite or affine type.
\end{remark}

\subsection{Factorization of light homomorphisms}\label{subs:factor light}
An analogue of Proposition~\ref{prop:classify light} for Artin monoids turns out to be more involved. 

\begin{theorem}\label{thm:classify light Artin}
Every light homomorphism of Artin monoids is a composition of one or more of the following
\begin{itemize}
    \item[-] a tautological homomorphism;
    \item[-] a parabolic projection;
    \item[-] a natural inclusion;
    \item[-] a Tits homomorphism;
    \item[-] a diagram automorphism;
    \item[-] a folding along a surjective map of index sets of Coxeter matrices.
\end{itemize}
In particular, every light homomorphism of Artin monoids is a composition of standard homomorphisms.
\end{theorem}

\begin{proof}
Let~$\wh M=(\wh m_{ij})_{i,j\in \wh I}\in\Cox{\wh I}$, $M\in\Cox I$ and let~$\Phi\in\Hom_{\mathscr A}(\wh M,M)$ be light.
By Lemma~\ref{lem:factor homs} we may assume that~$\Phi$ is optimal. 

Let~$\wh I_s=\{i\in \wh I\,:\,|[\Phi](i)|=s\}$, $s\in\{0,1\}$. Note first that~$P_{\wh I_1}$ is a well-defined homomorphism. Indeed, since~$\Phi\in\Hom_{\mathscr A}(\wh M,M)$, by Lemma~\partref{lem:elem Artin hom.c} if~$i\in\wh I_0$ then~$j\in\wh I_0$ for all~$j\in\wh I$ with~$\wh m_{ij}$ odd. Then~$\Phi=\Phi|_{\wh I_1}\circ P_{\wh I_1}$,
so we may assume without loss of generality that~$\wh I_0=\emptyset$. Furthermore, if~$\Phi$ is not fully supported, we can write it as a composition of a
fully supported light homomorphism with a natural inclusion.

Thus, we are assuming that~$\Phi$ is optimal, fully supported and~$|[\Phi](i)|=1$ for all~$i\in \wh I$. In particular, $[\Phi]$ can be regarded as a surjective map $\wh I\to I$.
Furthermore, by Lemma~\ref{lem:diagonal} we may assume that both~$M$ and~$\wh M$ are irreducible.

Given~$i\in \wh I$, define~$d_i\in\ZZ_{>0}$ by
$\Phi(\wh T_i)=T_{[\Phi](i)}^{d_i}$. Given $j\in I$ and~$d\in\ZZ_{>0}$, let
$\wh I(\Phi,j,d)=\{ i\in[\Phi]^{-1}(j)\,:\,d_i=d\}$. Clearly, 
$\wh I(\Phi,j,d)=\emptyset$ for all
but finitely many~$d\in\ZZ_{>0}$.
Let~$N(\Phi)=|\{i\in\wh I\,:\,d_i>1\}|$.
\begin{lemma}\label{lem:all powers 1}
Suppose that~$N(\Phi)=0$. Then 
$\wh M$ is foldable along~$[\Phi]$,
$M=\wh M^{[\Phi]}$ and~$\Phi=\mathbf F_{[\Phi]}$.
\end{lemma}
\begin{proof}
Since~$N(\Phi)=0$, $d_i=1$ for all~$i\in\wh I$.  
By the optimality of~$\Phi$, $\wh m_{ij}=
m_{[\Phi](i)[\Phi](j)}$ for all $i,j\in\wh I$ such that~$[\Phi](i)\not=[\Phi](j)$. The assertion is now immediate. 
\end{proof}
\begin{lemma}\label{lem:class light Artin 2}
Let $k\in I$, $d\in\ZZ_{>1}$
and~$i\in \wh I(\Phi,k,d)$.  
Set~$\widetilde I=
(\wh I\setminus\wh I(\Phi,k,d))\cup\{i\}$ and  define~$\varpi=\varpi_{k,d}:\wh I\to \widetilde I$
and~$\widetilde M=(\widetilde m_{i'i''})_{i',i''\in\widetilde I}$
by 
\begin{align*}
\varpi(j)&=\begin{cases}
j,&j\in\wh I\setminus\wh I(\Phi,k,d),\\
i,&j\in\wh I(\Phi,k,d),
\end{cases}
\\
\widetilde m_{i'i''}&=\begin{cases}
3,&i\in\{i',i''\},\,
m_{[\Phi](i')[\Phi](i'')}=3,\,d_{i'}d_{i''}\in\{2,3\},\\
\wh m_{i'i''},&\text{otherwise}
\end{cases}
\end{align*}
for all~$j\in\wh I$ and  for all~$i',i''\in\widetilde I$.
Then 
\begin{enmalph}
\item\label{lem:class light Artin 2.i} $\wh M$ is foldable along~$\varpi$;
\item\label{lem:class light Artin 2.ii} $\widetilde M\in\Cox{\widetilde I}$ and
$\wh M^\varpi=\widetilde M(\widetilde{\mathbf d})$ where~$\widetilde{\mathbf d}=(\widetilde d_j)_{j\in\widetilde I}$
with $\widetilde d_j=d^{\delta_{i,j}}$, $j\in\widetilde I$;
\item\label{lem:class light Artin 2.iii} The assignments $\widetilde T_j\mapsto T_{[\Phi](j)}^{d_j^{1-\delta_{i,j}}}$, $j\in\widetilde I$ define
$\widetilde\Phi\in\Hom_{\mathscr A}(\widetilde M,M)$ with~$N(\widetilde\Phi)<N(\Phi)$;
\item\label{lem:class light Artin 2.iv} $\Phi=\widetilde \Phi\circ\mathbf T_{i,d}\circ\mathbf F_\varpi$.
\end{enmalph}
\end{lemma}
\begin{proof}
Let~$j\in \wh I\setminus \wh I(\Phi,k,d)$ and let~$l=[\Phi](j)$.
We claim that for all~$i'\in \wh I(\Phi,k,d)$
\begin{equation}
\label{eq:class step 0c}
\wh m_{i'j}=\begin{cases}
2,&m_{kl}\le 2,\\
2d,&m_{kl}=3,\,d_j=1,\,d\in\{2,3\},\\
\infty,&\text{otherwise}.
\end{cases}
\end{equation}
Indeed, if~$l=k$ then, since~$\Phi(\wh T_{j})=T_k^{d_{j}}\not=T_k^{d}=\Phi(\wh T_{i'})$, $\wh m_{i'j}=2$ by the optimality of~$\Phi$. Suppose that~$l\not=k$. 
Then the restriction of~$\Phi$ to~$\Br^+_{\{i',j\}}(\wh M)\cong \Br^+(I_2(\wh m_{i'j}))$ is a homomorphism
to~$\Br^+_{\{k,l\}}(M)\cong 
\Br^+(I_2(m_{kl}))$. 
If~$m_{kl}=\infty$
then clearly
$\wh m_{i'j}=\infty$. 
If~$m_{kl}=2$ then, in particular, $\Phi(\wh T_{i
})=T_k^d\not=T_l^{d_j}=\Phi(\wh T_j)$ and so~$\wh m_{i'j}=2$ by the optimality of~$\Phi$.
Suppose that~$2<m_{kl}<\infty$.
If~$d_ {j}>1$ (respectively, if $d_{j}=1$
and either $d>3$ or~$m_{kl}>3$) then
$\wh m_{i'j}=\infty$ by~\cite{CP} (respectively, Proposition~\ref{prop:key converse Tits}). 
Finally, if $d_{j}=1$, $d\in\{2,3\}$ and~$m_{kl}=3$
then $\wh m_{i'j}=2d$ by the optimality of~$\Phi$ and Proposition~\ref{prop:key converse Tits}. 

Part~\ref{lem:class light Artin 2.i} is immediate from~\eqref{eq:class step 0c}.

The first assertion in~\ref{lem:class light Artin 2.ii} is obvious. Let~$j,j'\in\widetilde I$.
If~$i\notin\{j,j'\}$ then $\widetilde d_j=\widetilde d_{j'}=1$ and so
$\widetilde M(\widetilde{\mathbf d})_{jj'}=\widetilde m_{jj'}=\wh m_{jj'}=(\wh M^\varpi)_{jj'}$ by~\eqref{eq:Tits cover}. Suppose that, say, $j'=i\not=j$. Then~$\widetilde d_j=1$
and so by~\eqref{eq:Tits cover}
$$
\widetilde M(\widetilde{\mathbf d})_{ij}=\begin{cases}
\widetilde m_{ij},&\widetilde m_{ij}=2,\\
2d,&\widetilde m_{ij}=3,\,d\in\{2,3\},\\
\infty,&\text{otherwise}.
\end{cases}
$$
By definition of~$\widetilde M$, $\widetilde m_{ij}=2$
implies that~$\wh m_{ij}=2=(\wh M^\varpi)_{ij}$. If~$\widetilde m_{ij}=3$ then, since~$\wh m_{ij}\not=3$ by~\eqref{eq:class step 0c}, we must have
$m_{[\Phi](i)[\Phi](j)}=3$, $d_j=1$, $d\in\{2,3\}$ and then~$\wh m_{ij}=2d$ by~\eqref{eq:class step 0c}.
Finally, suppose that~$\widetilde m_{ij}>3$.
Then~$\widetilde m_{ij}=\wh m_{ij}$ by definition of~$\widetilde M$ which forces both of them to be equal to~$\infty$ by~\eqref{eq:class step 0c}. Thus, $\wh M^\varpi=\widetilde M(\widetilde{\mathbf d})$.

To prove~\ref{lem:class light Artin 2.iii}, we only need to show that 
for any~$j\in\widetilde I\setminus\{i\}$, the
assignments 
\begin{equation}\widetilde T_i\mapsto T_k,\quad \widetilde T_j\mapsto T_{l}^{d_j},\quad  l=[\Phi](j),\label{eq:Phi tild}
\end{equation}
define a homomorphism
$\Br^+_{\{i,j\}}(\widetilde M)\cong \Br^+(I_2(\widetilde m_{ij}))\to \Br^+_{\{k,l\}}(M)$.
If~$l=k$ then $\wh m_{ij}=2$ by~\eqref{eq:class step 0c},
$\widetilde m_{ij}=\wh m_{ij}=2$ and, obviously, $T_j T_j^{d_k}=
T_j^{d_k}T_j$. Suppose that~$l\not=k$
and so~$\Br^+_{\{k,l\}}(M)\cong\Br^+(I_2(m_{kl}))$.
If~$\widetilde m_{ij}=3$ then, as in the proof of~\ref{lem:class light Artin 2.ii} above, $m_{kl}=3$, $d_j=1$
and~\eqref{eq:Phi tild} defines an isomorphism of respective parabolic submonoids.
Otherwise, $\widetilde m_{ij}=\wh m_{ij}\in\{2,\infty\}$. In either case, \eqref{eq:Phi tild} defines a homomorphism of respective parabolic submonoids. By construction, $N(\Phi')
\le N(\Phi)-1<N(\Phi)$.

To prove~\ref{lem:class light Artin 2.iv}, note that for~$j\in \wh I\setminus\wh I(\Phi,k,d)$ we get $$(\widetilde\Phi\circ\mathbf T_{i,d}\circ\mathbf F_\varpi)(\wh T_j)=(\widetilde\Phi\circ\mathbf T_{i,d})(\wh T_j)=\widetilde\Phi(\widetilde T_j)=T_{[\Phi](j)}^{d_j}=\Phi(\wh T_j),$$ 
while 
for~$j\in\wh I(\Phi,k,d)$,
$$
(\widetilde\Phi\circ\mathbf T_{i,d}\circ\mathbf F_\varpi)(\wh T_j)
=(\widetilde\Phi\circ\mathbf T_{i,d})(\wh T_i)=
\widetilde \Phi(\widetilde T_i^{d})
=\widetilde \Phi(\widetilde T_i)^{d_i}
=T_{k}^{d_i}=\Phi(\wh T_j).
$$
This completes the proof of Lemma~\ref{lem:class light Artin 2}.
\end{proof}

Note that~$\widetilde\Phi$ obtained from~$\Phi$
by Lemma~\ref{lem:class light Artin 2} does not have to be optimal. However, by Lemma~\ref{lem:factor homs}, we can always factor out a tautological homomorphism. Thus, by applying Lemma~\ref{lem:class light Artin 2} repeatedly, we conclude that~$\Phi=\Phi'\circ\Phi''$ where~$\Phi'$ is optimal with $N(\Phi')=0$ and~$\Phi''$ is a composition of elementary Tits homomorphisms, foldings along surjective maps of index sets and tautological homomorphisms. But then~$\Phi'$ is a folding by Lemma~\ref{lem:all powers 1}.

All light homomorphisms listed in Theorem~\ref{thm:classify light Artin}, with the exceptions of Tits homomorphisms, are standard, while 
Tits ones are compositions of standard homomorphisms by Proposition~\ref{prop:Tits standard}.
\end{proof}
\begin{example}
Let $d\in\{2,3\}$, $m\in\ZZ_{>2}\cup\{\infty\}$ and 
let~$\wh M=\left(\begin{smallmatrix}
    1&\infty&2d&2d\\
    \infty&1&2&2\\
    2d&2&1&m\\
    2d&2&m&1
\end{smallmatrix}\right)$. 
It is easy to see that the assignments $\wh T_1\mapsto T_1$, $\wh T_2\mapsto T_2^{d'}$, $\wh T_3\mapsto T_2^{d}$,
$\wh T_4\mapsto T_2^{d}$, 
$d'\in\ZZ_{>3}$, define an optimal homomorphism $\Br^+(\wh M)\to \Br^+(A_2)$, which factorizes as
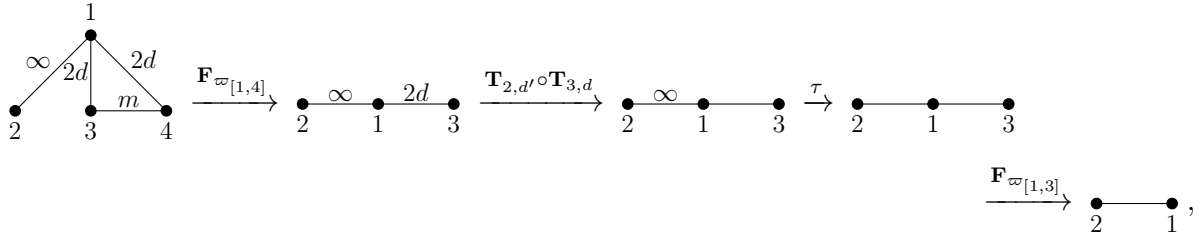
\begin{multline*}
\begin{dynkinDiagram}[name=upper,label directions={above}]A1
\node (current) at ($(upper root 1)+(-1cm,-1cm)$) {};
\dynkin[at=(current),name=left,labels={2},label directions={below}]A1
\node (current) at ($(left root 1)+(1cm,0cm)$) {};
\dynkin[at=(current),name=right,labels={3,4},label directions={below,below}]A2
\dynkinEdgeLabel{1}{2}{m}
\begin{pgfonlayer}{Dynkin behind}
\draw[/Dynkin diagram,edge]
($(upper root 1)$) -- ($(left root 1)$);
\draw[/Dynkin diagram,edge]
($(upper root 1)$) -- ($(right root 1)$);
\draw[/Dynkin diagram,edge]
($(upper root 1)$) -- ($(right root 2)$);
\end{pgfonlayer}
\draw[draw=none] (left root 1) to
                node[auto,%
                inner sep=0.8,%
                /Dynkin diagram/text style,%
                /Dynkin diagram/edge label]
                {\(\pgfkeys{/Dynkin diagram/label macro*=\infty}\)}%
                (upper root 1);%
\draw[draw=none] (right root 1) to
                node[auto,%
                inner sep=1,%
                /Dynkin diagram/text style,%
                /Dynkin diagram/edge label]
                {\(\pgfkeys{/Dynkin diagram/label macro*={2d}}\)}%
                (upper root 1);%
\draw[draw=none] (upper root 1) to
                node[auto,%
                inner sep=0.8,%
                /Dynkin diagram/text style,%
                /Dynkin diagram/edge label]
                {\(\pgfkeys{/Dynkin diagram/label macro*={2d}}\)}%
                (right root 2);%
\end{dynkinDiagram}
\xrightarrow{\mathbf F_{\varpi_{[1,4]}}}
\begin{dynkinDiagram}[labels={2,1,3}]A3
\dynkinEdgeLabel{1}{2}{\infty}
\dynkinEdgeLabel{2}{3}{2d}
\end{dynkinDiagram}
\xrightarrow{
\mathbf T_{2,d'}\circ\mathbf T_{3,d}}
\begin{dynkinDiagram}[labels={2,1,3}]A3
\dynkinEdgeLabel{1}{2}{\infty}
\end{dynkinDiagram}
\xrightarrow{\tau}
\dynkin[labels={2,1,3}]A3
\\
\xrightarrow{\mathbf F_{\varpi_{[1,3]}}}
\dynkin[labels={2,1}]A2,
\end{multline*}
where~$\varpi_{[1,k]}:[1,k]\to [1,k-1]$, $k>1$ is defined by $\varpi(i)=i-\delta_{i,k}$,
$i\in[1,k]$ and~$\tau$ is the tautological homomorphism.
\end{example}
\begin{example}
Let~$M=(m_{ij})_{i,j\in I}\in\Cox I$ and let~$\mathbf d=(d_i)_{i\in I}\in\ZZ_{\ge 0}^I$ with~$d_i=d_j$ whenever~$m_{ij}$
is odd. Then the character homomorphism~$\Xi_\mathbf X:\Br^+(M)\to\Br^+(A_1)$ with~$\mathbf X=(T_1^{d_i})_{i\in I}$ factorizes as follows. Let~$\sim_M$ be the transitive closure of the relation on~$I$ defined by~$i\sim_M j$ whenever~$m_{ij}$
is odd, $i,j\in I$; this is manifestly an equivalence relation. Let~$\underline I$ be the set of equivalence 
classes for~$\sim_M$ and denote the class of~$i\in I$
by~$\underline i$. Let~$\tilde M=(\tilde m_{ij})_{i,j\in I}$
with~$\tilde m_{ij}=m_{ij}$ if~$m_{ij}$ is odd and~$\tilde m_{ij}=2$
if~$m_{ij}$ is even, $i,j\in I$. Note that~$\tilde M$
is foldable along the canonical map~$\varpi:I\to\underline I$, $i\mapsto \underline i$, $i\in I$
with~$\tilde M^\varpi$ being the 
product of copies of~$A_1$ indexed by~$\underline I$
(cf. Lemma~\partref{lem:free product.a}). Then
$\Xi_{\mathbf X}=\mathbf F_{\underline\varpi}\circ 
\mathbf T_{\underline{\mathbf d}}\circ \mathbf F_{\varpi}
\circ\tau$, where~$\tau:\Br^+(M)\to\Br^+(\tilde M)$
is the tautological homomorphism, $\underline{\mathbf d}=
(d_{\underline i})_{\underline i\in\underline I}$
with~$d_{\underline i}=d_i$, $i\in I$, and~$\underline\varpi$
is the unique map~$\underline I\to \{1\}$.
\end{example}
\begin{proposition}
\label{prop:types of light}
The following exhausts
optimal fully supported light homomorphisms between Artin monoids of irreducible finite types
\begin{enmalph}

\item\label{prop:types of light.PBnAn-1} $P_{[1,n-1]}:\Br^+(B_n)\to \Br^+(A_{n-1})$;
\item\label{prop:types of light.PF4A2}  $P_{[1,2]}:\Br^+(F_4)\to \Br^+(A_2)$;
\item\label{prop:types of light.PBnA1}
$P_{\{n\}}:\Br^+(B_n)\to\Br^+(A_1)$;
\item\label{prop:types of light.PI2A1}
$P_{\{1\}}:\Br^+(I_2(2m))\to \Br^+(A_1)$;
\item\label{prop:types of light.FDn+1An} $\mathbf F_{\varpi_{(n,n+1)}}:
\Br^+(D_{n+1})\to \Br^+(A_n)$, $n\ge 2$, where 
$\varpi_{(n,n+1)}:[1,n+1]\to[1,n]$ is defined by~$\varpi_{(n,n+1)}(i)=i-\delta_{i,n+1}$, $i\in[1,n+1]$;
\item\label{prop:types of light.FD4A2} $\mathbf F_{\varpi_{(1,3,4)}}:
\Br^+(D_4)\to \Br^+(A_2)$ where 
$\varpi_{(1,3,4)}:[1,4]\to\{1,2\}$
is defined by $\varpi_{(1,3,4)}(i)=1$, $i\in\{1,3,4\}$, $\varpi_{(1,3,4)}(2)=2$;
\item\label{prop:types of light.A1A1}
$\mathbf T_{1,d}:\Br^+(A_1)\to\Br^+(A_1)$, $T_1\mapsto T_1^d$, $d\in\ZZ_{>1}$;
\item\label{prop:types of light.BnAn} $\mathbf T_{n,2}:\Br^+(B_n)\to\Br^+(A_n)$, $\wh T_i\mapsto T_i^{1+\delta_{i,n}}$, $i\in[1,n]$;
\item\label{prop:types of light.G2A2} $\mathbf T_{2,3}:\Br^+(G_2)\to \Br^+(A_2)$, $\wh T_1\mapsto T_1$, $T_2\mapsto T_2^3$;
\item\label{prop:types of light.len} The length
homomorphism~$\ell:\Br^+(M)\to (\ZZ_{\ge0},+)\cong
\Br^+(A_1)$.
\end{enmalph}

The only tautological homomorphisms
of irreducible Artin monoids of finite type are those from $\Br^+(I_2(dm))$ to~$\Br^+(I_2(m))$, 
$d>0$, $m\ge 3$.

All other light homomorphisms of 
Artin monoids of finite type are obtained as compositions of the above ones, up to natural inclusions, parabolic projections onto connected components, direct products
and diagram automorphisms.
\end{proposition}
\begin{proof}
It suffices to verify that the above list exhausts all ``elementary'' light homomorphisms listed in Theorem~\ref{thm:classify light Artin}. 
 By
Propositions~\ref{prop:parab proj Artin}, \ref{prop:types of light.PBnAn-1}--\ref{prop:types of light.PI2A1} are the only parabolic projections existing in finite irreducible types. 
By Proposition~\ref{prop:key converse Tits}, \ref{prop:types of light.A1A1}--\ref{prop:types of light.G2A2} exhaust
all elementary Tits homomorphisms in finite types.
By inspection, the only existing foldings with irreducible codomain are those in~\ref{prop:types of light.FDn+1An}, \ref{prop:types of light.FD4A2}
and \ref{prop:types of light.len}.
\end{proof}

\subsection{Light homomorphisms of Artin monoids are parabolic}\label{subs:light Artin->parabolic}
We now have all necessary ingredients to prove the following
\begin{theorem}\label{thm:light->parabolic}
Let~$\wh M$, $M$ be Coxeter matrices of finite type
and let~$\Phi\in\Hom_{\mathscr A}(\wh M,M)$ be light. Then
$\Phi$ is parabolic.
\end{theorem}
\begin{proof}
It suffices to prove the Theorem for
homomorphisms listed in Proposition~\ref{prop:types of light}.
\begin{proposition}\label{prop:PJ Bn An-1}
$P_{[1,n-1]}\in\Hom_{\mathscr A}(B_n,A_{n-1})$ is parabolic.
\end{proposition}
\begin{proof}
Let~$M=B_n$, $I=[1,n]$. 
We need the following
\begin{lemma}\label{lem:P Bn An-1 w0J}
Let~$J\subset [1,n]$ and let~$J'$
be the connected component of~$J$
containing~$n$ if~$n\in J$ and set~$J'=\emptyset$ if~$n\notin J$. Then
$$
P_{[1,n-1]}(T_{w_\circ^J})=T_{w_\circ^{J\setminus J'}}T_{w_\circ^{J'\setminus\{n\}}}^2
=T_{w_\circ^{J\setminus \{n\}}}T_{w_\circ^{J'\setminus\{n\}}}
=T_{w_\circ^{J'\setminus \{n\}}}T_{w_\circ^{J\setminus\{n\}}}.
$$
\end{lemma}
\begin{proof}
Since~$J\setminus J'$ and~$J'\setminus\{n\}$
are orthogonal and~$J\setminus\{n\}=(J\setminus J')
\cup (J'\setminus\{n\})$ the second and the third
equalities follow from the first. Clearly, it suffices
to prove the first one
for~$J=[a,b]$, $1\le a\le b\le n$. If~$b<n$
then the assertion is obvious.
Since 
$T_{w_\circ^{[a,n]}}=\Cx an{}^{h(M_{[a,n]})/2}=\Cx an{}^{n-a+1}
$
for all
$1\le a\le n$ by Proposition~\partref{prop:Coxeter splitting.b}, it follows that
$P_{[1,n-1]}(T_{w_\circ^{[a,n]}})=
\Cx a{(n-1)}{}^{n-a+1}=\Cx a{(n-1)}{}^{h(M_{[a,n-1]})}=T_{w_\circ^{[a,n-1]}}^2$ by Proposition~\partref{prop:Coxeter splitting.a}.
\end{proof}
Let~$J\subset I$. Then in~$\Br(M)$ we have 
$T_{w_{J}}=T_{w_\circ^J}^{-1}T_{w_\circ^I}$, whence 
by Lemma~\ref{lem:P Bn An-1 w0J} and Proposition~\partref{prop:fund elts BrSa.c}
\begin{align*}
P_{[1,n-1]}(T_{w_J})&=T_{w_\circ^{J\setminus\{n\}}}^{-1}
T_{w_\circ^{J'\setminus\{n\}}}^{-1} T_{w_\circ^{[1,n-1]}}^2
=T_{w_\circ^{J\setminus\{n\}}}^{-1}
T_{w_\circ^{[1,n-1]}}
T_{w_\circ^{\sigma(J'\setminus\{n\})}}^{-1}
T_{w_\circ^{[1,n-1]}}\\&
=T_{w_{J\setminus\{n\};[1,n-1]}}T_{w_{\sigma(J'\setminus\{n\});[1,n-1]}},
\end{align*}
where~$\sigma$ is the diagram automorphism of~$\Br^+_{[1,n-1]}(M)\cong \Br^+(A_n)$. Therefore,
$P_{[1,n-1]}(T_{w_J})
\in\mP(\Br^+_{[1,n-1]}(M))$.
\end{proof}

\begin{proposition}\label{prop:PJ F4 A2}
$P_{\{1,2\}}\in\Hom_{\mathscr A}(F_4,A_2)$ is parabolic.
\end{proposition}
\begin{proof}
Let~$M=F_4$.
We need the following
\begin{lemma}\label{lem:PJ F4 A2 w0J}
Let~$J\subset [1,4]$ and let~$J'$
be the connected component of~$J$ not contained in~$\{3,4\}$.
Then
$$
P_{\{1,2\}}(T_{w_\circ^J})=
\begin{cases}
T_1,&J'=\{1\},\\
T_{w_\circ^{\{1,2\}}}^{2^{k-2}},&J'=[1,k],\,2\le k\le 4,\\
T_2^{k-1},&J'=[2,k],\,2\le k\le 4.\\
\end{cases}
$$
\end{lemma}
\begin{proof}
Clearly, $J$ has at most two connected components. If $J'$, $J''$ are connected components of~$J$ with~$J''\subset\{3,4\}$ then~$P_{\{1,2\}}(T_{w_\circ^J})=P_{\{1,2\}}(T_{w_\circ^{J'}})$.

The assertion for~$J'\subset[1,3]$ was proven in Lemma~\ref{lem:P Bn An-1 w0J}. 
By Proposition~\partref{prop:Coxeter splitting.b}, $T_{w_\circ^{[1,4]}}=
\Cx14{}^{h(M)}=\Cx14{}^6$ and so
$P_{\{1,2\}}(T_{w_\circ^{[1,4]}})=
(T_1T_2)^6=T_{w_\circ^{[1,2]}}^4$
as $(T_1T_2)^3=T_{w_\circ^{\{1,2\}}}^2$.
Since~$T_{w_\circ^{[2,k]}}=
\Cx2k{}^{k-1}$, $2\le k\le 4$ 
(this is obvious for~$2\le k\le 3$ and follows from Proposition~\partref{prop:Coxeter splitting.b} for~$k=4$), it follows that~$P_{\{1,2\}}(T_{w_\circ^{[2,k]}})=T_2^{k-1}$, $2\le k\le 4$.
\end{proof}
We have $T_{w_J}=T_{w_\circ^J}^{-1}T_{w_\circ^I}$
in~$\Br(M)$ and so~$P_{\{1,2\}}(T_{w_J})=P_{\{1,2\}}(T_{w_\circ^{J'}})^{-1}T_{w_\circ^{\{1,2\}}}^4$ by Lemma~\ref{lem:PJ F4 A2 w0J}. 
If~$J'=[1,k]$, $2\le k\le 4$ then
$P_{\{1,2\}}(T_{w_J})=T_{w_\circ^{\{1,2\}}}^{4-2^{k-2}}\in\mP(\Br^+_{\{1,2\}}(M))$. If~$J'=[2,k]$, $2\le k\le 4$ then since $T_i^{-1}T_{w_\circ^{\{1,2\}}}=
T_{w_\circ^{\{1,2\}}}\sigma(T_i)^{-1}$,
$i\in\{1,2\}$ where~$\sigma$ be the diagram automorphism of~$\Br^+_{\{1,2\}}(M)\cong\Br^+(A_2)$,
\begin{align*}
P_{\{1,2\}}(T_{w_J})&=
T_2^{1-k}T_{w_\circ^{\{1,2\}}}^4
=
\big(\ascprod_{0\le r\le k-2}\sigma^{r}(T_2)^{-1}T_{w_\circ^{\{1,2\}}}\big)
T_{w_\circ^{\{1,2\}}}^{5-k}
\\
&=\big(\ascprod_{0\le r\le k-2}T_{w_{\{2-\overline r\};\{1,2\}}}\big)
T_{w_\emptyset;\{1,2\}}^{5-k}\in \mP(\Br^+_{\{1,2\}}(M)).
\end{align*}
Finally, if~$J'=\{1\}$, $P_{\{1,2\}}(T_{w_J})=T_1^{-1}T_{w_\circ^{\{1,2\}}}^4=T_{w_{\{1\};\{1,2\}}}T_{w_{\emptyset;{\{1,2\}}}}^3\in\mP(\Br^+_{\{1,2\}}(M))$.
\end{proof}
The homomorphisms from Proposition~\partref{prop:types of light.PBnA1} and~\partref{prop:types of light.PI2A1} are obviously parabolic
as the image of~$T_{w_\circ^J}$, $J\subset I$ under the corresponding homomorphism is always~$T_1^{d(J)}$
for some~$d(J)\in\ZZ_{\ge0}$ with
$d(I)=h(M)/2$ and~$d(J)\le d(I)$, $J\subsetneq I$.

Next we consider foldings listed in Proposition~\partref{prop:types of light.FDn+1An}\ref{prop:types of light.FD4A2}.
\begin{proposition}\label{prop:Dn+1 An parab}
Let~$M=D_{n+1}$, $n\ge 2$.
Then~$\mathbf F_{\varpi_{(n,n+1)}}:\Br^+(M)\to \Br^+_{[1,n]}(M)
\cong\Br^+(A_n)$
is a parabolic homomorphism.
\end{proposition}
\begin{proof}
Let~$\sigma$ be the diagram automorphism of~$\Br^+(M)$ corresponding
to the permutation $(n,n+1)$ of~$I$.
We need the following
\begin{lemma}\label{lem:Dn+1 An w0}
Let~$J\subset [1,n+1]$. Then
$$
\mathbf F_{\varpi_{(n,n+1)}}(T_{w_\circ^J})=
\begin{cases}
T_{w_\circ^J},& J\subset [1,n],\\
T_{w_\circ^{\sigma(J)}},& J\subset\sigma([1,n]),\\
T_{w_\circ^{J\setminus J'}}T_{w_\circ^{J'\setminus \{n+1\}}}^2=
T_{w_\circ^{J\setminus \{n+1\}}}
T_{w_\circ^{J'\setminus \{n+1\}}}\\
\quad =
T_{w_\circ^{J'\setminus \{n+1\}}}
T_{w_\circ^{J\setminus \{n+1\}}}
,&\{n,n+1\}\subset J,\\
\end{cases}
$$
where~$J'$ is the maximal interval $[i,n+1]$ contained in~$J$.
\end{lemma}
\begin{proof}
The first two cases are obvious. To prove the last,
since~$J\setminus J'$ and~$J'$ are orthogonal,
it suffices to prove the first equality and hence
to consider the case when~$J=J'$. If~$J=\{n,n+1\}$
then $T_{w_\circ^J}=T_nT_{n+1}$ and so
~$\mathbf F_{\varpi_{(n,n+1)}}(T_{w_\circ^J})=T_n^2$.
Otherwise, $J=[i,n+1]$ for some~$1\le i\le n-1$ and~$T_{w_\circ^J}$ is 
central in~$\Br^+_J(M)$ if~$n-i$ is even. If~$n-i$ is odd then~$T_j T_{w_\circ^J}=T_{w_\circ^J}T_j$
for all~$j\in [i,n-1]$ while
$T_j T_{w_\circ^J}=T_{w_\circ^J}T_{2n+1-j}$, $j\in\{n,n+1\}$. It follows that
$\mathbf F_{\varpi_{(n,n+1)}}(T_{w_\circ^{[i,n+1]}})$ is central in~$\Br^+_{[i,n]}(M)\cong \Br^+(A_{n+1-i})$. Since~$\ell(\mathbf F_{\varpi_{(n,n+1)}}(T))=\ell(T)$
for all~$T\in\Br^+(M)$,
we conclude that~$\ell(\mathbf F_{\varpi_{(n,n+1)}}(T_{w_\circ^{[i,n+1]}}))=
\ell(T_{w_\circ^{[i,n+1]}})=
(n+1-i)(n+2-i)=\ell(T_{w_\circ^{[i,n]}}^2)$. 
Since~$T_{w_\circ^{[i,n]}}^2$
generates the center of~$\Br^+_{[i,n]}(M)$ by Proposition~\partref{prop:fund elts BrSa.d}, it follows that
$\mathbf F_{\varpi_{(n,n+1)}}(T_{w_\circ^{[i,n+1]}})=T_{w_\circ^{[i,n]}}^2$.
\end{proof}
We claim that for any~$J\subset I$,
\begin{equation}\label{eq:Dn+1 An parab I}
\mathbf F_{\varpi_{(n,n+1)}}(T_{w_{J}})=\begin{cases}
T_{w_{J;[1,n]}}
T_{w_{\emptyset;[1,n]}},
&J\subset[1,n],\\
T_{w_{\sigma(J);[1,n]}}
T_{w_{\emptyset;[1,n]}},
&J\subset\sigma([1,n]),\\
T_{w_{J\setminus\{n+1\};[1,n]}}T_{w_{\sigma_{[1,n]}(J'\setminus\{n+1\});[1,n]}},&
\{n,n+1\}\subset J,
\end{cases}
\end{equation}
where~$\sigma_{[1,n]}$ is the diagram automorphism
of~$\Br^+_{[1,n]}(M)\cong \Br^+(A_n)$. Indeed,
we have, in $\Br(M)$, $T_{w_J}=T_{w_\circ^J}^{-1}
T_{w_\circ^I}$. If~$J\subset [1,n]$ we have 
by Lemma~\ref{lem:Dn+1 An w0}
$$
\mathbf F_{\varpi_{(n,n+1)}}(T_{w_J})=T_{w_\circ^J}^{-1} T_{w_\circ^{[1,n]}}^2
=T_{w_{J;[1,n]}}T_{w_\circ^{[1,n]}}.
$$
The case when~$J\subset \sigma([1,n])$ is proven similarly. Finally, if~$\sigma(J)=J$ then by Lemma~\ref{lem:Dn+1 An w0} and Proposition~\partref{prop:fund elts BrSa.c}
\begin{align*}
\mathbf F_{\varpi_{(n,n+1)}}(T_{w_J})&=T_{w_\circ^{J\setminus\{n+1\}}}^{-1}
T_{w_\circ^{J'\setminus \{n+1\}}}^{-1}
T_{w_\circ^{[1,n]}}^2
=T_{w_\circ^{J\setminus\{n+1\}}}^{-1}T_{w_\circ^{[1,n]}}
T_{w_\circ^{\sigma_{[1,n]}(J'\setminus \{n+1\})}}^{-1}
T_{w_\circ^{[1,n]}}\\
&=T_{w_{J\setminus\{n+1\};[1,n]}}T_{w_{\sigma_{[1,n]}(J'\setminus\{n+1\});[1,n]}}.
\end{align*}

Therefore, $\mathbf F_{\varpi_{(n,n+1)}}(T_{w_J})\in\mP(\Br^+_{[1,n]}(M))$
for all~$J\subset I$.
\end{proof}
\begin{proposition}\label{prop:D4 A2}
Let~$M=D_4$. Then~$\mathbf F_{\varpi_{(1,3,4)}}:
\Br^+(M)\to\Br^+(A_2)$ is a parabolic homomorphism.
\end{proposition}
\begin{proof}
Note that~$\mathbf F_{\varpi_{(1,3,4)}}=\mathbf F_{\varpi}\circ 
\mathbf F_{\varpi_{(3,4)}}$ where~$\varpi:[1,3]\to \{1,2\}$ is defined by $\varpi(1)=\varpi(3)=1$, $\varpi(2)=2$. Both homomorphisms are parabolic
by Proposition~\ref{prop:Dn+1 An parab}.
\end{proof}
The next class of homomorphisms to consider are elementary Tits homomorphisms from Proposition~\partref{prop:types of light.A1A1}\ref{prop:types of light.BnAn}\ref{prop:types of light.G2A2}. The first is obviously parabolic, while 
$\mathbf T_{n,2}\in
\Hom_{\mathscr A}(B_n,A_n)$ and~$\mathbf T_{1,3}\in\Hom_{\mathscr A}(G_2,A_2)$ are compositions 
of an unfolding (respectively, $\Br^+(B_n)\to \Br^+(D_{n+1})$ from~\eqref{eq:unfold Bn Dn+1} and 
$\Br^+(G_2)\to\Br^+(D_4)$ from Theorem~\partref{thm:adm finite class.even}), which is parabolic by
Theorem~\ref{thm:adm finite class}, with a 
folding (respectively, $\mathbf F_{\varpi_{(n,n+1)}}$ and
$\mathbf F_{\varpi_{(1,3,4)}}$)
whose parabolicity has already been established.

The length homomorphism is trivially parabolic. 

It remains to consider tautological homomorphisms
$\Phi:\Br^+(I_2(dm))\to \Br^+(I_2(m))$,
$m>2$, $d\ge 2$. Let~$I=\{1,2\}$.
We have $\Phi(\wh T_{w_\circ^I})
=\Phi(\brd{\wh T_1\wh T_2}{dm})=
\brd{T_1T_2}{dm}=(\brd{T_1T_2}{m})^d
=T_{w_\circ^I}^d$ by Lemma~\partref{lem:taut homs.a}.
Thus, $\Phi(\wh T_{w_{\emptyset}})\in 
\mP(\Br^+(I_2(m)))$. The remaining non-identity parabolic 
elements can be written, in~$\Br(I_2(dm))$, as 
$\wh T_{w_{\{i\}}}=\wh T_i^{-1}\wh T_{w_\circ^I}$,
$i\in \{1,2\}$ and so
$\Phi(\wh T_{w_{\{i\}}})=T_i^{-1}T_{w_\circ^I}^d
=T_{w_{\{i\}}}T_{w_{\emptyset}}^{d-1}\in 
\mP(\Br^+(I_2(m)))$.
\end{proof}

\subsection{Parabolic projections of Hecke monoids}
In this section, we fix a Coxeter matrix~$M$ over an index set~$I$ and
abbreviate $W=W(M)$ and~$W_J=W_J(M)$ for~$J\subset I$. The main
result of this section is
the following
\begin{theorem}\label{thm:pJwK is w0KstarJw0J}
For any $J \subset I$, $K\in\mathscr F(M)$, we have $p_J(w_{K}) = w_{ J\star_I K; J}$.  In particular, $p_J$ is a parabolic homomorphism.
\end{theorem}

\subsubsection{Reduction to connected subsets and corank one}
In this proof, we
will abbreviate $\mathscr F=\mathscr F(M)$ which in this case coincides with~$\mathscr P(I)$,
$J\star K=J\star_I K$ for $J,K\subset I$ and~$w_\circ=w_\circ^I$.

Define
$$
\plink{Good}\mathscr G=\mathscr G(M)=\{ (J,K)\in \mathscr F(M)\times \mathscr F(M)\,:\, p_J(w_K)=w_{J\star K;J}\}.
$$
Obviously, proving Theorem~\ref{thm:pJwK is w0KstarJw0J} amounts to proving that $\mathscr G=\mathscr P(I)\times\mathscr P(I)$.
Note that while $J\star K=K\star J$ for all $I,K\subset I$,
$p_J(w_K)$ and~$p_K(w_J)$ belong to different submonoids of~$W$
and do not need to be equal. Thus, $(J,K)\in\mathscr G$
does not immediately imply that~$(K,J)\in\mathscr G$.

The following proposition
is one of key ingredients of our proof, since it allows us
to consider only~$(J,K)\in \mathscr P(I)\times\mathscr P(I)$ such that both~$J$ and~$K$ are connected.
\begin{proposition} \label{prop:only connected}
\begin{enmalph}
\item\label{prop:only connected.a} Let $J',J''\subset I$ be orthogonal and let $K\subset I$.
If $(J',K),(J'',K)\in\mathscr G $ then $(J'\cup J'',K)
\in\mathscr G $.
\item\label{prop:only connected.b} Let $J\subset I$ and let $K',K''\subset I$ be orthogonal.
If $(J,K'),(J,K'')\in\mathscr G $ then $(J,K'\cup K'')\in
\mathscr G $.
\end{enmalph}
\end{proposition}
\begin{proof}
To prove part~\ref{prop:only connected.a}, denote
$J=J'\cup J''$. Since~$(J',K)\in\mathscr G$,
we have $p_{J'}(w_K)=w_{J'\star K;J'}=w_\circ^{J'\star K}w_\circ^{J'}$ and similarly for~$J''$.
Then by Lemmata~\ref{lem:p_J comp} and~\ref{lem:parab prod}
\begin{align*}
p_{J}(w_K)=p_{J'}(p_J(w_K))\times p_{J''}(p_J(w_K))=
p_{J'}(w_K)\times p_{J''}(w_K)
&=w_\circ^{J'\star K}w_\circ^{J'}\times w_\circ^{J''\star K}w_\circ^{J''}
\end{align*}
Since $S\star K\subset S$ for any $S\subset I$ and $J',J''$ are orthogonal, it follows that
$w_\circ^{J'}w_\circ^{J''\star K}=w_\circ^{J''\star K}w_\circ^{J'}$ and so $p_J(w_K)=w_\circ^{J'\star K}w_\circ^{J''\star K}w_\circ^J$.
Since $J'\star K\subset J'$ and~$J''\star K\subset J''$,
$J'\star K$ and $J''\star K$ are orthogonal whence
$(J'\star K)\cup (J''\star K)=J\star K$ by \cite{He09}*{Lemma~6}
and $w_\circ^{J'\star K}w_\circ^{J''\star K}=w_\circ^{J\star K}$.
Therefore,
$p_J(w_K)=w_\circ^{J\star K}w_\circ^J=w_{J\star K;J}$, which
proves part~\ref{prop:only connected.a}.

To prove part~\ref{prop:only connected.b} more work is needed.
\begin{lemma}\label{lem:wKL}
Let $K_1$, $K_2$ be orthogonal subsets of~$I$. Then $\downarrow w_{K_1}\cap \downarrow w_{K_2}=
\downarrow w_{K_1\cup K_2}$, that is, $w_{K_1\cup K_2}$ is the unique maximal element of
$\downarrow w_{K_1}\cap \downarrow w_{K_2}$.
\end{lemma}
\begin{proof}
Let $K=K_1\cup K_2$. Since~$K_1$ and~$K_2$ are orthogonal,
$w_\circ^K=w_\circ^{K_1}\times w_\circ^{K_2}$. In particular,
$w_\circ^{K_1},w_\circ^{K_2}\le w_\circ^K$ by Proposition~\partref{prop:Bruhat order.a}.

Since $w<w'$ if and only if $w'w_\circ^I<ww_\circ^I$ by Proposition~\partref{prop:Bruhat order.c} and
$w_S=w_\circ^S w_\circ$ for any $S\subset I$, the assertion is equivalent to
$$
\uparrow w_\circ^{K_1} \cap \uparrow w_\circ^{K_2} = \uparrow w_\circ^{K}.
$$
By the above, $w_\circ^K\le u$ implies $w_\circ^{K_s}\le u$, $s\in\{1,2\}$
and so
$\uparrow w_\circ^K\subset \uparrow w_\circ^{K_1}\cap \uparrow w_\circ^{K_2}$.

Conversely, suppose that $u\in \uparrow w_\circ^{K_1}\cap \uparrow w_\circ^{K_2}$.
Write~$u=s_{i_1}\cdots s_{i_r}$ where~$r=\ell(u)$ and~$i_1,\dots,i_r\in I$.
By
Proposition~\partref{prop:Bruhat order.a},
there exist $J_1,J_2\subset [1,r]$ such that $\ell(w_\circ^{K_p})=|J_p|$ and
$w_\circ^{K_p}=\ascprod_{t\in J_p} s_{i_t}$, $p\in \{1,2\}$. Note
that since~$K_1\cap K_2=\emptyset$, $J_1\cap J_2=\emptyset$. Furthermore, since~$K_1$ and~$K_2$ are orthogonal
$$\ascprod_{t\in J_1\cup J_2} s_{i_t}=
\big(\ascprod_{t\in J_1}s_{i_t}\big)\big( \ascprod_{t\in J_2} s_{i_t}\big)=
w_\circ^{K_1}w_\circ^{K_2}=
w_\circ^K
$$
and
so~$w_\circ^K\le u$ by Proposition~\partref{prop:Bruhat order.a}. Thus, $u\in\uparrow w_\circ^K$.
\end{proof}
We now show that parabolic projections are compatible with the strong Bruhat order.
\begin{lemma}\label{lem:proj comparison}
Let $w\le w'\in W(M)$ in the strong Bruhat order. Then $p_J(w)\le p_{J'}(w')$ for any $J\subset
J'\subset I$.
\end{lemma}
\begin{proof}
First, we prove that $p_J(w)\le p_J(w')$ for all $w\le w'$, $J\subset I$. The assertion is obvious if~$w=w'$.
By Proposition~\partref{prop:Bruhat order.b}, $w<w'$ implies that there exists a chain $w=u_0<\cdots<u_k=w'$ with
$k=\ell(w')-\ell(w)$. Thus, it suffices to prove the assertion for $w<w'$ with $\ell(w')=\ell(w)+1$.

Write $w'=s_{i_1}\times \cdots\times s_{i_r}$ with $r=\ell(w')$.
Since $\ell(w)=\ell(w')-1$ and by Proposition~\partref{prop:Bruhat order.a} every reduced expression for~$w'$ contains a reduced expression for~$w$, there exists $1\le t\le r$ such that $w'=u\times s_{i_t}\times v$
and $w=u\times v$ with $u=s_{i_1}\times\cdots\times s_{i_{t-1}}$ and $v=s_{i_{t+1}}\times\cdots \times s_{i_r}$. Then
$$
p_J(w')=p_J(u)\star p_J(s_{i_t})\star p_J(v),\quad p_J(w)=p_J(u)\star p_J(v).
$$
If $i_t\in I\setminus J$ then $p_J(s_{i_t})=1$ and so $p_J(w)=
p_J(w')$. If~$i_t\in J$ and so $p_J(s_{i_t})=s_{i_t}$ there
are two possibilities. If
$\ell(p_J(u)s_{i_t})<\ell(p_J(u))$
then by Proposition~\ref{prop:prod *}
$p_J(u)\star p_J(s_{i_t})=p_J(u)$ and so again~$p_J(w)=p_J(w')$.
Otherwise, $p_J(u)\star p_J(s_{i_t})=p_J(u)\times s_{i_t}$ with $\ell(p_J(u)s_{i_t})>\ell(p_J(u))$ and so $p_J(u)< p_J(u)\times s_{i_t}$. By Proposition~\partref{prop:Bruhat order *.a},
$p_J(w)\le p_J(w')$.

It remains to prove that $p_J(w)\le p_{J'}(w)$ for all $w\in W(M)$, $J\subset J'\subset I$.
Indeed, then we have $p_J(w)\le p_J(w')\le p_{J'}(w')$ for all $w\le w'$, $J\subset J'\subset I$.
The argument is by induction on~$\ell(w)$. For $\ell(w)=0$ there is nothing to prove while
for $w=s_i$, $i\in I$ either $p_J(w)=p_{J'}(w)$ or $p_J(w)=1$, $p_{J'}(w)=s_i$ and
the assertion holds.

Suppose that $\ell(w)>1$ and write $w=u\times s_i$ for some~$i\in I$, $u\in W(M)$. Then
$p_K(w)=p_K(u)\star p_K(s_i)$ for any~$K\subset I$. Since $p_J(u)\le p_{J'}(u)$ by the induction hypothesis and $p_J(s_i)\le p_{J'}(s_i)$ by the induction base, we have $p_J(w)=p_J(u)\star p_J(s_i)\le p_{J'}(u)\star p_{J'}(s_i)=
p_{J'}(w)$ by~Proposition~\partref{prop:Bruhat order *.a}.
\end{proof}

\begin{lemma}\label{lem:wKL proj}
Let $K_1,K_2\subset I$ be orthogonal and let $J\subset I$. Then
$$\downarrow p_J(w_{K_1\cup K_2})=\downarrow p_J(w_{K_1})\cap \downarrow p_J(w_{K_2}).$$
\end{lemma}
\begin{proof}
By Lemmata~\ref{lem:proj comparison} and~\ref{lem:wKL},
$\downarrow p_J(w_{K_1\cup K_2})\subset \downarrow p_J(w_{K_1})\cap\downarrow p_J(w_{K_2})$.

Let $u\in \downarrow p_J(w_{K_1})\cap \downarrow p_J(w_{K_2})$. In particular, $u\in W_J$ and so
$u=p_J(u)$. Then by Lemma~\ref{lem:proj comparison} we have
$$
u=p_J(u)\le p_J(w_{K_t})\le p_I(w_{K_t})=w_{K_t},\qquad t\in\{1,2\}.
$$
Therefore, $u\in \downarrow w_{K_1}\cap \downarrow w_{K_2}=
\downarrow w_{K_1\cup K_2}$
by Lemma~\ref{lem:wKL}. Then $u=p_J(u)\le p_J(w_{K_1\cup K_2})$ by Lemma~\ref{lem:proj comparison}.
\end{proof}

We now have all ingredients to finish our proof of part~\ref{prop:only connected.b} of Proposition~\ref{prop:only connected}. Suppose that
$K',K''\subset I$ are orthogonal and that $(J,K'),(J,K'')\in\mathscr G$.
Let $K=K'\cup K''$.
Then $p_J(w_{K'})=w_{J\star K';J}$, $p_J(w_{K''})=w_{J\star K'';J}$
and so $\downarrow p_J(w_K)=\downarrow w_{J\star K';J}\cap \downarrow w_{J\star K'';J}$.
But since $J\star K'$, $J\star K''$
are orthogonal subsets of~$J$ and $(J\star K')\cup (J\star K'')=J\star K$ by~\cite{He09}*{Lemma~6}, applying Lemmata~\ref{lem:len prop wJ K} and~\ref{lem:wKL} to~$W_J$ we conclude that
$$\downarrow w_{J\star K;J}=\downarrow w_{J\star K';J}\cap \downarrow w_{J\star K'';J}=
\downarrow p_J(w_K).$$
Thus, $(J,K)\in\mathscr G$ and part~\ref{prop:only connected.b} is proven.
\end{proof}

\begin{lemma}\label{lem:eq for w0 J*K}
Let $J,K\subset I$. Then $
w_\circ^J=w_\circ^{J\star K}\star p_J(w_K)$ and $w_{J\star K;J}\le p_J(w_K)$. In particular, if $J\star K=\emptyset$ then
$(J,K)\in\mathscr G$.
\end{lemma}
\begin{proof}
We have $w_\circ=w_\circ^{J\star K}\times w_{J\star K}=w_\circ^{J\star K}\star w_K\star w_J$. Since $w_\circ w_J^{-1}=w_\circ^J$ we obtain
$$
w_\circ^J=((w_\circ^{J\star K}\star w_K)\star w_J)w_J^{-1}.
$$
By~\cite{K14}*{Proposition~6}\footnote{In~\cite{K14}*{Proposition~6} the left-sided version
is proven. The right-sided version is proven similarly and is left to the reader as an exercise.}, this implies that $w_\circ^J\le w_\circ^{J\star K}\star w_K$.
Since $J\star K\subset J$, applying $p_J$ to both sides we obtain by Lemma~\ref{lem:proj comparison}
$$
w_\circ^J\le w_\circ^{J\star K}\star p_J(w_K).
$$
As $w_\circ^J$ is the unique maximal element of $W_J$ in the strong Bruhat order, the first assertion follows.
To prove the second, note that by Lemma~\ref{lem:len prop wJ K} we now have
$w_{J\star K;J}=w_\circ^{J\star K}(w_\circ^{J\star K}\star p_J(w_K))$ and
so $w_{J\star K;J}\le p_J(w_K)$ by~\cite{K14}*{Proposition~6}.
Since~$w_{\emptyset;J}=w_\circ^J$ and~$w_\circ^{\emptyset}=1$, the
last assertion is now trivial.
\end{proof}
In particular, $(\emptyset,K),(J,\emptyset)\in\mathscr G$ for all~$J,K\subset I$. Also, since $w_I=1$, $I\star J=J$ for all~$J\subset I$. Now, $p_J(w_I)=1=w_{J;J}$ for all~$J\subset I$
and $p_I(w_K)=w_K=w_{I\star K;K}$ for all~$K\subset I$. Thus,
$(J,I),(I,K)\in\mathscr G$ for all~$J,K\subset I$.
From now on, we assume that $J,K$ are proper non-empty subsets of~$I$.

The following Lemma allows us to use induction on (connected) subgraphs of the Coxeter graph of~$W$.
\begin{lemma}\label{lem:induction}
Suppose that
Theorem~\ref{thm:pJwK is w0KstarJw0J} is proven for~$W_{J}$ for some~$J\subsetneq I$. Then
$(J,K)\in\mathscr G$ for some~$K\subset I$ implies that $(J',K)\in\mathscr G$
for all~$J'\subset J$.
\end{lemma}
\begin{proof}
Since $J'\subset J$, we have $p_{J'}(w)=p_{J'}(p_J(w))$ for all $w\in W$. Since $(J,K)\in\mathscr G$, $p_J(w_K)=w_{J\star K;J}$.
Now, since Theorem~\ref{thm:pJwK is w0KstarJw0J} holds for~$W_J$, $p_{J'}(p_J(w_K))=w_{J'\star_J (J\star K);J'}$.
Since $J'\star_J(J\star K)=J'\star K$ by~\cite{He09}*{Lemmata~4 and~7}, the assertion follows.
\end{proof}

\begin{lemma}\label{eq:good products}
Let~$J,K,L\subset I$ and suppose that~$(J,K),(J,L)\in
\mathscr G$. Then $(J,K\star L)\in \mathscr G$.
\end{lemma}
\begin{proof}
We have
\begin{align*}
p_J(w_{K\star L})&=p_J(w_K)\star p_J(w_L)
=w_{J\star K; J}\star w_{J\star L; J}\\
&=w_{(J\star K)\star_J (J\star L); J}.
\end{align*}
By~\cite{He09}*{Lemmata~4 and~7} we have
$$
(J\star K)\star_J (J\star L)=(J\star K)\star L=
J\star (K\star L),
$$
and so~$p_J(w_{K\star L})=w_{J\star (K\star L); J}$ that is
$(J,K\star L)\in \mathscr G$.
\end{proof}

\subsubsection{Homomorphisms}
The following result allows us to significantly reduce the number of case-by-case arguments in proving Theorem~\ref{thm:pJwK is w0KstarJw0J}.
\begin{proposition}\label{prop:LCM hom good}
Let~$\wh M\in\Cox{\wh I}$, $M\in\Cox I$ be of finite type. Let $\Phi\in\Hom_{\mathscr A}(\wh M,M)$
be strictly parabolic and of Hecke type.
If~$(\wh J,\wh K)\in \mathscr G(\wh M)$ then $([\Phi](\wh J),[\Phi](\wh K))\in \mathscr G(M)$. Conversely, if
$\overline\Phi_\star$ is injective and
$([\Phi](\wh J),[\Phi](\wh K))\in\mathscr G(M)$ then $(\wh J,\wh K)
\in\mathscr G(\wh M)$.
\end{proposition}
\begin{proof}
We need the following
\begin{lemma}\label{lem:prj hom}\label{lem:diag aut proj}
Let~$\phi\in\Hom_{\mathscr H}(M',M)$ be disjoint. Then for any~$J'\subset I'$,
$\phi\circ p_{J'}=p_{[\phi](J')}\circ \phi$. In particular,
if~$\sigma$ is a diagram automorphism of~$(W(M),\star)$ then
$\sigma\circ p_J=p_{\sigma(J)}\circ\sigma$ for all~$J\subset I$.
\end{lemma}
\begin{proof}
It suffices to prove that $\phi(p_{J'}(s_i))=
p_{[\phi](J')}(\phi(s_i))$ for all~$i\in I'$.
By Lemma~\ref{lem:Hecke hom w0J},
$$
\phi(p_{J'}(s_i))=\begin{cases}
1,&i\in I'\setminus J'\\
w_\circ^{[\phi](i)},&i\in J'.
\end{cases}
$$
On the other hand, $p_{[\phi](J')}(\phi(s_i))=
p_{[\phi](J')}(w_\circ^{[\phi](i)})=
w_\circ^{[\phi](i)\cap [\phi](J')}$ by Lemma~\ref{lem:proj w0}.
If~$i\in J'$ then $[\phi](i)\cap[\phi](J')=[\phi](i)$.
Otherwise, since~$[\phi](i)\cap[\phi](j)=\emptyset$ for all~$i\not=j$, $[\phi](i)\cap [\phi](J')=\emptyset$. In either case, the assertion follows.
\end{proof}
We have
\begin{alignat*}{3}
\overline\Phi_\star(p_{\wh J}(w_{\wh K}))&=
\overline\Phi_\star(w_{\wh J\star_{\wh I}\wh K; \wh J})&\qquad &
\text{since $(\wh J,\wh K)\in\mathscr G(\wh M)$}\\
&=w_{[\Phi](\wh J\star_{\wh I}\wh K);[\Phi](\wh J)}&&\text{by Lemma~\ref{lem:parab preserving}}\\
&=w_{[\Phi](\wh J)\star_I [\Phi](\wh K));[\Phi](J)}&&\text{by Lemma~\ref{lem:hom parab submonoid}}.
\end{alignat*}
On the other hand, since~$\overline\Phi_\star\circ p_{\wh J}=
p_{[\Phi](\wh J)}\circ\overline \Phi_\star$ by Lemma~\ref{lem:prj hom}, we have by Lemma~\ref{lem:parab preserving}
\begin{equation*}
\overline\Phi_\star(p_{\wh J}(w_{\wh K}))=
p_{[\Phi](\wh J)}(w_{[\Phi](\wh K)}).
\end{equation*}
Thus, $p_{[\Phi](\wh J)}(w_{[\Phi](\wh K)})=w_{[\Phi](\wh J)\star_I [\Phi](\wh K));[\Phi](J)}$. Conversely,
\begin{alignat*}{3}
\overline\Phi_\star(p_{\wh J}(w_{\wh K}))&=
p_{[\Phi](\wh J)}(\overline\Phi_\star(w_{\wh K}))&&\text{by Lemma~\ref{lem:prj hom}}\\
&=p_{[\Phi](\wh J)}(w_{[\Phi](\wh K)})&&\text{by Lemma~\ref{lem:parab preserving}}\\
&=w_{[\Phi](\wh J)\star_I [\Phi](\wh K));[\Phi](J)}&\qquad&\text{since $([\Phi](\wh J),[\Phi](\wh K))\in\mathscr G(M)$}\\
&=w_{[\Phi](\wh J\star_{\wh I}\wh K);[\Phi](J)}&&\text{by Lemma~\ref{lem:hom parab submonoid}}\\
&=\overline\Phi_\star(w_{\wh J\star_{\wh I}\wh K;\wh J})&&
\text{by Lemma~\ref{lem:parab preserving}}.
\end{alignat*}
Since~$\overline\Phi_\star$ is injective, it follows that
$p_{\wh J}(w_{\wh K})=w_{\wh J\star_{\wh I}\wh K;\wh J}$.
\end{proof}
\begin{corollary}\label{cor:diag aut good}
Let~$\sigma$ be a diagram automorphism of~$W$ and the corresponding permutation of~$I$. Then
\begin{enmalph}
    \item\label{cor:diag aut good.a} $(J,K)\in \mathscr G$ if and only if~$(\sigma(J),\sigma(K))\in\mathscr G$;
    \item\label{cor:diag aut good.b} Suppose that~$J\subset I$ satisfies~$(J,K)\in\mathscr G$ for all~$K\subset I$. Then $(\sigma(J),K)\in\mathscr G$ for all~$K\subset I$;
    \item\label{cor:diag aut good.c} Suppose that~$K\subset I$ satisfies~$(J,K)\in\mathscr G$ for all~$J\subset I$. Then~$(J,\sigma(K))\in\mathscr G$ for all~$J\subset I$.
\end{enmalph}
\end{corollary}
\begin{proof}
Since~$\sigma(w_\circ)=w_\circ$, $\sigma(w_J)=w_{\sigma(J)}$ while
$\sigma(w_{K;J})=w_{\sigma(K);\sigma(J)}$ for all~$K\subset J\subset I$. We have
\begin{alignat*}{3}
p_{\sigma(J)}(w_{\sigma(K)})&=p_{\sigma(J)}(\sigma(w_K))\\
&=\sigma(p_J(w_K))&&\text{by Proposition~\ref{prop:LCM hom good}}\\
&=\sigma(w_{J\star K;J})&&\text{since $(J,K)\in\mathscr G$}\\
&=w_{\sigma(J\star K);\sigma(J)}\\
&=w_{\sigma(J)\star\sigma(K);\sigma(J)}&\qquad &\text{by Lemma~\ref{lem:hom parab submonoid}}.
\end{alignat*}
This proves part~\ref{cor:diag aut good.a}. Parts~\ref{cor:diag aut good.b} and~\ref{cor:diag aut good.c} follow from part~\ref{cor:diag aut good.a} since~$\sigma$ induces a bijection on~$\mathscr F(M)$.
\end{proof}

\subsubsection{Proof of Theorem~\ref{thm:pJwK is w0KstarJw0J} in
rank 2}
Let~$m=m_{12}=m_{21}$.
By Corollary~\partref{cor:diag aut good.c} it
suffices to prove that~$(J,\{1\})\in\mathscr G$ for~$J\in\{\{1\},\{2\}\}$. Since~$w_{\{i\}}=\brd{s_j\times s_i\times }{m-1}$
where~$\{i,j\}=\{1,2\}$,
it follows that $w_{\{1\}}\star w_{\{1\}}=w_\circ=
w_{\{1\}}\star w_{\{2\}}$ that is, $\{1\}\star \{1\}=\emptyset=\{1\}\star\{2\}$.
Thus, $(J,\{1\})\in\mathscr G$ for all~$J\subsetneq I$.\hfill\qedsymbol

\subsubsection{Proof of Theorem~\ref{thm:pJwK is w0KstarJw0J} for type \texorpdfstring{$A_n$}{An}}
We have, for any  $1\le a\le b\le n$,
\begin{equation}
w_\circ^{[a,b]}=\ascprodst_{a\le k\le b} \cxr ak=
\dscprodst_{a\le k\le b} \cx ak=
\ascprodst_{a\le k\le b} \cxr kb=
\dscprodst_{a\le k\le b} \cx kb.\label{eq:w0 type A}
\end{equation}
We need the following
\begin{proposition}\label{prop:wK explicit}
For any $1\le i\le j\le k\le l\le n$,
$$
w_{[j,k];[i,l]}
=\cx il{}^{\star(j-i)}\star \cxr il{}^{\star (l-k)}
= \cxr il{}^{\star (l-k)}\star\cx il{}^{\star(j-i)}.
$$
\end{proposition}
\begin{proof}
First we prove the Proposition for the case when either $i=j$ or $k=l$.
\begin{lemma}\label{lem:*powers of coxeter}
For all $1\le a\le b\le n$, $k\ge 0$ we have
\begin{align*}
&\cxr ab{}^{\star k}=w_{[a,b-k];[a,b]}=w_{[a+k,b];[a,b]}{}^{-1},
\\
&\cx ab{}^{\star k}=w_{[a,b-k];[a,b]}{}^{-1}=w_{[a+k,b];[a,b]}.
\end{align*}
\end{lemma}
\begin{proof}
We only prove the equality  $\cxr{a}b{}^{\star k}=w_{[a,b-k];[a,b]}$. The first equality
in the second line is proved similarly, while the remaining
equalities are obtained by applying the diagram automorphism
of~$W_{[a,b]}$.
The argument is by induction on~$k$, the case~$k=0$ being trivial.
For the inductive step, note that if $0\le k\le b-a$ then $w_\circ^{[a,b-k]}=\cx a{(b-k)}\times
w_\circ^{[a,b-(k+1)]}$ by~\eqref{eq:w0 type A}
and so $w_{[a,b-(k+1)];[a,b]}=\cxr a{(b-k)}\times w_{[a,b-k];[a,b]}$. Now, by the induction hypothesis
\begin{align*}
\cxr ab{}^{\star (k+1)}&=\cxr ab\star w_{[a,b-k];[a,b]}\\
&=\cxr{(b-k+1)}b
\star \cxr a{(b-k)}\star  w_{[a,b-k];[a,b]}\\&=
\cxr{(b-k+1)}b\star w_{[a,b-(k+1)];[a,b]}\\
&=w_{[a,b-(k+1)];[a,b]}.
\end{align*}
The last equality follows from Lemma~\ref{lem:wK absorption} since
$[b-k+1,b]\subset [a,b]\setminus[a,b-k-1]$.

In particular, we proved that $\cxr ab{}^{\star (b-a+1)}=w_\circ^{[a,b]}$.
Since $\cxr ab\star w_\circ^{[a,b]}=w_\circ^{[a,b]}$ by Lemma \ref{lem:char w_0 monoid},
it follows that $\cxr ab{}^{\star k}=w_\circ^{[a,b]}=w_{\emptyset;[a,b]}$ for all~$k\ge b-a+1$.
\end{proof}
To treat the general case, we use induction on~$j-i$
to show that
\begin{equation}\label{eq:wKexplicit inter}
w_{[j,k];[i,l]}
=\cx il{}^{\star(j-i)}\star w_{[i,k];[i,l]}.
\end{equation}
Once~\eqref{eq:wKexplicit inter} is established, the Proposition follows by Lemma~\ref{lem:*powers of coxeter}.

The case $j=i$ is trivial.
For the inductive step, note that by~\eqref{eq:w0 type A} for $j\le k$
$$
w_\circ^{[j,k]}=
w_\circ^{[j+1,k]}\times \cx jk=\cxr jk\times w_\circ^{[j+1,k]}
$$
and so
\begin{equation}\label{eq:w0step}
w_{[j+1,k];[i,l]}=\cx jk\times w_{[j,k];[i,l]}.
\end{equation}
Then
\begin{alignat*}{3}
w_{[j+1,k];[i,l]}&=\cx i{j-1}\star w_{[j+1,k];[i,l]}&&
\text{by Lemma~\ref{lem:wK absorption}}
\\
&=\cx i{(j-1)}\star \cx jk\times w_{[j,k];[i,l]}&\qquad&\text{by \eqref{eq:w0step}}\\
&=\cx ik\star w_{[j,k];[i,l]}\\
&=\cx il\star w_{[j,k];[i,l]}&&\text{by
Lemma~\ref{lem:wK absorption}}\\
&=\cx il{}^{\star (j-i+1)}\star w_{[i,k];[i,l]}&&\text{by the induction
hypothesis.}
\end{alignat*}
The inductive step is proven. The second equality is obtained from the first one using the diagram automorphism of~$W_{[i,l]}$.
\end{proof}
As an immediate byproduct, we obtain the following
\begin{corollary}[cf.~\cite{He09}]\label{cor:A J*K}
Let $J=[a',b']$, $K=[a,b]$, $1\le a\le b\le n$, $1\le a'\le b'\le n$.
Then $J\star_I K=[a+a'-1,b+b'-n]$.
\end{corollary}
\begin{proof}
We have, by Proposition~\ref{prop:wK explicit}
\begin{align*}
w_K\star w_J&=\cx1n{}^{\star(a-1)}\star \cxr1n{}^{\star(2n-b-b')}
\star \cx1n{}^{(a'-1)}\\
&=\cx1n{}^{\star(a+a'-2)}\star \cxr1n{}^{\star(n-(b+b'-n))},
\end{align*}
which, again by Proposition~\ref{prop:wK explicit}, is equal to
$w_{[a+a'-1,b+b'-n]}$.
\end{proof}
\begin{proof}[Proof of Theorem~\ref{thm:pJwK is w0KstarJw0J},
$W$ of type~$A_n$]
By Proposition~\ref{prop:only connected}, it suffices to
prove that~$(J,K)\in\mathscr G$ for
$J=[a',b']$, $1\le a'\le b'\le n$ and $K=[a,b]$, $1\le a\le b\le n$.
Since
$$
w_K=\cx1n{}^{\star(a-1)}\star \cxr1n{}^{\star(n-b)},
$$
by Proposition~\ref{prop:wK explicit} and
Corollary~\ref{cor:A J*K}
we have
\begin{equation*}
p_J(w_K)=\cx{a'}{b'}{}^{\star(a-1)}\star
\cxr{a'}{b'}{}^{\star(n-b)}=w_{[a+a'-1,b+b'-n];[a',b']}
=w_{J\star K;J}.\qedhere
\end{equation*}
\end{proof}

\subsubsection{Proof of Theorem~\ref{thm:pJwK is w0KstarJw0J} for type \texorpdfstring{$B_n$}{Bn}}

Let~$\Phi:\Br^+(B_n)\to\Br^+(A_{2n-1})$ be the injective Coxeter-Hecke type
homomorphism from~\eqref{eq:unfold Bn A2n-1}. Let~$\wh I=[1,n]$,
$I=[1,2n-1]$.
Note the following
\begin{lemma}[cf.~\cite{He09}\footnote{We provide a proof since there is a
misprint in~\cite{He09} in the second case}]\label{lem:B J*K}
Let $J=[a',b']$, $K=[a,b]$, $1\le b\le b'\le n$.
Then
$$
J\star_{\wh I} K=\begin{cases}
\emptyset,&b'<n,\\
[a+a'-1,b-a'+1],&b<b'=n,\\
[a+a'-1,n],&b=b'=n.
\end{cases}
$$
\end{lemma}
\begin{proof}
Note that $[\Phi]([a,b])=[a,b]\sqcup [2n-b,2n-a]$ if~$1\le a\le b<n$ while
$[\Phi]([a,n])=[a,2n-a]$, $1\le a\le n$, and that the intervals $[a,b]$, $[2n-b,2n-a]$
are orthogonal subsets of~$I$.

If~$b,b'<n$,
we have by~\cite{He09}*{Lemma~6} and by Corollary~\ref{cor:A J*K}
\begin{align*}
[\Phi](J\star_{\wh I} K)&=([a',b']\star_I [a,b])\cup
([a',b']\star_I [2n-b,2n-a])\cup\\
&\qquad ([a,b]\star_I [2n-b',2n-a'])
\cup ([2n-b',2n-a']\star_I [2n-b,2n-a])\\
&=[a+a'-1,b+b'-2n+1]\cup [2n-b+a'-1,b'-a+1]\\
&\qquad
\cup [a+2n-b'-1,b-a'+1]\cup [4n-b-b'-1,2n-a-a'+1].
\end{align*}
All these intervals are empty
since $b-a,b'-a'\le n-2$. Since~$[\Phi](i)\not=\emptyset$ for all~$i\in\wh I$, it follows that
$J\star_{\wh I} K=\emptyset$.

If~$b<b'=n$ then again by \cite{He09}*{Lemma~6} and Corollary~\ref{cor:A J*K},
\begin{align*}
[\Phi](J\star_{\wh I} K)&=([a',2n-a']\star_I [a,b])\cup
([a',2n-a']\star_I [2n-b,2n-a]\\
&=[a+a'-1,b-a'+1]\cup [2n+a'-b-1,2n-a-a'+1]
\\
&=[\Phi]([a+a'-1,b-a'+1]).
\end{align*}
Since~$\Phi$ is disjoint and~$[\Phi](i)\not=\emptyset$ for all~$i\in\wh I$, $J\star_{\wh I}K=
[a+a'-1,b-a'+1]$ by Lemma~\partref{lem:elem Artin hom.a}.
Finally, if $b=b'=n$,
\begin{align*}
[\Phi](J\star_{\wh I} K)&=[a',2n-a']\star [a,2n-a]\\
&=[a+a'-1,2n-a-a'+1]\\
&=[\Phi]([a+a'-1,n]),
\end{align*}
and it remains to apply Lemma~\partref{lem:elem Artin hom.a}.
\end{proof}
\begin{proof}[Proof of Theorem~\ref{thm:pJwK is w0KstarJw0J},
$W$ of type~$B_n$]
Since~$\Phi$ is injective and strongly square free,
$\overline\Phi_\star$ is injective by
Lemma~\partref{lem:sqf Hecke hom.d}.
The assertion is
immediate from Proposition~\ref{prop:LCM hom good}.
\end{proof}

\subsubsection{Proof of Theorem~\ref{thm:pJwK is w0KstarJw0J} for type \texorpdfstring{$D_{n+1}$}{Dn+1}}
Let~$\sigma$ be the diagram automorphism of~$W$ corresponding to the permutation~$(n,n+1)$ (cf.~\eqref{eq:diag aut}).
The following is easily checked
\begin{gather}
w_\circ = {\ascprodst_{1\le i\le n}}
\cx i{(n+1)}\times\cxr i{(n-1)}\label{eq:w0 D}\\
\label{eq:w1,n D}
w_{[1,n]}={\ascprodst_{1\le i\le n}}\sigma^i(\cxr{i}{n})
=\sigma(\cxr1n \times w_{[2,n];[2,n+1]})
\\
\intertext{
and}
\label{eq:w i,n+1 D}
w_{[i,n+1]}={\dscprodst_{1\le j\le i-1}}
(\cx j{(n+1)}\times\cxr j{(n-1)}),\qquad 1\le i\le n-1.
\end{gather}

\begin{proposition}\label{prop:1,n*K D}
Let $J=[1,n]$ and let $K\subset I$ be connected. We have
$$
J\star K=\begin{cases}
[1,n]_2,& K=J,\\
[1,n-1]_2,&K=\sigma(J),\\
[i,n-i+1],&K=[i,n+1],\,1\le i\le n-1,\\
\emptyset,&\text{otherwise}.
\end{cases}
$$
\end{proposition}
\begin{proof}
The argument is rather long, and we split it into several
Lemmata for the reader's convenience.
\begin{lemma}\label{lem:1,n*1,n D}
$\sigma(J)\star \sigma(J)=\sigma([1,n]_2)$ and
$\sigma(J)\star J=[1,n-1]_2$, $n\ge 3$.
\end{lemma}
\begin{proof}
We use induction on~$|I|$.
The induction base is the type $A_3$ ($n=2$), which identifies with $D_3$. To make the notation consistent with $D_{n+1}$ series, we label the nodes of the Coxeter graph of type~$A_3$ as follows
\begin{equation}\label{eq:A3 as D3}
\dynkin[text style/.style={scale=1},Coxeter,root radius=0.07,expand labels={2,1,3},edge length=1.2cm]D3.
\end{equation}
Then $J=\{1,2\}$, $\sigma(J)=\{1,3\}$.
Using Corollary~\ref{cor:A J*K}, we obtain
$J\star \sigma(J)=\{1\}=[1,1]_2$ and $\sigma(J)\star \sigma(J)=\{3\}=\sigma([1,2]_2)$.

For the inductive step, assume first that $n>2$ is odd.
Then $w_\circ$ is central in~$W$ and
since $w_{\sigma(J)}=w_\circ w_\circ^{\sigma(J)}$,
it follows from Lemma~\ref{lem:wK absorption} that $w_{\sigma(J)}\star s_n=w_{\sigma(J)}$ while
\begin{align}w_{\sigma(J)}\star s_i&=w_\circ^{\sigma(J)}w_\circ s_i=
w_\circ^{\sigma(J)}s_i w_\circ \nonumber\\
&=\begin{cases}
s_{n+1-i} w_{\sigma(J)},& 2\le i\le n-1,\\
s_{n+1} w_{\sigma(J)},&i=1,\\
s_1 w_{\sigma(J)},&i=n+1,
\end{cases}\nonumber\\
&=\begin{cases}
s_{n+1-i}\star w_{\sigma(J)},& 2\le i\le n-1,\\
s_{n+1}\star w_{\sigma(J)},&i=1,\\
s_1 \star w_{\sigma(J)},&i=n+1.
\end{cases}
\label{eq: wsJ si D}
\end{align}
Therefore,
$$
w_{\sigma(J)}\star \cxr1n=\cx{2}{(n+1)}\star w_{\sigma(J)},
$$
where we used that $s_n\star w_{\sigma(J)}=w_{\sigma(J)}$
by Lemma~\ref{lem:wK absorption}.
Using~\eqref{eq:w1,n D} and the induction hypothesis
we obtain
\begin{align*}
w_{\sigma(J)}\star w_{\sigma(J)}&=\cx2{(n+1)}\star \cxr1n\star w_{[2,n];[2,n+1]}\star w_{[2,n];[2,n+1]}\\
&=\cx2{(n+1)}\times \cxr1n\times w_{[2,n]_2;[2,n+1]}.
\end{align*}
Now, $W_{[2,n+1]}$ is
of type~$D_n$ and so
$w_\circ^{[2,n+1]}$ satisfies $s_i w_\circ^{[2,n+1]}=
w_\circ^{[2,n+1]}s_{\sigma(i)}$ for $i\in[2,n+1]$. As $n$ is odd, $[2,n]_2=[3,n]_2$, and so
\begin{align*}
w_{\sigma(J)}\star w_{\sigma(J)}&=\cx2{(n+1)}\times \cxr1{(n-1)}\times w_\circ^{[2,n+1]}w_\circ^{\sigma([3,n]_2)}\\
&=s_1 w_\circ^{[1,n+1]}w_\circ^{\sigma([3,n]_2)}=s_1 w_\circ^{\sigma([3,n]_2)} w_\circ
=w_\circ^{\sigma([1,n]_2)}w_\circ\\
&=w_{\sigma([1,n])_2}.
\end{align*}

If~$n>2$ is even then $w w_\circ=w_\circ \sigma(w)$, $w\in W$ and so for $1\le i\le n$ we have by Lemma~\ref{lem:wK absorption}
\begin{align}
w_{\sigma(J)}\star s_i&=w_\circ w_\circ^J\star s_i
=w_\circ w_\circ^Js_i
=w_\circ s_{n+1-i}w_\circ^J=s_{\sigma(n+1-i)}w_{\sigma(J)}\nonumber\\&
=s_{\sigma(n+1-i)}\star w_{\sigma(J)},\label{eq: wsJ si D even}
\end{align}
whence, as $s_n\star w_{\sigma(J)}=w_{\sigma(J)}$ by Lemma~\ref{lem:wK absorption},
$$
w_{\sigma(J)}\star \cxr1n=\sigma(\cx1n)\star w_{\sigma(J)}
=\sigma(\cx1n)\star s_n\star w_{\sigma(J)}=
\cx1{(n+1)}\star w_{\sigma(J)}
$$
Applying the induction hypothesis we obtain
$$
w_{\sigma(J)}\star w_{\sigma(J)}=
\cx{1}{(n+1)}\times \cxr1{(n-1)}\times w_{[2,n]_2;[2,n+1]}.
$$
Now, $w_\circ^{[2,n+1]}$ is central in~$W_{[2,n+1]}$ and
so
\begin{align*}
w_{\sigma(J)}\star w_{\sigma(J)}&=
\cx{1}{(n+1)}\times\cxr1{(n-1)}\times w_\circ^{[2,n+1]}
w_\circ^{[2,n]_2}\\
&=
w_\circ
w_\circ^{[2,n]_2}=
w_\circ^{\sigma([2,n]_2)}w_\circ
=w_{\sigma([1,n]_2)},
\end{align*}
since $[2,n]_2=[1,n]_2$ in this case.

The argument for $w_J\star w_{\sigma(J)}=w_{\sigma(J)}\star w_J$ is similar. If~$n>2$ is odd, we have by~\eqref{eq: wsJ si D}
$$
w_{\sigma(J)}\star \sigma(\cxr1n)=\cx1{(n+1)}\star w_{\sigma(J)}
$$
and so by~\eqref{eq:w1,n D} and by the induction hypothesis,
\begin{align*}
w_{\sigma(J)}\star w_J&=\cx1{(n+1)}\star \cxr1n\star w_{[2,n]_2;
[2,n+1]}\star \sigma(w_{[2,n]_2;[2,n+1]})\\
&=\cx1{(n+1)}\times \cxr1{(n-1)}\times w_{[2,n-1]_2}\\
&=\cx1{(n+1)}\times \cxr1{(n-1)}\times w_\circ^{[2,n+1]}
w_\circ^{[2,n-1]_2}\\
&=w_\circ^{[1,n+1]}w_\circ^{[2,n-1]_2}=w_\circ^{[2,n-1]_2}w_\circ^{[1,n+1]}\\
&=w_{[1,n-1]_2},
\end{align*}
as $n-1$ is even and so $[1,n-1]_2=[2,n-1]_2$. Similarly,
for $n>2$ even we obtain, using~\eqref{eq: wsJ si D even}
and $w_{\sigma(J)}\star s_{n+1}=w_{\sigma(J)}$,
$$
w_{\sigma(J)}\star \sigma(\cxr1n)=
w_{\sigma(J)}\star \cxr1{(n-1)}
=\cxr2{(n+1)} \star w_{\sigma(J)}
$$
whence, as in this case $[2,n-1]_2=[3,n-1]_2$,
\begin{align*}
w_{\sigma(J)}\star w_J&=\cx2{(n+1)}\star \cxr1n\star w_{[2,n]_2;
[2,n+1]}\star \sigma(w_{[2,n]_2;[2,n+1]})\\
&=\cx2{(n+1)}\times \cxr1{(n-1)}\times w_{[2,n-1]_2}\\
&=\cx2{(n+1)}\times \cxr1{(n-1)}\times w_\circ^{[2,n+1]}
w_\circ^{[3,n-1]_2}\\
&=s_1 w_\circ^{[1,n+1]}w_\circ^{[3,n-1]_2}=s_1 w_\circ^{[3,n-1]_2}w_\circ^{[1,n+1]}\\
&=w_{[1,n-1]_2}.\qedhere
\end{align*}
\end{proof}
\begin{lemma}\label{lem:1,n*i,j}
If $K\subsetneq J$ or $K\subsetneq \sigma(J)$ then
$w_J\star w_K=w_\circ$, that is, $J\star K=\emptyset$.
\end{lemma}
\begin{proof}
Suppose first that $K\subsetneq J$. Then
either $K\subset [1,n-1]$ or $K\subset [2,n]$ and so, by Lemma~\ref{lem:len prop wJ K},
either $w_K=w_{K;[1,n-1]}\star w_{[1,n-1]}$ or $w_K=w_{K;[2,n]}\star w_{[2,n]}$.
Thus, by Lemma~\ref{lem:char w_0 monoid} it suffices to prove that $w_{[1,n-1]}\star w_J=w_\circ$ and $w_{[2,n]}\star w_J=w_\circ$.

Using Lemmata~\ref{lem:wK absorption}, \ref{lem:len prop wJ K},  and~\ref{lem:1,n*1,n D}, we obtain
\begin{align*}
w_{[2,n]}\star w_J&=w_{[2,n];[1,n]}\star w_J\star w_J\\
&=\cx1n\star w_{[1,n]_2}=\cx1{(n-1)}\star s_n w_{[1,n]_2}=\cx1{(n-1)}\star w_{[1,n-2]_2}\\
&=\cx1{(n-2)}\star w_{[1,n-2]_2}.
\end{align*}
Continuing this way, we obtain $w_{[2,n]}\star w_J=w_\emptyset=w_\circ$. The computation for~$w_{[1,n-1]}$ is similar,
albeit a bit longer as it depends on the parity of~$n$,
and is omitted.

It remains to consider the case when $K=\sigma([i,n])$ for some~$2\le i\le n$. The same considerations as above show that it suffices to consider $K=\sigma([2,n])$. Then, by Lemmata~\ref{lem:len prop wJ K}, \ref{lem:wK absorption} and~\ref{lem:1,n*1,n D},
\begin{align*}
w_K\star w_J&=\sigma(w_{[2,n],[1,n]})\star w_{\sigma(J)}\star w_J=\sigma(\cx1n)\star w_{[1,n-1]_2}\\
&=\cx1{(n-1)}\star w_{[1,n-1]_2}=\cx1{(n-2)}\star s_{n-1}w_{[1,n-1]_2}=\cx1{(n-2)}\star w_{[1,n-3]_2}\\
&=\cx1{(n-3)}\star w_{[1,n-3]_2}.
\end{align*}
Continuing this way, we obtain $w_K\star w_J=w_\circ$.
\end{proof}
The last remaining case is
\begin{lemma}
We have $J\star [i,n+1]=[i,n+1-i]$, $1\le i\le n-1$.
\end{lemma}
\begin{proof}
We use induction on~$i$. The induction base is trivial
as $[1,n+1]=I$.
For the inductive step, note that $w_{[i,n+1]}=
\cx{(i-1)}{(n+1)}\times\cxr{(i-1)}{(n-1)}\times w_{[i-1,n+1]}$, $i>2$.
Therefore, using the induction hypothesis and Lemma~\ref{lem:wK absorption} we obtain for
$i\le (n+3)/2$
\begin{align*}
w_{[i,n+1]}\star w_{\sigma(J)}&=\cx{(i-1)}{(n+1)}\star\cxr{(i-1)}{(n-1)}\star w_{[i-1,n+1]}\star w_{\sigma(J)}\\
&=\cx{(i-1)}{(n+1)}\star\cxr{(i-1)}{(n-1)}\star w_{[i-1,n+2-i]}\\
&=\cx{(i-1)}{(n+1)}\star \cxr{(n+3-i)}{(n-1)}\star
\cxr{(i-1)}{(n+2-i)} w_{[i-1,n+2-i]}\\
&=\cx{(i-1)}{(n+1)}\star \cxr{(n+3-i)}{(n-1)}\star
w_{[i-1,n+1-i]}\\
&=\cx{(i-1)}{(n+1)}\star
w_{[i-1,n+1-i]}\\
&=\cx{(i-1)}{(n+1-i)}w_{[i-1,n+1-i]}\\
&=w_{[i,n+1-i]}.
\end{align*}
If $i>(n+3)/2$ then $[i-1,n+2-i]$ is empty, that is $w_{[i-1,n+2-i]}=w_\circ$, and so $w_{[i,n+1]}\star w_{\sigma(J)}=
\cx{(i-1)}{(n+1)}\star\cxr{(i-1)}{(n-1)}\star w_\circ=w_\circ$, that is $[i,n+1]\star \sigma(J)=\emptyset$. But then $i>(n+3)/2>(n+1)/2$ and so $[i,n+1-i]$ is also empty.
\end{proof}
This exhausts all connected~$K\subset I$.
\end{proof}

\begin{proposition}\label{prop:proj 1n D}
For any connected~$K\subset I$, $([1,n],K)\in\mathscr G$.
\end{proposition}
\begin{proof}
Let $J=[1,n]$. By Proposition~\ref{prop:1,n*K D} we only
need to consider the cases when $K=J$, $K=\sigma(J)$
and $K=[i,n+1-i]$, $i\le (n+1)/2$.

First, we use induction on rank of~$W$ to prove that $(J,J),(J,\sigma(J))\in\mathscr G$.
The case $n=2$ is actually type~$A_3$. Labeling the Coxeter graph  as in~\eqref{eq:A3 as D3}, we obtain $w_{\{1,2\}}=s_3 s_1 s_2$
$w_{\sigma(\{1,2\})}=s_2s_1s_3$ and so
$p_{\{1,2\}}(w_{\{1,2\}})=s_1 s_2=w_{\{2\};\{1,2\}}$,
$p_{\{1,2\}}(w_{\{1,3\}})=s_2s_1=w_{\{1\};\{1,2\}}$.
For the inductive step,
we have by~\eqref{eq:w1,n D}
\begin{align*}
p_{J}(w_{J})&=p_{J}(\sigma(\cxr1n))
\star p_{J}(\sigma(w_{[2,n],[2,n+1]}))\\
&=\cxr1{(n-1)}\star p_{[2,n]}(w_{\sigma([2,n]);[2,n+1]})\\
&=\cxr1{(n-1)}\times w_{[2,n-1]_2;[2,n]},
\\
\intertext{while}
p_{J}(\sigma(w_{J}))&=p_{J}(\cxr1n)
\star p_{J}(w_{[2,n];[2,n+1]})\\
&=\cxr1n\times w_{[2,n]_2;[2,n]}
\end{align*}
Now, $w_{[2,n-1]_2;[2,n]}=w_\circ^{[2,n]}w_\circ^{[3,n]_2}$
and so
$$
p_{J}(w_{J})=\cxr1{(n-1)}w_\circ^{[2,n]}
w_\circ^{[3,n]_2}=s_n w_\circ^{J}
w_\circ^{[3,n]_2}=w_\circ^{[1,n]_2}w_\circ^{J}=
w_{J\star J;J}
$$
where we used Proposition~\ref{prop:1,n*K D}.
Similarly,
\begin{equation*}
p_{J}(\sigma(w_{J}))=\cx1n w_\circ^{[2,n]}
w_\circ^{[2,n]_2}=w_\circ^{J}w_\circ^{[2,n]_2}
=w_\circ^{[1,n-1]_2}w_\circ^{J}
=w_{J\star\sigma(J);J}
\end{equation*}
also by Proposition~\ref{prop:1,n*K D}.

Now we use induction on~$i$ to prove that $(J,[i,n+1])\in\mathscr G$ for $1\le i\le (n+1)/2$. The induction base is
trivial as $[1,n+1]=I$.
For the inductive step, observe that, by~\eqref{eq:w i,n+1 D}, $w_{[i,n+1]}=\cx{(i-1)}{(n+1)}\times \cxr{(i-1)}{(n-1)}\times w_{[i-1,n+1]}$ for $i>1$.
Since by Lemma~\ref{lem:wK absorption}, $s_j\star w_{[a,n+1]}=w_{[a,n+1]}$, $1\le j<a$, we have
\begin{align*}
w_{[i,n]}&=\cx{1}{(i-2)}\star w_{[i,n]}=
\cx1{(n+1)}\star\cxr{(i-1)}{(n-1)}\star w_{[i-1,n+1]} \\
&=\cx1{(n+1)}\star\cxr1{(n-1)}\star w_{[i-1,n+1]}.
\end{align*}
Therefore,
\begin{align*}
p_J(w_{[i,n+1]})&=
\cx{1}{n}\star\cxr{1}{(n-1)}\star w_{[i-1,n+2-i],[1,n]}\\
&=\cx{1}n\star \cxr1n\star w_{[i-1,n+2-i],[1,n]}\\
&=w_{[2,n-1],[1,n]}\star w_{[i-1,n+2-i],[1,n]}\\
&=w_{[i,n-i+1],J}=w_{J\star [i,n+1],J},
\end{align*}
where we used Proposition~\ref{prop:wK explicit}, Corollary~\ref{cor:A J*K} and Proposition~\ref{prop:1,n*K D}.
\end{proof}

It remains to prove that $([2,n+1],K)\in\mathscr G$ for all
connected~$K\subset I$.
\begin{lemma}\label{lem:2,n+1*K D}
Let~$K\subset I$ be connected. Then for $J=[2,n+1]$
$$
J\star K=
\begin{cases}
[i+1,n+1],&K=[i,n+1],\,1\le i\le n-1\\
[i+1,j-1],&K=[i,j],\,1\le i\le j\le n-1\\
[i+1,n-1],&\text{$K=[i,n]$ or $K=\sigma([i,n])$, $1\le i\le n$}\\
\end{cases}
$$
\end{lemma}
\begin{proof}
Note that~$J$ is $\sigma$-invariant.
In the first two cases, $K=\sigma(K)$ and, since~$w_\circ$
is $\sigma$-invariant,
$w_K$, $w_J$ are $\sigma$-invariant. Yet the set of~$\sigma$-invariant elements in~$(W,\star)$ is isomorphic to~$(W(B_n),
\star)$ by Theorem~\ref{thm:adm finite class} and Lemma~\partref{lem:sqf Hecke hom.d}, and so we can apply Lemma~\ref{lem:B J*K}.

To prove the assertion for~$K=[i,n]$ we use induction on~$i$. The case $i=1$ has already been established in Proposition~\ref{prop:1,n*K D}. For the inductive step,
note that for $i>1$, $w_{[i,n]}=w_{[i,n];[i-1,n]}\star w_{[i-1,n]}$
by Lemma~\ref{lem:len prop wJ K},
whence
\begin{align*}
w_{[i,n]}\star w_J&=w_{[i,n];[i-1,n]}\star w_{[i-1,n]}\star w_J=\cx in\star w_{[i,n-1]}\\
&=\cx i{(n-1)}\star w_{[i,n-1]}=\cx i{(n-1)}w_\circ^{[i,n-1]}w_\circ\\
&=w_\circ^{[i+1,n-1]}w_\circ\\
&=w_{[i+1,n-1]},
\end{align*}
where we used Lemma~\ref{lem:wK absorption} and
the induction hypothesis, as well as the fact that $[i-1,n]$ and $[i,n-1]$ are of type~$A$. The result for~$K=\sigma([i,n])$ is now immediate.
\end{proof}
\begin{proposition}\label{prop:proj 2n+1 D}
Let $J=[2,n+1]$. Then $(J,K)\in\mathscr G$ for all connected~$K\subset I$.
\end{proposition}
\begin{proof}
If $K=\sigma(K)$ the assertion follows from the result in type~$B$ and Corollary~\ref{cor:diag aut good}. Thus, the only case to consider is that of $K=[i,n]$.
We use induction on~$i$. For~$i=1$ we have by~\eqref{eq:w1,n D}
\begin{align*}
p_{J}(w_{[1,n]})&=p_{J}(\sigma(\cxr1n))\star
p_{J}(\sigma(w_{[2,n];J}))=\sigma(\cxr2n)\star \sigma(w_{[2,n];J})\\
&=\sigma(\cxr2n\star w_{[2,n];J})=\sigma(\cxr2n w_{[2,n];J})\\
&=\sigma( \cxr2n w_\circ^{[2,n]}w_\circ^{J})=\sigma( w_\circ^{[2,n-1]}w_\circ^{J})\\
&=w_{[2,n-1];J}\\
&=w_{J\star [1,n];J}
\end{align*}
by Lemma~\ref{lem:2,n+1*K D}.
For~$i>1$, write $w_{[i,n]}=w_{[i,n];[i-1,n]}\times w_{[i-1,n]}=\cx{(i-1)}n\times w_{[i-1,n+1]}$ using Lemma~\ref{lem:len prop wJ K}. Then using the induction hypothesis and
Lemma~\ref{lem:wK absorption} we obtain
\begin{align*}
p_{J}(w_{[i,n]})&=p_{J}(\cx{(i-1)}n)\star
p_{J}(w_{[i-1,n];J})\\
&=\cx{\max(2,i-1)}n\star w_{[i,n-1];J}\\
&=\cx{\max(2,i-1)}{(n-1)}\star w_{[i,n-1];J}.
\end{align*}
If $i=2$ then we obtain
\begin{align*}
p_{J}(w_{[2,n]})&=\cx{2}{(n-1)}\star  w_{[2,n-1];J}=\cx{2}{(n-1)}w_\circ^{[2,n-1]}w_\circ^{J}=w_\circ^{[3,n-1]}w_\circ^{J}\\
&=w_{[3,n-1];J}\\
\intertext{while for $i>2$}
p_{J}(w_{[i,n]})&=\cx{(i-1)}{(n-1)}\star  w_{[i,n-1];J}=s_{i-1}\star\cx{i}{(n-1)}w_\circ^{[i,n-1]}w_\circ^{J}\\
&=s_{i-1}\star w_\circ^{[i+1,n-1]}w_\circ^{J}\\
&=w_{[i+1,n-1];J}.
\end{align*}
In either case, we obtain $p_J(w_K)=w_{J\star K;J}$
by Lemma~\ref{lem:2,n+1*K D}.
\end{proof}

\begin{proof}[Proof of Theorem~\ref{thm:pJwK is w0KstarJw0J},
$W$ of type~$D$]
By Lemma~\ref{lem:induction}, we only need to prove
that $(J,K)\in\mathscr G$ for all connected $K\subset I$ and
for all connected~$J\subset I$ with~$|J|=n$, that is for
$J\in\{[1,n],\sigma([1,n]),[2,n+1]\}$.
For~$J=[1,n]$, Theorem~\ref{thm:pJwK is w0KstarJw0J} for $W_J$
has already been proven since~$W_{[1,n]}$ is of type~$A$, while for~$J=[2,n+1]$
we can use induction~$|J|$ of~$W$,
the induction base being $D_3=A_3$. The result for~$J=\sigma([1,n])$ follows from that for~$J=[1,n]$ by Lemma~\ref{lem:diag aut proj}.
The assertion now follows from Propositions~\ref{prop:proj 1n D}
and~\ref{prop:proj 2n+1 D}.
\end{proof}

\subsubsection{Proof of Theorem~\ref{thm:pJwK is w0KstarJw0J} for
exceptional types}
By Theorem~\partref{thm:adm finite class.unfold}\ref{thm:adm finite class.H3}\ref{thm:adm finite class.H4}, Proposition~\ref{prop:LCM hom good}, Lemma~\partref{lem:sqf Hecke hom.d}, and Corollary~\ref{cor:diag aut good} it remains to
prove Theorem~\ref{thm:pJwK is w0KstarJw0J} for type~$E_n$,
$n\in\{6,7,8\}$.

\begin{proof}[Proof of Theorem~\ref{thm:pJwK is w0KstarJw0J},
$W$ of type~$E$]
First, let $W$ be of type~$E_6$ and let~$\sigma$ be
its diagram automorphism (cf.~\eqref{eq:diag aut}).
By Lemma~\ref{lem:induction}, it suffices to consider
all $J\subset I$
with~$|J|=5$, that is,
$$
J\in \mathscr J_{E_6}=\{ [1,5],\,[2,6], \sigma([2,6])\}
$$
and all connected~$K\subset I$ with $J\star K\not=\emptyset$.
Note that Theorem~\ref{thm:pJwK is w0KstarJw0J} has
already been proven for~$W_J$
with these~$J$ since $W_{[1,5]}$ is of type~$A_5$ while~$W_{[2,6]}$ and~$W_{\sigma([2,6])}$ are of type~$D_5$.
By Corollary~\ref{cor:diag aut good}, the assertion
for~$\sigma([2,6])$ follows from that for~$[2,6]$.

Using a Python program we developed for computations in Hecke monoids, we obtain
$J\star K=\emptyset$
for all connected~$K\subsetneq I$ and~$J\in \mathscr J_{E_6}$
except
\begin{alignat*}{3}
&[1,5]\star [2,6]=\{3,5\},
&\qquad &[2,6]\star \{2,3,4,6\}=
\{3\},\\
&[2,6]\star [2,6]=\{3, 4, 6\},&&[2,6]\star \sigma([2,6])=[2,4]
\end{alignat*}
and the products obtained from the above by applying~$\sigma$.
We have
\begin{align*}
w_{[1,5]}&=s_6s_3s_2s_1s_4s_3s_2s_5s_4s_3s_6s_3s_2s_1s_4s_3s_2s_5s_4s_3s_6,\\
w_{[2,6]}&=
s_1s_2s_3s_4s_5s_6s_3s_2s_1s_4s_3s_2s_6s_3s_4s_5,\\
w_{\{2,3,4,6\}}&=
s_1s_2s_3s_4s_5s_4s_3s_2s_1s_6s_3s_2s_1s_4s_3s_2s_5s_4s_6s_3s_2s_1s_4s_5
\end{align*}
and so
\begin{align*}
p_{[1,5]}(w_{[2,6]})
&=s_1s_2s_1s_3s_2s_4s_3s_2s_1s_5s_4s_3s_2
=w_\circ^{[1,3]}s_1 \cxr14\cxr15 s_1
=s_3 w_\circ^{[1,5]}s_1\\
&=w_{\{3,5\};[1,5]},\\
p_{[2,6]}(w_{[1,5]})&=s_2s_3s_4s_3s_2s_5s_4s_3s_6s_3s_2s_4s_3s_5s_4s_6s_3s_2
=s_2s_3s_4s_3s_2s_5s_4s_3 w_\circ^{[2,5]}w_\circ^{[2,6]}\\
&=\cx25 w_\circ^{[2,5]}s_4s_5s_3s_4 w_\circ^{[2,6]}
=w_\circ^{[3,5]}s_4s_5s_3s_4 w_\circ^{[2,6]}
=w_\circ^{\{3,5\}}w_\circ^{[2,6]}\\
&=w_{\{3,5\};[2,6]},\\
p_{[2,6]}(w_{[2,6]})&=
s_2s_3s_4s_5s_4s_3s_2s_6s_3s_2s_4s_3s_5s_6=
w_\circ^{[3,4]}w_\circ^{[2,6]}\cx24\\
&=w_\circ^{[3,4]}s_6 s_3 s_4 w_\circ^{[2,6]}
=w_\circ^{\{3,4,6\}}w_\circ^{[2,6]}\\
&=w_{\{3,4,6\};[2,6]},\\
p_{[2,6]}(w_{\{2.3,4,6\}})&=s_2s_3s_4s_3s_2s_5s_4s_3s_2s_6s_3s_2s_4s_3s_5s_4s_6s_3s_2\\
&=s_2 s_3 s_4 s_3 s_2 w_\circ^{[2,4]} w_\circ^{[2,6]}
=s_3 w_\circ^{[2,6]}\\
&=w_{\{3\};[2,6]},\\
p_{[2,6]}(w_{\sigma([2,6])})&=
s_5s_4s_3s_2s_6s_3s_2s_4s_3s_5s_4s_6s_3s_2=
w_\circ^{[2,4]}w_\circ^{[2,6]}\\
&=w_{[2,4];[2,6]}.
\end{align*}
This completes the proof of Theorem~\ref{thm:pJwK is w0KstarJw0J}
for~$W$ of type~$E_6$.

For type~$E_7$, we only need to consider pairs $(J,K)$
with~$J$ connected and of cardinality 6, that is
$$
J\in\mathscr J_{E_7}=\{[1,6],[2,7],[1,5]\cup\{7\}\},
$$
which are, respectively, of types $A_6$, $D_6$ and~$E_6$,
and $K\subset I$ connected such that~$J\star K\not=\emptyset$, which are
\begin{alignat*}{3}
&([1,5]\cup\{7\})\star ([1,4]\cup\{7\})=\{3\},&\qquad
&([1,5]\cup\{7\})\star ([2,5]\cup\{7\})=\{3\},\\
&([1,5]\cup\{7\})\star ([1,5]\cup\{7\})=[2,4]\cup\{7\},
&&([1,5]\cup\{7\})\star [2,7]=
[2,4],\\
&[2,7]\star [2,7]=\{4,6,7\}.
\end{alignat*}
We have
\begin{align*}
w_{[1,5]\cup\{7\}}&=s_6s_5s_4s_3s_2s_1s_7s_3s_2s_4s_3s_5s_4s_6s_5s_7s_3s_2s_1s_4s_3s_2s_7s_3s_4s_5s_6,\\
w_{[2,7]}&=s_1s_2s_3s_4s_5s_6s_7s_3s_2s_1s_4s_3s_2s_5s_4s_3s_7s_3s_2s_1s_4s_3s_2s_5s_4s_3s_6s_5s_4s_7s_3s_2s_1,\\
w_{[1,4]\cup\{7\}}&=s_5s_4s_3s_2s_1s_6s_5s_4s_3s_2s_1s_7s_3s_2s_1s_4s_3s_2s_5s_4s_3s_6s_5s_4s_7s_3s_2s_1s_4s_3s_2s_5s_4s_3\times \\
&\qquad s_7s_3s_2s_4s_3s_5s_4s_6s_5,\\
w_{[2,5]\cup\{7\}}&=
s_1s_2s_3s_4s_5s_6s_5s_4s_3s_2s_1s_7s_3s_2s_1s_4s_3s_2s_5s_4s_3s_6s_5s_7s_3s_2s_1s_4s_3s_2s_5s_4s_3s_6\times \\
&\qquad s_5s_4s_7s_3s_2s_1s_4s_5s_6,
\end{align*}
and so for~$J=[1,5]\cup \{7\}$
\begin{align*}
p_J(w_{[1,4]\cup\{7\}})&=s_1s_2s_1s_3s_2s_4s_3s_2s_1s_5s_4s_3s_2s_1s_7s_3s_2s_1s_4s_3s_2s_5s_4s_3s_7s_3s_2s_1s_4s_3s_2s_5s_4s_3s_7\\
&=s_1s_2s_1s_3s_2w_\circ^{[1,3]}w_\circ^{J}
=s_3 w_\circ^J\\
&=w_{\{3\};J},\\
p_J(w_{[2,4]\cup\{7\}})&=s_1s_2s_1s_3s_2s_4s_3s_2s_1s_5s_4s_3s_2s_1s_7s_3s_2s_1s_4s_3s_2s_5s_4s_3s_7s_3s_2s_1s_4s_3s_2s_5s_4s_3s_7\\
&=s_1 s_2 s_1 s_3 s_2 w_\circ^{[1,3]}w_\circ^J\\
&=w_{\{3\};J},\\
p_J(w_J)&=s_1s_2s_3s_4s_5s_4s_3s_2s_1s_7s_3s_2s_1s_4s_3s_2s_5s_4s_7s_3s_2s_1s_4s_5\\
&=w_\circ^{[2,4]}w_\circ^J s_7 s_3 s_4 s_2 s_3 s_7
=w_\circ^{[2,4]}s_7 s_3 s_2 s_4 s_3 s_7 w_\circ^J
=w_\circ^{[2,4]\cup\{7\}}w_\circ^J\\
&=w_{[2,4]\cup\{7\};J},\\
p_J(w_{[2,7]})
&=s_1s_2s_3s_4s_5s_4s_3s_2s_1s_7s_3s_2s_1s_4s_3s_2s_5s_4s_3s_7s_3s_2s_1s_4s_3s_2s_5s_4s_3s_7\\
&=w_{[2,4]; J}
\end{align*}
while for~$J=[2,7]$
\begin{align*}
p_J(w_{[1,5]\cup\{7\}})&=s_5s_4s_3s_2s_6s_5s_4s_3s_2s_7s_3s_2s_4s_3s_5s_4s_6s_5s_7s_3s_2s_4s_3s_7\\
&=w_{[2,4];[2,7]},\\
p_J(w_J)&=s_2s_3s_2s_4s_3s_5s_4s_3s_2s_6s_5s_4s_3s_7s_3s_2s_4s_3s_5s_4s_6s_5s_7s_3s_2s_4s_3\\
&=s_4 w_\circ^{[2,5]}w_\circ^{[2,6]} \cx25 s_7 w_\circ^J=s_4 \cxr26\cx25 s_7 w_\circ^J=s_4s_6s_7 w_\circ^J\\
&=w_{\{4,6,7\};J}.
\end{align*}
Finally, let~$W$ be of type~$E_8$. The connected~$J\subset I$
with~$|J|=7$ are
$$
J\in\mathscr J_{E_8}=\{[1,7],\,[2,8],\,
[1,6]\cup\{8\}\},
$$
which are, respectively, of type~$A_7$, $D_7$ and~$E_7$ and so,
in particular, Theorem~\ref{thm:pJwK is w0KstarJw0J} has already been established for~$W_J$ with~$J\in \mathscr J_{E_8}$.
We only need to consider connected~$K\subset I$ such that~$J\star K\not=\emptyset$ for some~$J\in\mathscr J_{E_8}$, which happens only for $J=K=[1,6]\cup\{8\}$
and
$$
J\star J=[2,4]\cup \{8\}.
$$
We have
\begin{align*}
w_J&=s_7s_6s_5s_4s_3s_2s_1s_8s_3s_2s_4s_3s_5s_4s_6s_5s_7s_6s_8s_3s_2s_1s_4s_3s_2s_5s_4s_3s_8s_3s_2s_1s_4s_3s_2\times\\
&\qquad
s_5s_4s_3s_6s_5s_4s_7s_6s_5s_8s_3s_2s_1s_4s_3s_2s_8s_3s_4s_5s_6s_7
\end{align*}
and
\begin{align*}
p_J(w_J)&=s_1s_2s_3s_4s_5s_4s_3s_2s_1s_6s_5s_4s_3s_2s_1s_8s_3s_2s_1s_4s_3s_2s_5s_4s_3s_6s_5s_4s_8s_3s_2s_1s_4s_3\times\\
&\qquad s_2s_5s_4s_3s_6s_5s_4s_8s_3s_2s_1s_4s_3s_5s_4s_6s_5\\
&=\cx15\cxr14 w_\circ^{[1,5]}w_\circ^J s_8s_3s_4s_2s_3s_8\\
&=w_\circ^{[2,4]}s_8s_3s_4s_2s_3s_8 w_\circ^J=
w_\circ^{[2,4]\cup\{8\}} w_\circ^J\\
&=w_{[2,4]\cup\{8\};J}.
\end{align*}
This completes the proof of Theorem~\ref{thm:pJwK is w0KstarJw0J}.
\end{proof}

\begin{remark}\label{rem:surj counterex}
The restriction of~$p_J$ to~$\mP(W(M),\star)$ needs not be surjective. For example, if $M=A_4$
and~$J=[1,3]$ then $\{1,3\}\subset J$ is not equal to
$J\star K$ for any~$K\subset [1,4]$ and
so $w_{\{1,3\};J}\not=p_J(w_K)$ for
any~$K\subset [1,4]$. Indeed,
by Corollary~\ref{cor:A J*K}, $J\star K$ is an interval for any connected~$K$, $J\star K=\{1\}$
if and only if~$K=[1,2]$ and $J\star K=\{3\}$
if and only if~$K=[2,3]$. Yet $[1,2]$ and~$[2,3]$
are not orthogonal.
\end{remark}

\subsection{Light homomorphisms of Hecke monoids are parabolic}\label{subs:light->parab}
We can now prove the following
\begin{theorem}\label{thm:light->parab}
Any light homomorphism of Hecke monoids is parabolic. 
\end{theorem}
\begin{proof}
In view of Proposition~\ref{prop:classify light},
Theorem~\ref{thm:pJwK is w0KstarJw0J} and
Lemma~\ref{lem:[phi]comp},
it remains to prove that~$\phi$ is parabolic when~$\phi$ is either a folding along some surjective map
or is tautological. Note that 
since a restriction of a light $\phi\in\Hom_{\mathscr H}(M',M)$ 
to any parabolic submonoid of~$W(M')$ is again light,
it suffices to consider the case when~$M'$ is of finite type and irreducible.

The only foldings in finite types with irreducible~$M'$ are
are $\mathbf f_{\varpi_{(n,n+1)}}:(W(D_{n+1}),\star)\to (W(A_n),\star)$
and $\mathbf f_{\varpi_{(1,3,4)}}:(W(D_4),\star)\to (W(A_2),\star)$ where $\varpi_{(n,n+1)}:[1,n+1]\to[1,n]$ and~$\varpi_{(1,3,4)}:
[1,4]\to [1,2]$ are defined as in~\S\ref{subs:light Artin->parabolic}. Since~$\mathbf f_{\varpi_{(n,n+1)}}=\overline{(\mathbf F_{\varpi_{(n,n+1)}})}_\star$, $\mathbf f_{\varpi_{(1,3,4)}}=\overline{(\mathbf F_{\varpi_{(1,3,4)}})}_\star$ and
$\mathbf F_{\varpi_{(n,n+1)}}$ (respectively,
$\mathbf F_{\varpi_{(1,3,4)}}$) is parabolic by
Proposition~\ref{prop:Dn+1 An parab}
(respectively, \ref{prop:D4 A2}), $\mathbf f_{\varpi_{(n,n+1)}}$ and~$\mathbf f_{\varpi_{(1,3,4)}}$ are parabolic by Lemma~\ref{lem:hecke parab induces parab}.

We now turn our attention to tautological homomorphisms. Note that if~$M=I_2(m)$ then
for any~$m'>m$ we have a tautological
$\phi\in\Hom_{\mathscr H}(I_2(m'),I_2(m))$.
Since
$$
\phi(\brd{s'_is'_j}{k})=\brd{s_i\star s_j\star}{k},
\qquad k\le m',
$$
and non-identity parabolic elements $w_{J;\{1,2\}}$
in~$I_2(m')$
correspond to~$k\in \{m'-1,m'\}\ge m$,
it follows that
$\phi(w_{J})=w_\circ^{\{1,2\}}$ and so~$\phi$ is parabolic.

Suppose now that~$|I|>2$ and that~$M'$ is irreducible. Note that if~$\phi\in\Hom_{\mathscr H}(M',M)$ is tautological
then~$\Gamma(M)$ is obtained from~$\Gamma(M')$
by either decreasing some labels or removing edges.

Suppose that the underlying graphs of~$\Gamma(M')$ and~$\Gamma(M)$ are isomorphic.
Then the only possibilities are:
\begin{enumerate}[label={$\arabic*^\circ$.},
ref={$\arabic*^\circ$}]
\item\label{taut-case-BnAn}
$M'=B_n$, $M=A_n$ and so~$\phi$
the composition of the unfolding $(W(B_n),\star)\to
(W(A_{2n-1}),\star)$ induced by~\eqref{eq:unfold Bn A2n-1} with~$p_{[1,n]}:(W(A_{2n-1}),\star)\to (W_{[1,n]}(A_{2n-1}),\star)\cong (W(A_n),\star)$
and hence is parabolic
by Proposition~\ref{prop:parabolic<->w0J}
and Theorem~\ref{thm:pJwK is w0KstarJw0J};

\item\label{taut-case-F4A4}
$M'=F_4$, $M=A_4$ and so~$\phi$ is the composition of the unfolding $(W(F_4),\star)\to (W(E_6),\star)$ induced by~\eqref{eq:unfold F4 E6}
with the parabolic projection $p_{\{1,2,3,6\}}:(W(E_6),\star)\to (W_{\{1,2,3,6\}}(E_6),\star)\cong (W(A_4),\star)$
hence the assertion holds in this case
by Proposition~\ref{prop:parabolic<->w0J}
and Theorem~\ref{thm:pJwK is w0KstarJw0J};

\item\label{taut-case-HnBn}
$M'=H_n$, $M=B_n$, $n\in\{3,4\}$.
We claim
that~$\phi(w_{J})=w_\circ^{[1,n]}$ for all~$J\subsetneq I$. By Lemma~\ref{lem:len prop wJ K},
it suffices to prove
the claim for~$J$ with~$|J|=n-1$.

If~$n=3$, let~$c=s_1s_3s_2$. Then~$w_\circ^{[1,3]}=c^{\times 5}$ in~$W(H_3)$ and so
\begin{align*}
&w_{\{1,2\}}=(s_1 s_2 s_3 s_2)\times c^{\times 3},
\quad w_{\{1,3\}}=s_2 \times c^{\times 4},\\
&w_{\{2,3\}}=s_1s_2s_3s_2s_1s_3s_2s_3s_2s_1.
\end{align*}
Since~$c^{\star 3}=w_\circ^{[1,3]}$ in~$(W(B_3),\star)$, the claim is obvious
for~$J\in\{\{1,2\},\{1,3\}\}$
while
\begin{align*}
\phi(w_{\{2,3\}})&=s'_1\star s'_2\star s'_3\star s'_2\star s'_1\star s'_3\star s'_2\star s'_3\star s'_2\star s'_1
\\
&=s'_1\star s'_2\star s'_3\star s'_2\star s'_1\star s'_2\star s'_3\star s'_2\star s'_3\star s'_1\\
&=s'_1\star s'_2\star s'_1\star s'_3\star s'_2\star s'_1\star s'_3\star s'_2\star s'_3\star s'_1
=w_\circ^{[1,3]}\star s'_1=w_\circ^{[1,3]}.
\end{align*}
Similarly, for~$n=4$ we have
\begin{align*}
&w_{[1,3]}=(w_\circ^{[1,3]}\cx14^3)\times \cx14^{\times 12}=s_4s_3s_4s_2s_3s_4\times \cx14^{\times 12},\\
&w_{\{1,2,4\}}=
s_3s_2s_1\times \cxr14^{\times 12}\times s_4s_3s_2s_4s_3,\\
&w_{\{1,3,4\}}=
s_2s_1s_3s_2\times \cxr14^{\times10}
\times \cxr24\times\cxr14\times\cxr24\\
&w_{\{2,3,4\}}=\cx13\times (\cxr14\times
\cxr24\times \cxr14)^{\times 3}\times \cxr24
\times\cxr34\times\cxr14.
\end{align*}
Since~$\cxr14^{\star 4}=w_\circ^{[1,4]}$ in~$(W(B_4),\star)$, the claim is immediate
for~$J\not=\{2,3,4\}$ with~$|J|=3$. Since
$\cxr14\star\cxr24\star\cxr14^{\star 2}\star \cxr24
=w_\circ^{[1,4]}$ in~$(W(B_3),\star)$,
the claim follows for~$J=\{2,3,4\}$ as well.

\item\label{taut-case-HnAn}
$M'=H_n$, $M=A_n$, $n\in\{3,4\}$.
Then~$\phi$ is the composition
of a homomorphism from~\ref{taut-case-HnBn} with
the respective homomorphism from~\ref{taut-case-BnAn}
and hence is parabolic.
\end{enumerate}

It remains to consider the case when~$\Gamma(M)$
is obtained from~$\Gamma(M')$ by removing some
edges. Let~$I_1,\dots,I_r$ be vertex sets 
of connected components of~$\Gamma(M)$. 
The following is immediate.
\begin{lemma}\label{lem:taut pJ commute}
Let~$M,M'\in\Cox I$ and let~$\phi\in\Hom_{\mathscr H}(M',M)$
be tautological. Let~$J\subset I$ and let 
$p_J:(W(M),\star)\to (W_J(M),\star)$,
$p'_J:(W(M'),\star)\to (W_J(M'),\star)$ be
respective parabolic projections. 
Then $p_J\circ \phi=\phi\circ p'_J$.
\end{lemma}
Let~$J\subset I$. Then by Lemmata~\ref{lem:parab prod} and~\ref{lem:taut pJ commute} and Theorem~\ref{thm:pJwK is w0KstarJw0J} 
\begin{align*}\phi(w_J)=
\prodst_{1\le i\le r} p_{I_i}(\phi(w_J))=
\prodst_{1\le i\le r}\phi(p'_{I_i}(w_J))
=\prodst_{1\le i\le r}\phi(w_{J\star I_i;I_i})
\end{align*}
where~$J\star I_i$, $1\le i\le r$ is taken in~$W(M')$. Since the restriction of~$\phi$ to~$W_{I_i}(M')$, $1\le i\le k$ is light and hence parabolic, as $\Gamma_{I_i}(M')$ and~$\Gamma_{I_i}(M)$ are connected, it follows that $\phi(w_{J\star I_i;I_i})=
w_{J'_i;I_i}$, $J'_i\subset I_i$ and so~$\phi(w_J)=
w_{J'_1\cup\cdots\cup J'_r}$ by Lemma~\ref{lem:prod orth parab}.
\end{proof}
\begin{remark}
Let~$\varpi_{(n,n+1)}:[1,n+1]\to[1,n]$ be
as in~\S\ref{subs:light Artin->parabolic}. It follows from~\eqref{eq:Dn+1 An parab I}
\begin{equation}\label{eq:explicit Dn+1 An parab}
\phi(w_J)=\begin{cases}
w_{J\star_I [1,n];[1,n]},&\{n,n+1\}\subset J,\\
w_{\emptyset;[1,n]},&\text{otherwise}.
\end{cases}
\end{equation}
Note that if~$\{n,n+1\}\not\subset J$,
$J\star_I [1,n]$ does not have to be empty (see Proposition~\ref{prop:1,n*K D}).
\end{remark}

\subsection{Remarks on multiparabolic elements}
\label{subs:finale}
We end this chapter with two conjectures
about multiparabolic elements which were verified in numerous examples.
\begin{conjecture}\label{conj:pJ multiparab}
Let~$M$ be a Coxeter matrix over~$I$. Then
for any~$J\subset I$ and
$K_1,\dots,K_r\subset K\in\mathscr F(M)$
we have $p_J(w_{K_1,\dots,K_r;K})=w_{K'_1,\dots,K'_r;K\cap J}$
for some~$K'_1,\dots,K'_r\subset K\cap J$. Thus,
parabolic projections of Hecke monoids map
multiparabolic elements to multiparabolic elements.
\end{conjecture}
\begin{conjecture}\label{conj:multiparab submon}
Let~$\tilde \mP(M)$ be the set of multiparabolic elements
in~$W(M)$. Then $\tilde \mP(M)\star \tilde \mP(M)$
is a submonoid of~$(W(M),\star)$.
\end{conjecture}
As before, since the standard unfoldings preserve multiparabolic elements and are injective, we only needs to consider
$W(A_n)$, $n\ge 1$, $W(D_{n+1})$, $n\ge 4$ and~$E_n$, $n\in\{6,7,8\}$.
The second conjecture was verified for~$W(A_n)$
with~$n\le 8$, $W(D_{n+1})$, $n\le 6$, $W(E_6)$ and~$W(E_7)$ (in the last case~$\tilde \mP(M)=\tilde \mP(M)\star \tilde \mP(M)$ and hence is a submonoid).
\begin{example}\label{ex:multipar submon}
By Theorem~\ref{thm:count multipar}, there are $68$ multiparabolic elements in~$W(A_4)$. The only $\star$-products of
multiparabolic elements which are not multiparabolic
are $w_{[2,4], \{1\}}\star w_{\{1\}, [2,4]}=s_1s_2s_3s_2s_4=
s_2s_3s_4s_3s_2w_\circ$ and $w_{[1,3],\{4\}}\star w_{\{4\},[1,3]}=s_2s_4s_3s_2s_1=s_1s_2s_3s_2s_1w_\circ$. The
set~$\tilde \mP(W)$ together with these two
elements is a submonoid of~$(W(A_4),\star)$.
For~$W(A_5)$ (respectively, $W(A_6)$, $W(A_7)$, $W(A_8)$)
we need $13$ (respectively, $53$, $176$, $510$)
additional pairwise products of multiparabolic
elements to obtain a submonoid of~$(W(A_n),\star)$.
\end{example}

\section{Standard homomorphisms}\label{sec:SQF AB}
We begin by classifying all disjoint fully supported standard homomorphisms $\Br^+(\wh M)
\to \Br^+(M)$ where~$\wh M$ is irreducible and of finite type
and~$M$ is of type~$A_n$ or~$B_n$. In view of Lemma~\ref{lem:factor homs}, we will only consider
optimal homomorphisms. 

\subsection{Two families of homomorphisms in type~\texorpdfstring{$B$}{B}} 
We will often use the following Lemma.
\begin{lemma}\label{lem:I2m iff cnd}
Let~$m>1$ and~$M\in\Cox I$. Let~$X_1,X_2\in \Br^+(M)$ be ${}^{op}$-invariant. Then the 
assignments $\wh T_i\mapsto X_i$, $i\in \{1,2\}$ define a homomorphism
$\Br^+(I_2(2m))\to \Br^+(M)$ if and only if $(X_1X_2)^m$ is ${}^{op}$-invariant.
\end{lemma}
\begin{proof}
These assignments define a homomorphism~$\Br^+(I_2(2m))\to\Br^+(M)$ if and only if 
$$
(X_1X_2)^m=\brd{X_1X_2}{2m}=\brd{X_2X_1}{2m}=(X_2X_1)^m.
$$
Since~$X_1$ and~$X_2$ are ${}^{op}$-invariant, this happens if and only if $((X_1X_2)^m)^{op}=(X_2X_1)^m=(X_1X_2)^m$.
\end{proof}
We begin by constructing infinite families
of disjoint Coxeter type parabolic and of
standard homomorphisms $\Br^+(I_2(2m))\to \Br^+(B_n)$
for~$2\le m\le n$.
\begin{proposition}  \label{prop:admissible hom from BrI22m to BrBn}
Let $2\le m\le n$.
\begin{enmalph}
\item\label{prop:admissible hom from BrI22m to BrBn.a} 
The assignments 
$$
\wh T_1\mapsto T_{w_\circ^{[1,m-1]_2}},\quad
\wh T_2\mapsto T_{w_\circ^{[1,m-2]_2}}\Cx mn\Cxr m{(n-1)}
$$
define a disjoint Coxeter type strict parabolic~$\Phi\in\Hom_{\mathscr A}(I_2(2m),B_n)$.
\item\label{prop:admissible hom from BrI22m to BrBn.b}
The assignments
$$
\wh T_1\mapsto T_{w_\circ^{[1,m-1]_2}},\quad
\wh T_2\mapsto T_{w_\circ^{[1,m-2]_2\cup [m,n]}}
$$
define a disjoint standard $\wh\Phi\in\Hom_{\mathscr A}(I_2(2m),B_n)$, which is strict parabolic if and only if~$m=n$.
\end{enmalph}
\end{proposition}
\begin{proof}
Abbreviate $T_{(i,n+1)}:=\Cxr in\Cxr i{(n-1)}$, $1\le 
i\le n$. Note that since $w_\circ^{[i,n]}
=\cx in\times \cxr i{(n-1)}\times w_\circ^{[i+1,n]}=
w_\circ^{[i+1,n]}\times \cx in\times \cxr i{(n-1)}$,
for all~$1\le i\le n-1$,
\begin{equation}\label{eq:B_n telescope 0}
T_{w_\circ^{[i,n]}}=T_{(i,n+1)}T_{w_\circ^{[i+1,n]}}= T_{w_\circ^{[i+1,n]}}T_{(i,n+1)}.
\end{equation}
We need the following
\begin{lemma}\label{lem:B2 I22m BrBn main id}
For all $2\le m\le n$ we have in~$\Br(B_n)$
\begin{equation}
( T_{w_\circ^{[1,m-1]_2}} T_{w_\circ^{[1,m-2]_2}} T_{(m,n+1)})^m  
=T_{w_\circ^{[m+1,n]}}^{-1} T_{w_\circ^{[1,n]}}=T_{w_\circ^{[1,n]}}
T_{w_\circ^{[m+1,n]}}^{-1} .\label{eq:prep I2m Bn}
\end{equation}
\end{lemma}
\begin{proof}
Since $m\ell(T_{w_\circ^{[1,m-1]_2}} T_{w_\circ^{[1,m-2]_2}} T_{(m,n+1)})=
(2n-m)m=n^2-(n-m)^2=\ell(T_{w_\circ^{[1,n]}})-
\ell(T_{w_\circ^{[m+1,n]}}),
$
it suffices to
prove that
\begin{equation}\label{eq:weyl grp id}
(w_\circ^{[1,m-1]_2}w_\circ^{[1,m-2]_2}\cx mn\cxr m{(n-1)})^m w_\circ^{[m+1,n]}=w_\circ^{[1,n]}
\end{equation}
in~$W(B_n)$. For that, since the reflection representation of~$W(B_n)$
on an~$n$-dimensional Euclidean space is faithful (see e.g.~\cite{Bou}*{Ch.~IV, \S4.4, Corollaire~2}),
it suffices to show that the left hand
side of~\eqref{eq:weyl grp id} acts as~$-1$ in the reflection representation since~$w_\circ^{[1,n]}$ acts this way.
Indeed, recall (see e.g.~\cite{Bou}*{Ch.~VI, \S4.5}) that~$W(B_n)$
is isomorphic to the semi-direct product of~$S_n$ with~$(\mathbb Z_2)^n$.
Let
$\epsilon_1,\dots,\epsilon_n$ be the standard basis of~$\mathbb R^n$.
Then~$S_n$ acts by permutations of the~$\epsilon_i$ and the~$i$th copy
of~$\mathbb Z_2$ acts as $\epsilon_j\mapsto (-1)^{\delta_{i,j}}\epsilon_j$,
$1\le j\le n$.
Then
\begin{equation}\label{eq:B refl 1}
w_\circ^{[m+1,n]}(\epsilon_i)=\begin{cases}
    \epsilon_i,&1\le i\le m\\
    -\epsilon_i,&m+1\le i\le n.
    \end{cases}
\end{equation}
while $\cx mn\cxr m{n-1}(\epsilon_i)=(-1)^{\delta_{i,m}}\epsilon_i$,
$1\le i\le n$. Since~$w_\circ^{[1,m-1]_2}w_\circ^{[1,m-2]_2}$
is a Coxeter element in~$W_{[1,m-1]}(B_n)\cong W(A_{m-1})$,
it identifies with a cycle of length~$m=h(A_{m-1})$ permuting all the
$\epsilon_1,\dots,\epsilon_{m}$.
It follows that
\begin{equation}\label{eq:B refl 2}
(w_\circ^{[1,m-1]_2}w_\circ^{[1,m-2]_2}\cx mn\cxr m{(n-1)})^m(\epsilon_i)=
\begin{cases}-\epsilon_i,&1\le i\le m\\
\epsilon_i,&m+1\le i\le n.
\end{cases}
\end{equation}
Together~\eqref{eq:B refl 1} and~\eqref{eq:B refl 2} imply that
the left hand side of~\eqref{eq:weyl grp id} acts as~$-1$.
\end{proof}
By Proposition~\ref{prop:fund elts BrSa}, $T_{w_\circ^{[1,n]}}$ is central 
in~$\Br(B_n)$ and $T_{w_\circ^{[1,n]}}$, $T_{w_\circ^{[m+1,n]}}$ are ${}^{op}$-invariant.
It follows that the left hand side of~\eqref{eq:prep I2m Bn}
is ${}^{op}$-invariant. Since~$T_{(m,n+1)}$ as  
well as the $T_{w_\circ^{[1,m-i]_2}}$, $i\in\{1,2\}$ are ${}^{op}$-invariant 
and $T_{(m,n+1)}$ commutes with~$T_{w_\circ^{[1,m-2]_2}}$, it 
follows that~$T_{w_\circ^{[1,m-2]_2}}T_{(m,n+1)}$ is ${}^{op}$-invariant. Since~$T_{w_\circ^{[1,n-1]_2}}$ is also ${}^{op}$-invariant,
it follows from Lemma~\ref{lem:I2m iff cnd} that
the assignments 
in part~\ref{prop:admissible hom from BrI22m to BrBn.a} define a homomorphism $\Phi:\Br^+(I_2(2m))\to\Br^+(B_n)$ which is 
of Coxeter type by Proposition~\partref{prop:elem prop Coxeter Hecke.a}. Moreover, by~\eqref{eq:prep I2m Bn},
$\Phi(\wh T_{w_\circ^{\{1,2\}}})=T_{w_{[m+1,n]}}$ 
and
$\Phi(\wh T_{w_{\{1\}}})=T_{w_\circ^{[1,m-1]_2}}^{-1}
T_{w_\circ^{[m+1,n]}}^{-1}T_{w_\circ^{[1,n]}}
=T_{w_{[1,m-1]_2\cup[m+1,n]}}$.
Finally, by~\eqref{eq:B_n telescope 0}, we have in~$\Br(B_n)$
\begin{align*}
\Phi(\wh T_{w_{\{2\}}})=
T_{w_\circ^{[1,m-2]_2}}^{-1}T_{(m,n+1)}^{-1}T_{w_\circ^{[m+1,n]}}^{-1}T_{w_\circ^{[1,n]}}
=T_{w_\circ^{[1,m-2]_2}}T_{w_\circ^{[m,n]}}^{-1}T_{w_\circ^{[1,n]}}
=T_{w_{[1,m-2]_2\cup[m,n]}}.
\end{align*}
Therefore, $\Phi$ is strict parabolic and the proof of part~\ref{prop:admissible hom from BrI22m to BrBn.a} is complete.

To prove part~\ref{prop:admissible hom from BrI22m to BrBn.b}, let~$z_1=1$ and~$z_2=T_{w_\circ^{[m+1,n]}}$.
By~\eqref{eq:B_n telescope 0}, $T_{w_\circ^{[1,m-2]_2\cup[m,n]}}=\Phi(\wh T_2)z_2$. 
We claim that~$\mathbf z=(z_1,z_2)$ is a decoration
of~$\Phi$ and so~$\wh\Phi=\Phi_{\mathbf z}
\in\Hom_{\mathscr A}(I_2(2m),B_n)$. Since~$z_1=1$,
by Lemma~\ref{lem:cent decor}
it suffices to prove that~$z_2$ commutes with~$\Phi(\wh T_i)$, $i\in\{1,2\}$ which is obvious for~$i=1$ and follows from~\eqref{eq:B_n telescope 0} for~$i=2$.

The second assertion in part~\ref{prop:admissible hom from BrI22m to BrBn.b} follows since~$m\ge 2$ and so
$\wh \Phi(\brd{\wh T_1\wh T_2}{2m})=T_{w_\circ^{[1,n]}}$
if and only if~$m=n$.
\end{proof}

\subsection{Key result}\label{subs:Key res}
Fix~$n>1$. We abbreviate~\plink{Brn}$\Br^+_{n+1}:=\Br^+(A_n)$, $\Br_{n+1} :=\Br(A_n)$ and~$\pi_n:=\pi_{A_n}$. Let~$\sigma$
be the diagram automorphism of~$\Br^+_{n+1}$, the corresponding
permutation of~$I=[1,n]$ being $\sigma(i)=n+1-i$, $i\in [1,n]$.
Note that if~$J=[a,b]\subset I$ satisfies~$\sigma(J)=J$ then
$b=n+1-a$ and so $n-|J|=2(a-1)$ is even.

Our present goal is to prove the following theorem which generalizes
the classical result from~\cite{BrSa}*{\S5.8} (cf. Theorem~\partref{thm:adm finite class.odd}%
\ref{thm:adm finite class.even} and
Proposition~\ref{prop:Coxeter splitting}).
\begin{theorem}\label{thm:main thm adm}
Let~$K\subsetneq I$ be an interval with~$|K|>1$
and let~$I'(K)\sqcup I''(K)$ be the unique partition of~$I\setminus K$ into self-orthogonal subsets such that~$I'(K)$
and~$K$ are orthogonal. Then  the assignments
$$
\wh T_1\mapsto T_{w_\circ^{I'(K)\cup K}},\qquad \wh T_2\mapsto T_{w_\circ^{I''(K)}}
$$
define an optimal (disjoint standard) homomorphism $\Br^+(I_2(2m(K)))\to \Br^+_{n+1}$
where
$$
m(K)=\begin{cases}
\frac12(n-|K|)+1,& \sigma(K)=K,\\
n-|K|+2,&\text{otherwise.}
\end{cases}
$$

Conversely, suppose that~$\Phi:\Br^+(I_2(N))\to \Br^+_{n+1}$ is an optimal disjoint standard homomorphism such that~$[\Phi](i)\not=\emptyset$,
$i\in \wh I=\{1,2\}$. Then
\begin{enmalph}
\item \label{thm:main them adm.a}
either both $[\Phi](1)$ and~$[\Phi](2)$ are self-orthogonal, or
exactly one of them contains a unique connected component
of rank~$>1$.

\item  \label{thm:main them adm.b} Suppose that
$[\Phi](1)$ and~$[\Phi](2)$ are self-orthogonal.
Then~$n=N-1$ and~$\Phi$ is the
homomorphism from Theorem~\partref{thm:adm finite class.odd} or~\ref{thm:adm finite class.even}.

\item\label{thm:main them adm.c} Suppose that precisely one
of the~$[\Phi](i)$, $i\in\{1,2\}$ contains a unique connected component~$K$
with~$|K|>1$. Then~$N=2m(K)$.
\end{enmalph}
\end{theorem}

\subsection{Transpositions in braid monoids}
Given $i\le j\in [1,n]$, denote \plink{T(i,j)}$T_{(i,j+1)}$ the unique square free element of~$\Br^+_{n+1}$
corresponding to the transposition~$(i,j+1)$ in~$S_{n+1}$ which identifies with~$W(A_n)$.
We use the convention that $T_{(i,j)}=1$ if~$i\ge j$.

\begin{proposition}\label{prop:elem prop transp}
Let $i\le j\in [1,n]$.
\begin{enmalph}
\item \label{prop:elem prop transp.a}$\ell(T_{(i,j+1)})=2(j-i)+1$.
    \item \label{prop:elem prop transp.b}
    $T_{(i,j+1)}=T_i T_{(i+1,j+1)}T_i=T_j T_{(i,j)} T_j$.
    \item \label{prop:elem prop transp.c}
    $T_{(i,j+1)}=\Cx ij\Cxr i{(j-1)}=\Cx i{(j-1)}\Cxr ij=
    \Cxr ij\Cx{(i+1)}j=\Cxr{(i+1)}j\Cx ij$.
    In particular, $T_{(i,j+1)}$ is ${}^{op}$-invariant.
    \item\label{prop:elem prop transp.e} $T_{k}T_{(i,j+1)}=T_{(i,j+1)}T_k$ for all $k\in [1,n]\setminus \{i-1,i,j,j+1\}$.
    \item\label{prop:elem prop transp.f}
    Let $k\le l\in[1,n]$. Then $T_{(k,l+1)}$ commutes with $T_{(i,j+1)}$
    provided that either $i<k$, $l<j$, or $l<i-1$ or~$j+1<k$.
    \item \label{prop:elem prop transp.g}
    $T_{w_\circ^{[i,j]}}=\prod_{0\le k\le \frac12(j-i)} T_{(i+k,j+1-k)}$.
\end{enmalph}
\end{proposition}
\begin{proof}
Since~$\ell((i,j+1))=2(j-i)+1$,
part~\ref{prop:elem prop transp.a} is obvious. Since
$(i,i+1)(i+1,j+1)(i,i+1)=(i,j+1)=(j,j+1)(i,j)(j,j+1)$
in~$S_{n+1}$ and
$\ell((i,j+1))=2+\ell((i+1,j+1))=2+\ell((i,j))$, part~\ref{prop:elem prop transp.b} follows. Part~\ref{prop:elem prop transp.c} follows from~\ref{prop:elem prop transp.b} by a straightforward induction.
The assertion of~\ref{prop:elem prop transp.e} is obvious
for $k\in [1,i-2]\cup[j+2,n]$. To prove it for $k\in[i+1,j-1]$, we
need the following
\begin{lemma}\label{lem:comm cox}
Let $i< j\in [1,n]$ and $k\in[i+1,j]$. Then $T_k\Cx ij=\Cx ij T_{k-1}$ and $T_{k-1}\Cxr ij=\Cxr ij T_k$.
\end{lemma}
\begin{proof}
We have
\begin{align*}
T_k \Cx ij&=\Cx i{(k-2)}T_k T_{k-1}T_k \Cx{(k+1)}j%
=\Cx i{(k-2)}T_{k-1}T_k T_{k-1}\Cx{(k+1)}j%
=\Cx ij T_{k-1}.
\end{align*}
The second identity is obtained from the first by applying~${}^{op}$.
\end{proof}
Thus, given $k\in [i+1,j-1]$ we have
$$
T_k T_{(i,j+1)}=T_k \Cx ij\Cxr i{(j-1)}=\Cx ij T_{k-1}\Cxr i{(j-1)}
=\Cx ij\Cxr i{(j-1)}T_k=T_{(i,j+1)}T_k.
$$
Part~\ref{prop:elem prop transp.f} is immediate from part~\ref{prop:elem prop transp.e}. Finally, since $w_\circ^{[i,j]}
=\prod_{0\le k\le \frac12(j-i)} (i+k,j+1-k)$ and
$\ell(w_\circ^{[i,j]})=\binom{j-i+2}2=\sum_{0\le k\le \frac12(j-i)}
(2(j-i-2k)+1)=\sum_{0\le k\le \frac12(j-i)}\ell((i+k,j+1-k))$,
part~\ref{prop:elem prop transp.g} follows.
\end{proof}

For $J=\{j_1 <j_2< \cdots <j_m\} \subset [1,n+1]$, set $$
\plink{TJ}T_J = \tilde \tau_1(J)\tilde \tau_0(J),
\qquad \tilde\tau_k(J)=\prod_{1\le r\le m\,:\, \bar{r}=k} T_{(j_r,j_{r+1})}
$$
By Proposition~\partref{prop:elem prop transp.e},
$\tilde\tau_k(J)$, $k\in\{0,1\}$ are well-defined. Note that
\begin{equation}\label{eq:len TJ}
\ell(T_J)=2\sum_{1\le r\le m-1} (j_{r+1}-j_r)-m+1=2(\max J-\min J)-|J|+1.
\end{equation}
We also set for $k\in\{0,1\}$\plink{tauk(J)}
\begin{equation}\label{eq:tau_i(J) defn}
\tau_k(J) = \prod_{1\le r\le m\,:\, \bar{r}=k} T_{w_\circ^{[j_r,j_{r+1}-1]}}.
\end{equation}
In particular, it follows from Proposition~\partref{prop:elem prop transp.g} that
\begin{equation}\label{eq:tilde tau tau}
\tau_k(J)=\tilde\tau_k(J) X_k(J),\qquad X_k(J)=\prod_{1\le r\le m\,:\,\bar r=k} T_{w_\circ^{[j_r+1,j_{r+1}-2]}},\qquad k\in\{0,1\}.
\end{equation}
Clearly, $X_1(J)$ and~$X_0(J)$ commute and also
commute with the~$\tilde\tau_k(J)$, $k\in\{0,1\}$.
The following Lemma is obvious
\begin{lemma}\label{lem:all disj TJ}
Up to renumbering of the generators, every fully supported disjoint standard homomorphism $\Psi:\Br^+(I_2(N))\to \Br^+_{n+1}$
satisfies $\Psi(\wh T_r)=\tau_{\bar r}(J)$, $r\in\{1,2\}$
for some~$\{1,n+1\}\subset J\subset [1,n+1]$;
\end{lemma}
The following Lemma is crucial for proving Theorem~\ref{thm:main thm adm}.
\begin{lemma}\label{lem:TJ hom}
The following are equivalent for~$J\subset[1,n+1]$
and~$m\ge 1$;
\begin{enmalph}
\item\label{lem:TJ hom.a}
$T_J^m$ is ${}^{op}$-invariant; 
\item\label{lem:TJ hom.b}
The assignments~$\wh T_k\mapsto \tilde\tau_{\bar k}(J)$, $k\in\{1,2\}$ define a homomorphism 
$\Br^+(I_2(2m))\to
\Br^+_{n+1}$;
\item\label{lem:TJ hom.c} The assignments 
$\wh T_k\mapsto \tau_{\bar k}(J)$, $k\in\{1,2\}$
define a homomorphism $\Br^+(I_2(2m))\to
\Br^+_{n+1}$.
\end{enmalph}
\end{lemma}
\begin{proof}
Assertions~\ref{lem:TJ hom.a} and~\ref{lem:TJ hom.b} 
are equivalent by Lemma~\ref{lem:I2m iff cnd}.
Suppose that the assignments in part~\ref{lem:TJ hom.b}
define~$\tilde\Phi\in\Hom_{\mathscr A}(I_2(2m),A_n)$. 
Then by Lemma~\ref{lem:cent decor}, $\mathbf z=(X_1(J),X_0(J))$ is a decoration of~$\tilde\Phi$
and then~$\tilde\Phi_{\mathbf z}(\wh T_r)=\tau_{\bar r}(J)$,
$r\in\{1,2\}$ by~\eqref{eq:tilde tau tau}. Finally, suppose that
the assignments in~\ref{lem:TJ hom.c} define~$\Phi\in\Hom_{\mathscr A}(I_2(2m),A_n)$. Extend it
to a homomorphism $\Br^+(I_2(2m))\to \Br_{n+1}$. Then
$\mathbf z^{-1}=(X_1(J)^{-1},X_0(J)^{-1})$ is a decoration of~$\Phi$ by Lemma~\ref{lem:cent decor}, and
$\Phi_{\mathbf z^{-1}}(\wh T_r)=\tilde\tau_{\bar r}(J)$,
$r\in\{1,2\}$.
\end{proof}

\subsection{Symmetries and conjugation}
By Lemma~\ref{lem:TJ hom}, to prove Theorem~\ref{thm:main thm adm}
we need to find a necessary and sufficient condition
for~$T_J^m$, $\{1,n+1\}\subset J\subset [1,n+1]$, $m\in\mathbb Z_{>0}$, to be~${}^{op}$-invariant.

Given $J\subset [1,n+1]$, let $\tilde\sigma(J)=\{n+2-j\,:\,j\in J\}$.
\begin{lemma}\label{lem:diag aut TJ}
Let $J\subset [1,n+1]$ and let $\sigma$ be the diagram automorphism
of~$\Br^+_{n+1}$. Then
$$
\sigma(T_J)=\begin{cases}
T_{\tilde\sigma(J)},& \text{$|J|$ is even},\\
T_{\tilde\sigma(J)}{}^{op},& \text{$|J|$ is odd}.
\end{cases}
$$
In particular, $T_J^m$ is ${}^{op}$-invariant if and only if
$T_{\tilde\sigma(J)}^m$ is ${}^{op}$-invariant.
\end{lemma}
\begin{proof}
Write $J=\{j_1,\dots,j_k\}$ where $j_1<j_2<\cdots<j_{k-1}<j_k$. Then $\tilde\sigma(J)=\{n+2-j_k,n+2-j_{k-1},\cdots,n+2-j_2,n+2-j_1\}$ and
$$
T_J=\prod_{1\le l\le k/2} T_{(j_{2l-1},j_{2l})} \prod_{1\le l\le (k-1)/2} T_{(j_{2l},j_{2l+1})}.
$$
Suppose first that~$k=2m$. Then
$$
T_J=\prod_{1\le l\le m} T_{(j_{2l-1},j_{2l})} \prod_{1\le l\le m-1} T_{(j_{2l},j_{2l+1})}
$$
Since $\sigma(T_{(a,b)})=T_{(n+2-b,n+2-a)}$
$$
\sigma(T_J)=\prod_{1\le l\le m} T_{(n+2-j_{2l},n+2-j_{2l-1})} \prod_{1\le l\le 2m-1} T_{(n+2-j_{2l+1},n+2-j_{2l})}=T_{\tilde\sigma(J)}.
$$
If~$k=2m+1$ then
\begin{align*}
\sigma(T_J)&=\prod_{1\le l\le m} \sigma(T_{(j_{2l-1},j_{2l})}) \prod_{1\le l\le m} \sigma(T_{(j_{2l},j_{2l+1})})\\&=
\prod_{1\le l\le m} T_{(n+2-j_{2l},n+2-j_{2l-1},j_{2l})} \prod_{1\le l\le m} T_{(n+2-j_{2l+1},n+2-j_{2l})}=T_{\tilde\sigma(J)}{}^{op}.\qedhere
\end{align*}
\end{proof}

Our next goal is to show that all the~$T_J$ with~$J$ of the same
cardinality are conjugate in~$\Br_{n+1}$ (eventually we will also see that the converse is true).
\begin{proposition}\label{prop:conj}
Let $J\subset [1,n+1]$ and let
$j\in J$ with $\min J<j<\max J$ and $j-1\notin J$. Then in~$\Br_{n+1}$
$$
T_{(J\setminus\{j\})\cup\{j-1\}}=T_{j-1}^{\epsilon} T_J T_{j-1}^{-\epsilon},
$$
where $\epsilon=(-1)^{|J\cap [1,j]|+1}$.
\end{proposition}
\begin{proof}
We need the following
\begin{lemma}\label{lem:conj}
Let $i,j,k\in [1,n+1]$ with $i<j-1$ and $j<k$. Then
in~$\Br_n$ we have
$$
T_{(j-1,k)}T_{(i,j-1)}=T_{j-1} T_{(j,k)}T_{(i,j)}T_{j-1}^{-1}.
$$
and
$$
T_{(i,j-1)}T_{(j-1,k)}=T_{j-1}^{-1} T_{(i,j)}T_{(j,k)}T_{j-1}.
$$
\end{lemma}
\begin{proof}
Using Proposition~\partref{prop:elem prop transp.b}
we obtain
$$
T_{(j-1,k)}T_{(i,j-1)}
=T_{j-1} T_{(j,k)}T_{j-1} T_{(i,j-1)}
=T_{j-1} T_{(j,k)} T_{(i,j)}T_{j-1}^{-1}.
$$
The second identity follows from the first by applying~${}^{op}$.
\end{proof}
Write $J=\{j_1,\dots,j_m\}$ where~$j_1<\cdots<j_m$ and
$j_k=j$. In particular, $|J\cap [1,j]|=k$.
Since~$\min J<j<\max J$, $2\le k\le m-1$ and so $j_{k-1}\le j-2$
as $j-1\notin J$.
Let~$J'=(J\setminus \{j\})\cup \{j-1\}$.
Suppose first that~$k$ is odd and so~$\epsilon=(-1)^{k+1}=1$. Then
$T_J=X T_{(j,j_{k+1})} T_{(j_{k-1},j)} X'$ and
$T_{J'}=X T_{(j-1,j_{k+1})} T_{(j_{k-1},j-1)} X'$ where
$$
X=\prod_{\substack{t\in[1,m]\setminus\{k\}\\\bar t=1}}
T_{(j_t,j_{t+1})},\qquad
X'=\prod_{\substack{t\in[1,m]\setminus\{k-1\}\\ \bar t=0}}
T_{(j_t,j_{t+1})}.
$$
By Proposition~\partref{prop:elem prop transp.e}, $T_{j-1}$ commutes with~$X$ and~$X'$. Then
by Lemma~\ref{lem:conj}
\begin{align*}
T_{J'}&=X T_{(j-1,j_{k+1})} T_{(j_{k-1},j-1)} X'\\
&=X T_{j-1}T_{(j,j_{k+1})} T_{(j_{k-1},j)}T_{j-1}^{-1} X'\\
&=T_{j-1}XT_{(j,j_{k+1})} T_{(j_{k-1},j)}X'T_{j-1}^{-1}=
T_{j-1} T_J T_{j-1}^{-1}.
\end{align*}
Similarly, if $k$ is even, $T_J=Y T_{(j_{k-1},j)}
T_{(j,j_{k+1})}Y'$ and~$T_{J'}=Y T_{(j_{k-1},j-1)}
T_{(j-1,j_{k+1})}Y'$ where
$$
Y=\prod_{\substack{t\in[1,m]\setminus\{k-1\}\\t\equiv 1\pmod 2}}
T_{(j_t,j_{t+1})},\qquad
Y'=\prod_{\substack{t\in[1,m]\setminus\{k\}\\ t\equiv 0\pmod 2}}
T_{(j_t,j_{t+1})}.
$$
In particular, $T_{j-1}$ commutes with~$Y$ and~$Y'$ and~$\epsilon=(-1)^{k+1}=-1$.
Using Lemma~\ref{lem:conj} we obtain
\begin{align*}
T_{J'}&=Y T_{(j_{k-1},j-1)} T_{(j-1,j_{k+1})} X'\\
&=YT_{j-1}^{-1} T_{(j_{k-1},j)} T_{(j,j_{k+1})}T_{j-1} Y'\\
&=T_{j-1}^{-1}YT_{(j_{k-1},j)} T_{(j,j_{k+1})}Y'T_J T_{j-1}.\qedhere
\end{align*}
\end{proof}

Given~$J\subset[1,n+1]$, denote \plink{gJ}$g(J)=|\{ j\in J\,:\,\min J<j<\max J,
\,j-1\notin J\}|$. For example, if $J=[1,n+1]$ then~$g(J)=0$
and $g([1,a]\cup[b+1,n+1])=1$ for all~$1\le a<b\le n$. Denote
$$
\plink{Cij(a)}\Cx ij^{(a)}=\ascprod_{i\le k\le j}T^a_k,\qquad \Cxr ij^{(a)}=\dscprod_{i\le k\le j}T^a_k,\qquad a\in\mathbb Z.
$$
\begin{corollary}\label{cor:conj J}
Let $J=\{j_0,\dots,j_{m+1}\}\subset [1,n+1]$ with $1=j_0<\cdots<j_{m+1}=n+1$. Then
$$
U(J) T_J U(J)^{-1} = T_{[1,m+1]\cup \{n+1\}}
$$
where\plink{U(J)}
\begin{equation}\label{eq:U(J) defn}
U(J)=\dscprod_{k\in[1,m]} \Cx{(k+1)}{(j_k-1)}^{((-1)^k)}.
\end{equation}
In particular, if $J,J'\subset [1,n+1]$ satisfy $|J|=|J'|$, $\min J=\min J'$ and~$\max J=\max J'$ then
$T_J$ and~$T_{J'}$ are conjugate in~$\Br_n$; thus,
all $\{1,n+1\}\subset J,J'\subset [1,n+1]$ with $|J|=|J'|$ are conjugate in~$\Br_n$.
\end{corollary}
\begin{proof}
Abbreviate~$J_m=[1,m+1]\cup\{n+1\}$ and $U_k(J)=\Cx{(k+1)}{(j_k-1)}^{((-1)^k)}$.
The argument is by induction on $g(J)$. Note that
$g(J)=|\{k\in[1,m]\,:\, j_{k-1}<j_k-1\}|$.
If $g(J)=0$ then $J=J_{n}$ and~$U(J)=1$.

For the inductive step, let $k>0$ be minimal such that $j_{k-1}<
j_k-1$. Then $j_s=s+1$ for all $0\le s<k$ and so
$U_s(J)=1$ for all~$1\le s<k$. By Proposition~\ref{prop:conj}, $U_k(J) T_J U_k(J)^{-1}
=T_{J'}$ where~$J'=J\setminus \{j_k\}\cup \{k+1\}$. Since
$g(J')=g(J)-1$, $U(J')T_{J'}U(J')^{-1}=T_{J_m}$ by the induction
hypothesis. It remains to observe that
$U_s(J')=1$ for all~$1\le s\le k$ and $U_s(J)=U_s(J')$ for all~$k+1\le s\le m$, whence $U(J')U_k(J)=U(J)$.
\end{proof}
\begin{corollary}\label{cor:TJ coxeter}
Let $J\subset [1,n+1]$. Then~$\pi_n(T_J)\in S_{n+1}$ is a cycle of length~$|J|$ and, in particular,
has order~$|J|$. Moreover, if $T_J^N$ is ${}^{op}$-invariant then
$|J|$ divides~$2N$.
\end{corollary}
\begin{proof}
By Corollary~\ref{cor:conj J} it suffices to prove the first assertion for~$J=J_m=[1,m+1]\cup\{n+1\}$, $1\le m\le n-2$.
It is easy to check that
$$
\pi_n(T_{J_m})=\begin{cases}
(1,2,4,\dots,m-1,m+1,n+1,m,m-2,\dots,3),&\text{$m$ is odd}\\
(1,2,4,\dots,m-2,m,n+1,m+1,m-1,\dots,3),&\text{$m$ is even}.
\end{cases}
$$
In either case, $\pi_n(T_{J_m})$ is a cycle of length~$m+2=|J_m|$.
By Lemma~\ref{lem:can image op inv}, if~$T_J^N$ is ${}^{op}$-invariant then
$\pi_n(T_J^N)=\pi_n(T_J)^N$ is an involution. Thus, $\pi_n(T_J)^{2N}=1$ and so the order of~$\pi_n(T_J)$ divides~$2N$.
\end{proof}

\subsection{Forward direction}\label{subs:forward direction}
Our present aim is to establish the forward direction of Theorem~\ref{thm:main thm adm}.
\begin{theorem}\label{thm:adm I2m}
For any~$\{1,n+1\}\subset J\subset [1,n+1]$ with~$g(J)=1$,
the assignments $\wh T_r\mapsto \tau_{\overline r}(J)$, $r\in\{1,2\}$
define an optimal fully supported disjoint standard homomorphism
$\Br^+(I_2(2m))\to \Br^+_{n+1}$  where
$$
m=m(J)=\begin{cases}
|J|/2,& J=\tilde\sigma(J),\\
|J|,& \text{otherwise}.
\end{cases}
$$
\end{theorem}
\begin{proof} Since~$g(J)=1$, we can write $J=[1,a]\cup [b+1,n+1]$ where
$1\le a<b\le n-1$.

Suppose first that $J=\tilde\sigma(J)=\{ n+2-j\,:\,j\in J\}$. This forces $b=n+1-a$.
Applying one of the unfolding homomorphisms~\eqref{eq:unfold Bn A2n-1} or~\eqref{eq:unfold Bn A2n}, depending on the parity of~$n$, to~\eqref{eq:prep I2m Bn}
we obtain in~$\Br_{n+1}$
$$
(T_{w_\circ^{[1,a-1]_2\cup \sigma([1,a-1]_2)}} T_{w_\circ^{[1,a-2]_2\cup
\sigma([1,a-2]_2)}}T_{(a,n+2-a)})^{a} = T_{w_\circ^{[1,n]}} T_{w_\circ^{[a+1,n-a]}}^{-1}
$$
with~$\sigma(i)=n+1-i$, $1\le i\le n$,
which immediately yields the following
\begin{corollary}\label{cor:symm even J}
Let~$\tilde J_m=[1,m]\cup \tilde\sigma([1,m])$ with $2m<n+1$. Then~$|\tilde J_m|=2m$ and
$T_{\tilde J_m}^m=T_{w_\circ^{[1,n]}} T_{w_\circ^{[m+1,n-m]}}^{-1}$ in~$\Br_{n+1}$.
\end{corollary}
Since $T_{w_\circ^{[m+1,n-m]}}$ commutes with~$T_{w_\circ^{[1,n]}}$
and both are~${}^{op}$-invariant by Proposition~\ref{prop:fund elts BrSa}
\ref{prop:fund elts BrSa.a}, by Lemma~\ref{lem:TJ hom}
this establishes the first case in Theorem~\ref{thm:adm I2m}.

To establish the second case, we need to prove the following
\begin{proposition}\label{prop:g J=1}
Let $J=[1,a]\cup[b+1,n+1]$, $1\le a<b\le n$. Then~$T_J^{|J|}=T_{w_\circ^{[1,n]}}^2 T_{w_\circ^{[a+1,b-1]}}^{-2}$ in~$\Br_n$.
\end{proposition}

The following Lemma allows one to reduce Proposition~\ref{prop:g J=1}
to a specially chosen~$J$.
\begin{lemma}\label{lem:move bubble}
Let  $[a,b],[a',b']\subset [1,n+1]\setminus\{1,n+1\}$ with~$b-a=b'-a'>1$.
Let~$J=[1,n+1]\setminus[a,b]$, $J'=[1,n+1]\setminus [a',b']$. Then
$T_J^{N}=T_{w_\circ^{[1,n]}}^2
T_{w_\circ^{[a,b-1]}}^{-2}$ for some~$N\ge 1$ if and only if
$T_{J'}^{N}=T_{w_\circ^{[1,n]}}^2 T_{w_\circ^{[a',b'-1]}}^{-2}$.
\end{lemma}
\begin{proof}
It suffices to prove the Lemma when $a'=a+1$.  We have, by Proposition~\ref{prop:conj},
$$
T_{J'}=\begin{cases}
\Cx ab T_J \Cx ab^{-1},&\text{$a$ is even}\\
\Cxr ab^{-1} T_J \Cxr ab,&\text{$a$ is odd}.
\end{cases}
$$
Since
$$
\Cx ab T_{w_\circ^{[a,b-1]}}^2=T_{w_\circ^{[a,b]}} T_{w_\circ^{[a,b-1]}}
=T_{w_\circ^{[a+1,b]}} T_{w_\circ^{[a,b]}}
=T_{w_\circ^{[a+1,b]}}^2 \Cx ab,
$$
it follows that
$$
\Cx ab T_{w_\circ^{[a,b-1]}}^{-2} \Cx ab^{-1}=T_{w_\circ^{[a+1,b]}}^{-2}
=\Cxr ab^{-1} T_{w_\circ^{[a,b-1]}}^{-2} \Cxr ab,
$$
where the second equality is obtained from the first by applying~${}^{op}$.
Since~$T_{w_\circ^{[1,n]}}^2$ is central in~$\Br^+_{n+1}$ by Proposition~\partref{prop:fund elts BrSa.d}, the assertion
is now immediate.
\end{proof}

Suppose first that~$|J|$ is even with~$|J|=2m$. By Corollary~\ref{cor:symm even J}, $T_{\tilde J_m}^{2m}=T_{w_\circ^{[1,n]}}^2 T_{w_\circ^{[m+1,n-1]}}^{-2}$. Thus, Proposition~\ref{prop:g J=1} for~$|J|$ even is proved.

We will now obtain a convenient expression for
$T_{w_\circ^{[1,n]}}^2 T_{w_\circ^{[r+1,n-1]}}^{-2}$ for
$1\le r\le n$.
For that we need more properties of the~$T_{(a,b+1)}$, $a\le b\in[1,n]$.
\begin{lemma}\label{lem:transp 1}
For any $2\le a\le b$,
$
T_{(a,b+1)}T_{(a+1,b+1)}T_{a-1}T_a T_{(a+1,b+1)}
$
is ${}^{op}$-invariant.
\end{lemma}
\begin{remark}
It is easy to check that for~$b>a$ the canonical image of the above element of~$\Br^+_{n+1}$ in~$W(A_n)$ is
$(a-1,b+1)$ which is of length~$2(b-a)+3$
while the original element of~$\Br^+_{n+1}$
has length~$6(b-a)+1$. 
\end{remark}
\begin{proof}
Let $X=T_{(a,b+1)}T_{(a+1,b+1)}T_{a-1}T_a T_{(a+1,b+1)}$.
For~$b=a$ we have $X=T_a T_{a-1} T_a$ which is manifestly
${}^{op}$-invariant. Thus, we may assume, without loss of generality, that $b>a$.

For simplicity, let $a=2$. Thus, we claim that
$$
T_{(2,b+1)}T_{(3,b+1)}T_1 T_2 T_{(3,b+1)}=T_{(3,b+1)}
T_2 T_1 T_{(3,b+1)}T_{(2,b+1)}.
$$
Since $T_{(3,b+1)}=\Cx 3{(b-1)}T_b \Cxr3{(b-1)}$,
and the $T_j$ with $i<j<k-1$ commute with $T_{(i,k)}$ by Proposition~\partref{prop:elem prop transp.e}, we have
$$
X=\Cx3{(b-1)}T_{(2,b+1)}\Cxr3b \Cx1b\Cxr3{(b-1)}
$$
while
$$
X^{op}=\Cx 3{(b-1)} \Cxr1b \Cx3b T_{(2,b+1)}\Cxr3{(b-1)}.
$$
Since~$\Br^+_{n+1}$ is cancellative, it suffices to prove that
$$
T_{(2,b+1)}\Cxr3b \Cx1b=
\Cxr1b \Cx3b T_{(2,b+1)}.
$$
Since $T_{(2,b+1)}=\Cxr2b\Cx3{b}$ and $\Cxr1b=\Cxr2b T_1$, the above equality
follows once we establish that
\begin{equation}\label{eq:interm transp 1}
\Cx3b\Cxr3b \Cx1b=T_1\Cx3b T_{(2,b+1)}.
\end{equation}
But the left hand side of~\eqref{eq:interm transp 1} is equal to
\begin{equation*}
\Cx3b\Cxr3b \Cx1b=\Cx3b\Cxr3b T_1\Cx2b=T_1\Cx3b\Cxr3b\Cx2b
=T_1 \Cx3bT_{(2,b+1)}.\qedhere
\end{equation*}
\end{proof}
\begin{lemma}\label{lem:transp 2}
For all $2\le a<b\le n$ we have
$$
T_{(a,b+1)} T_{a-1} T_{(a,b+1)}T_{(a+1,b+1)}T_a= T_{(a-1,b+1)}T_{(a,b+1)}T_{(a+1,b+1)}
$$
\end{lemma}
\begin{proof}
Since
$$
T_{(a,b+1)} T_{a-1}=\Cxr{a}b \Cx{(a+1)}b T_{a-1}=
\Cxr{(a-1)}b \Cx{(a+1)}{b}.
$$
and~$\Br^+_{n+1}$ is cancellative, the assertion is equivalent to
$$
\Cx{(a+1)}b T_{(a,b+1)}T_{(a+1,b+1)}T_a=
\Cx{a}b T_{(a,b+1)}T_{(a+1,b+1)}.
$$
Now,
\begin{align*}
\Cx{a}{b}T_{(a,b+1)}=\Cx{a} b T_a T_{(a+1,b+1)} T_a 
&=T_a T_{a+1} T_a \Cx{(a+2)}b T_{(a+1,b+1)} T_a\\
&=T_{a+1}\Cx ab T_{(a+1,b+1)}T_a.
\end{align*}
Suppose we proved that
\begin{equation}\label{eq:eq transp 2}
\Cx{a}{b}T_{(a,b+1)}=\Cx{(a+1)}k \Cx ab T_{(k,b+1)}\Cxr{a}{(k-1)}.
\end{equation}
for some~$k> a$ (the case~$k=a+1$ was established above).
Then
\begin{align*}
\Cx{a}{b}T_{(a,b+1)}&=\Cx{(a+1)}k \Cx ab T_k T_{(k+1,b+1)}\Cxr{a}{k}\\
&=\Cx{(a+1)}{k} \Cx a{(k-1)}(T_k T_{k+1} T_k)\Cx{(k+2)}b T_{(k+1,b+1)}\Cxr{a}{k}\\
&=\Cx{(a+1)}{k} \Cx a{(k-1)}(T_{k+1} T_{k} T_{k+1})
\Cx{(k+2)}bT_{(k+1,b+1)}\Cxr{a}{k}\\
&=\Cx{(a+1)}{(k+1)} \Cx abT_{(k+1,b+1)}\Cxr{a}{k}.
\end{align*}
Thus, \eqref{eq:eq transp 2} holds for all~$k\ge a+1$. In particular,
for~$k=b$ we obtain
$$
\Cx{a}{b}T_{(a,b+1)}=\Cx{(a+1)}{b} \Cx abT_{(b,b+1)}\Cxr{a}{b-1}
=\Cx{(a+1)}{b} \Cx ab\Cxr{a}{b}.
$$
Therefore, it suffices to prove that
$$
T_{(a,b+1)}T_{(a+1,b+1)}T_a=\Cx ab\Cxr{a}{b}T_{(a+1,b+1)}
$$
which, since $T_{(a,b+1)}=\Cx ab\Cxr a{(b-1)}$, is equivalent to
$$
\Cxr{a}{(b-1)}T_{(a+1,b+1)}T_a=\Cxr{a}{b}T_{(a+1,b+1)}.
$$
By Proposition~\ref{prop:elem prop transp} we obtain
\begin{align*}
\Cxr{a}{(b-1)}&T_{(a+1,b+1)}T_a=\Cxr{(a+1)}{(b-1)}T_a T_{(a+1,b+1)}T_a=\Cxr{(a+1)}{(b-1)}T_{(a,b+1)}\\
&=T_{(a,b+1)}\Cxr{(a+1)}{(b-1)}
=\Cxr ab\Cx{(a+1)}b \Cxr{(a+1)}{(b-1)}%
=\Cxr ab T_{(a+1,b+1)}.\qedhere
\end{align*}
\end{proof}

\begin{lemma}\label{lem:cox transp exchage}We have for all $r\in[1,n-1]$
\begin{equation}\label{eq:cox transp exchange}
\Big(\ascprod_{i\in[1,r+1]} T_{(i,n+1)}\Big) \Cx1{r}=\Cx1r
T_{(r+1,n+1)}\Big(\ascprod_{i\in[1,r]} T_{(i,n+1)}\Big)
\end{equation}
\end{lemma}
\begin{proof}
Suppose we proved that for some~$i\in[1,r]$
\begin{equation}\label{eq:cox transp exchange'}
\ascprod_{j\in[1,r+1]} T_{(j,n+1)}\Cx1r=
\Cx1{(i-1)} T_{(i,n+1)}\ascprod_{j\in[1,r+1]\setminus\{i\}} T_{(j,n+1)}\Cx ir
\end{equation}
(for~$i=1$ this is trivial). Then
\begin{align*}
\ascprod_{j\in[1,r+1]} T_{(j,n+1)}\Cx1r&=\Cx1{i} T_{(i+1,n+1)}T_i\ascprod_{j\in[1,r+1]\setminus\{i\}} T_{(j,n+1)}\Cx ir\\
&=\Cx1{i} T_{(i+1,n+1)}\ascprod_{j\in[1,i-1]}T_{(j,n+1)} T_i T_{(i+1,n+1)}T_i \ascprod_{j\in[i+2,r+1]} T_{(j,n+1)}\Cx{(i+1)} r\\
&=\Cx1{i} T_{(i+1,n+1)}\ascprod_{j\in[1,r+1]\setminus \{i+1\}} T_{(j,n+1)} \Cx{(i+1)} r.
\end{align*}
Thus, \eqref{eq:cox transp exchange'} holds for all~$i\in[1,r+1]$. But for~$i=r+1$ this is precisely~\eqref{eq:cox transp exchange}.
\end{proof}
Denote
\begin{equation}\label{eq:Z_r defn}
Z_r:=\Big(\ascprod_{i\in[1,r+1]} T_{(i,n+1)}\Big) \Cx1{r}T_{(r+1,n+1)}=\Cx1r
T_{(r+1,n+1)}\Big(\ascprod_{i\in[1,r+1]} T_{(i,n+1)}\Big),
\end{equation}
the equality following from~\eqref{eq:cox transp exchange}.
\begin{proposition}\label{prop:Z_r is central}
For all $r\in[1,n-2]$ we have
$$
Z_r T_{w_\circ^{[r+2,n-1]}}^2=T_{w_\circ^{[1,n]}}^2.
$$
\end{proposition}
\begin{proof}
We have
$$
\ell(Z_r)=\sum_{1\le i\le r+1}(2(n-i)+1)+2n-r-1=(r+2)(2n-r-1).
$$
and so
$$
\ell(Z_r T_{w_\circ^{[r+2,n-1]}}^2)=(n-r-1)(n-r-2)+(r+2)(2n-r-1)=n(n+1)
=\ell(T_{w_\circ^{[1,n]}}^2).
$$
Since~$T_{w_\circ^{[1,n]}}^2$ generates the center of~$\Br^+_{n+1}$
(cf. Proposition~\partref{prop:fund elts BrSa.d}),
it suffices to prove that $Z_r T_{w_\circ^{[r+2,n-1]}}^2$
is central in~$\Br^+_{n+1}$.
\begin{lemma}\label{lem:comm I}
We have $T_i Z_r=Z_r T_i$ for all $i\in [1,n]\setminus \{r+1,n\}$.
\end{lemma}
\begin{proof}
By Proposition~\partref{prop:elem prop transp.c}, the $T_j$, $j\in [r+2,n-1]$ commute with the $T_{(i,n+1)}$,
$i\in[1,r+1]$. It follows that
$T_i Z_r=Z_r T_i$ for all~$i\in [r+2,n-1]$.

Since~$\Br^+_{n+1}$ embeds into~$\Br_{n+1}$ (cf.~\cites{BrSa,Del,Par}),
it suffices to prove that~$T_i^{-1} Z_r T_i=Z_r$ for all~$i\in [1,r]$.

Let~$i\in [1,r-1]$. By Proposition~\ref{prop:elem prop transp}, $T_i^{-1}$ commutes with
the $T_{(j,n+1)}$ for all $1\le j\le i-1$ and $T_i^{-1} T_{(i,n+1)}=T_{(i+1,n+1)}T_i$. Therefore, %
\begin{multline*}
T_i^{-1} Z_r T_i = \Big(\ascprod_{j\in[1,i-1]} T_{(j,n+1)} \Big) T_{(i+1,n+1)}T_i T_{(i+1,n+1)}T_{(i+2,n+1)}\\
\Big(\ascprod_{j\in[i+3,r+1]} T_{(j,n+1)}\Big)\Cx1r T_i T_{(r+1,n+1)}.
\end{multline*}
Since $\Cx1r T_i=T_{i+1}\Cx1r$ by Lemma~\ref{lem:comm cox}, we obtain
\begin{align*}
T_i^{-1} &Z_r T_i = \Big(\ascprod_{j\in[1,i-1]} T_{(j,n+1)} \Big) T_{(i+1,n+1)}T_i T_{(i+1,n+1)}T_{(i+2,n+1)}T_{i+1}\\
&\mskip200mu
\Big(\ascprod_{j\in[i+3,r+1]} T_{(j,n+1)}\Big)\Cx1r T_{(r+1,n+1)}\\
&= \Big(\ascprod_{j\in[1,i-1]} T_{(j,n+1)} \Big) T_{(i,n+1)}T_{(i+1,n+1)}T_{(i+2,n+1)}
\Big(\ascprod_{j\in[i+3,r+1]} T_{(j,n+1)}\Big)\Cx1r T_{(r+1,n+1)}\\
&=Z_r,
\end{align*}
where we used Lemma~\ref{lem:transp 2}.

Let~$i=r$. Then
\begin{align*}
T_r^{-1} Z_r T_r&=\ascprod_{i\in[1,r-1]} T_{(i,n+1)} T_{(r+1,n+1)}T_r T_{(r+1,n+1)} \Cx1r T_{(r+1,n+1)} T_r\\
&=\ascprod_{i\in[1,r-1]}T_{(i,n+1)} T_{(r+1,n+1)}T_r T_{(r+1,n+1)} \Cx1{(r-1)} T_{(r,n+1)}\\
&=\ascprod_{i\in[1,r-1]}T_{(i,n+1)} \Cx1{(r-2)} T_{(r+1,n+1)}T_r T_{r-1} T_{(r+1,n+1)} T_{(r,n+1)}\\
&=\ascprod_{i\in[1,r-1]}T_{(i,n+1)} \Cx1{(r-2)} T_{(r,n+1)} T_{(r+1,n+1)}T_{r-1} T_r T_{(r+1,n+1)}\\
&=\ascprod_{i\in[1,r+1]}T_{(i,n+1)} \Cx1r T_{(r+1,n+1)}=Z_r,
\end{align*}
by Lemma~\ref{lem:transp 1}.
\end{proof}
\begin{lemma}\label{lem:comm II}
We have $T_i Z_r T_{(r+2,n)}^2=Z_r T_{(r+2,n)}^2 T_i$ for $i\in\{r+1,n\}$.
\end{lemma}
\begin{proof}
Since $T_{r+1}$ and~$T_n$ commute with $T_{(a,b)}$ for all $r+2<a<b<n-1$, $T_{w_\circ}^{[r+1,n]}=\prod_{1\le i\le \frac12(n-r)+1} T_{(r+i,n+2-i)}$ by Proposition~\ref{prop:elem prop transp} while $T_j T_{w_\circ^{[r+1,n]}}=T_{w_\circ^{[r+1,n]}}
T_{n+r+1-j}$ for all $j\in [r+1,n]$,
we conclude that
$$
T_{r+1} T_{(r+1,n+1)}T_{(r+2,n)}=T_{(r+1,n+1)}T_{(r+2,n)} T_n
$$
and
$$
T_n T_{(r+1,n+1)}T_{(r+2,n)}=T_{(r+1,n+1)}T_{(r+2,n)} T_{r+1}.
$$
Furthermore,
note that by Proposition~\partref{prop:elem prop transp.f}, $T_{(r+2,n)}$ commutes with all factors in
the formula defining~$Z_r$.

We have
\begin{align*}
T_{r+1} Z_r T_{(r+2,n)}^2 &=\ascprod_{i\in[1,r]} T_{(i,n+1)} T_{r+1} T_{(r+1,n+1)}T_{(r+2,n)} \Cx1r  T_{(r+1,n+1)}T_{(r+2,n)}\\
&=\ascprod_{i\in[1,r]} T_{(i,n+1)} T_{(r+1,n+1)}T_{(r+2,n)} T_n\Cx1r  T_{(r+1,n+1)}T_{(r+2,n)}\\
&=\ascprod_{i\in[1,r]} T_{(i,n+1)} T_{(r+1,n+1)}T_{(r+2,n)} \Cx1r T_n  T_{(r+1,n+1)}T_{(r+2,n)}\\
&=Z_r T_{(r+2,n)}^2 T_{r+1},
\\
\intertext{and}
T_{n} Z_r T_{(r+2,n)}^2 &=\Cx1r T_n T_{(r+1,n+1)}T_{(r+2,n)}\ascprod_{i\in[1,r]} T_{(i,n+1)} T_{(r+1,n+1)}T_{(r+2,n)} \\
&=\Cx1r T_{(r+1,n+1)}T_{(r+2,n)}T_{r+1}\ascprod_{i\in[1,r]} T_{(i,n+1)} T_{(r+1,n+1)}T_{(r+2,n)} \\
&=\Cx1r T_{(r+1,n+1)}T_{(r+2,n)}\ascprod_{i\in[1,r]}
T_{(i,n+1)} T_{r+1} T_{(r+1,n+1)}T_{(r+2,n)}\\
&=\Cx1r T_{(r+1,n+1)}T_{(r+2,n)}\ascprod_{i\in[1,r]}
T_{(i,n+1)} T_{(r+1,n+1)}T_{(m+2,n)}T_n\\
&=Z_r T_{(r+2,n)}^2 T_n.\qedhere
\end{align*}
\end{proof}
Together, Lemmata~\ref{lem:comm I} and~\ref{lem:comm II} imply that
$Z_r T_{w_\circ^{[r+2,n-1]}}^2$ is central
in~$\Br^+_{n}$, which completes the proof of Proposition~\ref{prop:Z_r is central}.
\end{proof}
\begin{proposition}\label{prop:TJm power central}
For all~$m\in [1,n-2]$,
$T_{[1,m+1]\cup\{n+1\}}^{m+2}=Z_m$.
\end{proposition}
\begin{proof}
The assertion follows immediately from Proposition~\ref{prop:Z_r is central} for $m$ even. Throughout the rest of this proof, we assume that~$m$ is odd. Let~$J_m=[1,m+1]\cup\{n+1\}$.
\begin{lemma}\label{lem:1st identity}
We have
$$
T_{J_m}^{m+1}=\dscprod_{j\in [1,m]_2} T_{(j,n+1)}\ascprod_{j\in [2,m+1]_2} T_{(j,n+1)}.
$$
\end{lemma}
\begin{proof}
First we prove by descending induction that for all $k\in [1,m]_2$,
\begin{equation}\label{eq:1st rewrite}
\begin{split}
T_{J_m}^{m+1}=\dscprod_{i\in [k,m]_2} &(\Cx im T_{(m+1,n+1)}) \Big(\ascprod_{j\in[k,m]_2} T_{w_\circ^{[1,j-2]_2}} T_{w_\circ^{[2,j-1]_2}}\Big)\\
&(T_{w_\circ^{[1,m]_2}}T_{w_\circ^{[2,m-1]_2}}T_{(m+1,n+1)})^{\frac{m+k}2}.
\end{split}
\end{equation}
Indeed, since the $T_i$, $i\in [1,m-1]$ commute with $T_{(m+1,n+1)}$
by Proposition~\partref{prop:elem prop transp.e}, we have
$$
T_{J_m}^{m+1}=T_m T_{(m+1,n+1)} (T_{w_\circ^{[1,m-2]_2}}T_{w_\circ^{[2,m-1]_2}}) (T_{w_\circ^{[1,m]_2}}T_{w_\circ^{[2,m-1]_2}}T_{(m+1,n+1)})^{m},
$$
which is~\eqref{eq:1st rewrite} with~$k=m$.

For the inductive step we have
\begin{align*}
T_{J_m}^{m+1}&=\dscprod_{i\in [k,m]_2} (\Cx im T_{(m+1,n+1)}) \Big(\ascprod_{j\in[k,m]_2} T_{w_\circ^{[1,j-2]_2}} T_{w_\circ^{[2,j-1]_2}}\Big)\\
&\mskip200mu(T_{w_\circ^{[1,m]_2}}T_{w_\circ^{[2,m-1]_2}} T_{(m+1,n+1)})^{\frac{m+k}2} \\
&=\dscprod_{i\in [k,m]_2} (\Cx im T_{(m+1,n+1)}) \Big(\ascprod_{j\in[k,m]_2} T_{w_\circ^{[1,j-4]_2}} T_{w_\circ^{[2,j-3]_2}} T_{j-2} T_{j-1} \Big)
T_m
T_{(m+1,n+1)}\\&\mskip200muT_{w_\circ^{[1,m-2]_2}}T_{w_\circ^{[2,m-1]_2}}(T_{w_\circ^{[1,m]_2}}T_{w_\circ^{[2,m-1]_2}} T_{(m+1,n+1)})^{\frac{m+k}2-1} \\
&=\dscprod_{i\in [k,m]_2} (\Cx im T_{(m+1,n+1)}) \Cx{(k-2)}mT_{(m+1,n+1)}\Big(\ascprod_{j\in[k-2,m-2]_2} T_{w_\circ^{[1,j-2]_2}} T_{w_\circ^{[2,j-1]_2}}\Big)
\\&\mskip200muT_{w_\circ^{[1,m-2]_2}}T_{w_\circ^{[2,m-1]_2}} (T_{w_\circ^{[1,m]_2}}T_{w_\circ^{[2,m-1]_2}} T_{(m+1,n+1)})^{\frac{m+k}2-1} \\
&=\dscprod_{i\in [k-2,m]_2} (\Cx im T_{(m+1,n+1)})\Big(\ascprod_{j\in[k-2,m]_2} T_{w_\circ^{[1,j-2]_2}} T_{w_\circ^{[2,j-1]_2}}\Big) \\
&\mskip200mu(T_{w_\circ^{[1,m]_2}}T_{w_\circ^{[2,m-1]_2}} T_{(m+1,n+1)})^{\frac{m+k-2}2}.
\end{align*}
This proves~\eqref{eq:1st rewrite}. Taking~$k=1$ yields
\begin{equation}\label{eq:1st rewrite final}
\begin{split}
T_{J_m}^{m+1}=\dscprod_{i\in [1,m]_2} (\Cx im T_{(m+1,n+1)}) &\Big(\ascprod_{j\in[1,m-2]_2} T_{w_\circ^{[1,j]_2}} T_{w_\circ^{[2,j+1]_2}}\Big)\\
&(T_{w_\circ^{[1,m]_2}}T_{w_\circ^{[2,m-1]_2}} T_m T_{(m+1,n+1)})^{\frac{m+1}2}.
\end{split}
\end{equation}
The next step is to show that for all $k\in[1,m]_2$,
\begin{equation}\label{eq:2nd rewrite}\begin{split}
(T_{w_\circ^{[1,m]_2}}T_{w_\circ^{[2,m-1]_2}} T_m T_{(m+1,n+1)})^{\frac{m+1}2}=
\Big(\dscprod_{j\in[k,m]_2} T_{w_\circ^{[1,j]_2}}T_{w_\circ^{[2,j-1]_2}}\Big) T_{(m+1,n+1)}\\\ascprod_{j\in[k+1,m-1]_2}(\Cxr jm T_{(m+1,n+1)})
(T_{w_\circ^{[1,m]_2}}T_{w_\circ^{[2,m-1]_2}} T_{(m+1,n+1)})^{\frac{k-1}2}.
\end{split}
\end{equation}
Again, we use descending induction on~$k$, the case~$k=m$ being trivial. For the inductive step, we have
\begin{align*}
(&T_{w_\circ^{[1,m]_2}}T_{w_\circ^{[2,m-1]_2}} T_m T_{(m+1,n+1)})^{\frac{m+1}2}=
\Big(\dscprod_{j\in[k,m]_2} T_{w_\circ^{[1,j]_2}}T_{w_\circ^{[2,j-1]_2}}\Big) T_{(m+1,n+1)}\\&\mskip50mu\ascprod_{j\in[k+1,m-1]_2}(\Cxr jm T_{(m+1,n+1)})
(T_{w_\circ^{[1,m]_2}}T_{w_\circ^{[2,m-1]_2}} T_{(m+1,n+1)})^{\frac{k-1}2}\\
&=
\Big(\dscprod_{j\in[k,m]_2} T_{w_\circ^{[1,j]_2}}T_{w_\circ^{[2,j-1]_2}}\Big) T_{(m+1,n+1)}\ascprod_{j\in[k+1,m-1]_2}(\Cxr jm T_{(m+1,n+1)})
 \\
&\mskip50muT_{w_\circ^{[1,k-2]_2}} T_{w_\circ^{[2,k-3]_2}}T_{w_\circ^{[k,m]_2}} T_{w_\circ^{[k-1,m]_2}} T_{(m+1,n+1)}
(T_{w_\circ^{[1,m]_2}}T_{w_\circ^{[2,m-1]_2}} T_{(m+1,n+1)})^{\frac{k-3}2}\\
&=
\Big(\dscprod_{j\in[k-2,m]_2} T_{w_\circ^{[1,j]_2}}T_{w_\circ^{[2,j-1]_2}}\Big) T_{(m+1,n+1)}\ascprod_{j\in[k+1,m-1]_2}(\Cxr jm T_{(m+1,n+1)})\\
&\mskip50mu
\Big(\ascprod_{i\in [k,m]_2} T_i T_{i-1}\Big)
T_{(m+1,n+1)}
(T_{w_\circ^{[1,m]_2}}T_{w_\circ^{[2,m-1]_2}} T_{(m+1,n+1)})^{\frac{k-3}2}\\
&=
\Big(\dscprod_{j\in[k-2,m]_2} T_{w_\circ^{[1,j]_2}}T_{w_\circ^{[2,j-1]_2}}\Big) T_{(m+1,n+1)}\ascprod_{j\in[k+1,m-1]_2}(\Cxr{(j-2)}m T_{(m+1,n+1)})\\
&\mskip50mu
T_m T_{m-1} T_{(m+1,n+1)}
(T_{w_\circ^{[1,m]_2}}T_{w_\circ^{[2,m-1]_2}} T_{(m+1,n+1)})^{\frac{k-3}2}\\
&=
\Big(\dscprod_{j\in[k-2,m]_2} T_{w_\circ^{[1,j]_2}}T_{w_\circ^{[2,j-1]_2}}\Big) T_{(m+1,n+1)}\ascprod_{j\in[k-1,m-1]_2}(\Cxr{j}m T_{(m+1,n+1)})\\
&\mskip50mu
(T_{w_\circ^{[1,m]_2}}T_{w_\circ^{[2,m-1]_2}} T_{(m+1,n+1)})^{\frac{k-3}2}.
\end{align*}
In particular, for~$k=1$ we obtain
\begin{equation}\label{eq:2nd rewrite final}
\begin{split}
T_{J_m}^{m+1}&=\dscprod_{i\in [1,m]_2} (\Cx im T_{(m+1,n+1)}) \Big(\ascprod_{j\in[1,m-2]_2} T_{w_\circ^{[1,j]_2}} T_{w_\circ^{[2,j+1]_2}}\Big)\\
&\Big(\dscprod_{j\in[1,m]_2} T_{w_\circ^{[1,j]_2}}T_{w_\circ^{[2,j-1]_2}}\Big) T_{(m+1,n+1)}\ascprod_{j\in[2,m-1]_2}(\Cxr jm T_{(m+1,n+1)}).
\end{split}
\end{equation}
Next we claim that
$$
\Big(\ascprod_{j\in[1,m-2]_2} T_{w_\circ^{[1,j]_2}} T_{w_\circ^{[2,j+1]_2}}\Big)
\Big(\dscprod_{j\in[1,m]_2} T_{w_\circ^{[1,j]_2}}T_{w_\circ^{[2,j-1]_2}}\Big)=T_{w_\circ^{[1,m]}}.
$$
We use induction on odd~$m$, the case~$m=1$ being trivial. Note that, since
$$
T_{w_\circ^{[1,r]}}=T_{w_\circ^{[1,r-1]}}\Cxr1r
$$
and $T_{w_\circ^J}$, $J\subset[1,n]$, is ${}^{op}$-invariant by~\cite{BrSa}*{Lemma~5.1}, it follows that
\begin{equation}\label{eq:2-side w0}
T_{w_\circ^{[1,m]}}=T_{w_\circ^{[1,m-1]}}\Cxr1m=\Cx1{(m-1)}T_{w_\circ^{[1,m-2]}}\Cxr1m.
\end{equation}
Then
\begin{align*}
\Big(\ascprod_{j\in[1,m-2]_2} &T_{w_\circ^{[1,j]_2}} T_{w_\circ^{[2,j+1]_2}}\Big)
\Big(\dscprod_{j\in[1,m]_2} T_{w_\circ^{[1,j]_2}}T_{w_\circ^{[2,j-1]_2}}\Big)\\
&=\Big(\ascprod_{j\in[1,m-2]_2} T_{w_\circ^{[1,j-2]_2}} T_{j}T_{j+1} T_{w_\circ^{[2,j-1]_2}}\Big)
\Big(\dscprod_{j\in[1,m]_2} T_{w_\circ^{[1,j-2]_2}}T_j T_{j-1} T_{w_\circ^{[2,j-3]_2}}\Big)\\
&=\Cx1{(m-1)}\Big(\ascprod_{j\in[1,m-4]_2} T_{w_\circ^{[1,j]_2}} T_{w_\circ^{[2,j+1]_2}}\Big)
\Big(\dscprod_{j\in[1,m-2]_2} T_{w_\circ^{[1,j]_2}}T_{w_\circ^{[2,j-1]_2}}\Big)\Cxr1m\\
&=\Cx1{(m-1)}T_{w_\circ^{[1,m-2]}}\Cxr1m,
\end{align*}
where we used the induction hypothesis and the convention that~$T_0=1$. It remains to use~\eqref{eq:2-side w0}.
Thus,
$$
T_{J_m}^{m+1}=\dscprod_{i\in [1,m]_2} (\Cx im T_{(m+1,n+1)}) T_{w_\circ^{[1,m]}} T_{(m+1,n+1)}\ascprod_{j\in[2,m-1]_2}(\Cxr jm T_{(m+1,n+1)}).
$$
Applying the diagram automorphism of the submonoid $\Br^+_{[1,m]}(A_n)\cong \Br^+_{m+1}$ of~$\Br^+_{n+1}$, we obtain from~\eqref{eq:2-side w0}
$$
T_{w_\circ^{[1,m]}}=\Cxr2{m}T_{w_\circ^{[3,m]}}\Cx1m.
$$
Since~$T_{w_\circ^{[1,m]}}$ is ${}^{op}$-invariant by Proposition~\partref{prop:fund elts BrSa.a}, this is also equal to
$$
T_{w_\circ^{[1,m]}}=\Cxr1{m}T_{w_\circ^{[1,m-2]}}\Cx2m.
$$
A straightforward induction now yields
$$
T_{w_\circ^{[1,m]}}=\Big(\ascprod_{j\in[1,m]_2}\Cxr jm\Big)\Big(\dscprod_{j\in[2,m-1]_2}\Cx jm\Big).
$$
Thus,
\begin{equation}\label{eq:3rd rewrite}
\begin{split}
T_{J_m}^{m+1}&=\dscprod_{i\in [1,m]_2} (\Cx im T_{(m+1,n+1)}) \Big(\ascprod_{j\in[1,m]_2}
\Cxr jm\Big)\\&\qquad\Big(\dscprod_{j\in[2,m-1]_2}\Cx jm\Big)
T_{(m+1,n+1)} \ascprod_{j\in[2,m-1]_2}(\Cxr jm T_{(m+1,n+1)}).
\end{split}
\end{equation}
Suppose we proved that, for some $k\in[1,m-2]_2$,
\begin{align}\label{eq:3rd rewrite final}
T_{J_m}^{m+1}&=\dscprod_{i\in [k,m]_2} (\Cx im T_{(m+1,n+1)}) \dscprod_{i\in [1,k-2]_2} T_{(j,n+1)}
\Big(\ascprod_{j\in[k,m]_2}
\Cxr jm\Big)\\&\qquad\Big(\dscprod_{j\in[k+1,m-1]_2}\Cx jm\Big)
\ascprod_{j\in[2,k-1]_2} T_{(j,n+1)} T_{(m+1,n+1)} \ascprod_{j\in[k+1,m-1]_2}(\Cxr jm T_{(m+1,n+1)}),\nonumber
\end{align}
the case~$k=1$ being just~\eqref{eq:3rd rewrite}. Since $\Cx im T_{(m+1,n+1)}\Cxr im=T_{(i,n+1)}$, $1\le i\le m$ and the $\Cx jm$, $i<j\le m$ commute with $T_{(i,n+1)}$, $1\le i\le m+1$,
we obtain
\begin{align*}
T_{J_m}^{m+1}&=\dscprod_{i\in [k+2,m]_2} (\Cx im T_{(m+1,n+1)}) \Cx km T_{(m+1,n+1)}\Cxr km \dscprod_{i\in [1,k-2]_2} T_{(j,n+1)}
\ascprod_{j\in[k+2,m]_2}
\Cxr jm\\&\qquad\Big(\dscprod_{j\in[k+3,m-1]_2}\Cx jm\Big)
\ascprod_{j\in[2,k-1]_2} T_{(j,n+1)}
\Cx{(k+1)}m T_{(m+1,n+1)}\Cxr{(k+1)}m\\
&\mskip300mu\ascprod_{j\in[k+3,m-1]_2}(\Cxr jm T_{(m+1,n+1)})\\
&=\dscprod_{i\in [k+2,m]_2} (\Cx im T_{(m+1,n+1)}) \dscprod_{i\in [1,k]_2} T_{(j,n+1)}
\ascprod_{j\in[k+2,m]_2}
\Cxr jm\\&\qquad\Big(\dscprod_{j\in[k+3,m-1]_2}\Cx jm\Big)
\ascprod_{j\in[2,k+1]_2} T_{(j,n+1)} T_{(m+1,n+1)} \ascprod_{j\in[k+3,m-1]_2}(\Cxr jm T_{(m+1,n+1)}).
\end{align*}
Thus, the identity~\eqref{eq:3rd rewrite final} holds for all~$k\in[1,m]_2$.
Taking~$k=m$ yields the assertion.
\end{proof}
\begin{lemma}\label{lem:2nd identity}
We have
\begin{equation}\label{eq:2nd identity}
T_{J_m}^{m+1} T_{w_\circ^{[1,m]_2}} T_{w_\circ^{[2,m-1]_2}}=\Big(\ascprod_{i\in[1,m+1]} T_{(i,n+1)}\Big) \Cx1m%
\end{equation}
\end{lemma}
\begin{proof}
The argument is by induction on odd~$m$, the case~$m=1$ being immediate from Lemma~\ref{lem:1st identity}. For the inductive step, by Lemma~\ref{lem:1st identity} we are proving that
$$
\dscprod_{j\in [1,m]_2} T_{(j,n+1)}\ascprod_{j\in [2,m+1]_2} T_{(j,n+1)}T_{w_\circ^{[1,m]_2}} T_{w_\circ^{[2,m-1]_2}}=\ascprod_{i\in[1,m+1]} T_{(i,n+1)} \Cx1m.
$$
The left hand side is equal to
\begin{align*}
T_{(m,n+1)} &\dscprod_{j\in [1,m-2]_2} T_{(j,n+1)}\ascprod_{j\in [2,m-1]_2} T_{(j,n+1)} T_{w_\circ^{[1,m-2]_2}} T_{w_\circ^{[2,m-3]_2}} T_{(m+1,n+1)}T_m T_{m-1}\\
&=T_{(m,n+1)} \ascprod_{i\in[1,m-1]} T_{(i,n+1)} \Cx1{(m-2)} T_{(m+1,n+1)}T_m T_{m-1}.
\end{align*}
where we used the induction hypothesis. Since $$
T_{(m-1,n+1)}T_{(m,n+1)}T_{(m+1,n+1)}=T_{(m,n+1)}T_{m-1} T_{(m,n+1)}T_{(m+1,n+1)}T_m
$$
by Lemma~\ref{lem:transp 2},
the right hand side equals to
\begin{multline*}
\ascprod_{i\in[1,m-2]} T_{(i,n+1)} T_{(m,n+1)}T_{m-1} T_{(m,n+1)}T_{(m+1,n+1)}T_m \Cx1m\\=\ascprod_{i\in[1,m-2]} T_{(i,n+1)} T_{(m,n+1)}T_{m-1} T_{(m,n+1)}
\Cx1{(m-1)} T_{(m+1,n+1)}T_m T_{m-1}.
\end{multline*}
Since the braid monoid is cancellative, it suffices to prove that
\begin{equation}
\label{eq:equiv id for reorder}
T_{(m,n+1)} \ascprod_{i\in[1,m-1]} T_{(i,n+1)} \Cx1{(m-2)}=\ascprod_{i\in[1,m-2]} T_{(i,n+1)} T_{(m,n+1)}T_{m-1} T_{(m,n+1)}
\Cx1{(m-1)}.
\end{equation}
Write $T_{(m,n+1)}=\Cx m{(n-1)}\Cxr mn$. Since~$\Cx m{(n-1)}$ commutes with the~$T_{(i,n+1)}$, $i\in [1,m-2]_2$, \eqref{eq:equiv id for reorder} is equivalent to
\begin{equation}
\label{eq:equiv id for reorder'}
\Cxr mn\ascprod_{i\in[1,m-1]} T_{(i,n+1)} \Cx1{(m-2)}=\ascprod_{i\in[1,m-2]} T_{(i,n+1)} \Cxr mn T_{m-1} T_{(m,n+1)}
\Cx1{(m-1)}.
\end{equation}
Since $\Cxr{(m-1)}n T_{(m,n+1)}=\Cxr{(m-1)}n\Cx mn\Cxr m{(n-1)}=T_{(m-1,n+1)}\Cxr m{(n-1)}$, that identity is equivalent to
\begin{equation}
\label{eq:equiv id for reorder''}
\Cxr mn\ascprod_{i\in[1,m-1]} T_{(i,n+1)} \Cx1{(m-2)}=\ascprod_{i\in[1,m-1]} T_{(i,n+1)} \Cxr m{(n-1)}
\Cx1{(m-1)}.
\end{equation}
Since $T_{(m-1,n+1)}$ commutes with $\Cxr m{(n-1)}$, we can rewrite the right hand side as
\begin{align*}
\ascprod_{i\in[1,m-2]} T_{(i,n+1)} &\Cxr m{(n-1)} T_{(m-1,n+1)}\Cx1{(m-1)}\\
&=\ascprod_{i\in[1,m-2]} T_{(i,n+1)} \Cxr m{(n-1)} T_{m-1} T_{(m,n+1)}T_{m-1}\Cx1{(m-1)}
\\
&=\ascprod_{i\in[1,m-2]} T_{(i,n+1)} \Cxr{(m-1)}{(n-1)} T_{(m,n+1)}T_{m-1}\Cx1{(m-1)}\\
&=\ascprod_{i\in[1,m-2]} T_{(i,n+1)} \Cxr{(m-1)}{(n-1)} T_{(m,n+1)}\Cx1{(m-1)}T_{m-2}.
\end{align*}
Therefore, \eqref{eq:equiv id for reorder''} is equivalent to
\begin{equation}
\label{eq:equiv id for reorder'''}
\Cxr mn\ascprod_{i\in[1,m-1]} T_{(i,n+1)} \Cx1{(m-3)}=\ascprod_{i\in[1,m-2]} T_{(i,n+1)} \Cxr{(m-1)}{(n-1)} T_{(m,n+1)}\Cx1{(m-1)}.
\end{equation}
Suppose we proved that~\eqref{eq:equiv id for reorder''} is equivalent to
\begin{equation}
\label{eq:equiv id for reorder'v}
\Cxr mn\ascprod_{j\in[1,m-1]} T_{(j,n+1)} \Cx1{(i-1)}=\ascprod_{j\in[1,i]} T_{(j,n+1)} \Cxr{(i+1)}{(n-1)} \ascprod_{j\in[i+2,m]} T_{(j,n+1)}\Cx1{(m-1)}.
\end{equation}
for some~$i\in[3,m-2]$.
We can rewrite the right hand side of~\eqref{eq:equiv id for reorder'v} as
\begin{align*}
\ascprod_{j\in[1,i-1]}& T_{(j,n+1)} \Cxr{(i+1)}{(n-1)} T_{(i,n+1)} \ascprod_{j\in[i+2,m]} T_{(j,n+1)}\Cx1{(m-1)}\\
&=\ascprod_{j\in[1,i-1]} T_{(j,n+1)} \Cxr{i}{(n-1)} T_{(i+1,n+1)} T_i \ascprod_{j\in[i+2,m]} T_{(j,n+1)}\Cx1{(m-1)}\\
&=\ascprod_{j\in[1,i-1]} T_{(j,n+1)} \Cxr{i}{(n-1)} T_{(i+1,n+1)} \ascprod_{j\in[i+2,m]} T_{(j,n+1)}T_i \Cx1{(m-1)}\\
&=\ascprod_{j\in[1,i-1]} T_{(j,n+1)} \Cxr{i}{(n-1)} T_{(i+1,n+1)} \ascprod_{j\in[i+2,m]} T_{(j,n+1)}\Cx1{(i-2)}T_i T_{i-1} T_i \Cx{(i+1)}{(m-1)}\\
&=\ascprod_{j\in[1,i-1]} T_{(j,n+1)} \Cxr{i}{(n-1)} T_{(i+1,n+1)} \ascprod_{j\in[i+2,m]} T_{(j,n+1)}\Cx1{(m-1)}T_{i-1}
\end{align*}
whence~\eqref{eq:equiv id for reorder'''} is equivalent to
$$
\Cxr mn\ascprod_{j\in[1,m-1]} T_{(j,n+1)} \Cx1{(i-2)}=\ascprod_{j\in[1,i-1]} T_{(j,n+1)} \Cxr{i}{(n-1)} \ascprod_{j\in[i+1,m]} T_{(j,n+1)}\Cx1{(m-1)}.
$$
Thus, \eqref{eq:equiv id for reorder'''} is equivalent to~\eqref{eq:equiv id for reorder'v} for all~$i\in[2,m-1]$. Taking~$i=2$
we conclude that~\eqref{eq:equiv id for reorder'''} is equivalent to
\begin{equation}\label{eq:last reduction}
\Cxr mn\ascprod_{j\in[1,m-1]} T_{(j,n+1)}=T_{(1,n+1)} \Cxr{2}{(n-1)} \ascprod_{j\in[3,m]} T_{(j,n+1)}\Cx1{(m-1)}.
\end{equation}
We now rewrite the right hand side of this identity as
\begin{align*}
T_{(1,n+1)} &\Cxr{2}{(n-1)} \ascprod_{j\in[3,m]} T_{(j,n+1)}\Cx1{(m-1)}\\
&=\Cxr1n\Cx 2n\Cxr2{(n-1)}\ascprod_{j\in[3,m]} T_{(j,n+1)}\Cx1{(m-1)}
=\Cxr1n \ascprod_{j\in[2,m]} T_{(j,n+1)}\Cx1{(m-1)}\\
&=\Cxr2n T_{(1,n+1)} \ascprod_{j\in[3,m]} T_{(j,n+1)}\Cx2{(m-1)}.
\end{align*}
Suppose we proved that for some~$i\in [2,m-1]$
\begin{align*}
T_{(1,n+1)} &\Cxr{2}{(n-1)} \ascprod_{j\in[3,m]} T_{(j,n+1)}\Cx1{(m-1)}=
\Cxr in \ascprod_{j\in[1,m]\setminus\{i\}} T_{(j,n+1)}\Cx i{(m-1)}.
\end{align*}
Then
\begin{align*}
T_{(1,n+1)} &\Cxr{2}{(n-1)} \ascprod_{j\in[3,m]} T_{(j,n+1)}\Cx1{(m-1)}\\
&=\Cxr{(i+1)} n \ascprod_{j\in[1,i-1]} T_{(j,n+1)} T_i T_{(i+1,n+1)} T_i \ascprod_{j\in[i+2,m]} T_{(j,n+1)}\Cx{(i+1)}{(m-1)}\\
&=\Cxr{(i+1)} n \ascprod_{j\in[1,i-1]} T_{(j,n+1)} T_{(i,n+1)} \ascprod_{j\in[i+2,m]} T_{(j,n+1)}\Cx{(i+1)}{(m-1)}\\
&=\Cxr{(i+1)}n \ascprod_{j\in[1,m]\setminus\{i+1\}} T_{(j,n+1)}\Cx{(i+1)}{(m-1)}.
\end{align*}
Therefore,
$$
T_{(1,n+1)} \Cxr{2}{(n-1)} \ascprod_{j\in[3,m]} T_{(j,n+1)}\Cx1{(m-1)}=\Cxr mn \ascprod_{j\in[1,m-1]} T_{(j,n+1)},
$$
which is the left-hand side of~\eqref{eq:last reduction}.
\end{proof}
Since $T_{J_m}=T_{w_\circ^{[1,m]_2}}T_{w_\circ^{[2,m-1]_2}} T_{(m+1,n+1)}$,
the assertion is immediate from Lemma~\ref{lem:2nd identity} and
the definition~\eqref{eq:Z_r defn} of~$Z_m$.
\end{proof}

Proposition~\ref{prop:g J=1} for~$J$ with~$g(J)=1$ and~$|J|$ odd follows from Proposition~\ref{prop:TJm power central} by Lemma~\ref{lem:move bubble}.

Thus, if $J=[1,a]\cup [b+1,n+1]$ with $1\le a<b\le n$,
$T_J^{|J|}$ is the product of two commuting ${}^{op}$-invariant
elements of~$\Br_n$ and, therefore, is~${}^{op}$-invariant.
It remains to apply Lemma~\ref{lem:TJ hom}. This completes the proof of
Theorem~\ref{thm:adm I2m}.
\end{proof}

\begin{proof}[Proof of forward direction in Theorem~\ref{thm:main thm adm}]
Since~$g(J)=1$, $J=[1,a]\cup [b+1,n+1]$ for some~$1\le a<b\le n+1$.
Let~$K=[a,b]$. Then, in the notation of Theorem~\ref{thm:main thm adm},
\begin{align*}
I'(K)&=[1,a-2]_2\cup \{ r\in [b+1,n+1]\,:\,
\overline{r-b}=0\},\\
I''(K)&=I\setminus (I'(K)\cup K)
=[1,a-1]_2\cup \{ r\in [b+1,n+1]\,:\, \overline{r-b}=1\},
\end{align*}
and so $\tau_{\overline a}(J)=T_{w_\circ^{I'(K)\cup K}}$,
$\tau_{\overline{a-1}}(J)=T_{w_\circ^{I''(K)}}$.
If~$\sigma(K)=K$, that is, $b=n+1-a$ then
$m(K)=\frac12(n-|K|)+1=a$ while~$|J|=2a$ and so~$m(K)=|J|/2$.
Otherwise, $m(K)=n-|K|+2=n-b+a+1=|J|$. Then
Theorem~\ref{thm:adm I2m} yields the desired homomorphism.
\end{proof}

Given $J\subset [1,n+1]$, $J=\{j_1,\dots,j_m\}$ with~$j_1<\cdots<j_m$, let $\Br^+_{n+1}[J]$ be the submonoid of~$\Br^+_{n+1}$
generated by the $T_{(j_k,j_{k+1})}$, $1\le k\le m-1$.
\begin{corollary}\label{cor:centrailty}
Let~$\{1,n+1\}\subset J\subset [1,n+1]$ with~$g(J)=1$. Then~$T_J^{|J|}$
is central in~$\Br^+_{n+1}[J]$.
\end{corollary}
\begin{proof}
Let~$J=[1,a]\cup[b+1,n+1]$, $1\le a<b\le n$. Then~$\Br^+(J)$
is generated by the $T_i$, $i\in [1,n]\setminus [a,b]$
and by $T_{(a,b+1)}$. By Proposition~\ref{prop:g J=1},
$T_J^{|J|}=T_{w_\circ^{[1,n]}}^2 T_{w_\circ^{[a+1,b-1]}}^{-2}$.
Since $T_{w_\circ^{[1,n]}}^2$ is central in~$\Br^+_{n+1}$
and $T_{w_\circ^{[a+1,b-1]}}^2$ commutes with the $T_i$,
$i\in [1,a-1]\cup[b+1,n]=[1,n]\setminus [a,b]$ and with $T_{(a,b+1)}$
by Proposition~\ref{prop:elem prop transp}, the assertion follows.
\end{proof}
\begin{proposition}\label{prop:Cox hom from TJ}
Let~$\{1,n+1\}\subset J\subset [1,n+1]$ with~$g(J)=1$ and let~$m=m(J)$ be 
as in Theorem~\ref{thm:adm I2m}. Then the
assignments $\wh T_r\mapsto \tilde\tau_{\overline r}(J)$, $r\in\{1,2\}$ define an optimal 
parabolic Coxeter type homomorphism~$\Phi:
\Br^+(I_2(2m))\to \Br^+_{n+1}$.
\end{proposition}
\begin{proof}
These assignments define a homomorphism
by Theorem~\ref{thm:adm I2m} and Lemma~\ref{lem:TJ hom}.
Since 
$\pi_n(\tilde\tau_r(J))$, $r\in\{0,1\}$ are manifestly involutions,
being products of commuting transpositions,
$\Phi$ is of Coxeter type by Proposition~\partref{prop:elem prop Coxeter Hecke.a}
and is optimal by Corollary~\ref{cor:TJ coxeter}.

The parabolicity of~$\Phi$ when $\tilde\sigma(J)=J$
follows from Proposition~\partref{prop:admissible hom from BrI22m to BrBn.a} and 
from that of the standard unfolding~$\Br^+(B_n)\to 
\Br^+(A_N)$, $n=\lfloor \frac12 N\rfloor$ (cf.~\eqref{eq:unfold Bn A2n}, \eqref{eq:unfold Bn A2n-1}).

Suppose now that~$\tilde\sigma(J)\not=J$.
Write $J=[1,a]\cup[b+1,n+1]$, $1\le a<b\le n$. 
By Proposition~\ref{prop:g J=1}, we have
in~$\Br_n$
\begin{align*}
\Phi(\wh T_{w_\circ^{\{1,2\}}})&=T_J^{|J|}
=T_{w_\circ^{[1,n]}}^2 T_{w_\circ^{[a+1,b-1]}}^{-2}
=T_{w_\circ^{[a+1,b-1]}}^{-1}
T_{w_\circ^{[1,n]}}
T_{w_\circ^{\sigma([a+1,b-1])}}^{-1}T_{w_\circ^{[1,n]}}\\
&=T_{w_{[a+1,b-1]}}T_{w_{\sigma([a+1,b-1])}}.
\end{align*}
Furthermore, let~$J_0=[1,a-2]_2\cup\{ r\in [b+1,n+1]\,:\,
\overline r=\overline b\}$ and let~$J_1=
[1,a-1]_2\cup ([b+1,n+1]\setminus J_0)$. Then in~$\Br_n$
\begin{align*}
\Phi(\wh T_{\{2-\overline a\}})&=
\tilde\tau_{\overline a}(J)^{-1}
T_{w_\circ^{[a+1,b-1]}}^{-1}
T_{w_\circ^{[1,n]}}
T_{w_\circ^{\sigma([a+1,b-1])}}^{-1}T_{w_\circ^{[1,n]}}\\
&=T_{w_\circ^{J_{\overline a}}}^{-1}T_{(a,b+1)}^{-1}
T_{w_\circ^{[a+1,b-1]}}^{-1}
T_{w_\circ^{[1,n]}}
T_{w_\circ^{\sigma([a+1,b-1])}}^{-1}T_{w_\circ^{[1,n]}}\\
&=T_{w_\circ^{J_{\overline a}\cup [a,b]}}^{-1}
T_{w_\circ^{[1,n]}}
T_{w_\circ^{\sigma([a+1,b-1])}}^{-1}T_{w_\circ^{[1,n]}}
=T_{w_{J_{\overline a}\cup [a,b]}}T_{w_{\sigma([a+1,b-1])}},\\
\intertext{while}
\Phi(\wh T_{\{1+\overline a\}})&=
\tilde\tau_{1-\overline a}(J)^{-1}
T_{w_\circ^{[a+1,b-1]}}^{-1}
T_{w_\circ^{[1,n]}}
T_{w_\circ^{\sigma([a+1,b-1])}}^{-1}T_{w_\circ^{[1,n]}}\\
&=T_{w_\circ^{J_{1-\overline a}}}^{-1}
T_{w_\circ^{[a+1,b-1]}}^{-1}
T_{w_\circ^{[1,n]}}
T_{w_\circ^{\sigma([a+1,b-1])}}^{-1}T_{w_\circ^{[1,n]}}
=T_{w_{J_{1-\overline a}\cup [a+1,b-1]}}T_{w_{\sigma([a+1,b-1])}}.\qedhere
\end{align*}
\end{proof}

\subsection{Symmetrized Burau representation and the converse}\label{subs:Burau} We now prove the converse in Theorem~\ref{thm:main thm adm}. The key ingredient is the following
\begin{theorem}\label{thm:adm I2m converse }
Let~$\{1,n+1\}\subset J\subset [1,n+1]$ with~$g(J)>1$. Then $T_J^m$ is not ${}^{op}$-invariant
for all~$m\in\mathbb Z_{>0}$.
\end{theorem}

To prove this Theorem, we use representation theory of braid monoids.
Let~$\kk$ be a field of characteristic zero and let~$q\in\kk^\times$
which is not a root of unity.
Let \plink{ei}$\{e_i\}_{1\le i\le n+1}$ be the standard basis of~$V=\kk^{n+1}$. Define~$T_i\in\End V$
by
$$
T_i(e_j)=e_j-(q \delta_{i,j}-\delta_{i+1,j})(q e_i-e_{i+1})
$$
for all $j\in [1,n+1]$, $i\in[1,n]$. It easy to see that the~$T_i$
are invertible with
$$
T_i^{-1}(e_j)=e_j-(\delta_{i,j}-q^{-1}\delta_{i+1,j})(e_i-q^{-1}e_{i+1}).
$$
\begin{proposition}\label{prop:Burau}
The operators $T_i$, $i\in [1,n]$ provide a representation of~$\Br_n$ on~$V$. Moreover, $T_i^2=(1-q^{2})T_i+q^{2}\id_V$, $T_i^{-1}=q^{-2} T_i+(1-q^{-2})\id_V$ and for any $T\in\Br_n$, the matrix of $T^{op}$
in the standard basis~$\{e_i\}_{1\le i\le n+1}$ is the transpose of the matrix of~$T$.
\end{proposition}
\begin{proof}
The operators~$T_i$ are easily seen to be modified Burau operators
(\cite{Bu}).
Clearly, $T_iT_j(e_k)=T_jT_i(e_k)$, $1\le k\le n+1$,
if~$|i-j|>1$. Since
\begin{align*}
T_i T_{i+1}T_i(e_j)&=T_i T_{i+1}(e_j-(\delta_{i,j}-q^{-1}\delta_{i+1,j})(e_i-q^{-1}e_{i+1})\\
&=T_i(e_j-(\delta_{i,j}-q^{-1}\delta_{i+1,j})e_i
-(\delta_{i+1,j}-q^{-1}\delta_{i+2,j})e_{i+1}
\\&\qquad+(q^{-2}\delta_{i,j}+q^{-1}(1-q^{-2})\delta_{i+1,j}-
q^{-2}\delta_{i+2,j})e_{i+2})\\
&=e_j-(\delta_{i,j}-q^{-2}\delta_{i+1,j})e_i-(1-q^{-2})(\delta _{i+1,j}-q^{-1}\delta _{i+2,j})e_{i+1}\\
&\qquad+(q^{-2}\delta_{i,j}+q^{-1}(1-q^{-2})\delta_{i+1,j}-
q^{-2}\delta_{i+2,j})e_{i+2})
\\
\intertext{and}
T_{i+1}T_iT_{i+1}(e_j)&=T_{i+1}T_i(e_j-(\delta_{i+1,j}-q^{-1}\delta_{i+2,j})(e_{i+1}-q^{-1}e_{i+2}))\\
 &=T_{i+1}(e_j-(\delta _{i,j}-q^{-2}\delta _{i+2,j})e_i+(q^{-1}\delta _{i,j}-\delta _{i+1,j}+q^{-1}(1-q^{-2})\delta _{i+2,j})e_{i+1}\\
&\qquad+q^{-1}(\delta_{i+1,j}-q^{-1}\delta_{i+2,j})e_{i+2})\\
 &=e_j-(\delta _{i,j}-q^{-2}\delta _{i+2,j})e_i
 -(1-q^{-2})(\delta _{i+1,j}-q^{-1}\delta _{i+2,j})e_{i+1}\\
 &\qquad+(q^{-2}\delta_{i,j}+q^{-1}(1-q^{-2})\delta_{i+1,j}-
 q^{-2}\delta_{i+2,j})e_{i+2}),
\end{align*}
for all $i\in[1,n-1]$, $j\in[1,n+1]$, it follows that all $T_i T_{i+1} T_i=T_{i+1}T_i T_{i+1}$.
Furthermore,
\begin{align*}
T_i^2(e_j)&=T_i(e_j-(q \delta_{i,j}-\delta_{i+1,j})(q e_i-e_{i+1})%
=e_j-(q\delta_{i,j}-\delta_{i+1,j})(1-q^{2})(q e_i-e_{i+1})\\
&=(1-q^2)T_i(e_j)+q^{2}e_j.
\end{align*}
The identity for~$T_i^{-1}$ is obvious. The last assertion follows since the matrix of~$T_i$ with respect to the
standard basis is obviously symmetric and thus the restriction of the transpose to the image of~$\Br_n$ in~$\End V$ coincides with~${}^{op}$.
\end{proof}

Given $i,j\in[1,n+1]$, define \plink{v[i,j]}$v_{[i,j]}=\sum_{i\le t\le j} q^t e_t$.

\begin{lemma}%
\label{lem:fixed point}
For all~$i\le j\in[1,n+1]$, $T(v_{[i,j]})=v_{[i,j]}$
for all $T\in\Br_n$ with $\supp(T)\subset [i,j-1]$. In particular, $T(v_{[1,n+1]})=v_{[1,n+1]}$ for all~$T\in\Br_n$.
\end{lemma}
\begin{proof}
It suffices to prove that~$T_k(v_{[i,j]})=v_{[i,j]}$ for all~$k\in [i,j-1]$. We have
\begin{equation*}
T_k(v_{[i,j]})=v_{[i,j]}-q^{k+1}(q e_k-e_{k+1})+q^{k+1}(q e_k-e_{k+1})=v_{[i,j]}.\qedhere
\end{equation*}
\end{proof}
Denote \plink{ui}$u_i=e_i-q^{-1}e_{i+1}$, $i\in [1,n]$ and let~$w^{(a)}_{[i,j]}=\sum_{i\le t\le j}q^{a t}u_t$, $a\in\mathbb Z$.
It is immediate from the definitions that for all~$i\in [1,n]$,
$T^{\pm1}_i(u_k)=u_k$ if $k\in[1,n+1]\setminus\{i-1,i,i+1\}$ while
\begin{equation}\label{eq:ind base cox w0}
\begin{array}{ll}
T^{\pm 1}_i(u_{i-1})=u_{i-1}+q^{\pm1} u_i,&  i\in [2,n],\\
T^{\pm 1}_i(u_i)=-q^{\pm 2} u_i,&i\in [1,n],\\
T^{\pm 1}_{i}(u_{i+1})=q^{\pm1} u_i+u_{i+1},&i\in[1,n-1].
\end{array}
\end{equation}

\begin{lemma}\label{lem:rev Cox act}
Let $i\le j\in [1,n]$, $k\in[1,n]$ and~$\epsilon\in\{1,-1\}$. Then
$$
\Cxr ij^{(\epsilon)}
(u_k)=\begin{cases}
u_k,&k\in[1,n]\setminus [i-1,j+1],\\
q^{\epsilon}u_{k-1}+\delta_{k,j+1}u_{j+1},&k\in[i+1,j+1],\\
q^{\epsilon(1-i)} w^{(\epsilon)}_{[i-1,j]},&k=i-1>0,\\
-q^{\epsilon(2-i)} w^{(\epsilon)}_{[i,j]},&k=i.
\end{cases}
$$
\end{lemma}
\begin{proof}
Since~$\supp(\Cxr{i}{ j}^{(\epsilon)})=[i,j]$, the assertion is obvious for~$k\in[1,n+1]\setminus [i-1,j+1]$. To prove it for~$k\in[i-1,j+1]$ we use induction on~$j-i$.
The induction base $j-i=0$ is~\eqref{eq:ind base cox w0}.
For the inductive step, for~$k\in [i+1,j]$ we have
\begin{align*}
\Cxr{i}{ j}^{(\epsilon)}(u_k)&=\Cxr{(k-1)}{ j}^{(\epsilon)}\Cxr{i}{(k-2)}^{(\epsilon)}(u_k)=\Cxr{k}{ j}^{(\epsilon)}T^{\epsilon}_{k-1}(u_k)
=\Cxr{k}{ j}^{(\epsilon)}(q^{\epsilon} u_{k-1}+u_{k})\\
&=\Cxr{(k+1)}{ j}^{(\epsilon)}(q^{\epsilon}(u_{k-1}+q^{\epsilon}u_k)-q^{2\epsilon} u_k)
=q^{\epsilon}u_{k-1}
\end{align*}
while for~$k=j+1$,
$$
\Cxr{i}{ j}^{(\epsilon)}(u_{j+1})=T_j^{\epsilon}(u_{j+1})=q^{\epsilon}u_j+u_{j+1}.
$$
For~$k=i-1$, using~\eqref{eq:ind base cox w0} and the induction hypothesis we obtain
\begin{align*}
\Cxr{i}{ j}^{(\epsilon)}(u_{i-1})&=
u_{i-1}+q^{\epsilon} \Cxr{(i+1)}{ j}^{(\epsilon)}(u_i)%
=u_{i-1}+q^{\epsilon(1-i)}w^{(\epsilon)}_{[i,j]}=q^{\epsilon(1-i)}w^{(\epsilon)}_{[i-1,j]}.
\end{align*}
Finally, for~$k=i$, \eqref{eq:ind base cox w0} and the induction hypothesis yield
\begin{align*}
\Cxr{i}{ j}^{(\epsilon)}(u_i)&=-q^{2\epsilon}\Cxr{(i+1)}{ j}^{(\epsilon)}(u_i)=-q^{\epsilon(2-i)} w^{(\epsilon)}_{[i,j]}.\qedhere
\end{align*}
\end{proof}
\begin{corollary}\label{cor:cox w +-}
Let $i\le j,k\le l\in [1,n]$ and $\epsilon,\epsilon'\in\{1,-1\}$. Then
$$
\Cxr ij^{(\epsilon)}(w^{(\epsilon')}_{[k,l]})=
\begin{cases}
w^{(\epsilon')}_{[k,l]},&k>j+1,\\
q^{\epsilon+\epsilon'}w^{(\epsilon')}_{[k-1,\min(l-1,j)]}+w^{(\epsilon')}_{[j+1,l]},&k\in[i+1,j+1],\\
q^{\epsilon+\epsilon'} w^{(\epsilon')}_{[i,\min(l-1,j]]}-
q^{2\epsilon+(\epsilon'-\epsilon)i}w^{(\epsilon)}_{[i,j]}
+w^{(\epsilon')}_{[j+1,l]},&k=i,\\
w^{(\epsilon')}_{[k,i-1]}+
q^{\epsilon+\epsilon'}w^{(\epsilon')}_{[i,\min(l-1,j)]}+w^{(\epsilon')}_{[j+1,l]}\\
\quad+q^{(\epsilon'-\epsilon)i}(q^{\epsilon-\epsilon'}-q^{2\epsilon}) w^{(\epsilon)}_{[i,j]},&k<i.
\end{cases}
$$
\end{corollary}
\begin{proof}
The first two cases are immediate from Lemma~\ref{lem:rev Cox act}. If~$k=i$,
\begin{align*}
\Cxr ij^{(\epsilon)}(w^{(\epsilon')}_{[i,l]})
&=q^{\epsilon' i}\Cxr ij^{(\epsilon)}(u_i)+q^{\epsilon+\epsilon'}w^{(\epsilon')}_{[i,\min(l-1,j)]}+w^{(\epsilon')}_{[j+1,l]}\\
&=-q^{(\epsilon'-\epsilon)i+2\epsilon} w^{(\epsilon)}_{[i,j]}+
q^{\epsilon+\epsilon'}w^{(\epsilon')}_{[i,\min(l-1,j)]}+w^{(\epsilon')}_{[j+1,l]}.
\end{align*}
Finally, if $k<i$,
\begin{align*}
\Cxr ij^{(\epsilon)}(w^{(\epsilon')}_{[k,l]})&=
w^{(\epsilon')}_{[k,i-2]}+q^{\epsilon'(i-1)}\Cxr ij^{(\epsilon)}(u_{i-1})-q^{(\epsilon'-\epsilon)i+2\epsilon} w^{(\epsilon)}_{[i,j]}+
q^{\epsilon+\epsilon'}w^{(\epsilon')}_{[i,\min(l-1,j)]}+w^{(\epsilon')}_{[j+1,l]}\\
&=
w^{(\epsilon')}_{[k,i-2]}+q^{(\epsilon'-\epsilon)(i-1)}
w^{(\epsilon)}_{[i-1,j]}-q^{(\epsilon'-\epsilon)i+2\epsilon} w^{(\epsilon)}_{[i,j]}+
q^{\epsilon+\epsilon'}w^{(\epsilon')}_{[i,\min(l-1,j)]}+w^{(\epsilon')}_{[j+1,l]}\\
&=w^{(\epsilon')}_{[k,i-1]}+
q^{(\epsilon'-\epsilon)i}(q^{\epsilon-\epsilon'}-q^{2\epsilon}) w^{(\epsilon)}_{[i,j]}+
q^{\epsilon+\epsilon'}w^{(\epsilon')}_{[i,\min(l-1,j)]}+w^{(\epsilon')}_{[j+1,l]}.\qedhere
\end{align*}
\end{proof}

\begin{lemma}\label{lem:Tw0ij act ui}
For all $i\le j\in[1,n]$, $k\in[1,n]$, $\epsilon\in\{1,-1\}$
\begin{equation}
T_{w_\circ^{[i,j]}}^{\epsilon}(u_k)=\begin{cases}
u_k,&k\in[1,n]\setminus[i-1,j+1],\\
-q^{\epsilon(j-i+2)}u_{i+j-k},&k\in[i,j],\\
q^{\epsilon(j+1)}w^{(-\epsilon)}_{[i,j+1]},&k=j+1,\\
q^{-\epsilon(i-1)}w^{(\epsilon)}_{[i-1,j]},&k=i-1.
\end{cases}\label{eq:Tw0 act uk}
\end{equation}
In particular,
\begin{equation}\label{eq:Tw0^2 eigen}\{ u_k\,:\, k\in[i,j]\}\subset
\ker(T_{w_\circ^{[i,j]}}^2-q^{2(j-i+2)}\id_V)
\end{equation}
and
\begin{equation}\label{eq:Tw0 eigen}
w^{(\epsilon)}_{[k,l]}\in\ker(T_{w_\circ^{[i,j]}}-\id_V),\qquad [i-1,j+1]\subset [k,l]\subset [1,n+1].
\end{equation}
\end{lemma}
\begin{proof}
The first case in~\eqref{eq:Tw0 act uk} is obvious. For the remaining cases, we use induction on~$j-i$.
The case $j=i$ has already been established in~\eqref{eq:ind base cox w0}.

Note that $T_{w_\circ^{[i,j]}}^{\epsilon}=T_{w_\circ^{[i,j]}}^{\epsilon}\Cxr{i}{ j}^{(\epsilon)}$. Indeed,
$T_{w_\circ^{[i,j]}}=T_{w_\circ^{[i,j]}}\Cxr ij$ while, since $T_{w_\circ^{[i,j]}}$ is
${}^{op}$-invariant by~\cite{BrSa}*{Lemma~5.1}, $T_{w_\circ^{[i,j]}}^{-1}=(\Cx ij T_{w_\circ^{[i,j-1]}})^{-1}=T_{w_\circ^{[i,j-1]}}^{-1}\Cx ij^{-1}$.

For the inductive step, for~$k\in[i+1,j]$
we have by Lemma~\ref{lem:rev Cox act} and the induction hypothesis
\begin{align*}
T_{w_\circ^{[i,j]}}^{\epsilon}(u_k)&=T_{w_\circ^{[i,j-1]}}^{\epsilon}\Cxr{i}{ j}^{(\epsilon)}(u_k)=q^{\epsilon} T_{w_\circ^{[i,j-1]}}^{\epsilon}(u_{k-1})%
=-q^{\epsilon(j-i+2)}u_{i+j-k},
\end{align*}
while for~$k=i$
\begin{align}
T_{w_\circ^{[i,j]}}^{\epsilon}(u_i)&=T_{w_\circ^{[i,j-1]}}^{\epsilon}\Cxr{i}{ j}^{(\epsilon)}(u_i)=-q^{\epsilon(2-i)} T_{w_\circ^{[i,j-1]}}^{\epsilon}(w^{(\epsilon)}_{[i,j]})
\nonumber\\
&=q^{\epsilon(2-i)}\sum_{i\le t\le j-1}q^{\epsilon(t-i+j+1)}u_{i+j-1-t}-q^{\epsilon(2+j-i)} T_{w_\circ^{[i,j-1]}}^{\epsilon}(u_j)
\nonumber\\
&=q^{\epsilon(2+2j-i)} w^{(-\epsilon1)}_{[i,j-1]}-q^{\epsilon(2+2j-i)} w^{(-\epsilon1)}_{[i,j]}%
=-q^{\epsilon(j-i+2)}u_{j}.\label{eq:Tw0ij ui}
\end{align}
Furthermore,
\begin{align*}
T_{w_\circ^{[i,j]}}^{\epsilon}(u_{j+1})&=T_{w_\circ^{[i,j-1]}}^{\epsilon}\Cxr{i}{ j}^{(\epsilon)}(u_{j+1})=T_{w_\circ^{[i,j-1]}}^{\epsilon}(q^{\epsilon} u_j+u_{j+1})\\
&=q^{\epsilon(j+1)}w^{(-\epsilon1)}_{[i,j]}+u_{j+1}=q^{\epsilon(j+1)} w^{(-\epsilon1)}_{[i,j+1]}.
\end{align*}
Finally, by Lemma~\ref{lem:rev Cox act}
\begin{align*}
T_{w_\circ^{[i,j]}}^{\epsilon}(u_{i-1})&=T_{w_\circ^{[i,j-1]}}^{\epsilon}\Cxr{i}{ j}^{(\epsilon)}(u_{i-1})
=q^{-\epsilon(i-1)} T_{w_\circ^{[i,j-1]}}^{\epsilon}(w^{(\epsilon)}_{[i-1,j]}).%
\end{align*}
As we already established in~\eqref{eq:Tw0ij ui},
$$
T_{w_\circ^{[i,j-1]}}^{\epsilon}(w^{(\epsilon)}_{[i,j]})=q^{\epsilon j} u_j.
$$
Using the induction hypothesis, we obtain
\begin{align*}
T_{w_\circ^{[i,j]}}^{\epsilon}(u_{i-1})&=T_{w_\circ^{[i,j-1]}}^{\epsilon}(u_{i-1})+q^{\epsilon(j-i+1)}u_j=q^{-\epsilon(i-1)}(w^{(\epsilon)}_{[i-1,j-1]}+q^{\epsilon j}u_j)%
=q^{-\epsilon(i-1)}w^{(\epsilon)}_{[i-1,j]}.
\end{align*}

The inclusion in~\eqref{eq:Tw0^2 eigen} is immediate. It is also clear that it suffices to prove~\eqref{eq:Tw0 eigen} for~$k=i-1$, $l=j+1$. Note that
$$
T_{w_\circ^{[i,j]}}^{-\epsilon}(u_{j+1})=q^{-\epsilon(j+1)}w^{(\epsilon)}_{[i-1,j+1]}-q^{-\epsilon(j-i+2)} u_{i-1}
$$
while
$$
T_{w_\circ^{[i,j]}}^{\epsilon}(u_{i-1})=q^{-\epsilon(i-1)}w^{(\epsilon)}_{[i-1,j+1]}-q^{\epsilon(j-i+2)}u_{j+1}
$$
whence
$$
w^{(\epsilon)}_{[i-1,j+1]}=q^{\epsilon(i-1)} T_{w_\circ^{[i,j]}}^{\epsilon}(u_{i-1})+q^{\epsilon(j+1)}u_{j+1}=
q^{\epsilon(j+1)}T_{w_\circ^{[i,j]}}^{-\epsilon 1}(u_{j+1})+q^{\epsilon(i-1)}u_{i-1}
$$
and so
$$
T_{w_\circ^{[i,j]}}^{\epsilon}(w^{(\epsilon)}_{[i-1,j+1]})=q^{\epsilon(j+1)}u_{j+1}+q^{\epsilon(i-1)}T_{w_\circ^{[i,j]}}^{\epsilon}(u_{i-1})
=w^{(\epsilon)}_{[i-1,j+1]},
$$
which immediately yields~\eqref{eq:Tw0 eigen}.
\end{proof}

\begin{lemma}\label{lem:Transp action}
For all $1\le i<j\le n$, $k\in[1,n]$,
$$
T_{(i,j+1)}(u_k)=\begin{cases}
				u_k,& k\in [1,n+1]\setminus [i-1,j+1],\\
				q^2 u_k,&k\in [i+1,j-1],\\
				-q^{j+2} w^{(-1)}_{[i,j-1]},&k=j,\\
				q^{j+1} w^{(-1)}_{[i,j+1]},&k=j+1,\\
				-q^{2-i} w^{(1)}_{[i+1,j]},&k=i,\\
				q^{1-i} w^{(1)}_{[i-1,j]},&k=i-1.
                 \end{cases}
$$
\end{lemma}
\begin{proof}
The assertion is obvious when $k\in[1,n+1]\setminus[i,j+1]$. For~$k\in [i-1,j+1]$
we use induction on~$j-i$. If~$j=i+1$, $T_{(i,j+1)}=T_{w_\circ^{[i,i+1]}}$
and the assertion follows from Lemma~\ref{lem:Tw0ij act ui}.

For the inductive step, recall that $T_{(i,j+1)}=T_i T_{(i+1,j+1)}T_i=T_j T_{(i,j)}T_j$.
Suppose first that~$k\in[i+2,j-1]$. Then
$$
T_{(i,j+1)}(u_k)=T_i T_{(i+1,j+1)}(u_k)=q^2 T_i(u_k)=q^2 u_k,
$$
while for~$k=i+1$ by~\eqref{eq:ind base cox w0} and the induction hypothesis
\begin{align*}
T_{(i,j+1)}(u_{i+1})&=T_i T_{(i+1,j+1)}(q u_i+u_{i+1})=T_i( q^{1-i}w^{(1)}_{[i,j]}-q^{1-i}w^{(1)}_{[i+2,j]})\\
&=T_i(q u_i+q^2 u_{i+1})=-q^3 u_i+q^2(q u_i+u_{i+1})=q^2 u_{i+1}.
\end{align*}
Furthermore, since $j>i+1$ and so $T_i(u_k)=u_k$ for~$k\in \{j,j+1\}$,
\begin{align*}
T_{(i,j+1)}(u_{j+1})%
&=q^{j+1} T_i(w^{(-1)}_{[i+1,j+1]})%
=q^{j-i} T_i(u_{i+1})+q^{j+1}w^{(-1)}_{[i+2,j+1]}\\
&=q^{j-i}(q u_i+u_{i+1})+q^{j+1}w^{(-1)}_{[i+2,j+1]}%
=q^{j+1} w^{(-1)}_{[i,j+1]},
\intertext{while}
T_{(i,j+1)}(u_j)&=-q^{j+2} T_i(w^{(-1)}_{[i+1,j-1]})=-q^{j+1-i} T_i(u_{i+1})-q^{j+2} w^{(-1)}_{[i+2,j-1]}\\
&=-q^{j+1-i}(q u_i+u_{i+1})-q^{j+2} w^{(-1)}_{[i+2,j-1]}=-q^{j+2} w^{(-1)}_{[i,j-1]}.
\end{align*}
Finally, for~$k\in \{i-1,i\}$, $T_j(u_k)=u_k$ and so
\begin{align*}
T_{(i,j+1)}(u_i)&=-q^{2-i} T_j ( w^{(1)}_{[i+1,j-1]})=-q^{2-i} w^{(1)}_{[i+1,j-2]}-q^{j+1-i} T_j(u_{j-1})\\
&=-q^{2-i} w^{(1)}_{[i+1,j-2]}-q^{j+1-i} (q u_{j-1}+u_j)%
=-q^{2-i} w^{(1)}_{[i+1,j]},
\intertext{while}
T_{(i,j+1})(u_{i-1})&=q^{1-i} T_j (w^{(1)}_{[i-1,j-1]})=q^{1-i}w^{(1)}_{[i-1,j-2]}+q^{j-i}T_j(u_{j-1})%
=q^{1-i}w^{(1)}_{[i-1,j]}.\qedhere
\end{align*}
\end{proof}

Now we describe eigenspaces of~$T_J^{|J|}$ for a special choice of~$J$.
\begin{proposition}\label{prop:TJ eigenvectors}
Let $m\in [1,n-1]$ and let~$J=[1,m]\cup \{n+1\}$.
Then
\begin{enmalph}
\item\label{prop:TJ eigenvectors.a} $\{ u_i\,:\,i\in[m+1,n-1]
\subset \ker(T_J-q^{2}\id_V)$;
\item\label{prop:TJ eigenvectors.b}
$\{u_i\,:\,i\in[1,m-1]\}\cup \{w^{(\epsilon)}_{[m,n]}\,:\,\epsilon\in\{1,-1\}\}$
is a basis of~$\ker(T_J^{|J|}-q^{2(n+1)}\id_V)$;
\item\label{prop:TJ eigenvectors.c} $T_J^{|J|}$ is diagonalizable on~$V$,
$\det(t\id_V-T_J^{|J|})=(t-1)(t-q^{2(n+1)})^{m+1}(t-q^{2m})^{n-m-1}$
and $\det(t\id_V-T_J)=(t-1)(t^{m+1}-q^{2(n+1)})(t-q^{2m})^{n-m-1}$.
\end{enmalph}
\end{proposition}
\begin{proof}
Let $i\in[m+1,n-1]$. Then $T_{(m,n+1)}(u_i)=q^2 u_i$ by
Lemma~\ref{lem:Transp action} and $T_j(u_i)=u_i$ for all~$j\in [1,m-1]$ by~\eqref{eq:ind base cox w0}.
Since $T_J$ is the product of~$T_{(m,n+1)}$ with the $T_j$, $j\in[1,m-1]$,
part~\ref{prop:TJ eigenvectors.a} follows.

We now prove~\ref{prop:TJ eigenvectors.b}.
By Proposition~\ref{prop:g J=1}, $T_J^{|J|}=T_{w_\circ^{[1,n+1]}}^2 T_{w_\circ^{[m+1,n-1]}}^{-2}$.
By Lemma~\ref{lem:Tw0ij act ui}, $T_{w_\circ^{[1,n]}}^2(u_k)=
q^{2(n+1)}u_k$ for all~$k\in [1,n]$, while $T_{w_\circ^{[m+1,n-1]}}(u_k)=u_k$,
$k\in[1,m-1]$ and $T_{w_\circ^{[m+1,n-1]}}(w^{(\pm1)}_{[m,n]})=w^{(\pm1)}_{[m,n]}$. Therefore,
$\{u_i\,:\,i\in [1,m-1]\}\cup\{w^{(1)}_{[m,n]},w^{(-1)}_{[m,n]}\}\subset \ker(T_J^{|J|}-q^{2(n+1)}\id_V)$.

Next we claim that $\{u_i\,:\,i\in [1,m-1]\}\cup\{w^{(1)}_{[m,n]},w^{(-1)}_{[m,n]}\}$
is linearly independent. Indeed, since
$w^{(-1)}_{[m,n]}=q^{-m}e_m-q^{-n-1}e_{n+1}$, $w^{(-1)}_m$ is not
contained in the span of the~$u_i$, $i\in [1,m-1]$, which are manifestly linearly independent. Since
the coefficient of~$e_i$, $i\in[m+1,n]$ in~$w^{(1)}_{[m,n]}$
is $q^i-q^{i-2}\not=0$, while the $u_i$, $i\in[1,m-1]$
and~$w^{(-1)}_{[m,n]}$ are contained in the span of~$\{ e_j\,:\,
j\in[1,m]\cup\{n+1\}\}$,
it follows that $w^{(1)}_{[m,n]}$, is not
in the span of~$\{ u_i\,:\,i\in[1,m-1]\}\cup\{w^{(-1)}_m\}$.

In particular, $\dim\ker(T_J^{|J|}-q^{2(n+1)}\id_V)\ge m+1$.
By part~\ref{prop:TJ eigenvectors.a}, $\dim
\ker(T_J^{|J|}-q^{2(m+1)})\ge n-m-1$. By Lemma~\ref{lem:fixed point}, $\dim\ker(T_J^{|J|}-\id_V)\ge 1$. Since $(n-m-1)+(m+1)+1=n+1=\dim V$,
and the sum of $\ker(T_J^{|J|}-\lambda\id_V)$ with
$\lambda\in\{1,q^{2(m+1)},q^{2(n+1)}\}$ is direct,
we conclude that
all these inequalities are in fact equalities.
Therefore, $\{ u_i\,:\, i\in[m+1,n]\}$ is a basis
of~$\ker(T_J^{|J|}-q^{2(m+1)}\id_V)$,
$\{u_i\,:\, i\in [1,m-1]\}\cup\{w^{(1)}_{[m,n]},w^{(-1)}_{[m,n]}\}$ is a
basis of~$\ker(T_J^{|J|}-q^{2(n+1)}\id_V)$ and
$v_{[1,n+1]}$ spans~$\ker(T_J^{|J|}-\id_V)$. The remaining assertions are
now immediate.
\end{proof}
The following is an immediate consequence of Proposition~\ref{prop:TJ eigenvectors} and Corollary~\ref{cor:conj J}.
\begin{corollary}\label{cor:TJ eigenvectors}
Let $\{1,n+1\}\subset J\subset [1,n+1]$ with $|J|\le n$. Then
$T_J^{|J|}$ is diagonalizable on~$V$ and
$$
\det(t\id_V-T_J^{|J|})=(t-1)(t-q^{2(n+1)})^{|J|}(t-q^{2|J|})^{n-|J|},
$$
while
$$
\det(t\id_V-T_J)=(t-1)(t^{|J|}-q^{2(n+1)})(t-q^{2})^{n-|J|}.
$$
In particular, $T_J$ is diagonalizable on~$V$ provided that~$\kk$ contains all $|J|$th roots of~$q^{2(n+1)}$. Finally,
$T_J$ is conjugate to~$T_{J'}$ in~$\Br_n$
if and only if~$|J|=|J'|$.
\end{corollary}

Let~\plink{<.|.>}$\la\cdot\vert\cdot\ra:V\tensor V\to \kk$ be the standard symmetric bilinear form defined by $\la e_i\vert e_j\ra=\delta_{i,j}$, $i,j\in [1,n+1]$. Then
\begin{equation}\label{eq:Euclid ui uj}
\la u_i\vert u_j\ra
=\delta_{i,j}(1+q^{-2})-q^{-1}(\delta_{i+1,j}+\delta_{i,j+1}),\qquad
i,j\in [1,n].
\end{equation}
Note the following
\begin{lemma}\label{lem:orth eigenspaces}
Let $X\in \Br_n$ be ${}^{op}$-invariant and let $\lambda\not=\mu\in\kk$.
Then $$\la\ker(X-\lambda\id_V)\,|\,\ker(X-\mu\id_V)\ra=\{0\}.
$$
\end{lemma}
\begin{proof}
The adjoint operator of~$X\in \End V$ with respect to~$\la\cdot|\cdot\ra$ is~$X^T$ which coincides with~$X^{op}$ for~$X\in\Br_n$ by Proposition~\ref{prop:Burau}. The assertion is now standard.
\end{proof}
\begin{proof}[Proof of Theorem~\ref{thm:adm I2m converse }]
By Corollary~\ref{cor:TJ coxeter}, $T_J^{k}$ is not ${}^{op}$-invariant
unless~$|J|$ divides~$2k$. Note that if~$|J|$ is even and
$X=T_J^{k|J|/2}$ is ${}^{op}$-invariant then so is~$X^2=T_J^{k|J|}$.
Thus, it suffices to prove that~$T_J^{k|J|}$ is not ${}^{op}$-invariant
for all~$k\in\mathbb Z_{>0}$.
By Lemma~\ref{lem:orth eigenspaces} and Corollary~\ref{cor:TJ eigenvectors}, it suffices to prove that
for all~$k\in\mathbb Z_{>0}$
$$
\la \ker(T_J^{k|J|}-q^{2k(n+1)}\id_V)\,\vert\,\ker(T_J^{k|J|}-q^{2k|J|}\id_V)\ra\not=\{0\}.
$$
Since~$\ker(T_J^{k|J|}-\lambda^k)=
\ker(T_J^{|J|}-\lambda)$ for all~$\lambda\in\kk$ by Corollary~\ref{cor:TJ eigenvectors}, it suffices to prove that
$$
\la \ker(T_J^{|J|}-q^{2(n+1)}\id_V)\,\vert\,\ker(T_J^{|J|}-q^{2|J|}\id_V)\ra\not=\{0\}.
$$
Let~$m=|J|-1$.
Since~$U(J)T_JU(J)^{-1}=T_{[1,m]\cup\{n+1\}}$ by Corollary~\ref{cor:conj J},
where~$U(J)$ is as defined in~\eqref{eq:U(J) defn}, it suffices to
prove the following
\begin{proposition}\label{prop:red}
Let~$\{1,n+1\}\subset J\subset [1,n+1]$ with~$g(J)>1$. Then
$$
\la U(J)^{-1}(u)\,\vert\, U(J)^{-1}(v)\ra\not=0
$$
for some~$u\in \ker(T_{[1,m]\cup\{n+1\}}^{m+1}-q^{2(m+1)}\id_V)$,
$v\in \ker(T_{[1,m]\cup\{n+1\}}^{m+1}-q^{2(n+1)}\id_V)$,
where~$m=|J|-1$.
\end{proposition}
\begin{proof}
Write
$J=\{j_0=1<j_1<\cdots<j_{m-1}<j_m=n+1\}$ and define\plink{betapm}
\begin{align*}
&\beta_-(J)=\max\{k\in[0,m-1]\,:\,j_k=k+1\}+1,\\
&\beta_+(J)=\min\{k\in[1,m-1]\,:\,j_t=j_{m-1}-m+t+1\}.
\end{align*}
Thus,
$$
J=[1,\beta_-(J)]\cup \{j_{\beta_-(J)},\dots,j_{\beta_+(J)-1}\}
\cup [j_{\beta_+(J)},j_{m-1}]\cup\{n+1\}.
$$
with~$j_{\beta_-(J)}\ge \beta_-(J)+2$ and~$j_{\beta_+(J)-1}\le j_{\beta_+(J)}-2$.
Note also that~$g(J)>1$ implies that~$|J|\le n-1$.

Given~$s\in\mathbb Z$ let \plink{qs}$q_s=q^{(-1)^s}$. We have
\begin{equation}\label{eq: q r+s}
q_{r+s}=q^{(-1)^{r+s}}=q_r^{(-1)^s},\qquad r,s\in\mathbb Z.
\end{equation}
\begin{lemma}\label{lem:TJ|J| eigenvectors}
Let~$J=\{j_0=1<j_1<\cdots<j_{m-1}<j_m=n+1\}\subset [1,n+1]$ where $m=|J|-1\le n-2$.
\begin{enmalph}
\item\label{lem:TJ|J| eigenvectors.a}
For all $r\in[1,m-1]$,
\begin{equation}\label{eq:UJ-1 ur}
U(J)^{-1}(u_{r})=\sum_{\beta_-(J)\le k\le r-1}
\sum_{j_{k-1}\le t\le j_k-1} q_r^{\overline{r-k}}
(q_r^{k-t}-q_r^{t-k})u_t
+\sum_{j_{r-1}\le t\le j_r-1} q_r^{r-t}u_t,
\end{equation}

\item\label{lem:TJ|J| eigenvectors.b} Suppose that~$g(J)>1$ and
$j_{m-1}<n$. Then~$j_{m-1}\ge m+1$ and
$$
U(J)^{-1}(u_{j_{m-1}})=\sum_{j_{\beta_+(J)}-1 \le t\le j_{m-1}} q_m^{\overline{j_{m-1}-t}} u_t.
$$
\item\label{lem:TJ|J| eigenvectors.c} If $j_{m-1}=n$ then
$$
U(J)^{-1}(w^{((-1)^{m+1})}_{[m+1,n-1]})=\sum_{\beta_-(J)\le k\le \beta_+(J)}
q_{m+1}^{m-k-\overline{m-k}} w^{((-1)^{m+1})}_{[j_{k-1}+\delta_{k,\beta_-(J)},j_k-1-\delta_{k,\beta_+(J)}]}.
$$
\end{enmalph}
\end{lemma}
\begin{proof}
Note that~$k+1\le j_k\le j_{k+1}-1$, $0\le k\le m-1$, and
so if~$j_k=k+1$ for some~$k>0$ then $j_s=s+1$ for all~$0\le s\le k$. %
In this proof we abbreviate $\beta_\pm=\beta_\pm(J)$. By definition of~$\beta_-$,
$$
U(J)^{-1}=\ascprod_{\beta_-\le k\le m-1} \Cxr{(k+1)}{(j_k-1)}^{((-1)^{k+1})}.
$$
Denote the right hand side of~\eqref{eq:UJ-1 ur} by~$S(J,r)$.
Then by Lemma~\ref{lem:rev Cox act}
\begin{align*}
U(J)^{-1}(u_r)&=\Big(\ascprod_{\beta_-\le k\le m-1} \Cxr{(k+1)}{(j_k-1)}^{((-1)^{k+1})}\Big)(u_r)=\Big(\ascprod_{\beta_-\le k\le r} \Cxr{(k+1)}{(j_k-1)}^{((-1)^{k+1})}\Big)(u_r).
\end{align*}
If~$r<\beta_-$ then~$U(J)^{-1}(u_r)=u_r$ while
$$
S(J,r)=\sum_{j_{r-1}\le t\le j_r-1} q_r^{r-t}u_t=\sum_{r\le t\le r} q_r^{r-t}u_t=u_r.
$$
Suppose that
$r\ge \beta_-$. Then by Lemma~\ref{lem:rev Cox act}
\begin{align*}
U(J)^{-1}(u_r)&=\Big(\ascprod_{\beta_-\le k\le r-1} \Cxr{(k+1)}{(j_k-1)}^{((-1)^{k+1})}\Big)\Cxr{(r+1)}{(j_r-1)}^{((-1)^{r+1})}(u_r)\\
&=\Big(\ascprod_{\beta_-\le k\le r-1} \Cxr{(k+1)}{(j_k-1)}^{((-1)^{k+1})}\Big)(\sum_{r\le t\le j_r-1} q_r^{r-t} u_t).
\end{align*}
If~$\beta_-=r$ then~$r=j_{r-1}$ and
$$
S(J,r)=\sum_{j_{r-1}\le t\le j_r-1} q_r^{r-t}u_t=\sum_{r\le t\le j_r-1} q_r^{r-t}u_t=U(J)^{-1}(u_r).
$$
If~$\beta_-<r$ then, again by Lemma~\ref{lem:rev Cox act}
\begin{align*}
U(&J)^{-1}(u_r)=\Big(\ascprod_{\beta_-\le k\le r-2} \Cxr{(k+1)}{(j_k-1)}^{((-1)^{k+1})}\Big)\Cxr{r}{(j_{r-1}-1)}^{((-1)^{r})}(\sum_{r\le t\le j_r-1} q_r^{r-t} u_t)\\
&=\Big(\ascprod_{\beta_-\le k\le r-2} \Cxr{(k+1)}{(j_k-1)}^{((-1)^{k+1})}\Big)\Big(\sum_{r+1\le t\le j_{r-1}-1} q_r^{r+1-t} u_{t-1}\\&\qquad+
q_r^{r-j_{r-1}-1}(q_r u_{j_{r-1}-1}+u_{j_{r-1}})+\sum_{j_{r-1}+1\le t\le j_r-1} q_r^{r-t}u_t%
-q_r^{2-r}\sum_{r\le t\le j_{r-1}-1}q_r^{t} u_t\Big)\\
&=\Big(\ascprod_{\beta_-\le k\le r-2} \Cxr{(k+1)}{(j_k-1)}^{((-1)^{k+1})}\Big)\Big(\sum_{r\le t\le j_{r-1}-1} q_r(q_r^{r-1-t}-q_r^{t-r+1}) u_t%
+\sum_{j_{r-1}\le t\le j_r-1} q_r^{r-t}u_t\Big).
\end{align*}
We claim that
\begin{align*}
U(J)^{-1}(u_r)&=\Big(\ascprod_{\beta_-\le s\le k} \Cxr{(s+1)}{(j_s-1)}^{((-1)^{s+1})}\Big)(\sum_{k+2\le t\le j_{k+1}-1} q_r^{\overline{r-k-1}}(q_r^{k+1-t}-q_r^{t-k-1}) u_t\\
&\qquad +\sum_{k+2\le l\le r-1} \sum_{j_{l-1}\le t\le j_l-1} q_r^{\overline{r-l}}
(q_r^{l-t}-q_r^{t-l})u_t
+\sum_{j_{r-1}\le t\le j_r-1} q_r^{r-t}u_t
\end{align*}
for all $\beta_--1\le k\le r-2$. The case~$k=r-2$ has already been established. For the inductive step, we obtain, using Lemma~\ref{lem:rev Cox act},
\begin{align*}
&\Cxr{(k+1)}{(j_k-1)}^{((-1)^{k+1})}\Big(\sum_{k+2\le t\le j_{k+1}-1} q_r^{\overline{r-k-1}}(q_r^{k+1-t}-q_r^{t-k-1}) u_t\Big)
\\&=\sum_{k+2\le t\le j_k-1} q^{(-1)^{k+1}}q_r^{\overline{r-k-1}}(q_r^{k+1-t}-q_r^{t-k-1}) u_{t-1}+\sum_{j_k+1\le t\le j_{k+1}-1} q_r^{\overline{r-k-1}}(q_r^{k+1-t}-q_r^{t-k-1}) u_t\\
&\qquad+q_r^{\overline{r-k-1}}(q_r^{k+1-j_k}-q_r^{j_k-k-1})(q^{(-1)^{k+1}} u_{j_k-1}+u_{j_k})\\
&=\sum_{k+2\le t\le j_k} q^{(-1)^{k+1}}q_r^{\overline{r-k-1}}(q_r^{k+1-t}-q_r^{t-k-1}) u_{t-1}
+\sum_{j_k\le t\le j_{k+1}-1} q_r^{\overline{r-k-1}}(q_r^{k+1-t}-q_r^{t-k-1}) u_t.
\end{align*}
Since $(-1)^{k+1}+(-1)^r \overline{r-k-1}=(-1)^r ((-1)^{k+1-r}+\overline{r-k-1})=(-1)^r \overline{r-k}$, we obtain
\begin{align*}
\Cxr{(k+1)}{(j_k-1)}^{((-1)^{k+1})}&\Big(\sum_{k+2\le t\le j_{k+1}-1} q_r^{\overline{r-k-1}}(q_r^{k+1-t}-q_r^{t-k-1}) u_t\Big)
\\
&=\sum_{k+1\le t\le j_k-1} q_r^{\overline{r-k}}(q_r^{k-t}-q_r^{t-k}) u_t+\sum_{j_k\le t\le j_{k+1}-1} q_r^{\overline{r-k-1}}(q_r^{k+1-t}-q_r^{t-k-1}) u_t,
\end{align*}
and so
\begin{align*}
U(J)^{-1}(u_r)&=\Big(\ascprod_{\beta_-\le s\le k-1} \Cxr{(s+1)}{(j_s-1)}^{((-1)^{s+1})}\Big)\Big(\sum_{k+1\le t\le j_{k}-1} q_r^{\overline{r-k}}(q_r^{k-t}-q_r^{t-k}) u_t\\
&\qquad +\sum_{k+1\le l\le r-1} \sum_{j_{l-1}\le t\le j_l-1} q_r^{\overline{r-l}}
(q_r^{l-t}-q_r^{t-l})u_t
+\sum_{j_{r-1}\le t\le j_r-1} q_r^{r-t}u_t\Big).
\end{align*}
Taking~$k=\beta_--1$ and noting that $j_{\beta_--1}=\beta_-$, we obtain
\begin{align*}
U(J)^{-1}(u_r)&=\sum_{\beta_-+1\le t\le j_{\beta_-}-1} q_r^{\overline{r-\beta_-}}(q_r^{\beta_--t}-q_r^{t-\beta_-}) u_t\\
&\qquad +\sum_{\beta_-+1\le l\le r-1} \sum_{j_{l-1}\le t\le j_l-1} q_r^{\overline{r-l}}(q_r^{l-t}-q_r^{t-l})u_t
+\sum_{j_{r-1}\le t\le j_r-1} q_r^{r-t}u_t\\
&=\sum_{\beta_-\le t\le j_{\beta_-}-1} q_r^{\overline{r-\beta_-}}(q_r^{\beta_--t}-q_r^{t-\beta_-}) u_t\\
&\qquad +\sum_{\beta_-+1\le l\le r-1} \sum_{j_{l-1}\le t\le j_l-1} q_r^{\overline{r-l}}(q_r^{l-t}-q_r^{t-l})u_t
+\sum_{j_{r-1}\le t\le j_r-1} q_r^{r-t}u_t\\
&=\sum_{j_{\beta_--1}\le t\le j_{\beta_-}-1} q_r^{\overline{r-\beta_-}}(q_r^{\beta_--t}-q_r^{t-\beta_-}) u_t\\
&\qquad +\sum_{\beta_-+1\le l\le r-1} \sum_{j_{l-1}\le t\le j_l-1} q_r^{\overline{r-l}}(q_r^{l-t}-q_r^{t-l})u_t
+\sum_{j_{r-1}\le t\le j_r-1} q_r^{r-t}u_t\\
&=S(J,r).
\end{align*}
Part~\ref{lem:TJ|J| eigenvectors.a} is proven.

To prove part~\ref{lem:TJ|J| eigenvectors.b}, we use
induction on~$m$. Since~$g(J)>1$, the induction base is~$m=2$, that is $J=\{1,j,n+1\}$ for some~$2<j<n$.
Then $U(J)^{-1}=\Cxr2{(j-1)}$ and $U(J)^{-1}(u_j)=q u_{j-1}+u_j=q_2 u_{j-1}+u_j$
by Lemma~\ref{lem:rev Cox act}.

Suppose the claim is proven for all $J'$ with~$|J'|=m+1$, $m\ge 2$ and that~$J$ with~$g(J)>1$ satisfies~$|J|=m+2$. Then by Lemma~\ref{lem:rev Cox act}
\begin{align*}
U(J)^{-1}(u_{j_m})&=U(J\setminus\{j_m\})^{-1}\Cxr{(m+1)}{(j_m-1)}^{((-1)^{m+1})}(u_{j_m})
=U(J\setminus\{j_m\})^{-1}(q_{m+1} u_{j_m-1}+u_{j_m}).
\end{align*}
If~$j_{m-1}<j_m-1$ then $U(J\setminus\{j_m\})^{-1}(q_{m+1} u_{j_m-1}+u_{j_m})=q_{m+1} u_{j_m-1}+u_{j_m}$. But in that case~$\beta_+(J)=m$ and the formula in~\ref{lem:TJ|J| eigenvectors.b} holds. Otherwise,
$j_{m-1}=j_m-1$. If~$g(J\setminus\{j_m\})=1$ then, as $j_{m-1}<n$, it follows
that $j_i=i+1$, $0\le i\le m$ which contradicts the assumption~$g(J)>1$.
Thus, $g(J\setminus\{j_m\})>1$, $\beta_+(J\setminus\{j_m\})=\beta_+
\le m-1$ and
and so by the induction hypothesis,
\begin{align*}
U(J)^{-1}(u_{j_m})&=%
U(J\setminus\{j_m\})^{-1}(q_m^{-1} u_{j_{m-1}})+u_{j_m}=q_m^{-1}\sum_{j_{\beta_+}-1\le t\le j_{m-1}} q_m^{\overline{j_{m-1}-t}}u_t+u_{j_m}\\
&=q_m^{-1}\sum_{j_{\beta_+}-1\le t\le j_{m}-1} q_m^{1-\overline{j_{m}-t}}u_t+u_{j_m}=\sum_{j_{\beta_+}-1\le t\le j_m} q_{m+1}^{\overline{j_m-t}}u_t.
\end{align*}
To prove part~\ref{lem:TJ|J| eigenvectors.c}, abbreviate~$\epsilon=(-1)^{m+1}$;
thus, $q_{m+1}=q^{\epsilon}$. First we claim that for all
$\beta_+\le k\le m-1$,
$\beta_+\le k\le m-1$,
\begin{align*}
\ascprod_{k\le t\le m-1} \Cxr{(t+1)}{(j_t-1)}^{((-1)^{t+1})}(
w^{(\epsilon)}_{[m+1,n-1]})&=
\ascprod_{k\le t\le m-1} \Cxr{(t+1)}{(n-m+t)}^{((-1)^{t+1})}(
w^{(\epsilon)}_{[m+1,n-1]})\\
&=
q^{\epsilon(m-k-\overline{m-k})}
w^{(\epsilon)}_{[k+1,n-m-1+k]}
\end{align*}
Indeed, for~$k=m-1$ we have $\Cxr{m}{(n-1)}^{(-\epsilon)}(w^{(\epsilon)}_{[m+1,n-1]})=
w^{(\epsilon)}_{[m,n-2]}$
by Corollary~\ref{cor:cox w +-}.
For the inductive step, again by Corollary~\ref{cor:cox w +-}
\begin{align*}
\Cxr{(k+1)}{(n-m+k)}^{((-1)^{k+1})}(q^{\epsilon(m-k-1-\overline{m-k-1})}
w^{(\epsilon)}_{[k+2,n-m+k]})
&=q^{\epsilon(m-k+(-1)^{k-m}-\overline{m-k-1})} w^{(\epsilon)}_{[k+1,n-m+k-1]}\\
&=q^{\epsilon(m-k+(-1)^{k-m}-\overline{m-k-1})} w^{(\epsilon)}_{[k+1,n-m+k-1]},
\end{align*}
since $(-1)^{k-m}-\overline{m-k-1}=-\overline{m-k}$.

Taking~$k=\beta_+$, we obtain
$$
\ascprod_{\beta_+\le t\le m-1} \Cxr{(t+1)}{(j_t-1)}^{((-1)^{t+1})}(w^{(\epsilon)}_{[m+1,n-1]}
=
q^{\epsilon(m-\beta_+-\overline{m-\beta_+})}
w^{(\epsilon)}_{[\beta_++1,n-m-1+\beta_+]}.
$$
Since $n-m+\beta_+=j_{\beta_+}-1$, we obtain
$$
\ascprod_{\beta_+\le t\le m-1} \Cxr{(t+1)}{(j_t-1)}^{((-1)^{t+1})}(w^{(\epsilon)}_{[m+1,n-1]}=q^{\epsilon(m-\beta_+-\overline{m-\beta_+})}
w^{(\epsilon)}_{[\beta_++1,j_{\beta_+}-2]}.
$$
Finally, we prove that for all $\beta_-\le k\le \beta_+$,
\begin{align}
\ascprod_{k\le t\le \beta_+} \Cxr{(t+1)}{(j_t-1)}^{((-1)^{t+1})}(w^{(\epsilon)}_{[m+1,n-1]})&=
q^{\epsilon(m-k-\overline{m-k})} w^{(\epsilon)}_{[k+1,j_k-1-\delta_{k,\beta_+}]}\nonumber\\
&\quad +
\sum_{k+1\le t\le \beta_+} q^{\epsilon(m-t-\overline{m-t})}w^{(\epsilon)}_{[j_t,j_{t+1}-1-\delta_{t,\beta_+}]}.\label{eq:intermed UJ wm+1n-1}
\end{align}
Indeed, the case~$k=\beta_+$ has already been established. For~$k<\beta_+$ note that by definition of~$\beta-+$, $j_k<j_{k+1}-\delta_{k+1,\beta_+}$, $\beta_-+1\le k\le \beta_+-1$.
Thus,
\begin{align*}
\Cxr{(k+1)}{(j_k-1)}^{(-1)^{k+1}}&(q^{\epsilon(m-k-1-\overline{m-k-1})} w^{(\epsilon)}_{[k+2,j_{k+1}-1-\delta_{k+1,\beta_+}]})\\&
=q^{\epsilon((-1)^{k-m}+m-k-\overline{m-k-1})} w^{(\epsilon)}_{[k+1,j_k-1]}+w^{(\epsilon)}_{[j_k,j_{k+1}-1-\delta_{k+1,\beta_+}]}\\
&=q^{\epsilon(m-k-\overline{m-k})} w^{(\epsilon)}_{[k+1,j_k-1]}+w^{(\epsilon)}_{[j_k,j_{k+1}-1-\delta_{k+1,\beta_+}]}.
\end{align*}
Since $\Cxr{(k+1)}{(j_k-1)}^{(-1)^{k+1}}(w^{(\epsilon)}_{[j_t,j_{t+1}-1-\delta_{t,\beta_+}]})=w^{(\epsilon)}_{[j_t,j_{t+1}-1-\delta_{t,\beta_+}]}$,
$k+1\le t\le \beta_+-1$, \eqref{eq:intermed UJ wm+1n-1} follows. It remains to observe that this implies the assertion since $j_{\beta_--1}=\beta_-$.
\end{proof}
\begin{lemma}\label{lem:Euclid jm-1<n}
Let~$\{1,n+1\}\subset J\subset [1,n+1]$ with~$g(J)>1$. Suppose that
$j=\max(J\setminus\{n+1\})<n$. Then
$\la U(J)^{-1}(u_{m-1})\,\vert\,U(J)^{-1}(u_j)\ra\not=0$.
\end{lemma}
\begin{proof}
Let~$m=|J|-1$ and write~$J=\{j_0=1<j_1<\cdots<j_{m-1}=j<j_m=n+1\}$.
Note that by Lemma~\partref{lem:TJ|J| eigenvectors.b},
$U(J)^{-1}(u_j)$ is contained in the span of the $u_t$, $t\in[j_{\beta_+}-1,
j]$ which is orthogonal to all the $u_s$, $s\in [1,j_{\beta_+}-3]$. Thus,
we can consider $U(J)^{-1}(u_{m-1})$ modulo $V'=\sum_{1\le s\le
j_{\beta_+}-3} \kk u_s$.

Suppose first that~$\beta_+(J)=m-1$. Then~$j_{m-2}\le j_{m-1}-2$, hence
by Lemma~\ref{lem:TJ|J| eigenvectors}
$U(J)^{-1}(u_{m-1})=q_m^{j-m-1}u_{j-2}+q_m^{j-m}u_{j-1}\pmod{V'}$,
$U(J)^{-1}(u_j)=q_m u_{j-1}+u_j$ and so by~\eqref{eq:Euclid ui uj}
\begin{align*}
\langle U(J)^{-1}(u_{m-1})\,\vert\,U(J)^{-1}(u_j)\rangle&=\langle q_m^{j-m-1}u_{j-2}+q_m^{j-m}u_{j-1},q_m u_{j-1}+u_j\rangle\\
&=q_m^{j-m}( -2q^{-1}+q_m(1+q^{-2}))=q^{(-1)^m(j-m)-1}(q_m^2-1),
\end{align*}
which is manifestly a non-zero Laurent polynomial in~$q$.

Suppose now that~$\beta_+=\beta_+(J)<m-1$.
Then, in particular, $j_{m-2}=j-1$ and
\begin{align*}
U(J)^{-1}(&u_{m-1})=\sum_{\beta_+\le k\le m-2}\sum_{j_{k-1}\le t\le j_k-1} q_{m-1}^{\overline{m-1-k}}(q_{m-1}^{k-t}-q_{m-1}^{t-k})u_t
+q_{m-1}^{m-j}u_{j-1}\pmod{V'}\\
&=\sum_{j_{\beta_+-1}\le t\le j_{\beta_+}-1} q_{m-1}^{\overline{m-1-\beta_+}}(q_{m-1}^{\beta_+-t}-q_{m-1}^{t-\beta_+})u_t\\
&\quad+\sum_{\beta_++1\le k\le m-2}q_{m-1}^{\overline{m-1-k}}(q_{m-1}^{k-j_{k-1}}-q_{m-1}^{j_{k-1}-k})u_{j_{k-1}}
+q_{m-1}^{m-j}u_{j-1}\pmod{V'}
\\
&=q_{m-1}^{\overline{m-1-\beta_+}}(q_{m-1}^{m+1-j}-q_{m-1}^{j-m-1})u_{j-m+\beta_+-1}
\\
&\quad+(q_{m-1}^{m-j}-q_{m-1}^{j-m})\sum_{\beta_+\le k\le m-2}q_{m-1}^{\overline{m-1-k}}u_{j-m+k}
+q_{m-1}^{m-j}u_{j-1}\pmod{V'}\\
&=q_{m}^{-\overline{m-1-\beta_+}}(q_{m}^{j-m-1}-q_{m}^{m+1-j})u_{j-m+\beta_+-1}
\\
&\quad+(q_{m}^{j-m}-q_{m}^{m-j})\sum_{j-m+\beta_+\le t\le j-2}q_{m}^{-\overline{j-t-1}}u_t
+q_{m}^{j-m}u_{j-1}\pmod{V'}.
\end{align*}
Then, using~\eqref{eq:Euclid ui uj}, we obtain
\begin{align*}
\langle U(&J)^{-1}(u_{m-1})\,\vert\,U(J)^{-1}(u_{j_{m-1}})\rangle =
\langle q_{m}^{-\overline{m-1-\beta_+}}(q_{m}^{j-m-1}-q_{m}^{m+1-j})u_{j-m+\beta_+-1}
\\
&\qquad+(q_{m}^{j-m}-q_{m}^{m-j})\sum_{j-m+\beta_+\le t\le j-2}q_{m}^{-\overline{j-t-1}}u_t
+q_{m}^{j-m}u_{j-1}\,\vert\,
\sum_{j-m+\beta_+ \le t\le j} q_m^{\overline{j-t}} u_t\rangle\\
&=-q_{m}^{2\overline{m-\beta_+}-1}(q_{m}^{j-m-1}-q_{m}^{m+1-j})q^{-1}\\
&\qquad+(q_{m}^{j-m}-q_{m}^{m-j})\langle \sum_{j-m+\beta_+\le t\le j-2}q_{m}^{-\overline{j-t-1}}u_t
+q_{m}^{j-m}u_{j-1}\,\vert\,\sum_{j-m+\beta_+ \le t\le j} q_m^{\overline{j-t}} u_t\rangle\\
&=-q_{m}^{2\overline{m-\beta_+}-1}(q_{m}^{j-m-1}-q_{m}^{m+1-j})q^{-1}+q_m^{j-m+1}(1+q^{-2})-2q^{-1} q_m^{j-m}\\\
&\qquad+(q_{m}^{j-m}-q_{m}^{m-j})\langle \sum_{j-m+\beta_+\le t\le j-2}q_{m}^{-\overline{j-t-1}}u_t
\,\vert\,\sum_{j-m+\beta_+ \le t\le j} q_m^{\overline{j-t}} u_t\rangle\\
&=-q^{-1}(q_{m}^{2\overline{m-\beta_+}}(q_{m}^{j-m-2}-q_{m}^{m-j})-q_m^{j-m}(q_m^2-1))\\\
&\qquad+(q_{m}^{j-m}-q_{m}^{m-j})\langle \sum_{j-m+\beta_+\le t\le j-2}q_{m}^{-\overline{j-t-1}}u_t\,
\vert\,\sum_{j-m+\beta_+ \le t\le j-1} q_m^{\overline{j-t}} u_t\rangle\\
&=-q^{-1}(q_{m}^{2\overline{m-\beta_+}}(q_{m}^{j-m-2}-q_{m}^{m-j})-q_m^{j-m}(q_m^2-1))\\\
&\qquad+(q_{m}^{j-m}-q_{m}^{m-j})q_m^{-1}\langle \sum_{j-m+\beta_+\le t\le j-2}q_{m}^{\overline{j-t}}u_t\,
\vert\,\sum_{j-m+\beta_+ \le t\le j-1} q_m^{\overline{j-t}} u_t\rangle\\
&=-q^{-1}(q_{m}^{2\overline{m-\beta_+}}(q_{m}^{j-m-2}-q_{m}^{m-j})-q_m^{j-m}(q_m^2-1)+q_m^{j-m}-q_m^{m-j})\\\
&\qquad+(q_{m}^{j-m}-q_{m}^{m-j})q_m^{-1}\langle \sum_{j-m+\beta_+\le t\le j-2}q_{m}^{\overline{j-t}}u_t\,
\vert\,\sum_{j-m+\beta_+ \le t\le j-2} q_m^{\overline{j-t}} u_t\rangle\\
&=-q^{-1}(q_{m}^{2\overline{m-\beta_+}}(q_{m}^{j-m-2}-q_{m}^{m-j})-q_m^{j-m}(q_m^2-1)+q_m^{j-m}-q_m^{m-j})\\\
&\qquad+(q_{m}^{j-m}-q_{m}^{m-j})q_m^{-1} \Big((1+q^{-2})\sum_{j-m+\beta_+\le t\le j-2} q_m^{2\overline{j-t}}
-2 q^{-1} \sum_{j-m+\beta_+\le t\le j-3} q_m^{\overline{j-t}+\overline{j-t+1}}\Big)\\
&=-q^{-1}\Big(q_{m}^{2\overline{m-\beta_+}}(q_{m}^{j-m-2}-q_{m}^{m-j})-q_m^{j-m}(q_m^2-1)%
+(q_m^{j-m}-q_m^{m-j})(2(m-\beta_+)-3))\Big)\\
&\qquad+(q_{m}^{j-m}-q_{m}^{m-j})q_m^{-1} (1+q^{-2})(q_m^2 \lfloor\tfrac12(m-\beta_+-1)\rfloor+\lfloor\tfrac12(m-\beta_+)\rfloor).
\end{align*}
Note that, since~$g(J)>1$, $j<m$ and so $q_m^{j-m}-q_m^{m-j}\not=0$. We can rewrite the above as
$$
\langle U(J)^{-1}(u_{m-1})\,\vert\, U(J)^{-1}(u_{j_{m-1}})\rangle =q_m^{j-m} Q^+_{m-\beta_+,(-1)^m}(q)-q_m^{m-j}Q^-_{m-\beta_+,(-1)^m}(q),
$$
where~$Q^\pm_{r,\epsilon}(q)\in\mathbb Z[q,q^{-1}]$, $r\in\mathbb Z$, $\epsilon\in\{1,-1\}$, are defined by
\begin{align*}
&Q^+_{r,\epsilon}(q)=-q^{-1}(q^{2\epsilon (\overline{r}-1)}-q^{2\epsilon}+2(r-1))+(1+q^{-2})(q^\epsilon\lfloor\tfrac12(r-1)\rfloor+q^{-\epsilon}\lfloor\tfrac12 r\rfloor),
\\
&Q^-_{r,\epsilon}(q)=-q^{-1}(q^{2\epsilon \overline{r}}+2(r-1)-1)+(1+q^{-2})(q^\epsilon \lfloor\tfrac12(r-1)\rfloor+q^{-\epsilon}\lfloor\tfrac12 r\rfloor).
\end{align*}
Since~$1+q^{-2}=q^{-1}(q^\epsilon+q^{-\epsilon})$, $\epsilon\in\{1,-1\}$, we have
\begin{align*}
q &Q^+_{r,\epsilon}(q)=q^{2\epsilon}(1+\lfloor\tfrac12(r-1)\rfloor)-q^{2\epsilon(\overline r-1)}+q^{-2\epsilon}\lfloor\tfrac12 r\rfloor+
2(1-r)+\lfloor\tfrac12(r-1)\rfloor+\lfloor\tfrac12 r\rfloor\\
&=q^{2\epsilon}(1+\lfloor\tfrac12(r-1)\rfloor)+q^{-2\epsilon}\lfloor\tfrac12 r\rfloor+
1-r-q^{2\epsilon(\bar r-1)}\\
&=q^{2\epsilon}(\lfloor\tfrac12r\rfloor+\bar r)+q^{-2\epsilon}\lfloor\tfrac12 r\rfloor+
1-2\lfloor\tfrac12 r\rfloor-\bar r-q^{2\epsilon(\bar r-1)}\\
&=\lfloor\tfrac12 r\rfloor(q^\epsilon-q^{-\epsilon})^2
+\bar r(q^{2\epsilon}-1)+1-q^{2\epsilon(\bar r-1)}%
=(q-q^{-1})(\lfloor\tfrac12 r\rfloor(q-q^{-1})+\epsilon q^{\epsilon(2\bar r-1)}).
\end{align*}
Similarly,
\begin{align*}
q Q^-_{r,\epsilon}(q)&=(q-q^{-1})^2\lfloor\tfrac12 r\rfloor+1-q^{2\epsilon \bar r}
+(1-\bar r)(1-q^{2\epsilon})\\
&=(q-q^{-1})( \lfloor\tfrac12 r\rfloor(q-q^{-1})-\epsilon q^\epsilon).
\end{align*}
Thus, $Q^\pm_{r,\epsilon}(q)/(1-q^{-2})\in \mathbb Z[q,q^{-1}]$ and
equals~$\pm \epsilon$ at~$q=1$.
It follows that
$$\la U(J)^{-1}(u_{m-1})\,\vert\, U(J)^{-1}(u_{j_{m-1}})\ra/(1-q^{-2})
\in\mathbb Z[q,q^{-1}],
$$
equals~$2(-1)^m$ at~$q=1$ and hence
is non-zero.
\end{proof}
\begin{lemma}\label{lem:Euclid jm-1=n}
Suppose that~$g(J)>1$ and~$j_{m-1}=n$. Let~$\epsilon=(-1)^{m-1}$. Then
$$
\langle U(J)^{-1}(u_{m-1})\,\vert\, U(J)^{-1}(w^{(\epsilon)}_{[m+1,n-1]})\rangle\not=0.
$$
\end{lemma}
\begin{proof}
As before, we abbreviate~$\beta_\pm=\beta_\pm(J)$. By definition, $\beta_-\le\beta_+$.
Since $g(J)>1$, we must have $\beta_-\le \beta_+-1$ for otherwise $\beta_-=\beta_+$ and so $J=[1,\beta_--1]\cup [n-m+\beta_-+1,n+1]$.
Write
$$
U(J)^{-1}(q^{-\epsilon} u_{m-1})=\sum_{\beta_-\le t\le n-1} c_t u_t,\qquad U(J)^{-1}( q^{-\epsilon m} w^{(\epsilon)}_{[m+1,n-1]})=\sum_{\beta_-+1\le t \le j_{\beta_+}-2} c'_t u_t,
$$
where by Lemma~\ref{lem:TJ|J| eigenvectors}
$$
c_t=q^{-\epsilon\overline{m-\kappa(t)}}(q^{\epsilon(\kappa(t)-t)}-(1-\delta_{m-1,\kappa(t)})q^{\epsilon(t-\kappa(t))}),\quad
c'_t=q^{\epsilon(t-\kappa(t)-\overline{m-\kappa(t)})},
$$
and $\kappa(t)=\{ k\in[1,m]\,:\, j_{k-1}\le t\le j_k-1\}$.
Thus, by~\eqref{eq:Euclid ui uj}
\begin{align*}
q^{-\epsilon(m+1)} \langle U(J)^{-1} &u_{m-1}\,\vert\, U(J)^{-1}(w^{(\epsilon)}_{[m+1,n-1]})\rangle=(1+q^{-2})\sum_{\beta_-+1\le t \le j_{\beta_+}-2} c_t c'_t
\\&-q^{-1}\sum_{\beta_-\le t\le j_{\beta_+}-3} c_t c'_{t+1}-q^{-1}\sum_{\beta_-+1\le t\le j_{\beta_+}-2} c_{t+1}c'_t
\end{align*}
Note that $\kappa(t+1)=\kappa(t)$ unless $t=j_{\kappa(t)}-1$ in which case $\kappa(t+1)=\kappa(t)+1$.
We have
\begin{align*}
&c_t c'_t=q^{-2\epsilon \overline{m-\kappa(t)}}(1-(1-\delta_{m-1,\kappa(t)})q^{2\epsilon(t-\kappa(t))})
\intertext{
while}
&c_t c'_{t+1}=\begin{cases}
q^{\epsilon}c_t c'_t,&t\not=j_{\kappa(t)}-1,\\
q^{-\epsilon}(1-(1-\delta_{m-1,\kappa(t)})q^{2\epsilon(t-\kappa(t))}),&t=j_{\kappa(t)}-1,
\end{cases}
\\
\intertext{and}
&c'_t c_{t+1}=\begin{cases}
q^{-\epsilon(1+2\overline{m-\kappa(t)})}(1-(1-\delta_{m-1,\kappa(t)})q^{2\epsilon(1+t-\kappa(t))}),&t\not=j_{\kappa(t)}-1,\\
q^{-\epsilon}(1-(1-\delta_{m-1,\kappa(t)+1})q^{2\epsilon(t-\kappa(t))}),&t=j_{\kappa(t)}-1.
\end{cases}
\end{align*}
Thus, for~$\beta_-<k<\beta_+$,
\begin{align*}
\sum_{j_{k-1}\le t\le j_k-1} &c_t c'_t
=q^{-2\epsilon\overline{m-k}}\Big(
j_k-j_{k-1}-q^{-2\epsilon k}\frac{q^{2\epsilon j_k}-q^{2\epsilon j_{k-1}}}{q^{2\epsilon}-1}\Big),\\
\sum_{j_{k-1}\le t\le j_k-1} c_t c'_{t+1}&=
q^{\epsilon(1-2\overline{m-k})}\Big(
j_k-j_{k-1}-1-q^{-2\epsilon k}\,\frac{q^{2\epsilon (j_k-1)}-q^{2\epsilon j_{k-1}}}{q^{2\epsilon}-1}\Big)\\
&\qquad+q^{-\epsilon}(1-q^{2\epsilon(j_k-k-1)}),\\
\sum_{j_{k-1}\le t\le j_k-1} c_{t+1} c'_{t}&=
q^{-\epsilon(1+2\overline{m-k})}\Big(j_k-j_{k-1}-1-q^{-2\epsilon k}\,\frac{q^{2\epsilon j_k}-q^{2\epsilon(j_{k-1}+1)}}{q^{2\epsilon}-1}\Big)\\
&\qquad
+q^{-\epsilon}(1-(1-\delta_{m-2,k})q^{2\epsilon(j_k-1-k)}).
\end{align*}
Since $q^{-1}(q^\epsilon+q^{-\epsilon})=1+q^{-2}$, $\epsilon\in\{1,-1\}$
we obtain, for $\beta_-<k<\beta_+$,
\begin{align*}
(1+&q^{-2})\sum_{j_{k-1}\le t\le j_k-1} c_t c'_t-q^{-1}\sum_{j_{k-1}\le t\le j_k-1} (c_t c'_{t+1}+c_{t+1}c'_t)\\
&=q^{-2\epsilon\overline{m-k}}\Big( (1+q^{-2})(j_k-j_{k-1})-q^{-1}(q^\epsilon+q^{-\epsilon})(j_k-j_{k-1})+q^{-1}(q^\epsilon+q^{-\epsilon})\\
&\quad-q^{-2\epsilon k}(q^{2\epsilon}-1)^{-1}\Big( (1+q^{-2})(q^{2\epsilon j_k}-q^{2\epsilon j_{k-1}})-q^{-1+\epsilon}(q^{2\epsilon (j_k-1)}-q^{2\epsilon j_{k-1}})\\
&\quad-q^{-1-\epsilon}(q^{2\epsilon j_k}-q^{2\epsilon(j_{k-1}+1)})\Big)\Big)-q^{-1-\epsilon}(2-(2-\delta_{m-2,k})q^{2\epsilon(j_k-1-k)})\\
&=q^{-2\epsilon\overline{m-k}}\Big( (1+q^{-2})
-q^{-2\epsilon k}(q^{2\epsilon}-1)^{-1}\Big( q^{2\epsilon j_k}((1+q^{-2})-2q^{-1-\epsilon})\\
&\quad-q^{2\epsilon j_{k-1}}((1+q^{-2})-2q^{-1+\epsilon})\Big)\Big)-q^{-1-\epsilon}(2-(2-\delta_{m-2,k})q^{2\epsilon(j_k-1-k)})\\
&=q^{-2\epsilon\overline{m-k}}( (1+q^{-2})-q^{-1-\epsilon(2k+1)}(q^{2\epsilon j_k}+q^{2\epsilon j_{k-1}}))%
-q^{-1-\epsilon}(2-(2-\delta_{m-2,k})q^{2\epsilon(j_k-1-k)}).
\end{align*}
For $k=\beta_-$ we have
\begin{align*}
(1+&q^{-2})\sum_{\beta_-+1\le t\le j_{\beta_-}-1} c_t c'_t-q^{-1}\sum_{\beta_-\le t\le j_{\beta_-}-1} c_t c'_{t+1}-q^{-1}\sum_{\beta_-+1\le t\le j_{\beta_-}-1}
c_{t+1}c'_t)\\
&=(1+q^{-2})q^{-2\epsilon\overline{m-\beta_-}}\Big(j_{\beta_-}-\beta_--1-\frac{q^{2\epsilon(j_{\beta_-}-\beta_-)}-q^{2\epsilon}}{q^{2\epsilon}-1}\Big)
\\&\qquad-q^{-1+\epsilon(1-2\overline{m-\beta_-})}\Big( j_{\beta_-}-\beta_--1-\frac{q^{2\epsilon(j_{\beta_-}-\beta_--1)}-1}{q^{2\epsilon-1}}\Big)%
-q^{-1-\epsilon}(1-q^{2\epsilon(j_{\beta_-}-\beta_--1)})\\
&\qquad-q^{-1-\epsilon(1+2\overline{m-\beta_-})}\Big(j_{\beta_-}-\beta_--2-\frac{q^{2\epsilon(j_{\beta_-}-\beta_-)}-q^{4\epsilon}}{q^{2\epsilon}-1}\Big)\\
&\qquad-q^{-1-\epsilon}(1-(1-\delta_{m-2,\beta_-})q^{2\epsilon(j_{\beta_-}-1-\beta_-)})\\
&=q^{-2\epsilon\overline{m-\beta_-}}\Big(q^{-1-\epsilon}-(q^{2\epsilon}-1)^{-1}\Big( (1+q^{-2})(q^{2\epsilon(j_{\beta_-}-\beta_-)}-q^{2\epsilon})\\
&\qquad+q^{-1+\epsilon}(q^{2\epsilon(j_{\beta_-}-\beta_--1)}-1)+q^{-1-\epsilon}(q^{2\epsilon(j_{\beta_-}-\beta_-)}-q^{4\epsilon})\Big)\Big)\\
&\qquad
-q^{-1-\epsilon}(2-(2-\delta_{m-2,\beta_-})q^{2\epsilon(j_{\beta_-}-1-\beta_-)})\\
&=q^{-(1+\epsilon(1+2\epsilon\overline{m-\beta_-})}(1-q^{2\epsilon(j_{\beta_-}-\beta_-)})
-q^{-1-\epsilon}(2-(2-\delta_{m-2,\beta_-})q^{2\epsilon(j_{\beta_-}-1-\beta_-)}),
\end{align*}
while for~$k=\beta_+$,
\begin{align*}
(1&+q^{-2})\sum_{j_{\beta_+-1}\le t\le j_{\beta_+}-2} c_t c'_t-q^{-1}\sum_{j_{\beta_+-1}\le t\le j_{\beta_+}-3} c_t c'_{t+1}-q^{-1} \sum_{j_{\beta_+-1}\le t\le j_{\beta_+}-2} c_{t+1}c'_t\\
&=(1+q^{-2})q^{-2\epsilon \overline{m-\beta_+}}\Big(j_{\beta_+}-j_{\beta_+-1}-1-(1-\delta_{m-1,\beta_+})q^{-2\epsilon\beta_+}\,\frac{q^{2\epsilon (j_{\beta+}-1)}-q^{2\epsilon j_{\beta_+-1}}}{q^{2\epsilon}-1}\Big)\\
&\qquad-q^{-1+\epsilon(1-2\overline{m-\beta_+})}\Big(j_{\beta_+}-j_{\beta_+-1}-2-(1-\delta_{m-1,\beta_+})q^{-2\epsilon\beta_+}\,\frac{q^{2\epsilon (j_{\beta+}-2)}-q^{2\epsilon j_{\beta_+-1}}}{q^{2\epsilon}-1}\Big)\\
&\qquad-q^{-1-\epsilon(1+2\overline{m-\beta_+})}\Big(j_{\beta_+}-j_{\beta_+-1}-1-(1-\delta_{m-1,\beta_+})q^{-2\epsilon\beta_+}\,\frac{q^{2\epsilon j_{\beta+}}-q^{2\epsilon (j_{\beta_+-1}+1)}}{q^{2\epsilon}-1}\Big)\\
&=
-(1-\delta_{m-1,\beta_+})q^{-2\epsilon (\beta_++\overline{m-\beta_+})}(q^{2\epsilon}-1)^{-1}\Big(q^{2\epsilon j_{\beta_+}}((1+q^{-2})q^{-2\epsilon}-q^{-1-\epsilon}-q^{-1-3\epsilon})
\\
&\qquad-q^{2\epsilon j_{\beta_--1}}((1+q^{-2})-2q^{-1+\epsilon})\Big)\\
&=q^{-1-\epsilon}(q^{2\epsilon(1-\overline{m-\beta_+})}-(1-\delta_{m-1,\beta_+})q^{2\epsilon(j_{\beta_+-1}-\beta_+-\overline{m-\beta_+})}).
\end{align*}
Let~$z=q^{2\epsilon}$. Since~$\beta_-\le \beta_+-1$ we obtain
\begin{align*}
&q^{1-\epsilon m} \la U(J)^{-1} u_{m-1}\,\vert\, U(J)^{-1}(w^{(\epsilon)}_{[m+1,n-1]})\ra
=z^{1-\overline{m-\beta_+}}-(1-\delta_{m-1,\beta_+})z^{j_{\beta_+-1}-\beta_+-\overline{m-\beta_+}}\\&\quad+
z^{-\overline{m-\beta_-}}(1-z^{j_{\beta_-}-\beta_-})
-(2-(2-\delta_{m-2,\beta_-})z^{j_{\beta_-}-1-\beta_-})\\
&\quad+\sum_{\beta_-+1\le k\le \beta_+-1} \Big(z^{-\overline{m-k}}( (1+z)-z^{- k}(z^{j_k}+z^{j_{k-1}}))-2+(2-\delta_{m-2,k})z^{j_k-1-k}\Big)\\
&=z^{1-\overline{m-\beta_+}}-(1-\delta_{m-1,\beta_+})z^{j_{\beta_+-1}-\beta_+-\overline{m-\beta_+}}\\&\quad+
z^{-\overline{m-\beta_-}}(1-z^{j_{\beta_-}-\beta_-})
+(2-\delta_{m-2,\beta_-})z^{j_{\beta_-}-1-\beta_-}-2(\beta_+-\beta_-)\\
&\quad+(1+z)\Big(z^{\overline{m-\beta_-}-1}\lfloor\tfrac12(\beta_+-\beta_-)\rfloor+z^{-\overline{m-\beta_-}}\lfloor\tfrac12(\beta_+-\beta_--1)\rfloor\Big)\\
&\quad-\sum_{\beta_-+1\le k\le \beta_+-1} \Big(z^{-\overline{m-k}-k}(z^{j_k}+z^{j_{k-1}}))-(2-\delta_{m-2,k})z^{j_k-1-k}\Big)\\
&=z^{1-\overline{m-\beta_+}}+z^{-\overline{m-\beta_-}}-2(\beta_+-\beta_-)\\
&\quad+(1+z)\Big(z^{\overline{m-\beta_-}-1}\lfloor\tfrac12(\beta_+-\beta_-)\rfloor+z^{-\overline{m-\beta_-}}\lfloor\tfrac12(\beta_+-\beta_--1)\rfloor\Big)\\
&\quad-\sum_{\beta_-+1\le k\le \beta_+} (1-\delta_{m-1,k}) z^{j_{k-1}-k-\overline{m-k}}%
-\sum_{\beta_-\le k\le \beta_+-1} z^{j_k-k}(z^{-\overline{m-k}}-(2-\delta_{k,m-2})z^{-1})\\
&=z^{1-\overline{m-\beta_+}}+z^{-\overline{m-\beta_-}}-2(\beta_+-\beta_-)\\
&\quad+(1+z)\Big(z^{\overline{m-\beta_-}-1}\lfloor\tfrac12(\beta_+-\beta_-)\rfloor+z^{-\overline{m-\beta_-}}\lfloor\tfrac12(\beta_+-\beta_--1)\rfloor\Big)\\
&\quad-\sum_{\beta_-\le k\le \beta_+-1} z^{j_k-k}(1-z^{\overline{m-k}-1})(1-(1-\delta_{k,m-2})z^{-1})\\
&=z^{1-\overline{m-\beta_+}}+z^{-\overline{m-\beta_-}}-2(\beta_+-\beta_-)-\sum_{\beta_-\le k\le \beta_+-1} z^{j_k-k}(1-z^{\overline{m-k}-1})^{2-\delta_{k,m-2}}\\
&\quad+(1+z)\Big(z^{\overline{m-\beta_-}-1}\lfloor\tfrac12(\beta_+-\beta_-)\rfloor+z^{-\overline{m-\beta_-}}\lfloor\tfrac12(\beta_+-\beta_--1)\rfloor\Big)%
.
\end{align*}
Denote this expression by~$Q_J(z)$. Clearly, $Q_J(z)\in\mathbb Z[z,z^{-1}]$
and, since $j_k-k\ge 2$ for all~$k\ge \beta_-$,
$z^{j_k-k}(1-z^{\overline{m-k}-1})^{2-\delta_{k,m-2}}\in
\mathbb Z[z]$ for all~$k\ge \beta_-$. Therefore,
$$
\operatorname{Res}_{z=0}Q_J(z)=
\begin{cases}
\lfloor\tfrac12(\beta_+-\beta_-)\rfloor,&\overline{m-\beta_-}=0,\\
1+\lfloor\tfrac12(\beta_+-\beta_-)\rfloor,&\overline{m-\beta_-}=1.
\end{cases}
$$
In particular, since $\beta_-\le \beta_+-1$, $\operatorname{Res}_{z=0} Q_J(z)
\not=0$, and hence~$Q_J(z)\not=0$, unless
$\beta_+=\beta_-+1$ and~$\overline{m-\beta_-}=0$. But in that case
\begin{equation*}
Q_J(z)=-z^{j_{\beta_-}-\beta_-}(1-z^{-1})^{2-\delta_{k,m-2}}\not=0.\qedhere
\end{equation*}
\end{proof}
Let~$m=|J|-1$.
By Proposition~\partref{prop:TJ eigenvectors.b}, $u_{m-1}\in\ker(T_{[1,m]\cup\{n+1\}}^{m+1}-q^{2(n+1)}\id_V)$.
If~$j_{m-1}<n$ then, since~$g(J)>1$, $j_{m-1}\ge m+1$ and so
$$u_{j_{m-1}}
\in \ker(T_{[1,m]\cup\{n+1\}}-q^2\id_V)=
\ker(T_{[1,m]\cup\{n+1\}}^{m+1}-q^{2(m+1)}\id_V)$$
by Proposition~\partref{prop:TJ eigenvectors.a}.
If~$j_{m-1}=n$ then $$w^{((-1)^m)}_{[m+1,n-1]}
\in \ker(T_{[1,m]\cup \{n+1\}}-q^2\id_V)=
\ker(T_{[1,m]\cup\{n+1\}}^{m+1}-q^{2(m+1)}\id_V)$$ by
Proposition~\partref{prop:TJ eigenvectors.a}.
Proposition~\ref{prop:red} is now immediate from Lemmata~\ref{lem:Euclid jm-1<n}
and~\ref{lem:Euclid jm-1=n}.
\end{proof}
Theorem~\ref{thm:adm I2m converse } is proven.
\end{proof}

\begin{proposition}\label{prop:eigenvectors even}
Let $\tilde J_m=[1,m]\cup[n+2-m,n+1]$, $1\le m\le \frac12 n$.
Then
\begin{enmalph}
\item\label{prop:eigenvectors even.a} $\{ u_i\,:\,m+1\le i\le n-m\}\subset \ker(T_{\tilde J_m}-q^{2}\id_V)$;
\item\label{prop:eigenvectors even.b} $\{ u_i\mp u_{n+1-i}:\,1\le i\le m-1\}\cup\{q^{n+1}w^{(-1)}_{[1,n]}\mp w^{(1)}_{[1,n]}\} \subset \ker(T_{\tilde J_m}^m\mp q^{n+1}\id_V)$,
\item\label{prop:eigenvectors even.c} $T_{\tilde J_m}^m$ is diagonalizable on~$V$ and
$$
\det(t \id_V-T_{\tilde J_m}^m)=(t-1)(t^2-q^{2(n+1)})^m (t-q^{2m})^{n-2m}.
$$
\end{enmalph}
\end{proposition}
\begin{proof}
By Lemma~\ref{lem:Transp action}, $T_{(m+1,n+2-m)}(u_i)=q^2 u_i$ for all~$i\in[m+1,n-m]$.
Since~$T_{\tilde J_m}$ is the product of~$T_{(m+1,n+2-m)}$ and the $T_j$ with $j\in[1,m]\cup[n+2-m,n]$ which fix the~$u_i$ with~$i\in[m+1,n-m]$, part~\ref{prop:eigenvectors even.a} follows.

To prove~\ref{prop:eigenvectors even.b},
recall that $
T_{\tilde J_m}^m=T_{w_\circ^{[1,n]}}T_{w_\circ^{[m+1,n-m]}}^{-1}
$
by Corollary~\ref{cor:symm even J}. Since $T_{w_\circ^{[1,n]}}(u_i)
=-q^{n+1}u_{n+1-i}$ for all~$i\in [1,n]$ while
$T_{w_\circ^{[m+1,n-1]}}(u_i)=u_i$ for all $i\in[1,m-1]\cup
[n+2-m,n]$ by Lemma~\ref{lem:Tw0ij act ui}, it follows that $u_i\pm u_{n+1-m}\in
\ker(T_{\tilde J_m}^m\mp q^{n+1}\id_V)$, $i\in[1,m-1]$.
Furthermore, by Lemma~\ref{lem:Tw0ij act ui}
$$
T_{w_\circ^{[1,n]}}(w^{(-1)}_{[1,n]})=
-\sum_{1\le t\le n}q^{-t+n+1}u_{n+1-t}=-w^{(1)}_{[1,n]},
$$
while
$$
T_{w_\circ^{[1,n]}}(w^{(1)}_{[1,n]})=-\sum_{1\le t\le n}q^{t+n+1}u_{n+1-t}
=-q^{2(n+1)}w^{(-1)}_{[1,n]},
$$
whence $q^{n+1}w^{(-1)}_{[1,n]}\mp w^{(1)}_{[1,n]}\in\ker(T_{w_\circ^{[1,n]}}\mp q^{n+1}\id_V)$. Since
$w^{(\pm1)}_{[1,n]}\in\ker(T_{w_\circ^{[m+1,n-m]}}-\id_V)$ by Lemma~\ref{lem:Tw0ij act ui}, part~\ref{prop:eigenvectors even.b} follows.

To prove~\ref{prop:eigenvectors even.c}, it remains to show that $q^{n+1}w^{(-1)}_{[1,n]}\mp w^{(1)}_{[1,n]}$ is not contained in the span of manifestly linearly
independent $\{ u_i\mp u_{n+1-i}\,:\, 1\le i\le m-1\}$. But we have
\begin{align*}
q^{n+1}w^{(-1)}_{[1,n]}\mp w^{(1)}_{[1,n]}&=\sum_{1\le t\le n} q^{n+1-t}u_t\mp \sum_{1\le t\le n} q^t u_t%
=\sum_{1\le t\le n} q^{n+1-t}(u_t\mp u_{n+1-t}).
\end{align*}
If~$n$ is odd, we obtain
$$
q^{n+1}w^{(-1)}_{[1,n]}-w^{(1)}_{[1,n]}=\sum_{1\le t\le \frac12(n-1)} (q^{n+1-t}- q^t)(u_t- u_{n+1-t})
$$
and
$$
q^{n+1}w^{(-1)}_{[1,n]}+w^{(1)}_{[1,n]}=\sum_{1\le t\le \frac12(n-1)} (q^{n+1-t}+ q^t)(u_t- u_{n+1-t})+2 q^{\frac12(n+1)} u_{\frac12(n+1)},
$$
while for~$n$ even
$$
q^{n+1}w^{(-1)}_{[1,n]}\mp w^{(1)}_{[1,n]}=\sum_{1\le t\le \frac12n} (q^{n+1-t}\mp q^t)(u_t\mp u_{n+1-t}).
$$
In either case, since all vectors appearing in the right hand side are linearly independent, it follows that $q^{n+1}w^{(-1)}_{[1,n]}\mp w^{(1)}_{[1,n]}$ is not
contained in the span of any proper subfamily of these vectors.
\end{proof}
\begin{corollary}
For any~$\{1,n+1\}\subset J\subset [1,n+1]$ with~$|J|$ even,
$T_J^{|J|/2}$ is diagonalizable on~$V$ and
$\det(t\id_V-T_J^{|J|/2})=(t-1)(t^2-q^{2(n+1)})^{|J|/2}
(t-q^{|J|})^{n-|J|}$.
\end{corollary}

\begin{proposition}\label{prop:|J| even}
Let~$\{1,n+1\}\subset J\subset [1,n+1]$ with~$2<|J|=2m<n+1$.
Then the assignments~$T_r\mapsto \tau_{\bar r}(J)$, $r\in\{1,2\}$
define a homomorphism~$\Br^+(I_2(|J|))\to\Br^+_{n+1}$
if and only if~$J=\tilde\sigma(J)$.
\end{proposition}
\begin{proof}
By Theorems~\ref{thm:adm I2m} and~\ref{thm:adm I2m converse },
it only remains to prove that~$T_J^{m}$ is not
${}^{op}$-invariant when~$g(J)=1$ and~$J\not=\tilde\sigma(J)$.
By Proposition~\ref{prop:eigenvectors even} and Lemma~\ref{lem:orth eigenspaces}, it suffices to prove that
$$
\la \ker(T_J^m-q^{n+1}\id_V)\,|\,\ker(T_J^m+q^{n+1}\id_V)\ra\not=\{0\}.
$$
By Lemma~\ref{lem:diag aut TJ}, it suffices to consider
$J=J(r,m):=[1,m-r]\cup[n+2-r-m,n+1]$ with $1\le r\le m-1$.
Denote $\tilde U(r,m):=U(\tilde J_m)^{-1}U(J(r,m))$
where~$\tilde J_m=J(0,m)=[1,m]\cup[n+2-m,n+1]=
[1,m]\cup\tilde\sigma([1,m])$.
Since by Corollary~\ref{cor:conj J}
$$
U(J(r,m))T_{J(r,m)}U(J(r,m))^{-1}=T_{[1,2m-1]\cup\{n+1\}}=U(\tilde J_m)T_{\tilde J_m}U(\tilde J_m)^{-1},
$$
we obtain
\begin{equation}\label{eq:T J(r,n) from T tilde J_m}
T_{J(r,m)}=\tilde U(r,m)^{-1} T_{\tilde J_m}\tilde U(r,m),
\end{equation}
where
$$
\tilde U(r,m):=U(\tilde J_m)^{-1}U(J(r,m))=\dscprod_{m-r+1\le k\le m} \Cx{k}{(n-2m+k)}^{((-1)^{k+1})}
$$
by~\eqref{eq:U(J) defn}.
Set
$x_{m-r}=\tilde U(r,m)^{-1}(u_{m-1})$, $1\le r\le m-1$.
Since $\tilde U(r,m)^{-1}(u_{n+2-m})=u_{n+2-m}$
by Lemma~\ref{lem:rev Cox act},
$x_{m-r}\pm u_{n+2-m}\in\ker(T_{J(r,m)}^m\pm q^{n+1}\id_V)$
by~\eqref{eq:T J(r,n) from T tilde J_m} and Proposition~\partref{prop:eigenvectors even.b}.
Therefore, it suffices to prove that
$$
\langle x_{m-r}-u_{n+2-m}\,|\, x_{m-r}+u_{n+2-m}\rangle=\langle x_{m-r}\,|\,x_{m-r}\rangle-(1+q^{-2})\not=0,\qquad 1\le r\le m-1.
$$

First, by Lemma~\ref{lem:rev Cox act}
$$
x_{m-1}=\Cxr m{(n-m)}^{((-1)^m)}(u_{m-1})=q_m^{-(m-1)} w^{((-1)^m)}_{[m-1,n-m]}=\sum_{m-1\le t\le n-m} q_m^{t-m+1}u_t.
$$
Then by~\eqref{eq:Euclid ui uj}
\begin{align*}
\langle x_{m-1}\,&|\,x_{m-1}\rangle=(1+q^{-2})\sum_{m-1\le t\le n-m} q_m^{2(t-m+1)}-2q^{-1}\sum_{m-1\le t\le n-m-1} q_m^{2(t-m+1)+1}
\\&=q^{-1}\Big((q_m+q_m^{-1})\frac{q_m^{2(n-2m+2)}-1}{q_m^2-1}-2q_m \frac{q^{2(n-2m+1)}-1}{q_m^2-1}\Big)
=q^{-1}(q_m^{-1}+q_m^{2(n-2m+1)+1}).
\end{align*}
Thus, since~$q$ is not a root of unity and~$2m<n+1$
$$
\langle x_{m-1}\,|\,x_{m-1}\rangle-(1+q^{-2})=q^{(-1)^m-1}(1-q_m^{2(n-2m+1)})\not=0.
$$
Next we claim that for~$2\le r\le m-1$,
\begin{align}
x_{m-r}&=q_m^{-\overline{r-1}}\sum_{m-r+1\le t\le n-m-r+1} (q_m^{t+r-m}-q_m^{m-r-t})u_t\nonumber\\
&\qquad+
(q_m^{n-2m+1}-q_m^{2m-n-1})\sum_{n-m-r+2\le t\le n-m-1} q_m^{-\overline{t-n+m}}u_t+q_m^{n-2m+1}u_{n-m}.\label{eq:x m-r}
\end{align}
Indeed, it is immediate from the definition of~$x_{m-r}$ that
$
x_{m-r-1}=\Cxr{m-r}{(n-m-r)}^{((-1)^{m+r})}(x_{m-r})
$.
Thus, by Lemma~\ref{lem:rev Cox act}
\begin{align*}
x_{m-2}=\Cxr{(m-1)}{(n-m-1)}^{((-1)^{m+1})}(&x_{m-1})=\sum_{m\le t\le n-m+1} q_m^{t-m}u_{t-1}
-\sum_{m-1\le t\le n-m-1} q_m^{m-t-3}\\
&=q_m^{-1}\sum_{m-1\le t\le n-m-1} (q_m^{t+2-m}-q_m^{m-t-2})u_t+q_m^{n-2m+1}u_{n-m},
\end{align*}
which is~\eqref{eq:x m-r} with~$r=2$.
For the inductive step we have
\begin{align*}
x_{m-r-1}&=\Cxr{m-r}{(n-m-r)}^{((-1)^{m+r})}\Big(q_m^{-\overline{r-1}}\sum_{m-r+1\le t\le n-m-r+1} (q_m^{t+r-m}-q_m^{m-r-t})u_t\Big)+
\\
&\qquad+(q_m^{n-2m+1}-q_m^{2m-n-1})\sum_{n-m-r+2\le t\le n-m-1} q_m^{-\overline{t-n+m}}u_t+q_m^{n-2m+1}u_{n-m}\\
&=q_m^{-\overline{r-1}}q_{m+r}\sum_{m-r+1\le t\le n-m-r} (q_m^{t+r-m}-q_m^{m-r-t})u_{t-1}\\
&\quad +q_m^{-\overline{r-1}}(q_m^{n-2m+1}-q_m^{2m-n+1})(q_{m+r}u_{n-m-r}+u_{n-m-r+1})\\
&\qquad+(q_m^{n-2m+1}-q_m^{2m-n-1})\sum_{n-m-r+2\le t\le n-m-1} q_m^{-\overline{t-n+m}}u_t+q_m^{n-2m+1}u_{n-m}.
\\
\intertext{Since $q_m^{-\overline{r-1}}q_{m+r}=q_m^{(-1)^r-\overline{r-1}}=q_m^{-\overline r}$ by~\eqref{eq: q r+s},}
x_{m-r-1}
&=q_m^{-\overline{r}}\sum_{m-r\le t\le n-m-r-1} (q_m^{t+1+r-m}-q_m^{m-r-1-t})u_{t}%
+q_m^{-\overline{r}}(q_m^{n-2m+1}-q_m^{2m-n+1})u_{n-m-r}\\
&\qquad+(q_m^{n-2m+1}-q_m^{2m-n-1})\sum_{n-m-r+1\le t\le n-m-1} q_m^{-\overline{t-n+m}}u_t+q_m^{n-2m+1}u_{n-m}\\
&=q_m^{-\overline{r}}\sum_{m-r\le t\le n-m-r} (q_m^{t+1+r-m}-q_m^{m-r-1-t})u_{t}\\
&\qquad+(q_m^{n-2m+1}-q_m^{2m-n-1})\sum_{n-m-r+1\le t\le n-m-1} q_m^{-\overline{t-n+m}}u_t+q_m^{n-2m+1}u_{n-m},
\end{align*}
which proves the inductive step and hence~\eqref{eq:x m-r}.
Using~\eqref{eq:Euclid ui uj} and~\eqref{eq:x m-r} we obtain
\begin{align*}
q\langle x_{m-r}&\,|\,x_{m-r}\rangle =(q_m + q_m^{-1}) \Big(\sum_{m-r+1\le t\le n-m-r+1}
     q_m^{-2 \overline{r - 1}} (q_m^{t + r - m} - q_m^{m - t - r})^2 \\
     &\quad+ (q_m^{n + 1 - 2 m} - q_m^{2 m - n - 1})^2 \sum_{n-m-r+2\le t\le n-m-1}
      q_m^{-2 \overline{t - n + m}} +
    q_m^{2 (n + 1 - 2 m)}\Big) \\
&\qquad-
 2 \Big (\sum_{m-r+1\le t\le n-m-r} q_m^{-2 \overline{r - 1}} (q_m^{t + r - m} -
        q_m^{m - t - r}) (q_m^{t + 1 + r - m} - q_m^{m - t - r - 1})\\
&\qquad\qquad + (q_m^{n + 1 - 2 m} - q_m^{2 m - n - 1})^2
      q_m^{-1}(r-2)+
    q_m^{n - 2 m} (q_m^{n + 1 - 2 m} - q_m^{2 m - n - 1})\Big)\\
&=(q_m + q_m^{-1}) \Big(\sum_{m-r+1\le t\le n-m-r+1}
     q_m^{-2 \overline{r - 1}} (q_m^{t + r - m} - q_m^{m - t - r})^2 \\
     &\qquad+ (q_m^{n + 1 - 2 m} - q_m^{2 m - n - 1})^2 \big( q_m^{-2\overline r}\lfloor\tfrac12(r-1)\rfloor+q_m^{-2\overline{r-1}}(\lfloor\tfrac12 r\rfloor-1)\big)+
    q_m^{2 (n + 1 - 2 m)}\Big) \\
&\qquad-
 2 \Big (\sum_{m-r+1\le t\le n-m-r} q_m^{-2 \overline{r - 1}} (q_m^{t + r - m} -
        q_m^{m - t - r}) (q_m^{t + 1 + r - m} - q_m^{m - t - r - 1})\\
&\qquad\qquad + (q_m^{n + 1 - 2 m} - q_m^{2 m - n - 1})^2
      q_m^{-1}(r-2)+q_m^{-1}(q_m^{2(n - 2 m+1)}-1)\Big)\\
&=(q_m + q_m^{-1}) \Big(\sum_{m-r+1\le t\le n-m-r+1}
     q_m^{-2 \overline{r - 1}} (q_m^{2(t + r - m)} + q_m^{2(m - t - r)})-2q_m^{-2\overline{r-1}} \\
     &\qquad+ (q_m^{n + 1 - 2 m} - q_m^{2 m - n - 1})^2 \big( q_m^{-2\overline r}\lfloor\tfrac12(r-1)\rfloor+q_m^{-2\overline{r-1}}(\lfloor\tfrac12 r\rfloor-1)\big)+
    q_m^{2 (n + 1 - 2 m)}\Big) \\
&\qquad-
 2 \Big (q_m^{-2 \overline{r - 1}}\sum_{m-r+1\le t\le n-m-r}  (q_m^{2(t + r - m)+1}+q_m^{2(m-r-t)-1})\\
&\qquad\qquad + (q_m^{n + 1 - 2 m} - q_m^{2 m - n - 1})^2
      q_m^{-1}(r-2) +
    q_m^{-1}(q_m^{2(n - 2 m+1)}-1)\Big)
\\
&=(q_m + q_m^{-1}) ((q_m^{n + 1 - 2 m} -
        q_m^{2 m - n - 1})^2 (q_m^{-2 \overline{r}} \lfloor\tfrac12(r-1)\rfloor +
       q_m^{-2 \overline{r - 1}}\lfloor \tfrac12 r\rfloor) +
    q_m^{2 (n + 1 - 2 m)})\\
&\qquad-
 2 ((q_m^{n + 1 - 2 m} - q_m^{2 m - n - 1)})^2 q_m^{-1} (r - 2) +
    q_m^{-1} (q_m^{2 (n + 1 - 2 m)} - 1))\\
&\qquad-(q_m - q_m^{-1}) q_m^{-2 \overline{r - 1}} \sum_{m-r+1\le t\le n-m-r}
   q_m^{2 (t + r - m)} - q_m^{2 (m - r - t)}\\
&=(q_m + q_m^{-1}) ((q_m^{n + 1 - 2 m} -
        q_m^{2 m - n - 1})^2 (q_m^{-2 \overline{r}} \lfloor \tfrac12 (r - 1) \rfloor +
       q_m^{-2 \overline{r - 1}} \lfloor \tfrac12 r \rfloor) +
    q_m^{2 (n + 1 - 2 m)}) \\
&\qquad-
 2 ((q_m^{n + 1 - 2 m} - q_m^{2 m - n - 1})^2 q_m^{-1} (r - 2) +
    q_m^{-1} (q_m^{2 (n + 1 - 2 m)} - 1))\\
&\qquad+
 q_m^{-2 \overline{r - 1}} (q_m + q_m^{-1} - q_m^{2 (2 m - n - 1) + 1} -
    q_m^{2 (n - 2 m + 1) - 1}).
\end{align*}
This can be rewritten as
\begin{align*}
q\langle x_{m-r}\,|\,x_{m-r}\rangle-(q+q^{-1})=
p_0(q_m)+p_+(q_m)q_m^{2(n-2m+1)}+p_-(q_m)q_m^{-2(n-2m+1)},
\end{align*}
where $p_0,p_\pm\in\mathbb Z[z,z^{-1}]$ are defined by
\begin{align*}
p_+(z)&=z-z^{-1}+(z+z^{-1})\Big(z^{-2\overline r}\lfloor\tfrac12(r-1)\rfloor+z^{-2\overline{r-1}}\lfloor\tfrac12r\rfloor\Big)-2z^{-1}(r-2)-z^{-1-2\overline{r-1}}\\
&=\lfloor\tfrac12(r+1)\rfloor z-\lfloor\tfrac12(r-1)\rfloor(2z^{-1}-z^{-3}),\\
p_-(z)&=(z+z^{-1})\Big(z^{-2\overline r}\lfloor\tfrac12(r-1)\rfloor+z^{-2\overline{r-1}}\lfloor\tfrac12r\rfloor\Big)-2z^{-1}(r-2)-z^{1-2\overline{r-1}}\\
&=\lfloor\tfrac12(r-2)\rfloor(z-2z^{-1})+\lfloor\tfrac12r\rfloor z^{-3},\\
p_0(z)&=-(r-1)(z+z^{-3})+2(r-2)z^{-1}.
\end{align*}
Since $n+1-2m\ge 1$ and~$r\ge 2$, it follows that $q\langle x_{m-r}\,|\,x_{m-r}\rangle-q-q^{-1}
$ is a Laurent polynomial in~$q_m$ with the leading term
$q_m^{2(n-2m+1)+1}\lfloor\tfrac12(r+1)\rfloor$ and hence is non-zero.
\end{proof}
\begin{theorem}\label{thm:adm I2m+1}
Let $\{1,n+1\}\subset J\subset [1,n+1]$.
Let~$\tau_i(J)$, $i\in\{1,2\}$ be as in~\eqref{eq:tau_i(J) defn}.
The assignments $\wh T_r\mapsto \tau_{\overline r}(J)$, $r\in\{1,2\}$
define a fully supported disjoint standard homomorphism $\Br^+(I_2(2m+1))\to \Br^+_{n+1}$, $m\ge 1$
if and only if~$J=[1,n+1]$.
\end{theorem}
\begin{proof}
The forward direction is well known (cf. Theorem~\ref{thm:adm finite class}). For the converse,
if these assignments define a homomorphism of Artin
monoids then, by Lemma~\partref{lem:elem Artin hom.c} we must have~$\ell(\tau_0(J))=\ell(\tau_1(J))$. Furthermore, Lemma~\ref{lem:TJ hom} implies that
$T_J^{2m+1}$ is ${}^{op}$-invariant. By Theorem~\ref{thm:adm I2m converse }
this forces~$g(J)=1$ that is $J=[1,a]\cup [a+r,n+1]$ for some~$1\le
a\le n-1$ and $2\le r\le n+1-a$. Then~$\ell(\tau_{\bar a}(J))=\binom{r+1}2+
k$ and~$\ell(\tau_{1-\bar a}(J))=k'$ where~$k+k'=n-r$ and $k=k'$ if~$n-r$
is even while $|k-k'|=1$ if~$n-r$ is odd.
Since~$\ell(\tau_1(J))=\ell(\tau_0(J))$ this
forces $\binom{r+1}2=k'-k$ which is impossible since~$\binom{r+1}2\ge 3$
for~$r\ge 2$.
\end{proof}

\begin{corollary}\label{cor:type B}
The homomorphisms $\Br^+(I_2(2m))\to \Br^+(B_n)$, $2\le m \le n$ from
Proposition~\ref{prop:admissible hom from BrI22m to BrBn} are
the only fully supported optimal disjoint standard homomorphisms
$I_2(N)\to \Br^+(B_n)$.
\end{corollary}
\begin{proof}
Since the composition of such a homomorphism
with one of the standard unfoldings $\Br^+(B_n)\to \Br^+(A_{2n})$ or~$\Br^+(B_n)\to \Br^+(A_{2n-1})$ (cf.~\eqref{eq:unfold Bn A2n}, \eqref{eq:unfold Bn A2n-1})
is again a homomorphism of the same type,
the assertion follows from Theorem~\ref{thm:main thm adm}.
\end{proof}

\subsection{Higher rank}
\label{subs:rank > 2}
We can now classify all fully supported 
disjoint standard homomorphisms $\Br^+(\wh M)\to\Br^+(M)$ where~$\wh M$ is irreducible
of finite type and~$M$ is of type~$A$ or~$B$.
\begin{theorem}\label{thm:higher rank adm AB}
Let~$\wh M$ be irreducible of finite type 
with~$m=|\wh I|>2$
and let~$\Phi:\Br^+(\wh M)\to \Br^+(M)$ be 
an optimal fully supported disjoint standard homomorphism. Suppose that~$M$ is of type~$A_n$ or~$B_n$.
\begin{enmalph}
 \item\label{thm:higher rank adm AB.a}
 If~$\wh M$ is not of type~$B$ then~$\wh M$ and~$M$ are both of type~$A$ and~$\Phi$ is an isomorphism;
 \item\label{thm:higher rank adm AB.b}
 If~$\wh M=B_m$ and~$M=B_n$ then $m\le n$
 and 
 $$
 \Phi(\wh T_i)=T_i,\quad  i\in [1,m-1],\qquad 
 \Phi(\wh T_m)=T_{w_\circ^{[m,n]}};
 $$
 \item\label{thm:higher rank adm AB.c}
 If~$\wh M=B_m$ and~$M=A_n$ then~$m\le \lfloor\frac n2\rfloor$ and $\Phi$ is the composition 
 of the 
 homomorphism~$\Br^+(\wh M)\to 
 \Br^+(B_{\lfloor \frac n2\rfloor})$
 from part~\ref{thm:higher rank adm AB.b}
 with the standard unfolding~$\Br^+(B_{\lfloor \frac n2\rfloor})\to 
 \Br^+(M)$ given by~\eqref{eq:unfold Bn A2n-1} or~\eqref{eq:unfold Bn A2n}, depending on
 the parity of~$n$.
 \end{enmalph}
\end{theorem}
\begin{proof}
The argument for types in which $\wh m_{ij}$
is odd for all~$i,j\in\wh I$ is the same as in classification of LCM homomorphisms.

Let~$\wh M$ be of type~$B_m$. It is easy
to see, using Proposition~\ref{prop:admissible hom from BrI22m to BrBn}, that the assignments in part~\ref{thm:higher rank adm AB.b} define 
a homomorphism $\Br^+(\wh M)\to \Br^+(B_n)$,
$n\ge m$.

Suppose first that~$M$ is of type~$A$. The restriction of~$\Phi$ to $\Br^+_{[m-1,m]}(\wh M)$ is 
a homomorphism~$\Br^+(I_2(4))\to \Br^+_J(M)$
where~$J=[\Phi](\{m-1,m\})$. By Lemma~\ref{lem:diagonal} and
Theorem~\ref{thm:main thm adm}, $J=\bigcup_{1\le i\le k} [a_i,b_i]$
where~$b_i-a_i\ge 2$, $a_i-b_{i-1}\ge 1$, $2\le i\le k$, $1\le a_1$, $b_k\le n$
and 
$$
[\Phi](m-1)=\bigcup_{1\le i\le k}\{a_i,b_i\},
\qquad [\Phi](m)=\bigcup_{1\le i\le k} [a_i+1,b_i-1].
$$
By Lemma~\ref{lem:diagonal}, Theorem~\ref{thm:main thm adm} and
Theorem~\partref{thm:adm finite class}, $[\Phi](\{m-1,m-2\})$ is the disjoint union of subsets of~$I$ corresponding to submatrices of type~$A_2$.
In particular, $[\Phi](m-2)$ and~$[\Phi](m-1)$ must be self-orthogonal.
Thus, $[\Phi](m-2)=\bigcup_{1\le i\le k}
\{a_i-1,b_i+1\}$. Continuing this way
we conclude that~$[\Phi](j)=
\bigcup_{1\le i\le k} \{ a_i-m+j+1,
b_i+m-j-1\}$, $j\in [1,m-1]$ and are self-orthogonal.
Thus, $a_i-m+2$ and~$b_{i-1}+m-2$ both belong to~$[\Phi](1)$ and so
we must have $a_i-m+2-(b_{i-1}+m-2)>1$. Yet
in that case $a_i-m+1$ does not belong to~$[\Phi](j)$ for any~$j\in [1,m]$ which 
contradicts the assumption that~$\Phi$ is 
fully supported. Thus, $k=1$,
$[\Phi](j)=\{j,n+1-j\}$, $j\in[1,m-1]$
and~$[\Phi](m)=[m,n+1-m]$. 
In particular, $\Phi$ is the composition
of the homomorphism
$\Br^+(B_m)\to \Br^+(B_{\lfloor \frac n2\rfloor})$ with one of the 
homomorphisms from~\eqref{eq:unfold Bn A2n-1},
\eqref{eq:unfold Bn A2n}, depending on 
the parity of~$n$. Furthermore, as 
any disjoint fully supported 
standard homomorphism $\Br^+(B_m)\to 
\Br^+(B_n)$ yields disjoint 
fully supported standard homomorphisms 
$\Br^+(B_m)\to \Br^+(A_{2n})$ 
and~$\Br^+(B_m)\to \Br^+(A_{2n-1})$, it 
follows that the homomorphisms 
described in parts~\ref{thm:higher rank adm AB.b} and~\ref{thm:higher rank adm AB.c} are the only ones with~$M$ of type~$A$ or~$B$.

Finally, let~$\wh M=F_4$ and let~$M=A_n$. A disjoint fully supported standard homomorphism 
$\Phi:\Br^+(\wh M)\to \Br^+(M)$ restricts 
to a disjoint standard homomorphism~$\Br^+_{[1,3]}(F_4)\to \Br^+(M)$. By Lemma~\ref{lem:diagonal}
and part~\ref{thm:higher rank adm AB.c},
$[\Phi]([1,3])=\bigcup_{1\le i\le k} [a_i,b_i]$ where~$a_i-b_{i-1}>1$, $2\le i\le k$ and
$\lceil \frac{b_i-a_i}2\rceil\ge 2$,
$1\le i\le k$. Furthermore, as
$I\setminus [\Phi]([1,3])=[\Phi](4)$ and is self-orthogonal, $b_{i-1}=
a_{i}-2$ for all~$2\le i\le k$ and 
so~$[\Phi](4)=\{a_1-1,\dots,a_k-1,b_k+1\}\cap [1,n]$. Since~$a_1,b_k\in [\Phi](1)$, we conclude that~$b_k=n$ and~$a_1=1$ for otherwise~$\Phi(\wh T_1)$
and~$\Phi(\wh T_4)$ do not commute. Then~$k>1$ and so~$b_{k-1}+1=a_k-1\in[\Phi](4)$,
$b_{k-1}=a_k-2,a_k\in [\Phi](1)$ which is 
again a contradiction. Thus, no such homomorphism exists. It remains to observe that a disjoint fully supported standard homomorphism~$\Br^+(F_4)\to\Br^+(B_n)$ yields a homomorphism~$\Br^+(F_4)\to \Br^+(A_{2n-r})$,
$r\in\{0,1\}$ of the same type by~\eqref{eq:unfold Bn A2n-1}, \eqref{eq:unfold Bn A2n}.
\end{proof}
\begin{proposition}\label{prop:Bm Bn parab}
Let~$2\le m\le n$.
The assignments $\wh T_i\mapsto T_i$, $i\in [1,m-1]$,
$\wh T_m\mapsto \Cx mn\Cxr m{(n-1)}$ define 
a strict parabolic Coxeter type homomorphism~$\Br^+(B_m)\to \Br^+(B_n)$,
and its composition with 
the standard unfolding $\Br^+(B_n)\to \Br^+(A_N)$,
$n=\lfloor \frac12N\rfloor$ is a parabolic 
Coxeter type homomorphism $\Br^+(B_m)\to \Br^+(A_N)$. 
\end{proposition}
\begin{proof}
It is easy to see from Proposition~\partref{prop:admissible hom from BrI22m to BrBn.a} that the above assignments define
a Coxeter type~$\Phi\in\Hom_{\mathscr A}(B_m,B_n)$. We need 
the following
\begin{lemma}
Let~$J\subset [1,m]$. Set~$J'=\emptyset$ if~$J\subset [1,m-1]$ and let~$J'$ be the connected component of~$J$
containing~$m$ otherwise. 
Then $\Phi(\wh T_{w_\circ^J})=
T_{w_\circ^{J\setminus J'}} T_{w_\circ^{[m+1,n]}}^{-1} T_{w_\circ^{[\Phi](J')}}$
in~$\Br(B_n)$.
\end{lemma}
\begin{proof}
Since, clearly, $\Phi(\wh T_{w_\circ^J})=T_{w_\circ^J}$ if~$J\subset [1,m-1]$,
it suffices to prove the assertion for~$J=[i,m]$, $1\le i\le m$. We use descending induction on~$1\le i\le m$. If~$i=m$ the assertion is immediate
from~\eqref{eq:B_n telescope 0}. For $i<m$  
we have by~\eqref{eq:B_n telescope 0} and by 
the induction hypothesis
\begin{align*}
\Phi(\wh T_{w_\circ^{[i,m]}})&=
\Phi(\wh T_{(i,m+1)}\wh T_{w_\circ^{[i+1,m]}})
=T_{(i,n+1)} T_{w_\circ^{[m+1,n]}}^{-1} T_{w_\circ^{[i+1,n]}}=T_{w_\circ^{[m+1,n]}}^{-1} T_{w_\circ^{[i,n]}}.
\qedhere 
\end{align*}
\end{proof}
Note that~$T_{w_\circ^{[\Phi](J')}}$ is central
in~$\Br^+_{[\Phi](J')}(B_n)$ and commutes with~$T_{w_\circ^{J\setminus J'}}$.
Thus, for any~$J\subset [1,m]$ we have by the above Lemma 
\begin{align*}
\Phi(\wh T_{w_J})&=\Phi(\wh T_{w_\circ^J}^{-1})
\Phi(\wh T_{w_\circ^{[1,m]}})=
T_{w_\circ^{J\setminus J'}}^{-1} T_{w_\circ^{[\Phi](J')}}^{-1} T_{w_\circ^{[1,n]}}=T_{w_{(J\setminus J')\cup[\Phi](J')}}
=T_{w_{[\Phi](J)}}.
\end{align*}
Thus, $\Phi$ is strict parabolic.
\end{proof}
Taking the composition of a 
homomorphism from Theorem~\partref{thm:higher rank adm AB.b}
with the standard unfolding
$\Br^+(B_n)\to\Br^+(D_{n+1})$ from~\eqref{eq:unfold Bn Dn+1},
we obtain infinite families
of optimal fully supported disjoint standard homomorphisms $\Br^+(B_m)\to \Br^+(D_{n+1})$
and~$\Br^+(I_2(2m))\to \Br^+(D_{n+1})$
for~$n\ge m$. It appears that these are the only
families existing for arbitrary~$n$. For small ranks there are other sporadic examples, for instance in type~$D_5$
the assignments
$$
\wh T_1\mapsto  T_i,\qquad \wh T_2\mapsto T_{w_\circ^{[1,5]\setminus\{i\}}},\qquad i\in \{4,5\}
$$
define homomorphisms~$\Br^+(I_2(8))\to \Br^+(D_5)$
while the assignments
$$
\wh T_1\mapsto T_{w_\circ^{\{1,i\}}},
\qquad \wh T_2\mapsto T_{w_\circ^{[2,n+1]\setminus\{i\}}},
\qquad i\in \{n,n+1\}
$$
define homomorphisms~$\Br^+(I_2(10))\to \Br^+(D_{n+1})$
for~$n\in\{4,5\}$. In higher ranks no such
homomorphisms seem to exists, apart from those obtained from type~$B$ via the standard unfolding~\eqref{eq:unfold Bn Dn+1}, but there are a lot
of apparently infinite families of non-disjoint standard 
homomorphisms (see~\S\ref{subs:conj families}).

\subsection{Families of non-disjoint parabolic
and standard homomorphisms}
\label{subs:mon braid}
The family constructed here is 
inspired by braidings of tensor powers of objects in braided monoidal
categories. We will
identify $\Br^+_k$ with the parabolic
submonoid~$\Br^+_{[1,k-1]}(A_{n-1})$ of~$\Br^+_n$ for all~$n>k$.
\begin{theorem}\label{thm:monomial brd}
Let~$m\in\ZZ_{>1}$, $n\in\ZZ_{>0}$
and let $J_i^{(m)}:=[(i-1)m+1,(i+1)m-1]$,
$i\in\mathbb Z_{>0}$.
\begin{enmalph}
\item \label{thm:monomial brd.a}
The assignments $T_i\mapsto 
T_{w_{J_i^{(m)}\setminus \{im\};J_i^{(m)}}}$,
$i\in[1,n-1]$ define a strict parabolic 
Coxeter type homomorphism
\plink{Phi(m)n}$\Phi_n^{(m)}:\Br^+_n\to\Br^+_{nm}$;

\item \label{thm:monomial brd.b}
The assignments $T_i\mapsto 
T_{w_\circ^{J_i^{(m)}}}$,
$i\in[1,n-1]$, define 
a standard homomorphism $\wh\Phi_n^{(m)}:\Br^+_n\to\Br^+_{nm}$.

Let~$\wh M=B_n$, $\widetilde M=B_{mn}$.
\item \label{thm:monomial brd.c}
    The assignments $\wh T_i\mapsto \widetilde T_{w_{J_i^{(m)}\setminus\{im\}};J_i^{(m)}}$,
    $i\in[1,n-1]$, $\wh T_n\mapsto \widetilde T_{w_{[(n-1)m+1,nm-1];[(n-1)m+1,nm]}}$ define a 
    parabolic Coxeter type homomorphism $\Br^+(\wh M)\to \Br^+(\widetilde M)$;
\item \label{thm:monomial brd.d}
     The assignments $\wh T_i\mapsto \widetilde T_{w_\circ^{J_i^{(m)}}}$,
    $i\in[1,n-1]$, $\wh T_n\mapsto 
    \widetilde T_{w_\circ^{[(n-1)m+1,nm]}}$ define a 
    standard homomorphism $\Br^+(\wh M)\to \Br^+(\widetilde M)$.  
\end{enmalph}
\end{theorem}
\begin{proof}
It will be convenient to consider all the $\Br^+_k$, $k\in\ZZ_{\ge 1}$ as parabolic submonoids 
of~$\Br^+_\infty=\Br^+(A_\infty)$, which is generated 
by the $T_i$, $i\in\mathbb Z_{>0}$ subject 
to relations $T_i T_j=T_j T_i$, $|i-j|>1$ and 
$T_i T_j T_i=T_j T_i T_j$, $|i-j|=1$ for all~$i,j\in\mathbb Z_{>0}$. Then we can consider~$\Phi^{(m)}$
as an endomorphism of~$\Br^+_\infty$. Likewise,
we consider symmetric groups~$S_n$ as parabolic 
subgroups of~$S_\infty$.

We need the following
\begin{lemma}\label{lem:prep Grassman perm}
Let~$I=[1,2m-1]$ and let
$t_m=\prod_{1\le i\le m}(i,m+i)\in S_{2m}\cong W(A_{2m-1})$.
Then 
\begin{enmalph}
 \item   \label{lem:prep Grassman perm.a} $t_m=\pi_{2m-1}(T_{w_{I\setminus\{m\};I}})$;
 \item 
 \label{lem:prep Grassman perm.b} 
 $t_m=\ascprod_{1\le i\le m}\cxr i{(m+i-1)}$
 and $T_{w_{I\setminus\{m\};I}}=
 \ascprod_{1\le i\le m}\Cxr i{(m+i-1)}$.
\end{enmalph}
\end{lemma}
\begin{proof}

Denote~$\eta_m:=t_m w_\circ^{[1,m-1]}w_\circ^{[m+1,2m-1]}$. 
Let~$i\in [1,m]$. Then $\eta_m(i) =t_m(m+1-i)=2m+1-i$. Similarly,
if~$i\in[m+1,2m]$, $\eta_m(i)=
t_m(3m+1-i)=2m+1-i$. Thus, $\eta_m=
\prod_{1\le i\le m} (i,2m+1-i)=w_\circ^{[1,2m-1]}$ and so
$t_m=w_{I\setminus\{m\};I}$. This proves~\ref{lem:prep Grassman perm.a}. 
To prove the first identity in part~\ref{lem:prep Grassman perm.b}, note that since $\cxr ij=(i,j+1,j,\dots,i+1)$  we have for~$1\le j\le m$
\begin{align*}
\Big(\ascprod_{1\le i\le m}\cxr i{(m+i-1)}\Big)(j)&=\Big(\ascprod_{1\le i\le j-1}\cxr i{(m+i-1)}\Big)
\cxr j{(m+j-1)}(j)\\
&=\Big(\ascprod_{1\le i\le j-1}
\cxr i{(m+i-1)}\Big)(m+j)=m+j
\\
\intertext{and}
\Big(\ascprod_{1\le i\le m}\cxr i{(m+i-1)}\Big)(m+j)&=
\Big(\ascprod_{1\le i\le m-1}\cxr i{(m+i-1)}\Big)(m+j-1)=\cdots
\\
&=\Big(\ascprod_{1\le i\le m-k}\cxr i{(m+i-1)}\Big)(m-k+j)=\cdots=j.
\end{align*}
The second identity follows from the first 
by Theorem~\partref{thm:Tits.b} since 
$\ell(T_{w_{I\setminus\{m\};I}})=
\ell(t_m)=\ell(w_\circ^{[1,2m-1]})-
\ell(w_\circ^{[1,m-1]})-\ell(w_\circ^{[m+1,2m-1]})=m^2=\sum_{1\le i\le m}
\ell(\Cxr i{(m+i-1)})$.
\end{proof}
Note that the assignments $T_i\mapsto T_{i+1}$,
$i\in\ZZ_{>0}$ define an endomorphism~$\xi$ of~$\Br^+_\infty$
which clearly descends to the Coxeter group.
We claim that the assignments
$$
s_i\mapsto t_{i,m}:=\xi^{(i-1)m}(t_m)=\prod_{1\le j\le m} ((i-1)m+j,
i m+j),
\qquad i\in[1,n-1]
$$
define an endomorphism of~$S_\infty$ which 
restricts to homomorphisms $S_n\to S_{nm}$ for any~$n\in\ZZ_{>0}$. Indeed, since the $t_{i,m}$, $t_{j,m}$ with~$|j-i|>1$ manifestly commute, it suffices to verify that $t_{1,m}t_{2,m}t_{1,m}=t_{2,m}t_{1,m}t_{2,m}$.
Note that 
$$
t_{i,m}(j)=\begin{cases}
j+m,&j\in[(i-1)m+1,im],\\
j-m,&j\in[im+1,(i+1)m],\\
j,&\text{otherwise}.
\end{cases}
$$
It follows that for all~$j\in[1,3m]$
\begin{equation}\label{eq:mon braiding rel}
t_{1,m}t_{2,m}t_{1,m}(j)=t_{2,m}t_{1,m}t_{2,m}(j)=\begin{cases}
j+2m,& j\in[1,m],\\
j,&j\in[m+1,2m],\\
j-2m,&j\in[2m+1,3m].
\end{cases}
\end{equation}
Furthermore, it is well-known (see e.g.~\cite{BjBr}*{Proposition~1.25})
that for~$w\in W(A_{k})\cong S_{k+1}$ where the isomorphism maps
$s_i$, $i\in[1,k]$ to the transposition~$(i,i+1)$, 
we have $\ell(w)=|\Inv(w)|$ where
$$\Inv(w)=\{ (i,j)\in[1,k+1]\times[1,k+1]\,:\,
i<j,\,w(i)>w(j)\}.
$$
It is immediate
from~\eqref{eq:mon braiding rel} that
\begin{multline*}
\Inv(t_{1,m}t_{2,m}t_{1,m})=([1,m]\times[m+1,2m])
\cup ([1,m]\times[2m+1,3m])\cup\\ ([m+1,2m]\times[2m+1,3m]).
\end{multline*}
Therefore, $\ell(t_{1,m}t_{2,m}t_{1,m})=
3m^2=\ell(t_{1,m})+\ell(t_{2,m})+\ell(t_{1,m})$. 
Then by Lemmata~\ref{lem:lifting to Cox-Hecke} and~\partref{lem:prep Grassman perm.a} the assignments in part~\ref{thm:monomial brd.a} define a homomorphism which is of Coxeter type
by Proposition~\partref{prop:elem prop Coxeter Hecke.a}.

To prove that~$\Phi^{(m)}_n$ is parabolic, we need the following
\begin{lemma}\label{lem:mon braid Tw0J}
For any~$m\in\ZZ_{\ge 1}$, $a\le b\in\ZZ_{\ge 1}$, we have in $\Br^+_\infty$
$$
\Phi^{(m)}(T_{w_\circ^{[a,b]}})=T_{w_{[\Phi^{(m)}]([a,b])\setminus\ZZ m;[\Phi^{(m)}]([a,b])}}.
$$
\end{lemma}
\begin{proof}
It suffices to prove the Lemma for~$[a,b]=[1,n]$,
$n\in\ZZ_{\ge 1}$; the general case follows by
applying the endomorphism~$\xi$. Note that
\begin{equation}\label{eq:Phi(m)[a,b]}
[\Phi^{(m)}]([a,b])=[(a-1)m+1,(b+1)m-1].
\end{equation}
First, we prove that
\begin{equation}\label{eq:induced version}
\overline{\Phi^{(m)}}(w_\circ^{[1,n]})=
w_{[1,(n+1)m-1]\setminus\ZZ m; [1,(n+1)m-1]}.
\end{equation}
To prove this identity we need to establish first that
\begin{equation}\label{eq:img Cox}
\overline{\Phi^{(m)}}(\cxr 1k)=
\ascprod_{1\le i\le m}\cxr i{(km+i-1)}, \qquad k\in\mathbb Z_{\ge 1}.
\end{equation}
Indeed, for~$k=1$ this was 
proven in Lemma~\partref{lem:prep Grassman perm.b}. Therefore, for~$k>1$
\begin{align*}
\overline{\Phi^{(m)}}(\cxr 1k)&=
\overline{\Phi^{(m)}}(s_k)\overline{\Phi^{(m)}}(\cxr 1{(k-1)})
=
\ascprod_{1\le i\le m}\cxr{((k-1)m+i)}{(km+i-1)}
\ascprod_{1\le i\le m}\cxr i{((k-1)m+i-1)}\\
&=
\ascprod_{1\le i\le m}
\cxr{((k-1)m+i)}{(km+i-1)}
\cxr i{((k-1)m+i-1)}
=\ascprod_{1\le i\le m}\cxr i{(km+i-1)}.
\end{align*}
It is now immediate from~\eqref{eq:img Cox} that
$$
\Big(\prod_{1\le j\le km-1}\cxr 1j\Big)
\overline{\Phi^{(m)}}(\cxr 1k)
\Big(\prod_{1\le j\le m-1}\cxr 1j\Big)
=\prod_{1\le j\le (k+1)m-1}\cxr 1j=
w_\circ^{[1,(k+1)m-1]},
$$
whence 
$$
\overline{\Phi^{(m)}}(\cxr 1k)=
w_\circ^{[1,km-1]}w_\circ^{[1,(k+1)m-1]}w_\circ^{[1,m-1]}=
w_{[1,(k+1)m-1]\setminus\{km\};[1,(k+1)m-1]}.
$$
We now use induction on~$n$ to prove~\eqref{eq:induced version}, the induction base being obvious. For the inductive step, since $w_\circ^{[1,n]}=w_\circ^{[1,n-1]}\cxr 1n$, we get
\begin{align*}
\overline{\Phi^{(m)}}(w_\circ^{[1,n]})&=w_\circ^{[1,nm-1]\setminus\ZZ m}w_\circ^{[1,nm-1]}
w_\circ^{[1,(n+1)m-1]\setminus\{nm\}}
w_\circ^{[1,(n+1)m-1]}
\\
&=
w_\circ^{[1,nm-1]\setminus\ZZ m}w_\circ^{[nm+1,(n+1)m-1]}
w_\circ^{[1,(n+1)m-1]}
=w_\circ^{[1,(n+1)m-1]\setminus\ZZ m}w_\circ^{[1,(n+1)m-1]}.
\end{align*}
To complete the proof of Lemma~\ref{lem:mon braid Tw0J} it remains to observe that
\begin{align*}
\ell(w_\circ^{[1,(n+1)m-1]\setminus\ZZ m}&w_\circ^{[1,(n+1)m-1]})=\tfrac12 (n+1)m((n+1)m-1)-\tfrac12(n+1)m(m-1)\\
&=\tfrac12 m^2 n(n+1)
=\ell(\Phi^{(m)}(T_{w_\circ^{[1,n]}})).\qedhere
\end{align*}
\end{proof}
Now we have all necessary ingredients to prove that
$\Phi^{(m)}_n$ is parabolic. Let~$J\subset [1,n]$
and write $J=\bigcup_{1\le i\le r}[a_i,b_i]$,
$a_i\le b_i$, $1\le i\le r$ and~$b_i<a_{i+1}$, $1\le i\le r-1$. Using Lemma~\ref{lem:mon braid Tw0J} and~\eqref{eq:Phi(m)[a,b]} we obtain in~$\Br_\infty$
\begin{align*}
\Phi^{(m)}(&T_{w_{J;[1,n]}})=
(\Phi^{(m)}(T_{w_\circ^J}))^{-1}\Phi^{(m)}(T_{w_\circ^{[1,n]}})\\
&
=\Big(
\prod_{1\le i\le r} T_{w_{[\Phi^{(m)}]([a_i,b_i])\setminus\ZZ m; [\Phi^{(m)}]([a_i,b_i])}}^{-1}\Big)
T_{w_{\Phi^{(m)}([1,n])\setminus\ZZ m; [\Phi^{(m)}]([1,n])}}\\
&=\Big(
\prod_{1\le i\le r} T_{w_\circ^{[\Phi^{(m)}]([a_i,b_i])}}^{-1}
T_{w_\circ^{[\Phi^{(m)}]([a_i,b_i])\setminus\ZZ m}}\Big)
T_{w_\circ^{\Phi^{(m)}([1,n])\setminus\ZZ m}}^{-1}
T_{w_\circ^{[\Phi^{(m)}]([1,n])}}\\
&=\Big(
\prod_{1\le i\le r} T_{w_\circ^{[\Phi^{(m)}]([a_i,b_i])}}^{-1}\Big)\Big(\prod_{1\le i\le r}
T_{w_\circ^{[\Phi^{(m)}]([a_i,b_i])\setminus\ZZ m}}\Big)
T_{w_\circ^{\Phi^{(m)}([1,n])\setminus\ZZ m}}^{-1}
T_{w_\circ^{[\Phi^{(m)}]([1,n])}}\\
&=
T_{w_\circ^{[\Phi^{(m)}](J)}}^{-1}
\Big(\prod_{1\le i\le r}\prod_{a_i\le s\le b_i+1}
T_{w_\circ^{[(s-1)m+1,sm-1]}}\Big)
\Big(\prod_{1\le s\le n+1} T_{w_\circ^{[(s-1)m+1,s m-1]}}^{-1}\Big)
T_{w_\circ^{[\Phi^{(m)}]([1,n])}}\\
&=
T_{w_\circ^{[\Phi^{(m)}](J)}}^{-1}
\Big(\prod_{s\in[1,n+1]\setminus\bigcup_{1\le i\le r}[a_i,b_i+1]}
T_{w_\circ^{[(s-1)m+1,sm-1]}}^{-1}\Big)T_{w_\circ^{[\Phi^{(m)}]([1,n])}}\\
&=
T_{w_\circ^{[\Phi^{(m)}](J)}}^{-1}
\Big(\prod_{s\in[1,n]\setminus\bigcup_{1\le i\le r}[a_i,b_i]}
T_{w_\circ^{[(s-1)m+1,(s+1)m-1]\setminus \{sm\}}}^{-1}\Big)T_{w_\circ^{[\Phi^{(m)}]([1,n])}}\\
&=T_{w_\circ^{[\Phi^{(m)}](J)}}^{-1}
T_{w_\circ^{[\Phi^{(m)}]([1,n]\setminus J)\setminus\ZZ m}}^{-1}T_{w_\circ^{[\Phi^{(m)}]([1,n])}}=T_{w_{[\Phi^{(m)}](J)\cup([\Phi^{(m)}]([1,n]\setminus J)\setminus\ZZ m);
[\Phi^{(m)}]([1,n])}}.
\end{align*}
This completes the proof of part~\ref{thm:monomial brd.a}.

As~$T_{w_\circ^{J_i^{(m)}}}$, $T_{w_\circ^{J_k^{(m)}}}$
with~$|i-k|>1$ manifestly commute,
it suffices to prove part~\ref{thm:monomial brd.b} for~$n=3$. Note that, since~$T_{w_\circ^{J_i^{(m)}\setminus\{im\}}}$ is invariant with respect to the 
diagram automorphism of~$\Br^+_{J_i^{(m)}}(A_\infty)$ and,
therefore, commutes with~$T_{w_\circ^{J_i^{(m)}}}$
by Proposition~\partref{prop:fund elts BrSa.e}, we have 
$$
T_{w_\circ^{J_i^{(m)}}}=\Phi^{(m)}(T_i)z_{i,i}^{(1)},\qquad 
z_{i,i}^{(1)}:=T_{w_\circ^{J_i^{(m)}\setminus\{im\}}}=
T_{w_\circ^{[(i-1)m+1,im-1]]\cup [im+1,(i+1)m-1]}}.
$$
Using~\eqref{eq:inv decoration}
and Proposition~\partref{prop:fund elts BrSa.e} we obtain in~$\Br^+_{3m}$
\begin{align*}
z_{i,k}^{(2)}&=\Phi^{(m)}(T_k)^{-1}T_{w_\circ^{J_i^{(m)}\setminus\{im\}}}
\Phi^{(m)}(T_k)=
T_{w_\circ^{J_k^{(m)}}}^{-1}T_{w_\circ^{J_k^{(m)}\setminus\{km\}}}T_{w_\circ^{J_i^{(m)}\setminus\{im\}}}
T_{w_\circ^{J_k^{(m)}\setminus\{km\}}}^{-1}T_{w_\circ^{J_k^{(m)}}}\\
&=T_{w_\circ^{[1,m-1]\cup[2m+1,3m-1]}},\\
z_{i,k}^{(2)}&=\Phi^{(m)}(T_i)^{-1}T_{w_\circ^{[1,m-1]\cup[2m+1,3m-1]}}
\Phi^{(m)}(T_i)=T_{w_\circ^{[m+1,2m-1]\cup[(k-i+1)m+1,(k-i+2)m+1]}},
\end{align*}
where~$\{i,k\}=\{1,2\}$, whence
\begin{align*}
z_{i,k}^{(3)}z_{k,i}^{(2)}&z_{i,k}^{(1)}
=T_{w_\circ^{[m+1,2m-1]\cup[2(2-i)m+1,(2(2-i)+1)m-1]}}
T_{w_\circ^{[1,m-1]\cup[2m+1,3m-1]}}T_{w_\circ^{[(i-1)m+1,im-1]\cup[im+1,(i+1)m-1]}}\\
&=T_{w_\circ^{[1,m-1]}}^2 T_{w_\circ^{[m+1,2m-1]}}^2
T_{w_\circ^{[2m+1,3m-1]}}^2.
\end{align*}
Therefore, 
$\mathbf z=(z_{1,1}^{(1)},z_{2,2}^{(1)})$
is a decoration of~$\Phi^{(m)}_3$ by Theorem~\ref{thm:decoration sufficient}, and~$\wh \Phi^{(m)}_3=(\Phi^{(m)}_3)_{\mathbf z}$.

To prove part~\ref{thm:monomial brd.c}, we need the following
\begin{lemma}\label{lem:diag aut mon brd}
Let~$\sigma_N$ be the diagram automorphism of~$\Br^+_N$ and let~$\Br^+_N{}^{\sigma_N}$ be 
the submonoid of~$\sigma_N$-invariant elements of~$\Br^+_N$. Then
$\Phi^{(m)}_n\circ \sigma_n=\sigma_{mn}\circ\Phi^{(m)}_n$
and
$\wh \Phi^{(m)}_n\circ \sigma_n=\sigma_{mn}\circ\wh\Phi^{(m)}_n$ for all~$m,n\in\mathbb Z_{>1}$. In particular,
$\Phi^{(m)}_n$ and~$\wh \Phi^{(m)}_n$
restrict to homomorphisms $\Br^+_n{}^{\sigma_n}\to 
\Br^+_{mn}{}^{\sigma_{mn}}$. 
\end{lemma}
\begin{proof}
Since~$\sigma_N$ corresponds to the permutation
$i\mapsto N+1-i$, $i\in [1,N-1]$, we have 
\begin{align*}
\sigma_{mn}(\wh\Phi^{(m)}_n(T_i))&=
\sigma_{mn}(T_{w_\circ^{[(i-1)m+1,(i+1)m-1]}})=
T_{w_\circ^{[nm-(i+1)m+1,nm-(i-1)m-1]}}\\
&=T_{w_\circ^{[(n-i-1)m+1,(n-i+1)m-1]}}
=\wh\Phi^{(m)}_n(T_{n-i})
=\wh\Phi^{(m)}_n(\sigma_n(T_i)).
\end{align*}
The argument for~$\Phi^{(m)}_n$ is similar and is omitted.
\end{proof}
Let~$\Upsilon_k:\Br^+(B_k)\to\Br^+_{2k}$ be the 
standard unfolding~\eqref{eq:unfold Bn A2n-1} which is an isomorphism
onto~$\Br^+_{2k}{}^{\sigma_{2k}}$
by Theorem~\ref{thm:adm finite class}.
Then $\Phi^{(m)}_{2n}\circ\Upsilon_n$ is a homomorphism
$\Br^+(\wh M)\to\Br^+_{2mn}$ whose 
image is contained in~$\Br^+_{2mn}{}^{\sigma_{2mn}}$.
It follows that
$\Upsilon_{mn}^{-1}\circ\Phi^{(m)}_{2n}\circ\Upsilon_n\in\Hom_{\mathscr A}(\wh M,\widetilde M)$ and is 
parabolic as the composition of parabolic homomorphisms. To obtain the explicit formulae, 
note that we have for~$i\in[1,n-1]$
\begin{align*}
(\Phi^{(m)}_{2n}\circ\Upsilon_n)(\wh T_i)&=
T_{w_{J_i^{(m)}\setminus\{im\};J_i^{(m)}}}
T_{w_{J_{2n-i}^{(m)}\setminus\{(2n-i)m\};J_{2n-i}^{(m)}}}\\
&=T_{w_{(J_i^{(m)}\setminus\{im\})\cup (J_{2n-i}^{(m)}\setminus\{(2n-i)m\});J_i^{(m)}\cup J_{2n-i}^{(m)}}}
=\Upsilon_{mn}(\widetilde T_{w_{J_i^{(m)}\setminus\{im\};J_i^{(m)}}})
\end{align*}
while
\begin{align*}
(\Phi^{(m)}_{2n}\circ\Upsilon_n)(\wh T_n)&=
T_{w_{J_n^{(m)}\setminus\{nm\};J_n^{(m)}}}
=T_{w_{[(n-1)m+1,[nm-1]\cup[nm+1,(n+1)m-1];[(n-1)m+1,(n+1)m-1]}}\\
&=\Upsilon_{nm}(\widetilde T_{w_{[(n-1)m+1,nm-1];
[(n-1)m+1,nm-1]}}),
\end{align*}
where we used Lemma~\ref{lem:parab preserving}.
In particular, this homomorphism is of Coxeter type.
The argument in
part~\ref{thm:monomial brd.d} is similar 
and is omitted.
\end{proof}

\begin{example}
Explicitly, we have
\begin{align*}&\Phi_n^{(2)}(T'_i)=T_{2i}T_{2i-1}T_{2i+1}T_{2i},\\
&\Phi_n^{(3)}(T'_i)=T_{3i}T_{3i-1}T_{3i-2}T_{3i+1}T_{3i}T_{3i-1}
T_{3i+2}T_{3i+1}T_{3i},\qquad i\in[1,n-1].
\end{align*}
\end{example}

\subsection{More infinite series of non-disjoint standard 
homomorphisms}\label{subs:inf ser non-disj}
We now use Theorem~\ref{thm:decoration sufficient} and
homomorphisms from Theorems~\partref{thm:monomial brd.b}\ref{thm:monomial brd.d} and~\partref{thm:higher rank adm AB.b}
to obtain additional infinite families of standard homomorphisms~$\Br^+_3\to\Br^+_{3m}$, $m\ge 1$,
and~$\Br^+(B_2)\to\Br^+(M)$ where~$M$ is of type~$A_n$,
$B_n$ or~$D_{n+1}$.

\begin{theorem}\label{thm:Hom A2}
For~$m\in\ZZ_{>0}$ and~$J\subset [1,m-1]$, 
the assignments 
$T_1\mapsto T_{w_\circ^{[1,2m-1]\cup (2m+J)}}$,
$T_2\mapsto T_{w_\circ^{[m+1,3m-1]\cup J}}$
define a homomorphism $\Br^+_3\to\Br^+_{3m}$.
\end{theorem}
\begin{proof}
 let~$M=A_3$,  $\mathsf M=\Br_{3m}$, $m\ge 2$ and let~$\Phi$ be the homomorphism from Theorem~\partref{thm:monomial brd.b}.
Thus, $\Phi(T_i)=T_{w_\circ^{[(i-1)m+1,(i+1)m-1]}}$,
$i\in\{1,2\}$.
Let~$\sigma_i$, $i\in\{1,2\}$ be the diagram 
automorphism of~$\Br^+_{[(i-1)m+1,(i+1)m-1]}(A_{3m-1})\cong \Br^+_{2m}$. By Proposition~\partref{prop:fund elts BrSa.e}, $\Phi(T_i)^{-1}(T_{w_\circ^K})\Phi(T_i)=\sigma_i(T_{w_\circ^K})=
T_{w_\circ^{2im-K}}$
for any~$K\subset[(i-1)m+1,(i+1)m-1]$. 
Let~$z_{i,i}^{(1)}=T_{w_\circ^{2m(2-i)+J}}$, 
$i\in\{1,2\}$
and define by~\eqref{eq:inv decoration}
\begin{align*}
z_{i,j}^{(2)}&=\Phi(T_j)^{-1}z_{i,i}^{(1)}\Phi(T_{j})=
T_{w_\circ^{2mj-2m(2-i)-J}}=
T_{w_\circ^{2m-J}},\\
z_{i,j}^{(3)}&=\Phi(T_{i})^{-1}z_{i,j}^{(2)}\Phi(T_{i})
=T_{w_\circ^{2(i-1)m+J}}=z_{j,j}^{(1)}
\end{align*}
where~$\{i,j\}=\{1,2\}$.
Since~$J\subset[1,m-1]$, $2m+J\subset[2m+1,3m-1]$ and~$2m-J\subset[m+1,2m-1]$, they are pairwise orthogonal. 
Then $z_{1,2}^{(3)}z_{2,1}^{(2)}z_{1,2}^{(1)}=
T_{w_\circ^{J}}T_{w_\circ^{2m-J}}T_{w_\circ^{2m+J}}
=T_{w_\circ^{2m+J}}T_{w_\circ^{2m-J}}T_{w_\circ^{J}}
=z_{2,1}^{(3)}z_{1,2}^{(2)}z_{2,1}^{(1)}$. Therefore,
$(z_{1,1}^{(1)},z_{2,2}^{(1)})=(T_{w_\circ^{2m+J}},T_{w_\circ^J})$ is a decoration of~$\Phi$ by Theorem~\ref{thm:decoration sufficient}
and the assertion follows.   
\end{proof}
\begin{theorem}\label{thm:Hom B2}
Let~$m\in\ZZ_{\ge 0}$ and~$J\subset[1,m-1]$.
\begin{enmalph}    
\item\label{thm:Hom B2.B2An}
Let~$n\ge 4m-1$ and suppose that~$K\subset[2m+1,n-2m]$
and~$n+1-K$ are weakly orthogonal.
Then the assignments $\wh T_1\mapsto T_{w_\circ^{[1,2m-1]
\cup[n+2-2m,n]\cup K}}$, $\wh T_2\mapsto T_{w_\circ^{[m+1,n-m]\cup J\cup (n+1-J)}}$ define a homomorphism~$\Br^+(B_2)\to\Br^+(A_n)$;

\item\label{thm:Hom B2.B2} Let~$n\ge 2m$ and~$K\subset[2m+1,n]$. Then
the assignments $\wh T_1\mapsto T_{w_\circ^{[1,2m-1]\cup K}}$, $\wh T_2\mapsto T_{w_\circ^{[m+1,n]\cup J}}$ define a homomorphism~$\Br^+(B_2)\to\Br^+(B_n)$;

\item\label{thm:Hom B2.B2Dn}
Let~$n\ge 2m$.
Let~$K\subset[2m+1,n+1]$ and, if~$n-m$ is even, assume in addition that~$K$ and~$\tau(K)$ are weakly orthogonal where~$\tau$
is the transposition~$(n,n+1)$.
Then the assignments 
$\wh T_1\mapsto T_{w_\circ^{[1,2m-1]\cup K}}$, $\wh T_2\mapsto T_{w_\circ^{[m+1,n+1]\cup J}}$ define a homomorphism~$\Br^+(B_2)\to\Br^+(D_{n+1})$.
\end{enmalph}
\end{theorem}
\begin{proof}
We need the following
\begin{proposition}\label{prop:basic Hom B2Bn}
Let~$m\in\ZZ_{\ge 0}$ and~$n\ge 2m$. Then the assignments
$\wh T_1\mapsto T_{w_\circ^{[1,2m-1]}}$,
$\wh T_2\mapsto T_{w_\circ^{[m+1,n]}}$ define 
a homomorphism $\Phi_n\in\Hom_{\mathscr A}(B_2,B_n)$.
\end{proposition}
\begin{proof}
For~$m=0$, these assignments define a character homomorphism.

Suppose that~$m>1$.
We use induction on~$n-2m$. 
For~$n=2m$, this is a special case of Theorem~\partref{thm:monomial brd.d}. 
For the inductive step, we need the following
\begin{lemma}\label{lem:img longest}
Let~$\Psi_n\in\Hom_{\mathscr A}(B_n,B_{n+1})$ be the homomorphism 
from Theorem~\partref{thm:higher rank adm AB.b}.
Then for any~$J\subset[1,n]$, $\Psi_n(\wh T_{w_\circ^J})=T_{w_\circ^{[\Psi_n](J)}}T_{n+1}^{|J'|-1}$ where~$J'$ is the connected
component of~$J$ containing~$n$.
\end{lemma}
\begin{proof}
Since~$[\Psi_n](J)=J$ if~$J\subset[1,n-1]$, $[\Psi_n](J)=J\cup\{n+1\}$ if~$n\in J$ and~$\Psi_n(\wh T_i)=T_i$,
$i\in[1,n-1]$, it follows that~$\Psi_n(\wh T_{w_\circ^{J\setminus J'}})=T_{w_\circ^{J\setminus J'}}$. Thus, it suffices to prove the lemma for~$J=J'=[1,n]$. By Proposition~\partref{prop:Coxeter splitting.b},
$\wh T_{w_\circ^{[1,n]}}=\Cx1n{}^n$ and since~$\Psi_n(T_i)=T_i$, $i\in[1,n-1]$, $\Psi_n(T_n)=T_nT_{n+1}T_nT_{n+1}$, we obtain
$\Psi_n(\wh T_{w_\circ^{[1,n]}})=(\Cx1{(n+1)}T_nT_{n+1})^n$.
We claim that for all~$1\le k\le n$ 
\begin{equation}\label{eq:Cox power Psi}
(\Cx1{(n+1)}T_nT_{n+1})^k=
(\Cx1{(n+1)})^k \Cx{(n+1-k)}{(n+1)}T_{n+1}^{k-1}.
\end{equation}
Indeed, for~$k=1$ there is nothing to prove. For the inductive step, we have 
\begin{align*}
(\Cx1{(n+1)}T_nT_{n+1})^{k+1}=
(\Cx1{(n+1)})^k \Cx{(n+1-k)}{(n+1)}T_{n+1}^{k-1}
\Cx1{(n+1)}T_n T_{n+1}.
\end{align*}
Since~$\Br^+(B_{n+1})$ is cancellative, it thus suffices to prove that 
\begin{equation}\label{eq:Cox power Psi 1}
\Cx{(n+1-k)}{(n+1)}T_{n+1}^{k-1}
\Cx1{(n+1)}T_n=\Cx1{(n+1)}\Cx{(n-k)}{(n+1)}
T_{n+1}^{k-1}.
\end{equation}
Since~$T_{n+1}^{k-1}\Cx1{(n+1)}T_n=\Cx1{(n-1)}
T_{n+1}^{k-1}T_nT_{n+1}T_n=
\Cx1{(n+1)}T_n T_{n+1}^{k-1}$, by cancellativity \eqref{eq:Cox power Psi 1}
is equivalent to
\begin{equation}\label{eq:Cox power Psi 2}
\Cx{(n+1-k)}{(n+1)}\Cx1{(n+1)}T_n=\Cx1{(n+1)}\Cx{(n-k)}{(n+1)}.
\end{equation}
We have 
\begin{align*}
\Cx{(n+1-k)}{(n+1)}\Cx1{(n+1)}T_n
&=\Cx{(n+1-k)}n\Cx1{(n-1)}T_{n+1} T_n T_{n+1}T_n\\
&=\Cx{(n+1-k)}n\Cx1{n}T_{n+1} T_{n}T_{n+1}
\end{align*}
while~$
\Cx1{(n+1)}\Cx{(n-k)}{(n+1)}=
\Cx1n\Cx{(n-k)}{(n-1)}T_{n+1}T_nT_{n+1}$.
Therefore, \eqref{eq:Cox power Psi 2} is equivalent
to
\begin{equation}\label{eq:Cox power Psi 3} 
\Cx{(n+1-k)}n\Cx1{n}=\Cx1n\Cx{(n-k)}{(n-1)}
\end{equation}
which, since both sides of~\eqref{eq:Cox power Psi 3}
are contained in~$\Br^+_{[1,n]}(B_{n+1})\cong\Br^+(A_n)$,
is immediate from Lemma~\ref{lem:comm cox}.

Taking~$k=n$ in~\eqref{eq:Cox power Psi}, we obtain
$\Psi_n(T_{w_\circ^{[1,n]}})=(\Cx1{(n+1)})^{n+1}T_{n+1}^n=
T_{w_\circ^{[1,n+1]}}T_{n+1}^n$, which completes the proof of Lemma~\ref{lem:img longest}.
\end{proof}
By induction hypothesis, the assignments $\wh T_1\mapsto
T_{w_\circ^{[1,2m-1]}}$, $\wh T_2\mapsto T_{w_\circ^{[m+1,n]}}$ define a homomorphism
$\Br^+(B_2)\to\Br^+(B_n)$. Taking its composition 
with~$\Psi_n$, we obtain a homomorphism~$\tilde\Psi_n:\Br^+(B_2)\to
\Br^+(B_{n+1})$. Let~$\mathsf M=\Br(B_{n+1})$ and
$\Phi=\tilde\Psi_n$. By Lemma~\ref{lem:img longest}, $\Phi(\wh T_1)=
T_{w_\circ^{[1,2m-1]}}$ and~$\Phi(\wh T_2)=T_{w_\circ^{[m+1,n+1]}}T_{n+1}^{n-m}$.
Let~$z_1=1$ and $z_2=T_{m+1}^{m-n}$. Since the~$z_i$, $\in\{1,2\}$ commute with~$\Phi(\wh T_j)$, $j\in\{1,2\}$, $\mathbf z=(z_1,z_2)$ is a decoration of~$\Phi$ by Lemma~\ref{lem:cent decor}. Then~$\Phi_{\mathbf z}$
is the desired homomorphism $\Br^+(B_2)\to\Br(B_{n+1})$.
Since its image is contained in~$\Br^+(B_{n+1})$, the assertion follows.
\end{proof}
To prove part~\ref{thm:Hom B2.B2An}, let~$N=4$, $\mathsf M=\Br^+(A_n)$, $r=\lceil \frac12 n\rceil$
and let~$\Phi:\Br^+(B_2)\to \mathsf M$ be the composition of the
homomorphism $\Br^+(B_2)\to\Br^+(B_{r})$
from Proposition~\ref{prop:basic Hom B2Bn} with
the standard unfolding $\Br^+(B_r)\to \Br^+(A_n)$ from~\eqref{eq:unfold Bn A2n-1} or~\eqref{eq:unfold Bn A2n} depending on the parity of~$n$. By Theorem~\ref{thm:adm finite class}, it follows that
$\Phi(\wh T_1)=T_{w_\circ^{[1,2m-1]\cup\tilde\sigma([1,2m-1])}}$ and~$\Phi(\wh T_2)=T_{w_\circ^{[m+1,n-m]}}$, where~$\tilde\sigma$ is  the diagram automorphism of~$\Br^+(A_n)$. In particular,
$\tilde\sigma(L)=n+1-L$ for any~$L\subset[1,n]$.
Note that in~$\Br(A_n)$ we have
$\Phi(\wh T_2)^{-1} x \Phi(\wh T_2)=\tilde\sigma(x)$
for any~$x\in\Br^+_{[m+1,n-m]}(M)$.
Let~$\sigma_1$
be the diagram automorphism of~$\Br^+_{[1,2m-1]}(A_n)$.
Then for any~$x\in \Br^+_{[1,2m-1]}(M)$, $\tilde x\in \Br^+_{[n+2-2m,n]}(M)$ we have 
$\Phi(\wh T_1)^{-1} x\Phi(\wh T_1)=
\sigma_1(x)$ and~$\Phi(\wh T_1)^{-1} \tilde x\Phi(\wh T_1)=
\tilde\sigma(\sigma_1(\tilde\sigma(\tilde x)))$.

Let~$z_{1,1}^{(1)}=T_{w_\circ^K}$ and~$z_{2,2}^{(1)}=
T_{w_\circ^{J\cup\tilde\sigma(J)}}$. 
Using~\eqref{eq:inv decoration}, we obtain
\begin{align*}
&z_{1,2}^{(2)}=\Phi(\wh T_2)^{-1}T_{w_\circ^K}
\Phi(\wh T_2)=T_{w_\circ^{\tilde\sigma(K)}},\\
&z_{1,2}^{(3)}=\Phi(\wh T_1)^{-1}T_{w_\circ^{\tilde\sigma(K)}}\Phi(\wh T_1)=
T_{w_\circ^{\tilde\sigma(K)}}=z_{1,2}^{(2)},\\
&z_{1,2}^{(4)}=\Phi(\wh T_2)^{-1}T_{w_\circ^{\tilde\sigma(K)}}\Phi(\wh T_2)
=T_{w_\circ^K}=z_{1,2}^{(1)},\\
&z_{2,1}^{(2)}=\Phi(\wh T_1)^{-1}T_{w_\circ^{J\cup\tilde\sigma(J)}}
\Phi(\wh T_1)=T_{w_\circ^{\sigma_1(J)\cup\tilde\sigma(\sigma_1(J))}},\\
&z_{2,1}^{(3)}=\Phi(\wh T_2)^{-1} T_{w_\circ^{\sigma_1(J)\cup\tilde\sigma(\sigma_1(J))}}
\Phi(\wh T_2)
=z_{2,1}^{(2)},\\
&z_{2,1}^{(4)}=\Phi(\wh T_1)^{-1}T_{w_\circ^{\sigma_1(J)\cup\tilde\sigma(\sigma_1(J)}}\Phi(\wh T_1)=
T_{w_\circ^{J\cup \tilde\sigma(J)}}=z_{2,1}^{(1)}.
\end{align*}
Then 
\begin{align*}
&z_{1,2}^{(4)}z_{2,1}^{(3)}z_{1,2}^{(2)}z_{2,1}^{(1)}
=z_{1,2}^{(1)}z_{2,1}^{(2)}z_{1,2}^{(2)}z_{2,1}^{(1)}
=T_{w_\circ^K}T_{w_\circ^{\sigma_1(J)\cup\tilde\sigma(\sigma_1(J))}}T_{w_\circ^{\tilde\sigma(K)}}T_{w_\circ^{J\cup\tilde\sigma(J)}}\\
&\qquad=T_{w_\circ^K}T_{w_\circ^{\tilde\sigma(K)}}T_{w_\circ^{\sigma_1(J)\cup\tilde\sigma(\sigma_1(J))}}T_{w_\circ^{J\cup\tilde\sigma(J)}},\\
&z_{2,1}^{(4)}z_{1,2}^{(3)}z_{2,1}^{(2)}z_{1,2}^{(1)}
=z_{2,1}^{(1)}z_{1,2}^{(2)}z_{2,1}^{(2)}z_{1,2}^{(1)}
=T_{w_\circ^{J\cup\tilde\sigma(J)}}T_{w_\circ^{\tilde\sigma(K)}}T_{w_\circ^{\sigma_1(J)\cup\tilde\sigma(\sigma_1(J))}}T_{w_\circ^K}
\\
&\qquad=T_{w_\circ^{\tilde\sigma(K)}}T_{w_\circ^K}T_{w_\circ^{J\cup\tilde\sigma(J)}}T_{w_\circ^{\sigma_1(J)\cup\tilde\sigma(\sigma_1(J))}}
=T_{w_\circ^{\tilde\sigma(K)}}T_{w_\circ^K}
T_{w_\circ^{\sigma_1(J)\cup\tilde\sigma(\sigma_1(J))}}T_{w_\circ^{J\cup\tilde\sigma(J)}}
\end{align*}
since~$\tilde\sigma(K),K\subset[2m+1,n-2m]$ while
$J\cup\tilde\sigma(J)\subset[1,m-1]\cup[n+2-m,n]$
and~$\sigma_1(J)\cup\tilde\sigma(\sigma_1(J))
\subset[m+1,2m-1]\cup [n+2-2m,n-m]$. Finally,
since~$K$ and~$n+1-K=\tilde\sigma(K)$
are weakly orthogonal, 
$T_{w_\circ^K}$ and~$T_{w_\circ^{\tilde\sigma(K)}}$ commute by Lemma~\ref{lem:weakly orthogonal}
and so the condition~\ref{thm:decoration sufficient.2} 
of Theorem~\ref{thm:decoration sufficient} holds. 
Then~$\mathbf z=(z_{1,1}^{(1)},z_{2,2}^{(1)})
=(T_{w_\circ^K},T_{w_\circ^{J\cup\tilde\sigma(J)}})$
is a decoration of~$\Phi$ and~$\Phi_{\mathbf z}$
is the desired homomorphism.

Part~\ref{thm:Hom B2.B2} follows from part~\ref{thm:Hom B2.B2An} by taking~$K=n+1-K$ and using Theorem~\ref{thm:adm finite class}.

To prove part~\ref{thm:Hom B2.B2Dn}, note first that 
the composition of the homomorphism from Proposition~\ref{prop:basic Hom B2Bn} with the 
standard unfolding~\eqref{eq:unfold Bn Dn+1} yields
a standard homomorphism~$\Phi:\Br^+(B_2)\to 
\Br^+(D_{n+1})$ satisfying $\Phi(\wh T_1)=T_{w_\circ^{[1,2m-1]}}$ and~$\Phi(\wh T_2)=
T_{w_\circ^{[m+1,n+1]}}$. Let~$\mathsf M=\Br(D_{n+1})$.
Then for any~$x\in\Br^+_{[m+1,n+1]}(M)$, 
$\Phi(\wh T_2)^{-1}x\Phi(\wh T_2)=\tau^{n-m+1}(x)$.
Let~$\sigma_1$ be as before
and let~$z_{1,1}^{(1)}=T_{w_\circ^K}$, $z_{2,2}^{(1)}=
T_{w_\circ^J}$. Then by~\eqref{eq:inv decoration}
\begin{alignat*}{2}
&z_{1,2}^{(2)}=\Phi(\wh T_2)^{-1}T_{w_\circ^K}\Phi(\wh T_2)=
T_{w_\circ^{\tau^{n-m+1}(K)}},&\quad &z_{2,1}^{(2)}=\Phi(\wh T_1)^{-1}T_{w_\circ^J}\Phi(\wh T_1)=T_{w_\circ^{2m-J}},\\
&z_{1,2}^{(3)}=\Phi(\wh T_1)^{-1}z_{1,2}^{(2)}\Phi(\wh T_1)=z_{1,2}^{(2)},&&z_{2,1}^{(3)}=\Phi(\wh T_2)^{-1}z_{2,1}^{(2)}\Phi(\wh T_2)=T_{w_\circ^{2m-J}},\\
&z_{1,2}^{(4)}=\Phi(\wh T_2)^{-1}z_{1,2}^{(2)}\Phi(\wh T_2)=z_{1,2}^{(1)},&
&z_{2,1}^{(4)}=\Phi(\wh T_1)^{-1}T_{w_\circ^{2m-J}}
\Phi(\wh T_1)=T_{w_\circ^J}=z_{2,1}^{(1)}
\end{alignat*}
and so, since~$J\subset[1,m-1]$, $2m-J\subset[m+1,2m-1]$ and~$K,\tau(K)\subset[2m+1,n+1]$,
\begin{align*}
&z_{1,2}^{(4)}z_{2,1}^{(3)}z_{1,2}^{(2)}z_{2,1}^{(1)}
=T_{w_\circ^K}T_{w_\circ^{2m-J}}z_{1,2}^{(2)}T_{w_\circ^J}
=T_{w_\circ^K}z_{1,2}^{(2)}T_{w_\circ^J}T_{w_\circ^{2m-J}},\\
&z_{2,1}^{(4)}z_{1,2}^{(3)}z_{2,1}^{(2)}z_{1,2}^{(1)}
=T_{w_\circ^J}z_{1,2}^{(2)}T_{w_\circ^{2m-J}}T_{w_\circ^K}
=z_{1,2}^{(2)}T_{w_\circ^K}z_{1,2}^{(2)}T_{w_\circ^J}T_{w_\circ^{2m-J}}.
\end{align*}
If~$n-m$ is odd then~$z_{1,2}^{(2)}=T_{w_\circ^K}$. Otherwise, as $K$ 
and~$\tau(K)$ are weakly orthogonal, 
$z_{1,2}^{(2)}$ commutes with~$T_{w_\circ^K}$ by Lemma~\ref{lem:weakly orthogonal}. In either case, the condition~\ref{thm:decoration sufficient.2} of Theorem~\ref{thm:decoration sufficient} is satisfied, whence~$\mathbf z=(z_1,z_2)=(T_{w_\circ^K},T_{w_\circ^J})$
is a decoration of~$\Phi$ and~$\Phi_{\mathbf z}$ yields the homomorphism in part~\ref{thm:Hom B2.B2Dn}.
\end{proof}
\begin{remark}\label{rem:non-std B2 An}
Composing the homomorphism from Theorem~\partref{thm:Hom B2.B2Dn} with the folding homomorphism~$\Br^+(D_{n+1})\to \Br^+(A_n)$  we obtain a non-standard homomorphism~$\Br^+(B_2)\to 
\Br^+(A_n)$ given by $$\wh T_1\mapsto T_{w_\circ^{[1,2m-1]}}T_{w_\circ^{K\setminus K'}}T_{w_\circ^{K'\setminus\{n+1\}}}^2,
\qquad \wh T_2\mapsto T_{w_\circ^J}
T_{w_\circ^{[m+1,n]}}^2 $$ 
where~$K'$ is the maximal interval~$[i,n+1]$ contained in~$K$ (see Lemma~\ref{lem:Dn+1 An w0}). In particular,
for~$J=K=\emptyset$ we obtain the homomorphism
satisfying $\wh T_1\mapsto T_{w_\circ^{[1,2m-1]}}$,
$\wh T_2\mapsto T_{w_\circ^{[m+1,n]}}^2$. It is very
tempting to factor it as a composition of the 
Tits homomorphism~$\Br^+(B_2)\to \Br^+(A_2)$, $\wh T_i\mapsto T_i^{i}$, $i\in \{1,2\}$ and a 
homomorphism~$\Br^+(A_2)\to \Br^+(A_n)$, $T_1\mapsto 
T_{w_\circ^{[1,2m-1]}}$, $T_2\mapsto T_{w_\circ^{[m+1,n]}}$. Alas, the latter assignments define a homomorphism if and only if~$n=3m-1$ (see Theorem~\ref{thm:Hom A2}).
\end{remark}
\begin{theorem}\label{thm:B2 A2n-1 spec}
For all~$n\ge 2$, $0\le k\le n-2$, $J=k+1-J\subset[1,k]$ and~$K=3n+k+1-K\subset[n+k+2,2n-1]$, 
the assignments $\wh T_1\mapsto T_{w_\circ^{[1,n+k]\cup K}}$, $\wh T_2\mapsto 
T_{w_\circ^{[k+2,2n-1]\cup J}}$
define a homomorphism~$\Br^+(B_2)\to\Br^+_{2n}$.
\end{theorem}
\begin{proof}
First, we prove the assertion for~$J=K=\emptyset$. 
By Lemma~\ref{lem:I2m iff cnd}, it is equivalent to proving that~$(T_{w_\circ^{[1,n+k]}}T_{w_\circ^{[k+2,2n-1]}})^2$ is~${}^{op}$-invariant. Abbreviate $X_{n,k}=T_{w_{[1,k];[1,n+k]}}T_{w_{[k+2,n-1];[k+2,2n-1]}}\in\Br^+_{2n}$. Then 
in~$\Br_{2n}$
\begin{align*}
X_{n,k}&=T_{w_\circ^{[1,k]}}^{-1}T_{w_\circ^{[1,n+k]}}T_{w_\circ^{[k+2,n-1]}}^{-1}T_{w_\circ^{[k+2,2n-1]}}
=(T_{w_\circ^{[1,k]}}T_{w_\circ^{[k+2,n-1]}})^{-1}T_{w_\circ^{[1,n+k]}}T_{w_\circ^{[k+2,2n-1]}}\\
&=T_{w_\circ^{[1,n+k]}}T_{w_\circ^{[k+2,2n-1]}}(T_{w_\circ^{[n+1,n+k]}}T_{w_\circ^{[n+k+2,2n-1]}})^{-1},
\end{align*}
whence 
\begin{equation}\label{eq:Xnk }
(T_{w_\circ^{[1,n+k]}}T_{w_\circ^{[k+2,2n-1]}})^2=T_{w_\circ^{[1,k]}}T_{w_\circ^{[k+2,n-1]}}X_{n,k}^2  T_{w_\circ^{[n+1,n+k]}}T_{w_\circ^{[n+k+2,2n-1]}}.   
\end{equation}
\begin{proposition}\label{prop:factor Tw0^2}
For all~$0\le k\le n-2$, $X_{n,k}^2=T_{w_\circ^{[1,2n-1]}}^2$.
\end{proposition}
\begin{proof}
Note first that, since~$\ell(w_{[a,b];[a,b+c]})=\binom{b-a+c+2}2-\binom{b-a+2}2=\frac12 c(c+2(b-a)+3)$,
$\ell(X_{n,k})=\frac12n(n+2k+1)+\frac12n(3n-2k-3)=
n(2n-1)=\ell(T_{w_\circ^{[1,2n-1]}})$. Thus, $\ell(X_{n,k}^2)=\ell(T_{w_\circ^{[1,2n-1]}}^2)$. 
Since~$T_{w_\circ^{[1,2n-1]}}^2$
generates the center of~$\Br^+_{2n}$ by Proposition~\partref{prop:fund elts BrSa.d}, it remains to prove that~$X_{n,k}^2$ commutes with the~$T_i$, $i\in[1,2n-1]$.

Suppose that~$i\in[1,2n-1]\setminus\{k+1,n,n+k+1\}$. If~$i\in[1,k]$ then using Proposition~\partref{prop:fund elts BrSa.c}
\begin{align*}
T_i X_{n,k}&=T_i T_{w_\circ^{[1,k]}}^{-1} T_{w_\circ^{[1,n+k]}}T_{w_\circ^{[k+2,n-1]}}^{-1} T_{w_\circ^{[k+2,2n-1]}}=T_{w_\circ^{[1,k]}}^{-1} T_{k+1-i} T_{w_\circ^{[1,n+k]}}T_{w_\circ^{[k+2,n-1]}}^{-1} T_{w_\circ^{[k+2,2n-1]}}\\
&=T_{w_{[1,k];[1,n+k]}} T_{n+i} T_{w_\circ^{[k+2,n-1]}}^{-1} T_{w_\circ^{[k+2,2n-1]}}=T_{w_{[1,k];[1,n+k]}}  T_{w_\circ^{[k+2,n-1]}}^{-1} T_{n+i} T_{w_\circ^{[k+2,2n-1]}}\\
&=X_{n,k} T_{n+k+1-i}
\end{align*}
with~$n+k+1-i\in[n+1,n+k]$. If~$i\in[n+1,n+k]$ then
\begin{align*}
T_i X_{n,k}&=T_i T_{w_\circ^{[1,k]}}^{-1} T_{w_\circ^{[1,n+k]}}T_{w_\circ^{[k+2,n-1]}}^{-1} T_{w_\circ^{[k+2,2n-1]}}\\
&=T_{w_\circ^{[1,k]}}^{-1} T_{w_\circ^{[1,n+k]}}T_{n+k+1-i} T_{w_\circ^{[k+2,n-1]}}^{-1} T_{w_\circ^{[k+2,2n-1]}}=X_{n,k}T_{n+k+1-i}
\end{align*}
with~$n+k+1-i\in[1,k]$. Thus, for~$i\in [1,k]\cup [n+1,n+k]$, $T_i X_{n,k}=X_{n,k}T_{n+k+1-i}$ whence~$T_i X_{n,k}^2=X_{n,k}^2 T_i$.

If~$i\in[k+2,n-1]$,
\begin{align*}
T_i X_{n,k}&=T_i T_{w_\circ^{[1,k]}}^{-1} T_{w_\circ^{[1,n+k]}}T_{w_\circ^{[k+2,n-1]}}^{-1} T_{w_\circ^{[k+2,2n-1]}}%
=
T_{w_\circ^{[1,k]}}^{-1} T_{w_\circ^{[1,n+k]}}T_{n+k+1-i} T_{w_\circ^{[k+2,n-1]}}^{-1} T_{w_\circ^{[k+2,2n-1]}}\\
&=T_{w_{[1,k];[1,n+k]}} T_{w_\circ^{[k+2,n-1]}}^{-1} T_i T_{w_\circ^{[k+2,2n-1]}}=X_{n,k} T_{2n+k+1-i}
\end{align*}
with~$2n+k+1-i\in [n+k+2,2n-1]$. Similarly, if~$i\in[n+k+2,2n-1]$ then
\begin{align*}
T_i X_{n,k}&=T_{w_{[1,k];[1,n+k]}} T_{w_\circ^{[k+2,n-1]}}^{-1} T_i T_{w_\circ^{[k+2,2n-1]}}=X_{n,k} T_{2n+k+1-i}
\end{align*}
with~$2n+k+1-i\in [k+2,n-1]$. Thus, $T_i X_{n,k}=X_{n,k}T_{2n+k+1-i}$ for~$i\in[k+2,n-1]\cup[n+k+2,2n-1]$ and
so~$T_iX_{n,k}^2=X_{n,k}^2 T_i$.

It remains to prove that~$X_{n,k}^2$ commutes with the $T_i$ for~$i\in\{k+1,n,n+k+1\}$. We need the following
\begin{lemma}\label{lem:move across}
Let~$M$, $i,j\in I$ be a Coxeter matrix and suppose that~$M_{[i,j]}$ is of type~$A$. Then in~$\Br_{[i,j]}(M)$
\begin{subequations}
\begin{alignat}{2}
&T_i T_{w_\circ^{[i+1,j]}}=T_{w_\circ^{[i+1,j]}}\Cx ij\Cx i{(j-1)}{}^{-1}=T_{w_\circ^{[i+1,j]}}\Cx{(i+1)}j{}^{-1}\Cx ij,&\qquad\label{eq:move across i [i+1 j]}\\
&T_{w_\circ^{[i+1,j]}}T_i=\Cxr i{(j-1)}{}^{-1}\Cxr ijT_{w_\circ^{[i+1,j]}}=\Cxr ij\Cxr{(i+1)}j{}^{-1}T_{w_\circ^{[i+1,j]}}\label{eq:move across [i+1 j] i},\\
&T_j T_{w_\circ^{[i,j-1]}}=T_{w_\circ^{[i,j-1]}}\Cxr ij\Cxr{(i+1)}{j}{}^{-1}=T_{w_\circ^{[i,j-1]}}\Cxr{i}{(j-1)}{}^{-1}\Cxr ij,&\label{eq:move across j [i j-1]}\\
&T_{w_\circ^{[i,j-1]}}T_j=\Cx{(i+1)}{j}{}^{-1}\Cx ijT_{w_\circ^{[i,j-1]}}=\Cx ij\Cx{i}{(j-1)}{}^{-1}T_{w_\circ^{[i,j-1]}},\label{eq:move across [i j-1] j}\\
&T_i^{-1} T_{w_\circ^{[i+1,j]}}=T_{w_\circ^{[i+1,j]}}\Cx i{(j-1)}\Cx ij{}^{-1}=T_{w_\circ^{[i+1,j]}}\Cx ij{}^{-1} \Cx{(i+1)}j,\label{eq:move across -i [i+1 j]}\\
&T_{w_\circ^{[i+1,j]}}T_i^{-1}=\Cxr ij{}^{-1}\Cxr i{(j-1)}T_{w_\circ^{[i+1,j]}}=\Cxr{(i+1)}j\Cxr ij{}^{-1}T_{w_\circ^{[i+1,j]}} ,\label{eq:move across [i+1 j] -i}\\
&T_j^{-1} T_{w_\circ^{[i,j-1]}}=T_{w_\circ^{[i,j-1]}}\Cxr{(i+1)}{j}\Cxr ij{}^{-1}=T_{w_\circ^{[i,j-1]}}\Cxr ij{}^{-1} \Cxr{i}{(j-1)},\label{eq:move across -j [i j-1]}\\
&T_{w_\circ^{[i,j-1]}}T_j^{-1}=\Cx ij{}^{-1}\Cx{(i+1)}{j}T_{w_\circ^{[i,j-1]}}=\Cx{i}{(j-1)}\Cx ij{}^{-1}T_{w_\circ^{[i,j-1]}} ,\label{eq:move across [i j-1] -j}
\end{alignat}
\end{subequations}
\end{lemma}
\begin{proof}
It suffices to prove~\eqref{eq:move across i [i+1 j]}. The remaining identities follow by applying~${}^{op}$ or the diagram automorphism of~$\Br^+_{[i,j]}(M)$ or taking inverses.
Note that
$$
T_{w_\circ^{[i,j]}}=T_{w_\circ^{[i+1,j]}}\Cx ij=\Cxr ij T_{w_\circ^{[i+1,j]}}.
$$
Then
\begin{equation*}
T_i T_{w_\circ^{[i+1,j]}}=T_i T_{w_\circ^{[i,j]}}\Cx ij{}^{-1}=T_{w_\circ^{[i,j]}}T_j\Cx ij{}^{-1}=T_{w_\circ^{[i+1,j]}}\Cx ij\Cx i{(j-1)}{}^{-1}.
\end{equation*}
To prove the second equality in~\eqref{eq:move across i [i+1 j]}, note that it is equivalent to $$\Cx{(i+1)}j \Cx ij=\Cx ij\Cx i{(j-1)}$$ which in turn is immediate from Lemma~\ref{lem:comm cox}.
\end{proof}
Our aim is to show that for~$i\in\{k+1,n,n+k+1\}$,
$T_i^{-1}X_{n,k}=X_{n,k}U_{i,n,k}$ while $X_{n,k}T_i=U_{i,n,k}^{-1} X_{n,k}$ for some~$U_{i,n,k}\in\Br_{2n}$. Then~$X_i^{-1}X_{n,k}^2 X_i=X_{n,k}^2$ in~$\Br_{2n}$. 

For~$i=k+1$ we obtain, using the inverse of~\eqref{eq:move across [i j-1] j}
\begin{align*}
T_i^{-1}&X_{n,k}=T_{k+1}^{-1} T_{w_\circ^{[1,k]}}^{-1} T_{w_\circ^{[1,n+k]}} T_{w_{[k+2,n-1];[k+2,2n-1]}}\\
&=T_{w_\circ^{[1,k]}}^{-1} \Cx{1}{k}\Cx 1{(k+1)}{}^{-1}T_{w_\circ^{[1,n+k]}} T_{w_{[k+2,n-1];[k+2,2n-1]}}\qquad\\
&=T_{w_\circ^{[1,k]}}^{-1} T_{w_\circ^{[1,n+k]}}\Cxr{(n+1)}{(n+k)}\Cxr n{(n+k)}{}^{-1} T_{w_{[k+2,n-1];[k+2,2n-1]}}\\
&=T_{w_{[1,k];[1,n+k]}}\Cxr{(n+1)}{(n+k)}T_n^{-1}  T_{w_\circ^{[k+2,n-1]}}^{-1}\Cxr{(n+1)}{(n+k)}{}^{-1}T_{w_\circ^{[k+2,2n-1]}}\\
&=T_{w_{[1,k];[1,n+k]}}\Cxr{(n+1)}{(n+k)}T_{w_\circ^{[k+2,n-1]}}^{-1}\Cx{(k+2)}{(n-1)}\Cx{(k+2)}n{}^{-1}\Cxr{(n+1)}{(n+k)}{}^{-1}T_{w_\circ^{[k+2,2n-1]}}\\
&=T_{w_{[1,k];[1,n+k]}}T_{w_\circ^{[k+2,n-1]}}^{-1}\Cxr{(n+1)}{(n+k)}\Cx{(k+2)}{(n-1)}\Cx{(k+2)}n{}^{-1}\Cxr{(n+1)}{(n+k)}{}^{-1}T_{w_\circ^{[k+2,2n-1]}}\\
&=X_{n,k}\Cxr{(n+k+2)}{(2n-1)}\Cx{(n+1)}{(n+k)}\Cxr{(n+k+1)}{(2n-1)}{}^{-1}\Cx{(n+1)}{(n+k)}{}^{-1}=
X_{n,k} U_{k+1,n,k},
\end{align*}
where~$U_{k+1,n,k}=\Cxr{(n+k+2)}{(2n-1)}\Cx{(n+1)}{(n+k)}\Cx{(n+1)}{(n+k+1)}{}^{-1}\Cxr{(n+k+2)}{(2n-1)}{}^{-1}$.

On the other hand,
\begin{align*}
X_{n,k}&T_{k+1}=T_{w_{[1,k];[1,n+k]}}T_{w_\circ^{[k+2,2n-1]}} T_{w_\circ^{[n+k+2,2n-1]}}^{-1} T_{k+1}
=T_{w_{[1,k];[1,n+k]}}T_{w_\circ^{[k+2,2n-1]}}T_{k+1} T_{w_\circ^{[n+k+2,2n-1]}}^{-1}\\
&=T_{w_{[1,k];[1,n+k]}}\Cxr{(k+1)}{(2n-1)}\Cxr{(k+2)}{(2n-1)}{}^{-1} T_{w_{[k+2,n-1];[k+2,2n-1]}} \\
&=T_{w_\circ^{[1,k]}}^{-1} T_{w_\circ^{[1,n+k]}} \Cxr{(n+k+1)}{(2n-1)}\Cxr{(k+1)}{(n+k)}\Cxr{(k+2)}{(2n-1)}{}^{-1}  T_{w_{[k+2,n-1];[k+2,2n-1]}}\\
&=\Cxr{(n+k+2)}{(2n-1)} T_{w_\circ^{[1,k]}}^{-1} T_{w_\circ^{[1,n+k]}}T_{n+k+1} \Cxr{(k+1)}{(n+k)}\Cxr{(k+2)}{(2n-1)}{}^{-1}  T_{w_{[k+2,n-1];[k+2,2n-1]}}\\
&=\Cxr{(n+k+2)}{(2n-1)} T_{w_\circ^{[1,k]}}^{-1} \Cx1{(n+k+1)}\Cx{1}{(n+k)}{}^{-1} T_{w_\circ^{[1,n+k]}} \\
&\qquad\qquad\qquad \Cxr{(k+1)}{(n+k)}\Cxr{(k+2)}{(2n-1)}{}^{-1}  T_{w_{[k+2,n-1];[k+2,2n-1]}}\\
&=\Cxr{(n+k+2)}{(2n-1)} T_{w_\circ^{[1,k]}}^{-1} \Cx1{(n+k+1)}\Cx{(n+1)}{(n+k)}{}^{-1} T_{w_\circ^{[1,n+k]}}\\
&\qquad \qquad\qquad \Cxr{(k+2)}{(2n-1)}{}^{-1}  T_{w_{[k+2,n-1];[k+2,2n-1]}},
\end{align*}
where we used~\eqref{eq:move across [i+1 j] i} and~\eqref{eq:move across [i j-1] j}. 
We have by~\eqref{eq:move across [i j-1] -j}
and the inverse of~\eqref{eq:move across -j [i j-1]}
\begin{align*}
T_{w_\circ^{[1,k]}}^{-1} &\Cx1{(n+k+1)}\Cx{(n+1)}{(n+k)}{}^{-1}=\Cxr1k T_{w_\circ^{[1,k]}}^{-1} T_{k+1}\Cx{(k+2)}{(n+k+1)}\Cx{(n+1)}{(n+k)}{}^{-1}\\
&=\Cxr1k \Cxr{1}{k}{}^{-1}\Cxr 1{(k+1)} \Cx{(k+2)}{(n+k+1)} T_{w_\circ^{[1,k]}}^{-1}\Cx{(n+1)}{(n+k)}{}^{-1}\\
&=\Cxr 1{(k+1)} \Cx{(k+2)}{(n+k+1)} \Cx{(n+1)}{(n+k)}{}^{-1} T_{w_\circ^{[1,k]}}^{-1}\\
&=\Cx{(k+1)}{(n+k+1)} \Cx{(n+1)}{(n+k)}{}^{-1} \Cxr1k T_{w_\circ^{[1,k]}}^{-1},
\\
T_{w_\circ^{[1,n+k]}} &\Cxr{(k+2)}{(2n-1)}{}^{-1}=\Cx{1}{(n-1)}{}^{-1} T_{w_\circ^{[1,n+k]}}T_{n+k+1}^{-1} \Cxr{(n+k+2)}{(2n-1)}{}^{-1}\\
&=\Cx{1}{(n-1)}{}^{-1} \Cx{1}{(n+k)}\Cx1{(n+k+1)}{}^{-1}\Cxr{(n+k+2)}{(2n-1)}{}^{-1} T_{w_\circ^{[1,n+k]}}\\
&=\Cx{n}{(n+k)}\Cx1{(n+k+1)}{}^{-1}\Cxr{(n+k+2)}{(2n-1)}{}^{-1} T_{w_\circ^{[1,n+k]}},
\end{align*}
whence
\begin{align*}
&T_{w_\circ^{[1,k]}}^{-1} \Cx1{(n+k+1)}\Cx{(n+1)}{(n+k)}{}^{-1}T_{w_\circ^{[1,n+k]}}\Cxr{(k+2)}{(2n-1)}{}^{-1}\\
&=\Cx{(k+1)}{(n+k+1)} \Cx{(n+1)}{(n+k)}{}^{-1}\Cx n{(n+k)} \Cxr1k T_{w_\circ^{[1,k]}}^{-1}\Cx1{(n+k+1)}{}^{-1}\Cxr{(n+k+2)}{(2n-1)}{}^{-1} T_{w_\circ^{[1,n+k]}}.
\end{align*}
Next, using the inverse~\eqref{eq:move across j [i j-1]}
\begin{align*}
\Cxr1k &T_{w_\circ^{[1,k]}}^{-1}\Cx1{(n+k+1)}{}^{-1}=\Cxr1k\Cx{(k+2)}{(n+k+1)}{}^{-1} T_{w_\circ^{[1,k]}}^{-1}T_{k+1}^{-1}\Cx1k{}^{-1}\\
&=\Cxr1k \Cx{(k+2)}{(n+k+1)}{}^{-1}\Cxr1{(k+1)}{}^{-1}\Cxr{1}{k}\Cxr1k{}^{-1} T_{w_\circ^{[1,k]}}^{-1}\\
&=\Cx{(k+2)}{(n+k+1)}{}^{-1}\Cxr1k\Cxr1{(k+1)}{}^{-1} T_{w_\circ^{[1,k]}}^{-1}=\Cx{(k+1)}{(n+k+1)}{}^{-1} T_{w_\circ^{[1,k]}}^{-1}.
\end{align*}
Putting all the above together we conclude that $$
X_{n,k}T_{k+1}=\Cxr{(n+k+2)}{(2n-1)}\tilde U_{n,k}\Cxr{(n+k+2)}{(2n-1)}{}^{-1} X_{n,k}$$ 
where
$\tilde U_{n,k}=\Cx{(k+1)}{(n+k+1)} \Cx{(n+1)}{(n+k)}{}^{-1}\Cx n{(n+k)}\Cx{(k+1)}{(n+k+1)}{}^{-1}$. 
We have, as in the proof of Lemma~\ref{lem:move across}
$$
\Cx{(n+1)}{(n+k)}{}^{-1} \Cx{n}{(n+k)}=\Cx n{(n+k)}\Cx{n}{(n+k-1)}{}^{-1},
$$
and then by Lemma~\ref{lem:comm cox}
\begin{align*}
\tilde U_{n,k}&=\Cx{(k+1)}{(n+k+1)} \Cx n{(n+k)}\Cx{n}{(n+k-1)}{}^{-1}\Cx{(k+1)}{(n+k+1)}{}^{-1}\\
&=\Cx{(n+1)}{(n+k+1)}\Cx{(n+1)}{(n+k)}{}^{-1}.
\end{align*}
Therefore, $\Cxr{(n+k+2)}{(2n-1)}\tilde U_{n,k}\Cxr{(n+k+2)}{(2n-1)}{}^{-1}=
U_{k+1,n,k}^{-1}$.

Next, for~$i=n$ we obtain, using~\eqref{eq:move across -i [i+1 j]} and the inverse of~\eqref{eq:move across [i+1 j] i}
\begin{align*}
T_i^{-1} &X_{n,k}=T_{w_{[1,k];[1,n+k]}} T_{k+1}^{-1} T_{w_{[k+2,n-1];[k+2,2n-1]}}\\
&=T_{w_{[1,k];[1,n+k]}} T_{w_\circ^{[k+2,n-1]}}^{-1}\Cxr{(k+2)}{(n-1)}\Cxr{(k+1)}{(n-1)}{}^{-1}T_{w_\circ^{[k+2,2n-1]}}\quad \\
&=T_{w_{[1,k];[1,n+k]}} T_{w_\circ^{[k+2,n-1]}}^{-1}\Cxr{(k+2)}{(n-1)}T_{k+1}^{-1}T_{w_\circ^{[k+2,2n-1]}} \Cx{(n+k+2)}{(2n-1)}{}^{-1}\\
&=T_{w_{[1,k];[1,n+k]}} T_{w_\circ^{[k+2,n-1]}}^{-1}\Cxr{(k+2)}{(n-1)}T_{w_\circ^{[k+2,2n-1]}}\\
&\qquad\qquad\qquad\Cx{(k+1)}{(2n-1)}{}^{-1} \Cx{(k+2)}{(2n-1)}\Cx{(n+k+2)}{(2n-1)}{}^{-1}%
=X_{n,k}U_{n,n,k},
\end{align*}
where~$U_{n,n,k}=\Cx{(k+1)}{(n+k+1)}{}^{-1} \Cx{(k+2)}{(n+k+1)}$. On the other hand,
using~\eqref{eq:move across [i j-1] j}
and the inverse of~\eqref{eq:move across -j [i j-1]}, we obtain
\begin{align*}
X_{n,k}&T_n=T_{w_{[1,k];[1,n+k]}}T_{w_\circ^{[k+2,2n-1]}}T_{w_\circ^{[k+2+n,2n-1]}}^{-1} T_n =
T_{w_{[1,k];[1,n+k]}}T_{n+k+1} T_{w_{[k+2,n-1];[k+2,2n-1]}}\\
&=T_{w_\circ^{[1,k]}}^{-1}\Cx 1{(n+k+1)}\Cx1{(n+k)}{}^{-1}T_{w_\circ^{[1,n+k]}}T_{w_{[k+2,n-1];[k+2,2n-1]}} \\
&=\Cxr1k T_{w_\circ^{[1,k]}}^{-1}\Cx{(k+1)}{(n+k+1)}\Cx1{(n+k)}{}^{-1}T_{w_\circ^{[1,n+k]}}T_{w_{[k+2,n-1];[k+2,2n-1]}}\\
&=\Cxr 1{(k+1)}\Cx{(k+2)}{(n+k+1)} T_{w_\circ^{[1,k]}}^{-1}\Cx1{(n+k)}{}^{-1}T_{w_\circ^{[1,n+k]}}T_{w_{[k+2,n-1];[k+2,2n-1]}}\\
&=\Cxr 1{(k+1)}\Cx{(k+2)}{(n+k+1)}\Cx{(k+2)}{(n+k)}{}^{-1} T_{w_\circ^{[1,k]}}^{-1}\Cx1{(k+1)}{}^{-1}T_{w_\circ^{[1,n+k]}}T_{w_{[k+2,n-1];[k+2,2n-1]}}\\
&=\Cxr 1{(k+1)}\Cx{(k+2)}{(n+k+1)}\Cx{(k+2)}{(n+k)}{}^{-1}\Cxr1{(k+1)}{}^{-1}X_{n,k}\\
&=\Cx{(k+1)}{(n+k+1)}\Cx{(k+1)}{(n+k)}{}^{-1}X_{n,k}\\
&
=\Cx{(k+2)}{(n+k+1)}{}^{-1}\Cx{(k+1)}{(n+k+1)}X_{n,k}=U_{n,n,k}^{-1} X_{n,k}.
\end{align*}
Finally, for~$i=n+k+1$ we have, by~\eqref{eq:move across -j [i j-1]},
\begin{align*}
T_i^{-1}&X_{n,k}=T_{w_\circ^{[1,k]}}^{-1} T_{n+k+1}^{-1} T_{w_\circ^{[1,n+k]}} T_{w_{[k+2,n-1];[k+2,2n-1]}}\\
&=T_{w_{[1,k];[1,n+k]}}\Cxr1{(n+k+1)}{}^{-1} \Cxr{1}{(n+k)}T_{w_{[k+2,n-1];[k+2,2n-1]}}.
\end{align*}
Furthermore, since by~\eqref{eq:move across i [i+1 j]} and the inverse of~\eqref{eq:move across [i+1 j] -i},
\begin{align*}
\Cxr{1}{(n+k)}&T_{w_{[k+2,n-1];[k+2,2n-1]}}=\Cxr{(k+1)}{(n+k)} T_{w_\circ^{[k+2,2n-1]}} T_{w_\circ^{[n+k+2,2n-1]}}^{-1} \Cxr1k\\
&=T_{w_\circ^{[k+2,2n-1]}}\Cx{(n+1)}{(2n-1)}\Cx{(k+2)}{(2n-1)}{}^{-1}\Cx{(k+1)}{(2n-1)} T_{w_\circ^{[n+k+2,2n-1]}}^{-1} \Cxr1k\quad \\
&=T_{w_\circ^{[k+2,2n-1]}}\Cx{(k+2)}{n}{}^{-1}\Cx{(k+1)}{(n+k+1)} T_{w_\circ^{[n+k+2,2n-1]}}^{-1}\Cxr{(n+k+2)}{(2n-1)} \Cxr1k\\
&=T_{w_\circ^{[k+2,2n-1]}}\Cx{(k+2)}{n}{}^{-1}\Cx{(k+1)}{(n+k)}T_{w_\circ^{[n+k+2,2n-1]}}^{-1}\Cxr{(n+k+1)}{(2n-1)} \Cxr1k\\
&=T_{w_{[k+2,n-1];[k+2,2n-1]}}\Cx{(k+2)}{n}{}^{-1}\Cx{(k+1)}{(n+k)}\Cxr{(n+k+1)}{(2n-1)} \Cxr1k
\end{align*}
while~\eqref{eq:move across -i [i+1 j]}
and the inverse of~\eqref{eq:move across [i+1 j] i} yield
\begin{align*}
&\Cxr1{(n+k+1)}{}^{-1}T_{w_{[k+2,n-1];[k+2,2n-1]}}=\Cxr1{(n+k+1)}{}^{-1}T_{w_\circ^{[k+2,2n-1]}}T_{w_\circ^{[n+k+2,2n-1]}}^{-1}
\\
&=\Cxr1{(k+1)}{}^{-1}T_{w_\circ^{[k+2,2n-1]}}\Cx n{(2n-1)}{}^{-1}T_{w_\circ^{[n+k+2,2n-1]}}^{-1}\\
&=T_{w_\circ^{[k+2,2n-1]}}\Cxr1{k}{}^{-1}\Cx{(k+1)}{(2n-1)}{}^{-1} \Cx{(k+2)}{(n-1)}T_{w_\circ^{[n+k+2,2n-1]}}^{-1}\\
&=T_{w_\circ^{[k+2,2n-1]}}\Cxr1{k}{}^{-1}\Cx{(n+k+2)}{(2n-1)}{}^{-1}T_{n+k+1}^{-1}T_{w_\circ^{[n+k+2,2n-1]}}^{-1}\Cx{(k+1)}{(n+k)}{}^{-1}\Cx{(k+2)}{(n-1)}\\
&=T_{w_\circ^{[k+2,2n-1]}}\Cxr1{k}{}^{-1}T_{w_\circ^{[n+k+2,2n-1]}}^{-1}\Cxr{(n+k+1)}{(2n-1)}{}^{-1}\Cx{(k+1)}{(n+k)}{}^{-1}\Cx{(k+2)}{(n-1)}\\
&=T_{w_{[k+2,n-1];[k+2,2n-1]}}\Cxr1{k}{}^{-1}\Cxr{(n+k+1)}{(2n-1)}{}^{-1}\Cx{(k+1)}{(n+k)}{}^{-1}\Cx{(k+2)}{(n-1)}.
\end{align*}
It follows that $T_{n+k+1}^{-1}X_{n,k}=X_{n,k}U_{n+k+1,n,k}$
where
\begin{align*}
U_{n+k+1,n,k}
&=\Cxr1{k}{}^{-1}\Cxr{(n+k+1)}{(2n-1)}{}^{-1}\Cx{(k+1)}{(n+k)}{}^{-1}\Cx{(k+2)}{(n-1)}\times\\
&\qquad\qquad \Cx{(k+2)}{n}{}^{-1}\Cx{(k+1)}{(n+k)}\Cxr{(n+k+1)}{(2n-1)} \Cxr1k\\
&=\Cxr1{k}{}^{-1}\Cx{(k+1)}{(n-2)}\Cx{(k+1)}{(n-1)}{}^{-1}\Cxr1k\\
&=\Cxr1{k}{}^{-1}\Cx{(k+1)}{(n-1)}{}^{-1}\Cxr1k\Cx{(k+2)}{(n-1)}
\end{align*}
 by Lemma~\ref{lem:comm cox}. On the other hand, by the inverse of~\eqref{eq:move across -j [i j-1]}
\begin{align*}
X_{n,k}&T_{n+k+1}=T_{w_{[1,k];[1,n+k]}}T_{w_\circ^{[k+2,n-1]}}^{-1} T_n T_{w_\circ^{[k+2,2n-1]}}\\
&=T_{w_{[1,k];[1,n+k]}}\Cxr{(k+2)}{(n-1)}{}^{-1}\Cxr{(k+2)}n T_{w_{[k+2,n-1];[k+2,2n-1]}}\\
&=\Cx{(k+2)}{(n-1)}{}^{-1} T_{w_\circ^{[1,k]}}^{-1} \Cx{(k+1)}{(n-1)} T_{w_\circ^{[1,n+k]}} T_{w_{[k+2,n-1];[k+2,2n-1]}}\\
&=\Cx{(k+2)}{(n-1)}{}^{-1}\Cxr{1}{k}{}^{-1}\Cxr1{(k+1)}\Cx{(k+2)}{(n-1)} X_{n,k}\\
&=\Cx{(k+2)}{(n-1)}{}^{-1}\Cxr{1}{k}{}^{-1}\Cx{(k+1)}{(n-1)}\Cxr1{k} X_{n,k}=U_{n+k+1,n,k}^{-1}X_{n,k}.
\end{align*}
Thus, $T_i^{-1}X_{n,k}^2T_i=X_{n,k}^2$ for $i\in\{k+1,n,n+k+1\}$ which completes the proof of Proposition~\ref{prop:factor Tw0^2}.
\end{proof}
\begin{remark}
While~$\ell(X_{n,k})=\ell(w_\circ^{[1,2n-1]})$, the~$X_{n,k}$ are not square free and hence are not equal to~$T_{w_\circ^{[1,2n-1]}}$. For example,
$X_{2,0}=T_1T_2T_1T_2T_3T_2=T_2T_1^2T_3T_2$.
\end{remark}
By~\eqref{eq:Xnk } and Proposition~\ref{prop:factor Tw0^2},
\begin{align}
(T_{w_\circ^{[1,n+k]}}T_{w_\circ^{[k+2,2n-1]}})^2&=T_{w_\circ^{[1,k]}}T_{w_\circ^{[k+2,n-1]}}T_{w_\circ^{[1,2n-1]}}^2  T_{w_\circ^{[n+1,n+k]}}T_{w_\circ^{[n+k+2,2n-1]}}
\nonumber\\
&=T_{w_\circ^{[1,2n-1]}}^2 T_{w_\circ^{[1,2n-1]\setminus\{k+1,n,n+k+1\}}},\label{eq: str hom main}
\end{align}
which is manifestly~${}^{op}$ invariant, being the product of two commuting ${}^{op}$-invariant elements of~$\Br^+_{2n}$.
\begin{corollary}\label{cor:strange homs}
For all~$n\ge 2$, $0\le k\le n-2$ the assignments 
$
\wh T_1\mapsto T_{w_{[1,k];[1,n+k]}}$, $\wh T_2\mapsto T_{w_{[k+2,n-1];[k+2,2n-1]}}
$ (respectively, $
\wh T_1\mapsto T_{w_{[n+1,n+k];[1,n+k]}}$, $\wh T_2\mapsto T_{w_{[n+k+2,2n-1];[k+2,2n-1]}}
$)
define square free homomorphisms $\Br^+(B_2)\to\Br^+_{2n}$ 
which are neither Coxeter nor Hecke type. 
\end{corollary}
\begin{proof}
The first assignments define a homomorphism by Proposition~\ref{prop:factor Tw0^2} and Lemma~\ref{lem:cradicals}. The second is obtained from the first by applying~${}^{op}$.
\end{proof}

Now let~$J=k+1-J\subset[1,k]$, $K=3n+k+1-K\subset [n+k+2,2n-1]$. Set~$z_{1,1}^{(1)}=T_{w_\circ^K}$, 
$z_{2,2}^{(1)}=T_{w_\circ^J}$. Then, using~\eqref{eq:inv decoration}, set
\begin{alignat*}{2}
&z_{1,2}^{(2)}=T_{w_\circ^{[k+2,2n-1]}}^{-1}z_{1,2}^{(1)}T_{w_\circ^{[k+2,2n-1]}}=T_{w_\circ^{2n+k+1-K}},&\quad&z_{2,1}^{(2)}=T_{w_\circ^{[1,n+k]}}^{-1}z_{2,1}^{(1)}T_{w_\circ^{[1,n+k]}}=T_{w_\circ^{n+k+1-J}},\\
&z_{1,2}^{(3)}=T_{w_\circ^{[1,n+k]}}^{-1}z_{1,2}^{(2)}T_{w_\circ^{[1,n+k]}}=
T_{w_\circ^{-n+K}},&&
z_{2,1}^{(3)}=T_{w_\circ^{[k+2,2n-1]}}^{-1}z_{2,1}^{(2)}T_{w_\circ^{[k+2,2n-1]}}=T_{w_\circ^{n+J}},\\
&z_{1,2}^{(4)}=T_{w_\circ^{[k+2,2n-1]}}^{-1}z_{2,1}^{(3)}T_{w_\circ^{[k+2,2n-1]}}
=T_{w_\circ^{3n+k+1-K}},&&
z_{2,1}^{(4)}=T_{w_\circ^{[1,n+k]}}^{-1}z_{2,1}^{(3)}T_{w_\circ^{[1,n+k]}}
=T_{w_\circ^{k+1-J}}.
\end{alignat*}
Since~$J=k+1-J$ and~$K=3n+k+1-K$, we get
\begin{align*}
z_{1,2}^{(4)}z_{2,1}^{(3)}z_{1,2}^{(2)}z_{2,1}^{(1)}&=T_{w_\circ^K}T_{w_\circ^{n+J}}
T_{w_\circ^{2n+k+1-K}}T_{w_\circ^J}
\\
&=T_{w_\circ^K} T_{w_\circ^{-n+K}} T_{w_\circ^{n+J}} T_{w_\circ^J}
=T_{w_\circ^J}T_{w_\circ^{-n+K}}T_{w_\circ^{n+k+1-J}}T_{w_\circ^K}
=z_{2,1}^{(4)}z_{1,2}^{(3)}z_{2,1}^{(2)}z_{1,2}^{(1)}.
\end{align*}
Thus, $\mathbf z=(T_{w_\circ^K},T_{w_\circ^J})$
is a decoration of the basic 
homomorphism~$\Phi:\
\Br^+(B_2)\to \Br^+_{2n}$, $\Phi(\wh T_1)=T_{w_\circ^{[1,n+k]}}$,
$\Phi(\wh T_2)=T_{w_\circ^{[k+2,2n-1]}}$,
and $\Phi_{\mathbf z}$ is the desired homomorphism. This completes 
the proof of Theorem~\ref{thm:B2 A2n-1 spec}.
\end{proof}
\subsection{Combinatorics of standard homomorphisms}\label{subs:comb std}
It is obvious that the total number of distinct homomorphisms~$\Br^+_3\to\Br^+_{3m}$ described in Theorem~\ref{thm:Hom A2} is~$2^m$, as they are parametrized, up to the diagram automorphism of~$\Br^+_3$,
by subsets of~$[1,m-1]$. 

Enumerating homomorphisms from~$\Br^+(B_2)$ requires more effort and yields some interesting sequences.
\begin{theorem}\label{thm:combinatoric std}
\begin{enmalph}
\item \label{thm:combinatoric std.a}
The number of  homomorphisms~$\Br^+(B_2)\to\Br^+(A_{n}) $ described in Theorem~\partref{thm:Hom B2.B2An} is 
equal to
\begin{equation}\label{eq:H n}
H_n:=\tfrac12 \boldsymbol w_{\lfloor \frac12(n+3)\rfloor}-\overline{\lfloor \tfrac12(n+1)\rfloor}2^{\frac12\lfloor \frac12(n-1)\rfloor},\qquad n\ge 1
\end{equation}
where~$\boldsymbol w_0=1$, $\boldsymbol w_1=4$, $\boldsymbol w_{r+1}=2(\boldsymbol w_r+\boldsymbol w_{r-1})$, $r\ge 1$ and
\begin{equation}\label{eq:w_r seq}
\boldsymbol w_r=\tfrac12((1+\sqrt3)^{r+1}+(1-\sqrt3)^{r+1})=\sum_{k\ge 0}\binom{r+1}{2k}3^k,\qquad r\ge 0.
\end{equation}

\item\label{thm:combinatoric std.b}
The number of homomorphisms
$\Br^+(B_2)\to \Br^+(B_n)$, $n\ge 2$
described in Theorem \partref{thm:Hom B2.B2} is
$3\cdot 2^{n}-2^{\lceil \frac n2\rceil}$.

\item\label{thm:combinatoric std.c}
The number of homomorphisms
$\Br^+(B_2)\to \Br^+(D_{n+1})$, $n\ge 3$
described in Theorem~\partref{thm:Hom B2.B2Dn} is
$$
\begin{cases}
\tfrac13(34\cdot 2^{n-1}-(5-\overline{\frac12(n-1)})2^{\frac12(n+1)}),&\bar n=1\\
\vphantom{\dfrac13}\tfrac13(29\cdot 2^{n-1}-(4+\overline{\frac12 n})2^{\frac12 n}),&\bar n=0.
\end{cases}
$$

\item \label{thm:combinatoric std.e}
The number of homomorphisms~$\Br^+(B_2)\to \Br^+(A_{2n-1})$ 
described 
in Theorem~\ref{thm:B2 A2n-1 spec} is equal to
\begin{align*}
\begin{cases}
2^{\frac n2}(\frac32n-2),&\bar n=0,\\
2^{\frac12(n+1)}(n-1),&\bar n=1.
\end{cases}
\end{align*}
\end{enmalph}
In all cases, the sequence grows exponentially as a function of the rang of the codomain.
\end{theorem}
\begin{proof}
For~$M\in\Cox I$ fixed, we will write~$J\perp K$ (respectively,
$J\perp_w K$) for $J,K\subset I$ if~$J$ and~$K$ are 
orthogonal (respectively, weakly orthogonal). 
We need the following
\begin{lemma}\label{lem:weakly orthogonal comb}
Let~$M=A_r$ and define
$$
\mathcal W_r=\{ (K_1,K_2)\in\mathscr P([1,r])\times \mathscr P([1,r])\,:\,K_1\perp_w K_2\},\quad r\ge 0.
$$
Then~$|\mathcal W_r|=\boldsymbol w_r$ for all~$r\in\ZZ_{\ge 0}$.
\end{lemma}
\begin{proof}
Clearly, $\mathcal W_0=\{(\emptyset,\emptyset)\}$, 
$\mathcal W_1=
\mathscr P(\{1\})\times \mathscr P(\{1\})$ and so~$|\mathcal W_0|=1=w_0$ and~$|\mathcal W_1|=4=w_1$. 
Note that~$\mathcal W_{r+1}$, $r\ge 1$
is the disjoint union of the following sets:
\begin{align*}
\mathcal W_{r+1}^{(0)}&=\{ (K_1,K_2)\in\mathcal W_{r+1}\,:\, K_1,K_2\subset [1,r]\},\\
\mathcal W_{r+1}^{(1)}&=\{ (K_1,K_2)\in\mathcal W_{r+1}\,:\, r+1\in K_1,\,K_2\subset [1,r]\},\\
\mathcal W_{r+1}^{(2)}&=\{ (K_1,K_2)\in\mathcal W_{r+1}\,:\, (K_2,K_1)\in\mathcal W_{r+1}^{(1)}\},\\
\mathcal W_{r+1}^{(3)}&=\{ (K_1,K_2)\in\mathcal W_r\,:\,
r+1\in K_1\cap K_2\}.
\end{align*}
Clearly, $\mathcal W^{(0)}_{r+1}=\mathcal W_r$ and~$
|\mathcal W^{(2)}_{r+1}|=|\mathcal W^{(1)}_{r+1}|$.
We claim that 
\begin{align}\label{eq: W r+1 1}
\mathcal W_{r+1}^{(1)}&=\{([1,r+1],\emptyset)\}
\sqcup \bigsqcup_{2\le a\le r+1} \{ (K_1\cup [a,r+1],K_2)\,:\, (K_1,K_2)\in \mathcal W_{a-2}\},\\
\mathcal W_{r+1}^{(3)}&=\{([1,r+1],[1,r+1])\}\nonumber
\\&\qquad \qquad \sqcup \bigsqcup_{2\le a\le r+1} \{ (K_1\cup [a,r+1],K_2\cup [a,r+1])\,:\, (K_1,K_2)\in \mathcal W_{a-2}\},
\label{eq: W r+1 3}
\end{align}
whence $|\mathcal W_{r+1}^{(i)}|=1+\sum_{0\le j\le r-1} \boldsymbol w_j$, $1\le i\le 3$ and so
\begin{equation}\label{eq:w r long recursion}
\boldsymbol w_{r+1}=\boldsymbol w_r+3\sum_{0\le j\le r-1} \boldsymbol w_j+3.
\end{equation}
To prove the claim, note that the right hand side of~\eqref{eq: W r+1 1} (respectively, \eqref{eq: W r+1 3}) is obviously contained in~$\mathcal W^{(1)}_{r+1}$
(respectively, in~$\mathcal W^{(3)}_{r+1}$). To prove
the opposite inclusion in~\eqref{eq: W r+1 1},
let~$(K_1,K_2)\in\mathcal W_{r+1}^{(1)}$ and
write~$K_1=K'_1\cup [a,r+1]$ where~$K'_1\subset [1,a-2]$.
If~$a=1$, that is $K_1=[1,r+1]$ then~$K_2=\emptyset$. Indeed, otherwise, $i:=\max K_2\le r$ as~$r+1\notin K_2$
and so $i+1\in K_1\setminus K_2$, which therefore is not orthogonal to~$K_1$ as~$i\in K_1$. 

Suppose that~$2\le a\le r+1$. 
Since~$r+1\in K_1\setminus K_2$ and $(K_1\setminus K_2)\perp K_2$, we conclude that~$r\notin K_2$.
Furthermore, if~$a-1\in K_2$ then~$a-1\in K_2\setminus K_1$
which thus cannot be orthogonal to~$K_1$ as~$a\in K_2$. Thus, $K_2=K'_2\cup K''_2$ with~$K'_2\subset [1,a-2]$
and~$K''_2\subset [a,r]$. Clearly, $K'_1$, $K'_2$
are weakly orthogonal, and so are~$K''_2$ and~$[a,r+1]$. Then the same argument as was used for~$a=1$ shows that~$K''_2=\emptyset$. 

To prove the opposite inclusion in~\eqref{eq: W r+1 3},
let~$(K_1,K_2)\in\mathcal W_{r+1}^{(3)}$ and write 
$K_i=K'_i\cup [a_i,r+1]$, $1\le a_i\le r+1$, $K'_i\subset [1,a_i-2]$, $i\in\{1,2\}$. It remains to observe that~$a_1=a_2$. If say~$a_1<a_2$ then $a_2-1\in K_1\setminus K_2$ which is a contradiction as~$(K_1\setminus K_2)\perp K_2$ and~$a_2\in K_2$.

We now use~\eqref{eq:w r long recursion} to prove that~$\boldsymbol w_{r+1}=2(\boldsymbol w_r+\boldsymbol w_{r-1})$ for all~$r\ge 1$. Indeed, for~$r=1$ we have $2(\boldsymbol w_1+\boldsymbol w_0)=10=\boldsymbol w_1+3 \boldsymbol w_0+3$. For the inductive step, we have 
$$
\boldsymbol w_{r+1}-\boldsymbol w_r=\boldsymbol w_r+3\sum_{0\le j\le r-1} \boldsymbol w_j-\boldsymbol w_{r-1}-3\sum_{0\le j\le r-2}\boldsymbol w_r=\boldsymbol w_r+2 \boldsymbol w_{r-1}.
$$
The first equality in~\eqref{eq:w_r seq} 
follows from the recursion by elementary linear algebra while the second is immediate from the first.
\end{proof}

\begin{lemma}\label{lem:K sigma K w.o.}
Let~$M=A_n$ and let~$\mathcal U_n=\{ K\in [1,n]\,:\,
K\perp_w (n+1-K)\}$. Then~$|\mathcal U_n|=\boldsymbol u_{\lfloor\frac12(n+1)\rfloor}$
where~$\boldsymbol u_0=1$, $\boldsymbol u_1=2$, $\boldsymbol u_{r+1}=2(\boldsymbol u_r+\boldsymbol u_{r-1})$, $r\ge 1$ and 
\begin{equation}\label{eq:u_r sequence}
\boldsymbol u_r=
\frac{(1+\sqrt 3)^{r+1}-
(1-\sqrt 3)^{r+1}}{2\sqrt 3}
=\sum_{k\ge 0}\binom{r+1}{2k+1}3^k,\qquad r\ge 0.
\end{equation}
\end{lemma}
\begin{proof}
Let~$k=\lfloor \frac12(n+1)\rfloor$, $\mathcal U_n'=
\{ K\in\mathcal U_n\,:\, k\in K\}$ and $\mathcal U_n''
=\mathcal U_n\setminus \mathcal U'_n$. Our first observation is that, for~$n$ even, $\mathcal U_n''
=\{ K\in\mathcal U_n\,:\, k,k+1\notin K\}$
and~$\mathcal U'_n=\{ K\in\mathcal U_n\,:\, \{k,k+1\}\subset K\}$. Indeed, let~$K\in\mathcal U_n''$ and 
suppose that~$k+1\in K$.  Then~$k\in n+1-K$ and~$k+1\notin n+1-K$, whence~$k\in (n+1-K)\setminus K$ which
is a contradiction as~$(n+1-K)\setminus K\perp K$.
Likewise, if~$K\in\mathcal U'_n$ and~$k+1\notin K$ then
$k+1\in (n+1-K)\setminus K$ which thus cannot 
be orthogonal to~$K$ as~$k\in K$.
Thus, if~$K\in \mathcal U''_n$ then~$K=K_1\cup (n+1-K_2)$ where~$K_1,K_2\subset [1,k-1]$. Since~$K$ and~$n+1-K$
are weakly orthogonal, it follows that~$K_1$ and~$K_2$
are weakly orthogonal. Thus, $|\mathcal U''_n|=w_{k-1}$.

Suppose now that~$K\in \mathcal U'_n$ and let~$[a,b]\subset K$
be its connected component containing~$k$. By the above,
$a\le k\le n+1-k\le b$.
Then $K=K_1\cup K'\cup K_2$ with~$K_1\subset [1,a-2]$ and~$K_2\subset [b+2,n]$. We claim that~$b=n+1-a$. Indeed,
if $n+1-b<a$ then~$n+1-b\in n+1-K$. Since~$a\le k$
we have~$a-1\in [n+1-b,n+1-a]\subset (n+1-K)$.  
Then~$a-1\in (n+1-K)\setminus K$ which is a contradiction
as~$(n+1-K)\perp K$ and~$a\in K$. Similarly, if~$n+1-b>a$
then~$b+1\in[n+1-b,n+1-a]$ and so~$b+1\in (n+1-K)\setminus K$ which is also a contradiction. 

Thus, $K=K_1\cup [a,n+1-a]\cup K_2$ with~$K_2\subset n+1-[1,a-2]$ and so we can write~$K_2=n+1-K'_1$ where~$K'_1\in [1,a-2]$. It is now immediate that~$K\perp_w (n+1-K)$
if and only if $K_1\perp K'_1$. Thus,
$$
\mathcal U'_n=\{[1,n]\} \sqcup \bigsqcup_{2\le a\le k}
\{ K_1\cup [a,n+1-a]\cup (n+1-K'_1)\,:\, (K_1,K'_1)\in\mathcal W_{a-2}\}
$$
in the notation of Lemma~\ref{lem:weakly orthogonal comb}.
Therefore, $|\mathcal U'_n|=1+\sum_{0\le a\le k-2} \boldsymbol w_a$
and so
\begin{equation*}
|\mathcal U_n|=|\mathcal U'_n|+|\mathcal U''_n|=
1+\sum_{0\le a\le k-1}\boldsymbol w_a=1+\sum_{0\le a\le \lfloor \frac12 (n-1)\rfloor} \boldsymbol w_a.    
\end{equation*}
In particular, $|\mathcal U_n|=\boldsymbol u_{\lfloor\frac12(n+1)\rfloor}$ where 
\begin{equation}\label{eq:w_r->u_r}
\boldsymbol u_r=1+\sum_{0\le i\le r-1} \boldsymbol w_i,\quad r\ge 0.
\end{equation}
Then $\boldsymbol{u}_0=1$,
$\boldsymbol{u}_1=2$ and~$\boldsymbol{u}_2=
1+\boldsymbol{w}_0+\boldsymbol{w}_1=6=
2(\boldsymbol{u}_0+\boldsymbol{u}_1)$.
Suppose that~$\boldsymbol u_r=2(\boldsymbol{u}_{r-1}+\boldsymbol{u}_{r-2})$ for some~$r\ge 2$. Since~$\boldsymbol{u}_{r+1}-\boldsymbol{u}_r=
\boldsymbol w_r=2(\boldsymbol{w}_{r-1}+\boldsymbol w_{r-2})=
2\boldsymbol u_r-2\boldsymbol u_{r-2}
$ by~\eqref{eq:w_r->u_r}, it follows that~$\boldsymbol u_{r+1}=3\boldsymbol{ u}_r-2\boldsymbol u_{r-2}=2\boldsymbol u_r+
\boldsymbol u_r-2\boldsymbol{u}_{r-2}
=2(\boldsymbol u_r+\boldsymbol{u}_{r-1})$.
The equalities in~\eqref{eq:u_r sequence}
are now routine.
\end{proof}
We now have all necessary ingredients to 
finish the proof of part~\ref{thm:combinatoric std.a}. 
By Theorem~\partref{thm:Hom B2.B2An}
$$
H_n=2\sum_{0\le m\le \lfloor\frac14(n+1)\rfloor}
|\{ (J,K)\,:\, J\subset [1,m-1],\, K\subset [2m+1,n-2m],\,
K\perp_w n+1-K\}|,
$$
where the first factor accounts for the diagram
automorphism of~$\Br^+(B_2)$. By Lemma~\ref{lem:K sigma K w.o.}
$$
H_n=2\boldsymbol u_{\lfloor\frac12(n+1)\rfloor}+\sum_{1\le m\le 
\lfloor\frac14(n+1)\rfloor} 2^m \boldsymbol u_{\lfloor\frac12(n+1)\rfloor-2m}=
2\boldsymbol u_{\lfloor\frac12(n+1)\rfloor}+\sum_{1\le m\le 
\lfloor\frac12\lfloor\frac12(n+1)\rfloor\rfloor} 2^m \boldsymbol u_{\lfloor\frac12(n+1)\rfloor-2m}.
$$
Thus~$H_n=\boldsymbol h_{\lfloor \frac12(n+1)\rfloor}$ where
\begin{equation}\label{eq:hr seq}
\boldsymbol h_r=2 \boldsymbol u_r+\sum_{1\le m\le \lfloor \frac12 r\rfloor} 2^m \boldsymbol u_{r-2m},\qquad r\ge 0.
\end{equation}
We claim that~$\boldsymbol h_r=\boldsymbol u_{r+1}-\boldsymbol u_{r-1}-\bar r\,2^{\frac12(r-1)}$, $r\ge 0$ where~$\boldsymbol u_{-1}=0$. 
Indeed, $\boldsymbol h_0=2\boldsymbol u_0=2=\boldsymbol u_1-\boldsymbol u_{-1}$ and
$\boldsymbol h_1=2\boldsymbol u_1=4=\boldsymbol u_2-\boldsymbol u_0-1$.
Furthermore,
\begin{align*}
\boldsymbol h_{r+2}&=2 \boldsymbol u_{r+2}+\sum_{1\le m\le \lfloor\frac12 r\rfloor+1}2^m \boldsymbol u_{r+2-2m}=2 \boldsymbol u_{r+2}+2\boldsymbol u_r+2\sum_{1\le m\le \lfloor\frac 12r\rfloor}
2^m \boldsymbol u_{r-2m}\\
&=2 \boldsymbol u_{r+2}+2 \boldsymbol h_r-2\boldsymbol u_r.
\end{align*}
Thus, if~$\boldsymbol h_r=\boldsymbol u_{r+1}-\boldsymbol u_{r-1}-c_r$ with~$c_{r+2}=2c_r$
then~$\boldsymbol h_{r+2}=
2 \boldsymbol u_{r+2}+2(\boldsymbol u_{r+1}-\boldsymbol u_{r-1})-2c_r-2\boldsymbol u_r= \boldsymbol u_{r+3}-\boldsymbol u_{r+1}
-c_{r+2}$.
Since~$c_r=\bar r\, 2^{\frac12(r-1)}$ clearly satisfies
$c_{r+2}=2 c_r$ and~$\boldsymbol h_r=\boldsymbol u_{r+1}-\boldsymbol u_{r-1}-c_r$, $r\in\{0,1\}$,
it follows that~$\boldsymbol h_r=\boldsymbol u_{r+1}-\boldsymbol u_{r-1}-c_r$ for all~$r\ge 0$.

Using~\eqref{eq:w_r->u_r} we obtain
$$
\boldsymbol h_r=\boldsymbol w_{r}+\boldsymbol w_{r-1}-\overline r\,2^{\frac12 (r-1)}=\tfrac12 \boldsymbol w_{r+1}-\overline r\,2^{\frac12 (r-1)},
$$
which immediately yields~\eqref{eq:H n}. Part~\ref{thm:combinatoric std.a} is proven.

To prove part~\ref{thm:combinatoric std.b}, note 
that this family of homomorphisms is 
parametrized by $(J,K)$ with $J\subset [1,m-1]$
and~$K\subset [2m+1,n]$. Thus, taking into
account the diagram automorphism of~$\Br^+(B_2)$
we conclude that the total number of homomorphisms
$\Br^+(B_2)\to\Br^+(B_n)$ is 
$$
2\cdot 2^{n}+\sum_{1\le m\le \lfloor\frac12n\rfloor} 2^{n-m}=3\cdot 2^n-2^{\lceil\frac12 n\rceil}.
$$

To prove part~\ref{thm:combinatoric std.c}, we need 
the following
\begin{lemma}\label{lem:comb}
Let~$M=D_{n+1}$ and~$a\in[1,n+1]$. Let~$\mathcal T_{n,a}=
\{ K\subset [a,n+1]\,:\, K\perp_w \tau(K)\}$. Then
\begin{equation}\label{eq:card T}
|\mathcal T_{n,a}|=\begin{cases}
3\cdot 2^{n-a},&1\le a\le n-1,\\
4,&a=n,\\
2,&a=n+1.
\end{cases}
\end{equation}
\end{lemma}
\begin{proof}
Clearly, if~$K\subset [a,n+1]$ satisfies ~$K=\tau(K)$ then~$K\in\mathcal T_{n,a}$. Furthermore, $\tau(K)=K$
if either~$K\subset [a,n-1]$ or~$K=K'\cup\{n,n+1\}$
with~$K'\subset [a,n-1]$. Thus, every subset of~$[a,n-1]$
yields precisely two $\tau$-invariant subsets of~$[a,n+1]$
and so
$$
|\{ K\subset [a,n+1]\,:\,\tau(K)=K\}|=\begin{cases}
2^{n+1-a},&1\le a\le n,\\
1,& a=n+1.
\end{cases}
$$
Let~$\mathcal T'_{n,a}=\mathcal T_{n,a}\setminus \{ K
\subset [a,n+1]\,:\, K=\tau(K)\}$. Clearly
$\mathcal T'_{n,n+1}=\{\{n+1\}\}$ and~$\mathcal T'_{n,n}=\{\{n\},\{n+1\}\}$. Let~$1\le a\le n-1$ and let~$K\in\mathcal T'_{n,a}$. Then~$|K\cap \{n,n+1\}|=1$. If say~$n\in K$ then~$n+1\in\tau(K)\setminus K$ and, since~$\tau(K)\setminus K\perp K$, it follows that~$n-1\notin K$.
Thus, $K=K'\cup\{n\}$ with~$K'\subset [a,n-2]$ and so~$\mathcal T'_{n,a}=
\{ K'\cup\{n\},K'\cup\{n+1\}\,:\,
K'\subset [a,n-2]\}$.
Therefore, 
$$
|\mathcal T'_{n,a}|=\begin{cases}1,&a=n+1,\\
2,&a=n,\\
2^{n-a},&1\le a\le n-1.
\end{cases}
$$
The assertion is now immediate.
\end{proof}

The number of homomorphisms from Theorem~\partref{thm:Hom B2.B2Dn} is therefore equal to
\begin{equation}\label{eq: Dn+1 n odd}
\sum_{0\le k\le \lfloor
\frac12(l-1)\rfloor} 2^{2k+\delta_{k,0}} |\mathscr P([4k+1,2l+2])|\\
+\sum_{1\le k\le \lfloor\frac12 l\rfloor+1} 2^{2k-1} |
\mathcal T_{2l+1,4k-1}|
\end{equation}
if~$n=2l+1$, $l\ge 1$ and 
\begin{equation}\label{eq: Dn+1 n even}
\sum_{0\le k\le \lfloor 
\frac12l\rfloor} 2^{2k+\delta_{k,0}} 
|\mathcal T_{2l,4k+1}|
+\sum_{1\le k\le \lfloor\frac12(l+1)\rfloor} 2^{2k-1}   
|\mathscr P([4k-1,2l+1])|
\end{equation}
if~$n=2l$, $l\ge 2$.
By~\eqref{eq:card T} the expression~\eqref{eq: Dn+1 n odd}
is equal to
\begin{align*}
\sum_{0\le k\le \lfloor\frac12(l-1)\rfloor}&2^{\delta_{k,0}+2l+2-2k}+3\sum_{1\le k\le \lfloor \frac12 l\rfloor} 2^{2l+1-2k}+4\cdot 2^l
\\&=2^{2 l+3}+\tfrac13 (2^{2 l+2}-2^{l+4-\bar l})+
2^{2l+1}+2^{l+1}(1-\bar l)
=\tfrac13(34\cdot 2^{2l}-(5-\bar l)2^{l+1}).
\end{align*}
Similarly, \eqref{eq: Dn+1 n even} becomes
\begin{align*}
3\cdot 2^{2l}+3 \sum_{1\le k\le \lfloor\frac12(l-1)}
2^{2 l - 2 k - 1}+2^{l+1}+\sum_{1\le k\le \lfloor\frac12l\rfloor}
2^{2l+2-2k}=\tfrac13(29\cdot 2^{2l-1}-2^l(4+\bar l)).
\end{align*}

To prove part~\ref{thm:combinatoric std.e} we need the following
\begin{lemma}\label{lem:no inv subsets}
There are~$2^{\lfloor\frac12 n\rfloor}$
subsets of~$[1,n-1]$ which are invariant 
with respect to the diagram automorphism of~$\Br^+_n$, $n\ge 2$.
\end{lemma}
\begin{proof}
For~$n=2$ the assertion is obvious.
Clearly, $J\subset [1,n-1]$ is invariant
with respect to the diagram automorphism
of~$\Br^+_n$ if and only if~$J=J'\cup J''\cup (n-J')$
where~$J'$, $n-J'$ and~$J''$ are pairwise orthogonal
and~$J''=n-J''$ is connected. If~$J''=\emptyset$ then~$J'\subset [1,\lfloor\frac12 n\rfloor-1]$ and
so there are~$2^{\lfloor \frac12 n\rfloor-1}$ of them. Otherwise, $J''=[i,n-i]$ with~$1\le i\le \lfloor\frac12n\rfloor$ and
then~$J'\subset [1,i-2]$. Therefore, the total number of invariant subsets is
\begin{equation*}
1+\sum_{2\le i\le \lfloor\frac12n\rfloor}2^{i-2}+2^{\lfloor \frac12 n\rfloor-1}=2^{\lfloor\frac12 n\rfloor}.\qedhere
\end{equation*}
\end{proof}
The homomorphisms from Theorem~\ref{thm:B2 A2n-1 spec} being parametrized by pairs~$(J,K)$ where~$J=k+1-J\subset [1,k]$ and~$K=3n+k+1-K\subset [n+k+2,2n-1]$, $0\le k\le n-2$, by Lemma~\ref{lem:no inv subsets} their total number is
equal to
\begin{align*}
2\sum_{0\le k\le n-2} &2^{\lfloor\frac12(k+1)\rfloor+
\lfloor\frac12(n-k-1)\rfloor}.
\end{align*}
Suppose first that~$n=2r$. Then
$\lfloor\frac12(k+1)\rfloor+\lfloor\frac12(n-k-1)\rfloor
=r+\lfloor\frac12(k+1)\rfloor-\lceil\frac12(k+1)\rceil=
r-\overline{k+1}=r-1+\bar k$, whence
$$
2\sum_{0\le k\le n-2} 2^{\lfloor\frac12(k+1)\rfloor+
\lfloor\frac12(n-k-1)\rfloor}=2^{r}
\sum_{0\le k\le 2(r-1)} 2^{\overline k}
=2^{r}(r+2(r-1))=2^r(3r-2).
$$
If~$n=2r+1$ then 
$\lfloor\frac12(k+1)\rfloor+\lfloor\frac12(n-k-1)\rfloor
=\lfloor\frac12(k+1)\rfloor+r-\lceil\frac12 k\rceil=r$,
and so our sum is equal to~$2^{r+2}r$.
\end{proof}
\begin{remark}\label{rem:OEIS}
The sequences~$\boldsymbol w_n$ and~$\boldsymbol u_n$, $n\ge 0$ coincide up to a shift with, respectively, \OEIS{A026150} and~\OEIS{A002605}
and admit a number of combinatorial interpretations (see e.g.~\cites{GK,Leh,BoMo}).
The sequence~$\boldsymbol h_n+\bar n\, 2^{\frac12(n-1)}$, $n\ge 0$ (cf.~\eqref{eq:hr seq}) coincides with~\OEIS{A052945}, up to the first term of the latter.  
The even (respectively, odd) numbered terms in the sequence from part~\ref{thm:combinatoric std.e}
form the sequence~\OEIS{A130129}
(respectively, \OEIS{A058922}).
\end{remark}

\subsection{Sporadic standard homomorphisms}\label{subs:non-disj Hecke}
First we describe all homomorphisms from $\Br^+(A_2)$ and~$\Br^+(B_2)$
to Artin monoids of finite exceptional types.
\begin{proposition}\label{prop:Hom A2|B2 EF}
Let~$m\in \{3,4\}$ and~$M\in\Cox I$ be either~$F_4$ or~$E_n$, $n\in\{6,7,8\}$. Furthermore, let~$J_2=I\setminus\{1\}$ and let~$J_1
\subset I$ be 
as in the following table
$$
\begin{array}{c|c|c}
m&M&J_1\\
\hline
3&E_6&I\setminus\{5\}\\
\hline
4&F_4&I\setminus\{4\}\\
\hline
4&E_7&I\setminus\{5,6\}\\
\hline
4&E_8&I\setminus\{7\}
\end{array}
$$
\begin{enmalph}
   \item\label{prop:Hom A2|B2 EF.a}
The assignments~$\wh T_i\mapsto T_{w_{J_1\cap J_2;J_i}}$, $i\in\{1,2\}$ define a strict parabolic Coxeter type homomorphism $\Phi_{m,M}:\Br^+(I_2(m))\to\Br^+(M)$;
 \item\label{prop:Hom A2|B2 EF.b}
The assignments~$\wh T_1\mapsto T_{w_\circ^{J_1\cup K}}$, $\wh T_2\mapsto T_{w_\circ^{J_2}}$,
where~$K\in\{\emptyset,\{6\}\}$ if~$m=4$ and~$M=E_7$
and~$K=\emptyset$ otherwise, define a standard
homomorphism~$\wh\Phi_{m,M,K}:\Br^+(I_2(m))\to\Br^+(M)$;
\item \label{prop:Hom A2|B2 EF.c}
The only optimal fully supported standard homomorphisms~$\Br^+(I_2(m))\to\Br^+(M)$ where
$m\in\{3,4\}$ and~$M=F_4$ or~$M=E_n$, $n\in\{6,7,8\}$ are, up to the diagram automorphism of~$I_2(m)$, the 
homomorphisms~$\wh \Phi_{m,M,K}$ listed in part~\ref{prop:Hom A2|B2 EF.b}, together with
the composition of~$\wh\Phi_{4,F_4,\emptyset}$ with
the standard
unfolding~\eqref{eq:unfold F4 E6}, and homomorphisms
$\Br^+(B_2)\to \Br^+(E_6)$ given by $\wh T_1\mapsto T_{w_\circ^{J}}$, $\wh T_2\mapsto T_{w_\circ^{[1,6]}}$
where either~$J=[1,6]$ or~$J\not=\tau(J)$ are weakly orthogonal for
$\tau=(1,5)(2,4)$.
\end{enmalph}
\end{proposition}
\begin{proof}
To prove~\ref{prop:Hom A2|B2 EF.a}, it is easy to check, for example using our Python program for calculations in Coxeter groups and Hecke monoids, that the assignments $\wh s_i\mapsto w_{J_1\cap J_2;J_i}$
define a homomorphism $\phi=\phi_{m,M}:W(I_2(m))\to W(M)$ and
that $\ell(\phi(\wh w_\circ))=3\ell(\phi(\wh s_1))=
3\ell(\phi(\wh s_2))$
if~$m=3$ while $\ell(\phi(\wh w_0))=2(\ell(\phi(\wh s_1))+\ell(\phi(\wh s_2)))$ if~$m=4$. Then 
the assignments from part~\ref{prop:Hom A2|B2 EF.a} define a homomorphism $\Br^+(I_2(m))\to\Br^+(M)$
by Lemma~\ref{lem:lifting to Cox-Hecke}. 
We need the following
\begin{lemma}\label{lem:centrality}
For all pairs~$(m,M)$ listed in Proposition~\ref{prop:Hom A2|B2 EF}, $T_{w_\circ^{J_1\cap J_2}}$
commutes with~$T_{w_\circ^{J_i}}$, $i\in\{1,2\}$.
\end{lemma}
\begin{proof}
For~$m=3$ and~$M=E_6$ (respectively, $m=4$ and~$M=F_4$), the $\Br^+_{J_i}(M)$, $i\in\{1,2\}$ 
are of type~$D_4$ (respectively, $B_3$) and so the~$T_{w_\circ^{J_i}}$, $i\in\{1,2\}$ are central
in the corresponding parabolic submonoids by Proposition~\partref{prop:fund elts BrSa.e}. Let~$m=4$ and~$M=E_7$. Then~$\Br^+_{J_2}(M)$ is of type~$D_6$
and so~$T_{w_\circ^{J_2}}$ is central in~$\Br^+_{J_2}(M)$. On the other hand, $\Br^+_{J_1}(M)$ is of type~$D_5$ and so $xT_{w_\circ^{J_1}}=T_{w_\circ^{J_1}}\Sigma_{J_1}(x)$
by Proposition~\partref{prop:fund elts BrSa.e}. Since~$J_1\cap J_2=\{2,3,4,7\}$ is invariant with respect to the diagram automorphism of~$\Br^+_{J_1}(M)$ which corresponds to the transposition $(4,7)$, it follows that~$T_{w_\circ^{J_1}}$ commutes with~$T_{w_\circ^{J_1\cap J_2}}$. Finally, for~$m=4$ and~$M=E_8$, $\Br^+_{J_1}(M)$ is of type~$E_7$ and 
so~$T_{w_\circ^{J_1}}$ is central in~$\Br^+_{J_1}(M)$
by Proposition~\partref{prop:fund elts BrSa.e}. On the other hand, $\Br^+_{J_2}(M)$ is of type~$D_7$, the 
corresponding diagram automorphism being the transposition~$\tau=(2,8)$. Since~$J_1\cap J_2=\{2,3,4,5,8\}$ and hence is $\tau$-invariant,
it follows 
from Proposition~\partref{prop:fund elts BrSa.e} that~$T_{w_\circ^{J_1\cap J_2}}$
commutes with~$T_{w_\circ^{J_2}}$
\end{proof}
It follows from Lemma~\ref{lem:centrality} that the~$w_{J_1\cap J_2;J_i}$, $i\in\{1,2\}$ are products
of commuting involutions and so the~$\Phi_{m,M}$ are of Coxeter type by Proposition~\partref{prop:elem prop Coxeter Hecke.a}.
Since~$\phi_{m,M}(\wh w_\circ)=w_{J_1\cap J_2}$
in all these cases and~$\ell(\Phi_{m,M}(\wh T_{w_\circ}))=\ell(w_{J_1\cap J_2})$
$\Phi_{m,M}(\wh T_{w_\circ})=T_{w_{J_1\cap J_2}}$ 
and so is strongly square free by Lemma~\partref{lem:elem Artin hom.b}. Then
to prove that it is parabolic it suffices to verify
that the~$\phi(\wh s_i\wh w_\circ)$, $i\in \{1,2\}$
are parabolic elements of~$W(M)$. But we have~$\phi(\wh s_i\wh w_\circ)=w_{J_1\cap J_2;J_i}w_{J_1\cap J_2}=
w_\circ^{J_1\cap J_2}w_\circ^{J_i}w_\circ^{J_1\cap J_2}w_\circ^I=w_\circ^{J_i}w_\circ^I=w_{J_i}$,
where we used Lemma~\ref{lem:centrality}.

To prove part~\ref{prop:Hom A2|B2 EF.b}, assume first that~$K=\emptyset$. Note that~$\wh T_{w_\circ^{J_i}}=
T_{w_\circ^{J_1\cap J_2}}T_{w_{J_1\cap J_2;J_i}}=
\Phi_{m,M}(\wh T_i)T_{w_\circ^{J_1\cap J_2}}$
by Lemma~\ref{lem:centrality}. By Lemmata~\ref{lem:centrality} and~\ref{lem:cent decor}, $\mathbf z=(T_{w_\circ^{J_1\cap J_2}},T_{w_\circ^{J_1\cap J_2}})$
is a decoration of~$\Phi_{m,M}$ and~$\wh\Phi_{m,M,\emptyset}=(\Phi_{m,M})_{\mathbf z}$. Finally, if~$m=4$, $M=E_7$
and~$K=\{6\}$, note that~$T_6$ commutes with~$T_{w_\circ^{J_2}}$ which is central in~$\Br^+_{J_2}(M)\cong \Br^+(D_6)$ and 
with~$T_{w_\circ^{J_1}}$ since~$J_1$ and~$K$ are 
orthogonal. Then by Lemma~\ref{lem:cent decor},
$\wh\Phi_{4,E_7,\{6\}}$ is obtained as the decoration
of~$\wh\Phi_{4,E_7,\emptyset}$ by
to~$\mathbf z=(T_6,1)$. 

To prove~\ref{prop:Hom A2|B2 EF.c},
recall (Lemma~\partref{lem:sqf Hecke hom.a}) that a standard homomorphism $\Phi:\Br^+(I_2(m))\to\Br^+(M)$
is uniquely determined by~$J_i:=[\Phi](i)$,
$i\in\{1,2\}$
and induces a homomorphism of Coxeter groups and of Hecke monoids. In addition, for~$m=3$, we must have 
$\ell(w_\circ^{J_1})=\ell(w_\circ^{J_2})$ by Lemma~\partref{lem:sqf Hecke hom.b}. If say~$J_1\subset J_2$ then, since~$\Phi$ is fully supported, we must have~$J_2=I$.
Then, by the optimality of~$\Phi$, either~$J_1=I$, or
$J_1\subsetneq I$ and~$J_1\not=\emptyset$, whence~$m=4$, and~$T_{w_\circ^{J_1}}$ cannot commute with~$T_{w_\circ^I}$. Since~$T_{w_\circ^I}$
is central in~$\Br^+(M)$ for~$M=F_4$, $E_7$ or~$E_8$,
it follows that~$M=E_6$ and~$\tau(J_1)\not=J_1$. Then~$\Phi$
is the decoration with~$\mathbf z=(T_{w_\circ^{J_1}},1)$ of the character 
homomorphism $\Br^+(B_2)\to \Br^+(E_6)$, $\wh T_1\mapsto 1$, $\wh T_2\mapsto T_{w_\circ^{[1,6]}}$,
and so by Theorem~\ref{thm:decoration sufficient} and Corollary~\ref{cor:dec iff}
we conclude that~$T_{w_\circ^{J_1}}$ must commute 
with~$T_{w_\circ^{\tau(J_1)}}$ hence either~$J_1=[1,6]$ or~$J_1\not=\tau(J_1)$ are weakly orthogonal. Otherwise, if $J_1,J_2\not=I$,
the only pairs $J_1$, $J_2$, up to renumbering, satisfying all conditions discussed above are 
the ones listed in Proposition~\ref{prop:Hom A2|B2 EF} together with
$J_1=\{1, 2, 3, 4, 5, 7\}$, $J_2=[2,7]$ for~$m=4$
and~$M=E_7$. Yet one can verify, for example in the reflection representation of the Hecke algebra of~$W(E_7)$  (cf.~\cite{CIK}*{Proposition~9.8}) with, say, $q=17$, that~$(T_{w_\circ^{J_1}}T_{w_\circ^{J_2}})^2\not=
(T_{w_\circ^{J_2}}T_{w_\circ^{J_1}})^2$ in that case.
\end{proof}
\begin{remark}
One can easily verify that there are no fully supported optimal standard homomorphisms $\wh\Phi_{m,H_k}:\Br^+(I_2(m))\to \Br^+(H_k)$,
$m,k\in\{3,4\}$ except $\wh T_i\mapsto T_{w_\circ^I}$,
$i\in \{1,2\}$. Indeed, if say~$\wh T_1\mapsto T_{w_\circ^I}$ then for~$m=3$ Lemma~\partref{lem:elem Artin hom.c} forces the image 
of~$\wh T_2$ to be also~$T_{w_\circ^I}$, while 
for~$m=4$, since~$T_{w_\circ^I}$ is central in~$\Br^+(H_k)$, such a homomorphism will not be optimal unless~$\wh T_2$ is also mapped to~$T_{w_\circ^I}$.
So, we may assume that~$\wh T_i\mapsto T_{w_\circ^{K_i}}$, $i\in\{1,2\}$ with~$K_i\not=[1,k]$.
Let~$\Psi_3:\Br^+(H_3)\to
\Br^+(D_6)$ and~$\Psi_4:\Br^+(H_4)\to \Br^+(E_8)$
be standard unfoldings~\eqref{eq:unfold H3D6} and~\eqref{eq:unfold H4E8}, respectively and let~$J_i=[\Psi_k](K_i)$, $i\in \{1,2\}$. 
Then $\Psi_k\circ\wh\Phi_{m,H_k}$ would 
be a fully supported standard homomorphism
$\Br^+(I_2(m))\to \Br^+(D_6)$ for~$k=3$ (respectively, $\Br^+(I_2(m))\to \Br^+(E_6)$ for $k=4$) with the following
property
\begin{equation}
J_i\cap [\Psi_k](j)\not=\emptyset\implies [\Psi_k](j)\subset J_i,\qquad i\in\{1,2\},\,j\in[1,k].
\label{eq:H34 cnd}
\end{equation}
Yet the unique standard homomorphism
$\wh\Phi_{4,E_8}:\Br^+(B_2)\to \Br^+(E_8)$ described in Proposition~\ref{prop:Hom A2|B2 EF} does not satisfy~\eqref{eq:H34 cnd} since~$J_1=[1,7]\setminus\{7\}$ while $[\Psi_4](1)=\{1,7\}$. 
For~$k=3$, there are~$45$, up to renumbering, pairs  
$J_1,J_2\subsetneq [1,6]$
such that $(w_\circ^{J_1}w_\circ^{J_2})^4=1$,
$(w_\circ^{J_1}\star w_\circ^{J_2})^{\star 2}=(w_\circ^{J_2}\star w_\circ^{J_1})^{\star 2}$.
Yet none of them satisfies~\eqref{eq:H34 cnd}.
\end{remark}

\begin{remark}\label{rem:undercoration}
In our proof of Theorems~\ref{thm:adm I2m}, 
\ref{thm:monomial brd} and Proposition~\ref{prop:Hom A2|B2 EF} we used the following procedure of ``undecoration'' which in rank~2 boils down to the following algorithm. Given a standard homomorphism~$\Phi\in\Hom_{\mathscr A}(I_2(m),M)$, $M\in\Cox I$
which by Lemma~\partref{lem:sqf Hecke hom.a} is uniquely determined by the~$K_i:=[\Phi](i)\subset I$, $i\in\{1,2\}$, we take~$J_i\subset K_i$ to be maximal
such that~$T_{w_\circ^{J_i}}$ commutes with~$T_{w_\circ^{K_j}}$, $\{i,j\}=\{1,2\}$. Quite surprisingly, it so happened in all aforementioned cases $\mathbf z=(z_1,z_2)$
where~$z_i=T_{w_\circ^{J_i}}^{-1}$, $i\in\{1,2\}$
was a decoration of~$\Phi$ regarded as a homomorphism~$\Br^+(I_2(m))\to \Br(M)$. Yet
$\Phi_{\mathbf z}$ turned out to be a homomorphism
of monoids $\Br^+(I_2(m))\to\Br^+(M)$ and, on top of that,
parabolic (Proposition~\ref{prop:Cox hom from TJ}, Theorem~\partref{thm:monomial brd.a}\ref{thm:monomial brd.c}, Proposition~\partref{prop:Hom A2|B2 EF.a}).
We expect that this picks up all ``missing'' parabolic homomorphisms, apart from LCM and light ones whose parabolicity was established in Theorems~\ref{thm:adm finite class} and~\ref{thm:light->parabolic}.
\end{remark}

We now describe some irregular families of non-disjoint standard homomorphisms.
\begin{proposition}\label{prop:exotic homs}
Let $\mathsf M$ be a multiplicative monoid. Suppose that $t_0, t_1, \tau, S \in \mathsf M$ satisfy $t_iS=St_{1-i}$, $t_i \tau = \tau t_i$, $i\in\{0,1\}$, $(\tau S)^{m_1} = (S\tau)^{m_1}$ and $\brd{t_0t_1}{m_2}=\brd{t_1t_0}{m_2}$. Then the assignments $T_1 \mapsto t_1\tau$, $T_2 \mapsto S$ define a homomorphism $\Br^+(I_2(2\operatorname{lcm}(m_1,m_2)) \to \mathsf M$.
\end{proposition}
\begin{proof}
We need the following
\begin{lemma} \label{lem:weakly admissible lemma t1tauS to power m}
\label{lem:weakly admissible lemma St1tau to power m}  
Suppose that $t_0, t_1, \tau, S \in \mathsf{M}$ satisfy $t_iS=St_{1-i}$, $t_i \tau = \tau t_i$, $i\in\{0,1\}$. Then for any~$m\in\ZZ_{>0}$
$$
(t_1 \tau S)^m %
= \brd{  t_1t_0}m (\tau S)^m,
\quad (St_1\tau)^m%
=\brd{t_0t_1}m(S\tau)^m.
$$
\end{lemma}

\begin{proof}
We only prove the first equality, the argument for the second one being similar. The case~$m=1$ is trivial.
Suppose that the identity holds for some~$m\ge 1$. Then
\begin{align*}
(t_1\tau S)^{m+1}&=t_1\tau S\brd{t_1t_0}m(\tau S)^m=t_1\brd{t_0t_1}m (\tau S)^{m+1}
=\brd{t_0t_1}{m+1} (\tau S)^{m+1}.\qedhere
\end{align*}
\end{proof}

Let~$m=\operatorname{lcm}(m_1,m_2)$.
Then $(\tau S)^m=(S\tau)^m$ and
$\brd{t_1t_0}{m}=\brd{t_0t_1 }{m}$
by Lemma~\partref{lem:taut homs.a} and so
$(t_1\tau S)^m=(S t_1\tau)^m$ by
Lemma~\ref{lem:weakly admissible lemma t1tauS to power m}.
\end{proof}

\begin{corollary}
The assignments $\wh T_1 \mapsto T_1 T_3$, $\wh T_2 \mapsto T_{w_0^{[2,n-1]}}$ define a homomorphism $I_2(12) \to \Br^+_{n}$. 
\end{corollary}
\begin{proof}
Let~$\mathsf M=\Br^+_n$, $S=T_{w_\circ^{[2,n-2]}}$, $t_1=T_3$, $t_0=T_{n-2}$ and~$\tau=T_1$. 
Then the relation~$(S\tau)^{m_1}=(\tau S)^{m_1}$
with $m_1=3$ follows from Theorem~\ref{thm:adm I2m} while the remaining relations of Proposition~\ref{prop:exotic homs} are
immediate, with $m_2=2$.
\end{proof}

\begin{proposition}
Let~$M\in\Cox I$ be of finite type, let~$i\in I$ and suppose that $T_i T_{w_\circ^I}=T_{w_\circ^I}T_j$
for some~$j\not=i\in I$. The assignments $\wh T_1\mapsto T_i$, $\wh T_2\mapsto T_{w_\circ^I}$ define an optimal
homomorphism $\Br^+(I_2(2m))\to \Br^+(M)$ if and only if $m=\operatorname{lcm}(2,m_{ij})$.
\end{proposition}
\begin{proof}
Since~$T_j T_{w_\circ^I}=T_{w_\circ^I} T_i$ by Proposition~\partref{prop:fund elts BrSa.c},
using Lemma~\ref{lem:weakly admissible lemma t1tauS to power m} with~$\mathsf M=\Br^+(M)$, $t_1=T_i$, $t_0=T_j$, $S=T_{w_\circ^I}$ and~$\tau=1$
we obtain
$
(T_i T_{w_\circ^I})^m=\brd{T_iT_j}{m} T_{w_\circ^I}^m
$
whence by Proposition~\partref{prop:fund elts BrSa.a}\ref{prop:fund elts BrSa.d}
$$
((T_iT_{w_\circ}^I)^m)^{op}=T_{w_\circ^I}^m\brd{T_jT_i}{m}=\begin{cases}
\brd{T_iT_j}m T_{w_\circ^I}^m,&\bar m=1,\\
\brd{T_jT_i}m T_{w_\circ^I}^m,&\bar m=0.
\end{cases}
$$
Thus, if $m$ is odd, $(T_iT_{w_\circ^I})^m$ is automatically ${}^{op}$-invariant, while for even~$m$, since~$\Br^+(M)$ is cancellative, $(T_iT_{w_\circ^I})^m$ is ${}^{op}$-invariant
if and only if $\brd{T_iT_j}{m}=\brd{T_jT_i}m$, which
by Lemma~\partref{lem:taut homs.b} happens if and only if~$m_{ij}$ divides~$m$. The assertion follows by Lemma~\ref{lem:I2m iff cnd}.
\end{proof}

\begin{proposition}
For any~$n\in\mathbb Z_{\ge 4}$, the assignments
$\wh T_1\mapsto T_1T_n, \wh T_2\mapsto  T_{w_\circ^{[1,n-1]}}$
define a standard homomorphism $\Br^+(I_2(10))\to \Br^+_{n+1}$.
\end{proposition}
\begin{proof} We need the following
 \begin{lemma} \label{lem:TUS5=UST5}
 Let $\mathsf M$ be a multiplicative monoid and suppose that $S,T,U,V\in \mathsf M$ satisfy $(ST)^3=(TS)^3$, $UV=VU$, $US=SU$, $TU=VT$, $TV=UT$ and $SVS=VSV$. Then $(TUS)^5=(UST)^5$.
\end{lemma}
\begin{proof}We have
\begin{align*}(T&US)^5=(VTS)(TSU)^4
=VTSTSTVSUTSUTSU
=VTSTSTVSVUTSTSU\\
&=VTSTSTSVSUTSTSU
=VSTSTSTVSUTSTSU
=VSTSVTSTSUTSTSU\\
&=VSTSVTSTSTVSVTS
=VSTSVTSTSTSVSTS
=VSTSVSTSTSTVSTS\\
&=VSTVSVTSTSTVSTS
=UVSTSVTSVTSTSTS
=UVSTSVTSVSTSTST\\
&=UVSTSVTVSVTSTST
=UVSTSUTUSVTSTST
=UVSVTSTUSVTSTST\\
&=USVSTSTUSVTSTST
=USVSTSTSTVSTVST
=USVTSTSTSVSTVST\\
&=USVTSTSTVSVTVST
=USTUSTSTVSVTVST
=USTUSTSTVSUTUST\\
&=USTUSTUSTSUTUST=(UST)^5.\qedhere
\end{align*}
\end{proof}
Let $\mathsf M=\Br_n^+$, $S=T_n$, $U=T_1$, $V=T_{n-1}$
and $T=T_{w_\circ^{[1,n-1]}}$. All relations in the Lemma
are immediate except $(ST)^3=(TS)^3$. But
for~$J=\{1,n-1,n+1\}$
we have $\tau_1(J)=T_{w_\circ^{[1,n-1]}}=T$, $\tau_0(J)=T_n=S$
and then $(\tau_1(J)\tau_0(J))^3=(\tau_0(J)\tau_1(J))^3$
by Theorem~\ref{thm:adm I2m}.
\end{proof}

\subsection{Some conjectural families of non-disjoint standard homomorphisms}\label{subs:conj families}
In this section, we list several yet conjectural infinite families of standard homomorphisms from Artin monoids
of type~$I_2(N)$ to $\Br^+(A_n)$, $\Br^+(B_n)$ and $\Br^+(D_{n+1})$. So far we have verified these conjectures for $n \le 15$. 
\begin{conjecture}\label{conj:all homs A2 An}
Let~$M\in\Cox I$ be irreducible of finite type and let~$\Phi\in\Hom_{\mathscr A}(A_2,M)$ be fully supported and 
standard with~$[\Phi](1)\not=[\Phi](2)$. Then, up to diagram automorphisms,  either~$M=A_{3m-1}$ for some~$m\ge 1$ and~$\Phi=\Phi_{m,J}$, $J\subset[1,m-1]$ from Proposition~\ref{thm:Hom A2}
or~$M=E_6$ and~$\Phi=\wh \Phi_{3,E_6}$ from Proposition~\partref{prop:Hom A2|B2 EF.b}.
\end{conjecture}
\begin{conjecture}\label{conj:all homs B2 M}
 Homomorphisms from either Theorem~\ref{thm:Hom B2} or~\ref{thm:B2 A2n-1 spec} or Proposition~\partref{prop:Hom A2|B2 EF.b},
exhaust, up to diagram automorphisms, all fully supported standard homomorphisms from $\Br^+(B_2)$ to any Artin monoid of irreducible finite type.
\end{conjecture}

\begin{conjecture}
Let~$1\le b\le a<n-1$ and~$a>1$. The assignments~$\wh T_1\mapsto T_{w_\circ^{[1,a]}}$, $\wh T_2\mapsto T_{w_\circ^{[b,n-1]}}$ define a homomorphism
$\Br^+(I_2(N))\to \Br^+_{n}$ if one of the 
following holds
\begin{itemize}
\item $N=2n/(n+b-a-2)\in 2\ZZ_{>0}$;
\item
$a+b=n$ and~$N=n/(b-1)\in1+2\ZZ_{>0}$;
\item 
$a+1=2(b-1)$, $a+b< n$ and~$N=6$.
\end{itemize}
\end{conjecture}
 
 \begin{conjecture}
For any~$i\in [2,n-2]$, the assignments $\wh T_1 \mapsto T_i T_n$, $\wh T_2 \mapsto T_{w_0^{[1,n-1]}}$, $i \in [2,n-2]$, define a standard homomorphism $\Br^+(I_2(12/d)) \to \Br^+_{n+1}$, where $d=1+\delta_{2i,n}+\delta_{2i,n-1}+\delta_{2i,n+1}$.
\end{conjecture}

We conclude with a list of conjectural families of standard
homomorphisms $\Br^+(I_2(N))\to 
\Br^+(D_{n+1})$. Here we use the following labeling of the Coxeter graph of type~$D_{n+1}$
$$
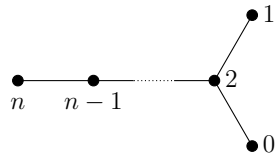
\dynkin[Coxeter,expand labels={n,n-1,2,1,0},label directions={,,right,,},make indefinite edge={2-3}]D5
$$
\begin{conjecture}
\begin{enmalph}
\item 
The assignments
$\wh T_i\mapsto T_{w_\circ^{K_i\setminus(2i+4\ZZ_{\ge 0})}}$, $i\in\{1,2\}$
define a standard homomorphism $\Br^+(I_2(n+1))\to \Br^+(D_{n+1})$ for the following
 $K_1,K_2\subset [0,n]$:
\begin{enumerate}[leftmargin=*,label={$\bullet$}]

\item $K_1 =[1,n]$, $K_2 = \{0\} \cup [2,n-2]$ if $n\in 1+4\ZZ_{>0}$,

\item $K_1=[1,n]$, $K_2 = \{0\} \cup [2,n]$ if $n\in 1+ 2\ZZ_{> 0}$,

\item  $K_1=[1,n-2]$, $K_2 = \{0\} \cup [2,n]$ if $n\in 3+4\ZZ_{\ge 0}$.

\end{enumerate}
\item The assignments $\wh T_i\mapsto T_{w_\circ^{K_i\setminus(2i-1+4\ZZ_{\ge 0})}}$, $i\in\{1,2\}$
define a standard homomorphism $\Br^+(I_2(2n))\to \Br^+(D_{n+1})$ for the following
pairs $K_1,K_2\subset[0,n]$:
\begin{enumerate}[leftmargin=*,label={$\bullet$}]

\item $K_1=[0,n]$,
$K_2= \{1\} \cup [4,n-2]$ if $n\in 4\ZZ_{>0}$,

\item $K_1=[0,n]$, $K_2= \{1\} \cup [4,n]$ if $n\in 2\ZZ_{>1}$,

\item  $K_1=[0,n-2]$, $K_1= \{1\} \cup [4,n]$ if $n\in 2+4\ZZ_{>0}$.

\end{enumerate}
\end{enmalph}

\end{conjecture}

\section*{List of symbols}
\def\bqq{{\setbox0\hbox{$\widehat{U}_q^+$}\setbox2\null\ht2\ht0\dp2\dp0\box2}}
\def\hr#1{\hyperlink{#1}{\pageref*{page:#1}}}

\noindent
{
\scriptsize
\begin{tabular}{p{1.55in}@{\bqq}l@{\hskip.25in}p{1.55in}@{\bqq}l@{\hskip.25in}p{1.55in}@{\bqq}l}
$\bar s$&p.~\hr{bar s}&$[a,b]_2$&p.~\hr{[a,b]2}&$\mathscr P(S)$&p.~\hr{PS}\\
$\ascprod$, $\dscprod$&p.~\hr{ascp}&
$\brd{xy}{m}$&p.~\hr{brd}&$B(x,y)$&p.~\hr{B(x,y)}\\
$\Cox I$&p.~\hr{Cox I}&$\Gamma(M)$&p.~\hr{Gamma(M)}&
$\Br^+(M)$, $\Br(M)$&p.~\hr{Br+(M)} \\
$\ell$&p.~\hr{ell}&${}^{op}$&p.~\hr{op}&$W(M)$&p.~\hr{W(M)}\\
$\pi_M$&p.~\hr{piM}&
$\SQF^+(M)$&p.~\hr{SQF}&
$\Br^+_J(M)$, $W_J(M)$&p.~\hr{Br+J(M)}\\
$\iota_J$&p.~\hr{iotaJ}&$\mathscr F(M)$&p.~\hr{F(M)}&
$\supp$&p.~\hr{supp}\\
$w_\circ^J$&p.~\hr{w0J}&
$\pi^\star_M$&p.~\hr{pi*M}&
$(W(M),\star)$&p.~\hr{HeMon}\\
$\mathscr A$, $\mathscr C$, $\mathscr H$&p.~\hr{A C H}&
$\times$&p.~\hr{times}&
$\downarrow w$, $\uparrow w$&p.~\hr{dar}\\
$D_L(w)$,$D_R(w)$&p.~\hr{DL(w)}&
$\cx ab$, $\cxr ab$, $\Cx ab$, $\Cxr ab$&p.~\hr{cab}&
$D_L(X)$&p.~\hr{I(X)}\\
$h(M)$&p.~\hr{h(M)}&
$w_{J;K}$&p.~\hr{wJ;K}&
$\mP(W_K(M),\star)$, $\mP(\Br^+_K(M))$&p.~\hr{Psbm}\\
$[\phi]$&p.~\hr{[phi]}&
$\mu_M$&p.~\hr{mu M}&
$\Lambda(M',M)$&p.~\hr{LM'M}\\
$\Theta_\xi$&p.~\hr{Theta xi}&
$p_J$&p.~\hr{pJ}&
$[\Phi]$&p.~\hr{[Phi]}\\
$\Xi_{\mathbf X}$&p.~\hr{char hom}&
$\Phi_{\mathbf z}$&p.~\hr{Phi z}&
$\overline\Phi$,
$\overline\Phi_\star$&p.~\hr{barPhi}\\
$\mathscr{AH}$, $\mathscr{AC}$&p.~\hr{AH AC}&
$\mathsf H$, $\mathsf C$&p.~\hr{H C}&
$P_J$&p.~\hr{PJ}\\
$\star_J$&p.~\hr{*J}&
$w_{J_1,\dots,J_k;J}$&p.~\hr{multiparab}&
$M^{\varpi}$&p.~\hr{MII}\\
$\mathbf f_\varpi$&p.~\hr{phi fold}&
$\mathbf F_{\varpi}$&p.~\hr{F wpi}&
$M(\mathbf d)$&p.~\hr{M(d)}\\
$\mathbf T_{\mathbf d}$&p.~\hr{T d}&
$\mathbf T_{i,d}$&p.~\hr{T i,d}&
$\mathscr G$&p.~\hr{Good}\\
$\Br^+_n$, $\Br_n$&p.~\hr{Brn}&
$T_{(i,j)}$&p.~\hr{T(i,j)}&
$T_J$,$\tilde\tau_k(J)$&p.~\hr{TJ}\\
$\tau_k(J)$&p.~\hr{tauk(J)}&
$\Cx ij^{(a)}$, $\Cxr ij^{(a)}$&p.~\hr{Cij(a)}&
$g(J)$&p.~\hr{gJ}\\
$U(J)$&p.~\hr{U(J)}&
$e_i$&p.~\hr{ei}&
$v_{[i,j]}$&p.~\hr{v[i,j]}\\
$u_i$, $w_{[i,j]}^{(a)}$&p.~\hr{ui}&
$\la\cdot|\cdot\ra$&p.~\hr{<.|.>}&
$\beta_\pm(J)$&p.~\hr{betapm}\\
$q_s$&p.~\hr{qs}&
$\Phi^{(m)}_n$,
$\wh \Phi^{(m)}_n$&p.~\hr{Phi(m)n}
\end{tabular}
}

\renewcommand{\PrintDOI}[1]{%
    DOI: \href{https://dx.doi.org/#1}{#1}%
}

\renewcommand{\MR}[1]{\relax}

\end{document}